\documentclass[11pt]{amsart}

\usepackage[usenames,dvipsnames,svgnames,table]{xcolor}
\usepackage[colorlinks=true, pdfstartview=FitV, linkcolor=blue, citecolor=blue, urlcolor=darkblue]{hyperref}

\usepackage{geometry}                
\usepackage{graphicx}
\usepackage{mathrsfs}
\usepackage{amssymb}
\usepackage{amsfonts}
\usepackage{mathrsfs}
\usepackage{epstopdf}
\usepackage{lscape}
\usepackage[utf8]{inputenc}
\usepackage{tikz,caption}
\usepackage{enumitem}
\setlist[itemize]{leftmargin=2em}
\setlist[enumerate]{leftmargin=2em}
\usepackage{booktabs}
\usepackage{multirow}
\usepackage{mathtools}
\usepackage[linesnumbered,ruled]{algorithm2e}

\usepackage{soul}
\usepackage{upgreek}

\definecolor{darkblue}{rgb}{0.0,0,0.7} 
\definecolor{darkred}{rgb}{0.7,0,0} 
\definecolor{darkgreen}{rgb}{0, .6, 0} 

\newcommand{\defncolor}{\color{darkred}}
\newcommand{\defn}[1]{{\defncolor\emph{#1}}} 

\newtheorem{theorem}{Theorem}[section]
\newtheorem{prop}[theorem]{Proposition}
\newtheorem{cor}[theorem]{Corollary}
\newtheorem{lemma}[theorem]{Lemma}

\theoremstyle{definition}
\newtheorem{definition}[theorem]{Definition}
\newtheorem{example}[theorem]{Example}
\newtheorem{remark}[theorem]{Remark}
\numberwithin{equation}{section}

\newcommand{\idiot}[1]{\vspace{5 mm}\par \noindent
\marginpar{\textsc{Note}}
\framebox{\begin{minipage}[c]{0.95 \textwidth}
#1 \end{minipage}}\vspace{5 mm}\par}

\renewcommand{\idiot}[1]{}

\def\CC{{\mathbb C}}
\def\ZZ{{\mathbb Z}}

\def\o{\overline}

\def\ind{{\mathfrak I}}
\def\vtype{{\overrightarrow{\mathsf{type}}}}
\def\XXX{{\mathbf{X}}}
\def\EEE{{\mathbf{E}}}
\def\epart{{\text{\footnotesize$\varnothing$}}}

\newcommand{\cf}{\textit{cf.} }
\newcommand{\ie}{\textit{i.e.}}

\newcommand{\class}{\operatorname{Cl}}

\usepackage[colorinlistoftodos]{todonotes}

\usepackage{tikz}
\usepackage{tikz-cd}
\usetikzlibrary{calc}
\usetikzlibrary{arrows}
\usepackage{tikzsymbols}
\usetikzlibrary{backgrounds}
\tikzstyle{partition-diagram}=[scale=0.4, thick, baseline={(0,-1ex/2)}, 
]
\tikzstyle{idempotent}=[background rectangle/.style={ultra thick, draw=ForestGreen, rounded corners}]
\tikzstyle{vertex} = [shape = circle, minimum size = 3pt, inner sep = 1pt]
\tikzstyle{representative}=[background rectangle/.style={ultra thick, draw=RoyalBlue, rounded corners}]

\tikzstyle{partition-diagram-small}=[scale=0.15, thin, baseline={(0,-1ex/2)}]
\tikzstyle{vertex-small} = [shape = circle, minimum size = 2pt, inner sep = 1pt]

\usepackage{ytableau}
\ytableausetup{smalltableaux, aligntableaux=center, textmode}

\def\unprotectedboldentry#1{\textcolor{Red}{\large{#1}}}
\def\boldentry{\protect\unprotectedboldentry}
\usepackage{tikz}
\usetikzlibrary{calc}
\usepackage{ifthen}
\newcommand{\tikztableauinternal}[1]{
    \def\newtableau{#1}
    \coordinate (x) at (-0.5,0.5);
    \coordinate (y) at (-0.5,0.5);
    \foreach \row in \newtableau {
        \foreach \entry in \row {
            \ifthenelse{\equal{\entry}{X} \OR \equal{\entry}{None}}
               {
                \node (y) at ($(y) + (1,0)$) {};
                \fill[color=gray!30] ($(y)-(0.5,0.5)$) rectangle +(1,1);
                \draw[color=gray, dotted] ($(y)-(0.5,0.5)$) rectangle +(1,1);
               }
               {
                \ifthenelse{\equal{\entry}{\boldentry X}}
                   {
                    \node (y) at ($(y) + (1,0)$) {};
                    \fill[color=gray] ($(y)-(0.5,0.5)$) rectangle +(1,1);
                    \draw ($(y)-(0.5,0.5)$) rectangle +(1,1);
                   }
                   {
                    \node (y) at ($(y) + (1,0)$) {\entry};
                    \draw ($(y)-(0.5,0.5)$) rectangle +(1,1);
                   }
               }
            }
        \coordinate (x) at ($(x)-(0,1)$);
        \coordinate (y) at (x);
        }
}

\newcommand{\tikztableau}[2][scale=0.6,every node/.style={font=\small}]{%
    \begin{tikzpicture}[#1]%
        \tikztableauinternal{#2}%
    \end{tikzpicture}%
}

\newcommand{\tikztableausmall}[1]{\tikztableau[scale=0.45,every node/.style={font=\rm\small}]{#1}}

\newdimen\squaresize \squaresize=10pt
\newdimen\thickness \thickness=0.4pt

\def\square#1{\hbox{\vrule width \thickness
     \vbox to \squaresize{\hrule height \thickness\vss
        \hbox to \squaresize{\hss#1\hss}
     \vss\hrule height\thickness}
\unskip\vrule width \thickness}
\kern-\thickness}

\def\vsquare#1{\vbox{\square{$#1$}}\kern-\thickness}

\def\thisbox#1{\kern-.09ex\fbox{#1}}
\def\downbox#1{\lower1.200em\hbox{#1}}
\newdimen\Squaresize \Squaresize=20pt
\newdimen\Thickness \Thickness=0.4pt

\def\Square#1{\hbox{\vrule width \Thickness
     \vbox to \Squaresize{\hrule height \Thickness\vss
        \hbox to \Squaresize{\hss#1\hss}
     \vss\hrule height\Thickness}
\unskip\vrule width \Thickness}
\kern-\Thickness}

\def\Vsquare#1{\vbox{\Square{$#1$}}\kern-\Thickness}

\title[plethysm and UBP]{Plethysm and the algebra of uniform block permutations}

\author[Orellana]{Rosa Orellana}
\address[R. Orellana]{Mathematics Department, Dartmouth College, 
Hanover, NH 03755, U.S.A.}
\email{Rosa.C.Orellana@dartmouth.edu}
\urladdr{\href{https://math.dartmouth.edu/~orellana/}{https://math.dartmouth.edu/~orellana/}}

\author[Saliola]{Franco Saliola}
\address[F. Saliola]{D\'epartement de math\'ematiques,
Universit\'e du Qu\'ebec \`a Montr\'eal, Canada}
\email{saliola.franco@uqam.ca}
\urladdr{\href{http://lacim.uqam.ca/~saliola/}{http://lacim.uqam.ca/~saliola/}}

\author[Schilling]{Anne Schilling}
\address[A. Schilling]{Department of Mathematics, University of California, One Shields
Avenue, Davis, CA 95616-8633, U.S.A.}
\email{anne@math.ucdavis.edu}
\urladdr{\href{http://www.math.ucdavis.edu/~anne}{http://www.math.ucdavis.edu/~anne}}

\author[Zabrocki]{Mike Zabrocki}
\address[M. Zabrocki]{Department of Mathematics and Statistics,  York University, 4700 Keele Street, Toronto,
Ontario M3J 1P3, Canada}
\email{zabrocki@mathstat.yorku.ca}
\urladdr{\href{http://garsia.math.yorku.ca/~zabrocki/}{http://garsia.math.yorku.ca/~zabrocki/}}

\allowdisplaybreaks

\begin{document}

\maketitle

\begin{abstract}
We study the representation theory of the uniform block permutation algebra in the context of the representation theory
of factorizable inverse monoids. The uniform block permutation algebra is a subalgebra of the partition algebra and is also
known as the party algebra. We compute its characters and provide a Frobenius characteristic map to symmetric functions.
This reveals connections of the characters of the uniform block permutation algebra and plethysms of Schur functions.
\end{abstract}

\setcounter{tocdepth}{2}
\tableofcontents

\section{Introduction}
\label{Introduction}

The partition algebra arose in the early 1990s in the work of Martin~\cite{Martin.1991,Martin.1994, Martin.1996, Martin.2000} and Jones~\cite{Jones.1994}
in the context of the Potts model in statistical mechanics. It is a generalization of the Temperley--Lieb algebra and can be formulated
in terms of an important question in invariant theory: If a group $G$ acts on an $n$-dimensional vector space $V$, how does $V^{\otimes k}$
decompose into irreducible representations of $G$? This question can be studied using the centralizer algebra $\operatorname{End}_G(V^{\otimes k})$.
The partition algebra is isomorphic to this centralizer algebra when the group $G$ is the symmetric group $\mathfrak{S}_n$~\cite{Jones.1994,Martin.1994}, that is, $\operatorname{End}_{\mathfrak{S}_n}(V^{\otimes k}) \simeq P_k(n)$.

Inspired by this work, Tanabe~\cite{Tanabe.1997} considered the case when $G$ is a unitary reflection group $G(r,p,n)$, where $G(r,1,n)$ is a group
of $n \times n$ monomial matrices whose nonzero entries are $r$-th roots of unity and $G(r,p,n)$ is a subgroup of index $p$ in $G(r,1,n)$.
Kosuda~\cite{Kosuda.2000, Kosuda.2006} studied the party algebra $\mathcal{U}_k$, which corresponds to the subcase $p=1$, $n \geqslant k$
and $r>k$. The party algebra is a subalgebra of the partition algebra $P_k(n)$. Elements in the party algebra can be viewed as bijections between blocks
of the same size of two set partitions of $\{1,2,\ldots,k\}$ of the same type. To quote from Kosuda~\cite{Kosuda.2000}:
\begin{quote}
Suppose that there exist two parties each of which consists of $n$ members. The parties hold meetings splitting into several small groups.
Every group consists of the members of each party of the same number. The set of such decompositions into small groups makes an algebra
called the party algebra under a certain product.\footnote{Clearly in arriving at the name
`party algebra' Kosuda was not imagining a party of introverted mathematicians for
which we would likely see each vertex in the diagram isolated.}
\end{quote}
Since the block sizes of the two set  partitions are required to match, this algebra is also known as the uniform block permutation
algebra~\cite{FitzGerald.2003}, which is the terminology we will use in this paper.

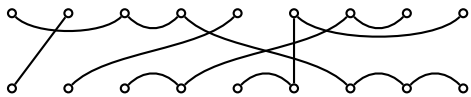
\begin{figure}[h]
\begin{tikzpicture}[scale = 0.5,thick, baseline={(0,-1ex/2)}]
\node[vertex] (G--9) at (12.0, -1) [shape = circle, draw] {};
\node[vertex] (G--8) at (10.5, -1) [shape = circle, draw] {};
\node[vertex] (G--6) at (7.5, -1) [shape = circle, draw] {};
\node[vertex] (G-1) at (0.0, 1) [shape = circle, draw] {};
\node[vertex] (G-3) at (3.0, 1) [shape = circle, draw] {};
\node[vertex] (G-4) at (4.5, 1) [shape = circle, draw] {};
\node[vertex] (G--7) at (9.0, -1) [shape = circle, draw] {};
\node[vertex] (G--5) at (6.0, -1) [shape = circle, draw] {};
\node[vertex] (G-7) at (9.0, 1) [shape = circle, draw] {};
\node[vertex] (G-9) at (12.0, 1) [shape = circle, draw] {};
\node[vertex] (G--4) at (4.5, -1) [shape = circle, draw] {};
\node[vertex] (G--3) at (3.0, -1) [shape = circle, draw] {};
\node[vertex] (G-6) at (7.5, 1) [shape = circle, draw] {};
\node[vertex] (G-8) at (10.5, 1) [shape = circle, draw] {};
\node[vertex] (G--2) at (1.5, -1) [shape = circle, draw] {};
\node[vertex] (G-5) at (6.0, 1) [shape = circle, draw] {};
\node[vertex] (G--1) at (0.0, -1) [shape = circle, draw] {};
\node[vertex] (G-2) at (1.5, 1) [shape = circle, draw] {};
\draw[] (G-1) .. controls +(0.6, -0.6) and +(-0.6, -0.6) .. (G-3);
\draw[] (G-3) .. controls +(0.5, -0.5) and +(-0.5, -0.5) .. (G-4);
\draw[] (G-4) .. controls +(1, -1) and +(-1, 1) .. (G--7);
\draw[] (G--9) .. controls +(-0.5, 0.5) and +(0.5, 0.5) .. (G--8);
\draw[] (G--8) .. controls +(-0.5, 0.5) and +(0.5, 0.5) .. (G--7);
\draw[] (G-7) .. controls +(0.5, -0.5) and +(-0.5, -0.5) .. (G-8);
\draw[] (G--6) .. controls +(-0.5, 0.5) and +(0.5, 0.5) .. (G--5);
\draw[] (G-6) .. controls +(0, -1) and +(0, 1) .. (G--6);
\draw[] (G-6) .. controls +(0.8, -0.8) and +(-0.8, -0.8) .. (G-9);
\draw[] (G--4) .. controls +(-0.5, 0.5) and +(0.5, 0.5) .. (G--3);
\draw[] (G--4) .. controls +(1, 1) and +(-1, -1) .. (G-7);
\draw[] (G-5) .. controls +(-1, -1) and +(1, 1) .. (G--2);
\draw[] (G-2) .. controls +(-0.75, -1) and +(0.75, 1) .. (G--1);
\end{tikzpicture}
\caption{An example diagram representing an element of the uniform block permutation monoid $\mathcal{U}_9$.
The connected components of the graph visually represent the blocks of the set partitions.
Each connected component contains the same number of elements in the top row as in the bottom row.}
\end{figure}

In this paper, we study the representation theory of the uniform block permutation algebra $\mathcal{U}_k$. It is an interesting, nontrivial example
of a factorizable inverse monoid. We use the general theory of finite inverse monoids to develop the representation theory of $\mathcal{U}_k$. 
This relies on theorems due to Clifford~\cite{Clifford}, Munn~\cite{Munn} and Ponizovski\u{\i}~\cite{Ponizovski} and the explicit constructions
given in~\cite{RZ1991,LP1969}. The exposition and notation we follow here is found in~\cite{GMS09, Steinberg.2016}. 
In particular, by characterizing the idempotents, the maximal subgroups, the $\mathscr{J}$-classes and the $\mathscr{L}$-classes of $\mathcal{U}_k$,
the Sch\"utzenberger representations can be employed to construct the irreducible representations of $\mathcal{U}_k$. The 
representations that we obtain very nicely extend Young's construction of the irreducible representations of symmetric groups:
instead of a symmetric group action on standard tableaux we obtain a monoid action on
sequences of set-valued tableaux. The action can be described on a basis indexed by combinatorial objects
using familiar and well-used relations on tableaux, rather than operations on paths in a
Bratteli diagram as appears in the construction by Kosuda~\cite{Kosuda.2006}.

We also compute the irreducible characters of the uniform block permutation algebra $\mathcal{U}_k$
and relate them to symmetric functions by defining a Frobenius characteristic that maps
a class function of $\mathcal{U}_k$ to an element of the $k$-fold tensor product of the symmetric functions.
More precisely, each irreducible representation of $\mathcal{U}_k$ is indexed by a
$k$-tuple of partitions $\vec{\lambda}=(\lambda^{(1)},\lambda^{(2)},\ldots,\lambda^{(k)})$
such that $\sum_{i=1}^k i | \lambda^{(i)}| =k$, and the associated symmetric function
of the character of the restriction of this representation to $\mathfrak{S}_k$ is
\[
    s_{\lambda^{(1)}}[s_1] s_{\lambda^{(2)}}[s_2] \ldots s_{\lambda^{(k)}}[s_k],
\]
where $s_\lambda$ is the Schur function indexed by a partition $\lambda$ and $s_\lambda[s_h]$ is the plethysm of $s_\lambda$ with
the Schur function $s_h$ indexed by a single row. In this sense, the representation theory of
$\mathcal{U}_k$ gives a novel representation theoretic approach to plethysm.

Furthermore, having the image of the characters under the Frobenius map reduces
the computation of the characters to a computation on symmetric functions.
In a 2005 talk, Naruse presented (without proof) several results on the
characters of the Tanabe algebra, and hence as a special case the uniform block
permutation algebra $\mathcal{U}_k$~\cite{Naruse.2005}.
This included character tables for $\mathcal{U}_k$ for small values of $k$.
Using the symmetric function connection that we establish, these tables can be verified.
We are not aware of any other proofs of Naruse's results in the literature.

In a subsequent paper, we will consider the restriction from the general linear group $\operatorname{GL}_n$ to the symmetric group $\mathfrak{S}_n$
which involves the same restriction coefficients as the restriction of the partition algebra $P_k(n)$ to the symmetric group $\mathfrak{S}_k$.
The uniform block permutation algebra can be viewed as an intermediate step in this restriction, see \cite[Section 4.1]{Har}.
The restriction from the partition algebra $P_k(n)$ to the uniform block permutation algebra $\mathcal{U}_k$ involves the Littlewood--Richardson rule,
whereas the restriction from $\mathcal{U}_k$ to $\mathfrak{S}_k$ involves the plethysm operation.

To conclude, let us compare our approach in this paper with existing constructions in the literature.
Irreducible matrix representations of $\mathcal{U}_k$ were previously constructed by
Kosuda~\cite{Kosuda.2006} by defining a tower of algebras $\mathcal{U}_k \subseteq \mathcal{U}_{k+1}$.
In Kosuda's approach~\cite{Kosuda.2006}, the rows of the matrices are indexed by paths in the Bratteli diagram of the tower 
of monoid algebras and the action is defined over the field
$\mathbb{Q}(\sqrt{2}, \sqrt{3}, \ldots, \sqrt{k})$.
In contrast, we use the theory of finite monoids to construct the irreducible representations
of $\mathcal{U}_k$ in terms of tuples of set-valued tableaux
and the action is expressed in the basis with coefficients that are integers.
The bijection between the tuples of set-valued tableaux that we use here and the path model
used by Kosuda \cite{Kosuda.2006} is similar to the bijection described in~\cite{COSSZ} relating the path model for diagram algebras
and set-valued tableaux. Set-valued tableaux also appear in the construction of the representations 
of the partition algebra and sub-diagram algebras by Halverson and Jacobson~\cite{HJ}.

This paper is organized as follows. In Section~\ref{section.monoid}, we introduce uniform block permutations and describe the monoid
structure on them. In particular, we provide a presentation of the monoid of uniform block permutations $\mathcal{U}_k$ and show that it is an inverse monoid.
In Section~\ref{section.irreps of Uk}, we compute the maximal subgroups, $\mathscr{J}$- and $\mathscr{L}$-classes of $\mathcal{U}_k$.
Using Sch\"utzenberger representations, this makes it possible to construct the irreducible representations of $\mathcal{U}_k$.
The characters of $\mathcal{U}_k$ are computed in Section~\ref{section.characters}. Finally, in Section~\ref{sec:charsSym} the connection of the 
characters with symmetric functions is established.

\subsection*{Acknowledgements}
This project benefitted tremendously from the support by the American Institute for Mathematics (AIM) through their AIM SQuaRE program.
We thank Laura Colmenarejo for discussions.

RO was partially supported by NSF grant DMS--300512.
AS was partially supported by NSF grant DMS--1760329 and DMS--2053350.
FS and MZ are supported by NSERC/CNSRG.

\section{The monoid of uniform block permutations}
\label{section.monoid}

After some preliminary notation in Sections~\ref{section.partitions} and~\ref{section.set partitions} on partitions and set partitions,
we define uniform block permutations $\mathcal{U}_k$ in Section~\ref{section.UBP} and give its monoid structure in Section~\ref{sec:monoidstructure}. 
We show in Section~\ref{section.factorizable} that every element of $\mathcal{U}_k$ is a product of an idempotent and a permutation. We recall a 
presentation of $\mathcal{U}_k$ in Section~\ref{sec:presentation} and we conclude in Section~\ref{sec:inverse-monoid} with a proof that $\mathcal{U}_k$ 
is an inverse monoid.

\subsection{Partitions}
\label{section.partitions}
A \defn{partition} of a positive integer $k$ is a nonincreasing sequence $\lambda = (\lambda_1, \ldots, \lambda_\ell)$ of positive integers such that 
$\lambda_1 + \cdots + \lambda_\ell = k$.
We write $|\lambda|$ for $\lambda_1 + \cdots + \lambda_\ell$ and call $\lambda_i$ the \defn{parts} of $\lambda$.
The \defn{length} of the partition $\lambda$ is $\ell(\lambda)=\ell$.
We write $\lambda \vdash k$ to mean that $\lambda$ is a partition of $k$.
We declare that the empty sequence $()$ is the unique partition of $0$, and we denote this by $\epart$.
We will often use \defn{exponential notation} for partitions in which $b$ consecutive occurrences of the part $i$ is denoted by $i^b$: 
for example, $(4,4,4,2,1,1,1,1)$ can be denoted $(1^4 2^1 4^3)$.

We use Young diagrams to represent partitions.  If $\lambda = (\lambda_1, \ldots, \lambda_\ell)$ is a partition of $k$, then the \defn{Young diagram}
of $\lambda$ is the left-justified array of $k$ cells (or boxes) with $\lambda_i$ cells in the $i$-th row. We use French notation, so that
the largest row is at the bottom. This may be upside down from what is sometimes used in representation theory.

For every nonnegative integer $k$, we define
\begin{equation}
\label{equation.Ik}
    I_k = \left\{ \left(\lambda^{(1)}, \lambda^{(2)}, \ldots, \lambda^{(k)}\right) : \lambda^{(i)}
    \hbox{ are partitions such that }
    \sum_{i=1}^k i |\lambda^{(i)}| = k \right\}.
\end{equation}
We denote elements in $I_k$ as $\vec{\lambda} = (\lambda^{(1)},\ldots,\lambda^{(k)})$.
We will see that the elements of $I_k$ index the irreducible representations of
the uniform block permutation algebra $\mathcal{U}_k$.  In examples the elements of
$I_k$ will be expressed by dropping the trailing empty partitions in the list of partitions
so that, for instance, the element $(\epart, (2), \epart, \epart)$ of $I_4$ will be displayed
without loss of information as $(\epart, (2))$.

\subsection{Set partitions}
\label{section.set partitions}
A \defn{set partition} $\pi$ of a set $X$ is a collection of nonempty subsets $\{\pi_1, \ldots, \pi_\ell\}$ of $X$ such that
$\pi_i \cap \pi_j = \emptyset$ for all $i \neq j$ and $\bigcup_{i=1}^{\ell} \pi_i = X$.  We use $\pi \vdash X$ to denote that $\pi$
is a set partition of $X$. The subsets $\pi_i$ are called the \defn{blocks} of $\pi$.

If $\pi = \{\pi_1, \ldots, \pi_\ell\}$ is a set partition of $[k] = \{1, 2, \ldots, k\}$,
then we order the blocks in $\pi$ using the \defn{graded last letter order}: if $A$ and $B$ are two blocks, then
$A\leqslant B$ in the graded last letter order if either $|A|<|B|$ or if $|A|=|B|$, then $\mathsf{max}(A)\leqslant \mathsf{max}(B)$.
For example, the blocks of $\pi = \{\{4\}, \{1,6\},\{3,8\},\{2,5,7\}\}$ are listed in graded last letter order.

To simplify notation, we often write the set partitions by separating the blocks by vertical lines. For example, the set partition
$\pi = \{\{4\}, \{1,6\},\{3,8\},\{2,5,7\}\}$ will also be denoted by $\pi = 4|16|38|257$.

The \defn{type} of a set partition $\pi$, denoted $\mathsf{type}(\pi)$,
is the (integer) partition formed by the sizes of the blocks of $\pi$.
For example,
\begin{equation*}
    \mathsf{type}(4|16|38|257) = (3,2,2,1) = (1, 2^2, 3).
\end{equation*}
The number of set partitions of type $\lambda = (1^{a_1}, 2^{a_2}, \ldots, k^{a_k})$ is
\begin{equation}\label{eq:nosetpartitions}
    \mathsf{sp}_k(\lambda) = \frac{k!}{a_1! \cdots a_k!(1!)^{a_1}(2!)^{a_2} \cdots (k!)^{a_k}}.
\end{equation}

Set partitions are ordered by \defn{refinement}: if $\pi$ and $\gamma$ are two set partitions, then we say that $\pi$ is \defn{finer}
than $\gamma$ and that $\gamma$ is \defn{coarser} than $\pi$, if every block of $\pi$ is a subset of some block of $\gamma$. In that case,
we write $\pi\leqslant \gamma$.

Given a set partition $\pi = \{\pi_1, \ldots, \pi_\ell\} \vdash [k]$
and permutation $\sigma \in \mathfrak{S}_k$, we define
$\sigma(\pi) = \{\sigma(\pi_1), \ldots, \sigma(\pi_\ell)\}$,
where $\sigma(A) = \{\sigma(a) : a\in A\}$ for every subset $A \subseteq [k]$.
Since $|A| = |\sigma(A)|$ for all $A \subseteq [k]$, it follows that
$\pi$ and $\sigma(\pi)$ have the same type.

\subsection{Uniform block permutations}
\label{section.UBP}

We define the set of uniform block permutations $\mathcal{U}_k$ and give
three equivalent ways to view its elements:
as set partitions of $[k]\cup [\bar{k}]$;
as size-preserving bijections between the blocks of two set partitions;
as two-row diagrams.
We will use these interpretations interchangeably throughout the paper.

\subsubsection{Set partitions of $[k] \cup [\overline{k}]$}
For nonzero $k\in \mathbb{N}$, define
\begin{equation*}
    [k]= \{1, \dots, k\}
    \qquad\text{and}\qquad
    [\overline{k}] = \{\overline{1}, \ldots, \overline{k}\}.
\end{equation*}
For each $a \in [k]$, we define $\overline{\overline{a}} = a$ so that
$a \mapsto \overline{a}$ as an involution on $[k] \cup [\overline{k}]$.

Let $d = \{d_1, d_2, \ldots, d_\ell\}$ be a set partition of $[k]\cup [\bar{k}]$.
We say that $d$ is \defn{uniform} if
$|d_i\cap [k]|= |d_i\cap [\bar{k}]|$ for all $1\leqslant i\leqslant \ell$.
Let $\mathcal{U}_k$ be the set of uniform set partitions of $[k] \cup [\bar{k}]$:
\[
    \mathcal{U}_k = \big\{ d \vdash [k]\cup[\bar{k}]: \text{$d$ uniform}\big\}.
\]

Let $\mathsf{top}(d)$ be the set partition of $[k]$ consisting of the blocks $d_i \cap [k]$
for $1 \leqslant i \leqslant \ell$ and $\mathsf{bot}(d)$ the set
partition of $[k]$ containing the blocks $\overline{d_i} \cap [k]$
for $1 \leqslant i \leqslant \ell$, where $\overline{d_i} = \{ \overline{a} : a \in d_i\}$. For example,
\begin{align*}
    d & = \{ \{2,\bar{4} \}, \{5,\bar{7} \}, \{1,3,\bar{1},\bar{2} \}, \{4,6,\bar{3}, \bar{6}\}, \{7, 8, 9, \bar{5}, \bar{8}, \bar{9} \}  \}, \\
\mathsf{top}(d) &=\{\{2\}, \{5\}, \{ 1,3\},  \{4, 6\},  \{7, 8, 9\}\}, \\
\mathsf{bot}(d) &=\{ \{ 4\}, \{7\},  \{1,2\},\{3,6\},  \{5, 8, 9\}\}.
\end{align*}
When writing set partitions, we list the blocks in graded last letter order.

\subsubsection{Size-preserving bijections between set partitions of $[k]$}
\label{sssection:bijection}
It is useful to think of $d$ as the size-preserving bijection $d \colon
\mathsf{top}(d) \rightarrow \mathsf{bot}(d)$ that maps
$d_i \cap [k]$ to $\overline{d_i} \cap [k]$.
Continuing the previous example, the bijection associated with $d$, expressed
in two-line notation, is
\begin{align*}
    \left(
        \begin{array}{ccccc}
            \{2\} & \{5\} & \{ 1,3\} & \{4, 6\} & \{7, 8, 9\} \\
            \{4\} & \{7\} & \{1,2\}  & \{3,6\}  & \{5, 8, 9\}
        \end{array}
    \right).
\end{align*}
For this reason, the elements of $\mathcal{U}_k$ are called \defn{uniform block
permutations}. With this interpretation, it follows that the number of elements
in $\mathcal{U}_k$ is
\[
     |\mathcal{U}_k| = \sum_{\lambda = (1^{a_1}, \ldots, k^{a_k})\vdash k} \mathsf{sp}_k(\lambda)^2 a_1! \cdots a_k!.
\]
Starting with $k=0$, the sequence of $ |\mathcal{U}_k|$ begins
\[
1, 1, 3, 16, 131, 1496, 22482, 426833,\ldots
\]
and is listed as sequence \href{https://oeis.org/A023998}{A023998}
in the Online Encyclopedia of Integer Sequences \cite{OEIS}.

\subsubsection{Diagrams}
A graph on the vertex set $[k] \cup [\bar{k}]$ \defn{represents}
a set partition $d \vdash [k] \cup [\bar{k}]$ if (the vertices of) the
connected components of the graph are the blocks of $d$.
We draw these graphs by arranging the vertices in two rows:
$1, 2, \ldots, k$ appear from left to right in the top row; and
$\bar1, \bar2, \ldots, \bar{k}$ from left to right in the bottom row.
In this way, the graph
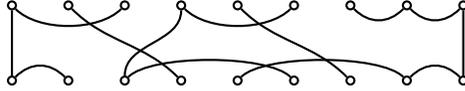
\begin{center}
\begin{tikzpicture}[scale = 0.5,thick, baseline={(0,-1ex/2)}]
\node[vertex] (G--9) at (12.0, -1) [shape = circle, draw] {};
\node[vertex] (G--8) at (10.5, -1) [shape = circle, draw] {};
\node[vertex] (G--5) at (6.0, -1) [shape = circle, draw] {};
\node[vertex] (G-5) at (6.0, 1) [shape = circle, draw] {};
\node[vertex] (G-8) at (10.5, 1) [shape = circle, draw] {};
\node[vertex] (G-9) at (12.0, 1) [shape = circle, draw] {};
\node[vertex] (G--7) at (9.0, -1) [shape = circle, draw] {};
\node[vertex] (G-7) at (9.0, 1) [shape = circle, draw] {};
\node[vertex] (G--6) at (7.5, -1) [shape = circle, draw] {};
\node[vertex] (G--3) at (3.0, -1) [shape = circle, draw] {};
\node[vertex] (G-3) at (3.0, 1) [shape = circle, draw] {};
\node[vertex] (G-6) at (7.5, 1) [shape = circle, draw] {};
\node[vertex] (G--4) at (4.5, -1) [shape = circle, draw] {};
\node[vertex] (G--1) at (0.0, -1) [shape = circle, draw] {};
\node[vertex] (G-1) at (0.0, 1) [shape = circle, draw] {};
\node[vertex] (G-4) at (4.5, 1) [shape = circle, draw] {};
\node[vertex] (G--2) at (1.5, -1) [shape = circle, draw] {};
\node[vertex] (G-2) at (1.5, 1) [shape = circle, draw] {};
\draw[] (G-7) .. controls +(0.5, -0.5) and +(-0.5, -0.5) .. (G-8);
\draw[] (G-8) .. controls +(0.5, -0.5) and +(-0.5, -0.5) .. (G-9);
\draw[] (G-9) .. controls +(0, -1) and +(0, 1) .. (G--9);
\draw[] (G--9) .. controls +(-0.5, 0.5) and +(0.5, 0.5) .. (G--8);
\draw[] (G--8) .. controls +(-0.7, 0.7) and +(0.7, 0.7) .. (G--5);
\draw[] (G-5) .. controls +(1, -1) and +(-1, 1) .. (G--7);
\draw[] (G-4) .. controls +(0.7, -0.7) and +(-0.7, -0.7) .. (G-6);
\draw[] (G--6) .. controls +(-0.7, 0.7) and +(0.7, 0.7) .. (G--3);
\draw[] (G-1) .. controls +(0.7, -0.7) and +(-0.7, -0.7) .. (G-3);
\draw[] (G-4) .. controls +(0, -1) and +(0, 1) .. (G--3);
\draw[] (G--2) .. controls +(-0.5, 0.5) and +(0.5, 0.5) .. (G--1);
\draw[] (G--1) .. controls +(0, 1) and +(0, -1) .. (G-1);
\draw[] (G-2) .. controls +(1, -1) and +(-1, 1) .. (G--4);
\end{tikzpicture}
\end{center}
represents the set partition $\{ \{1,3,\bar{1},\bar{2} \}, \{2,\bar{4} \}, \{4,6,\bar{3}, \bar{6}\}, \{5,\bar{7} \}, \{7, 8, 9, \bar{5}, \bar{8}, \bar{9} \}  \}$.
We call this the \defn{(two-row) diagram} of the set partition.
Notice that it is possible that more than one graph represents a given set
partition; therefore, a diagram represents a class of labeled graphs that have
the same connected components.

\subsection{Monoid structure}
\label{sec:monoidstructure}
We next define a monoid structure on the set of all set partitions of $[k] \cup
[\bar{k}]$. It follows from this definition that the product of two uniform
block permutations is again a uniform block permutation, from which we obtain
a monoid structure on $\mathcal{U}_k$.

Let $d, d' \in \mathcal{U}_k$ (or more generally, any pair of set partitions
$[k] \cup [\bar{k}]$), which we view as diagrams. The product $d d'$ is defined
as follows:
\begin{itemize}
    \item
        stack $d$ on top of $d'$ so that the bottom vertices of $d$
        line up with the top vertices of $d'$;

    \item
        compute the connected components of the resulting three-row diagram;

    \item
        eliminate the vertices of the middle row from the connected components.
\end{itemize}

\begin{example}
    We illustrate the product of the following two set partitions:
    \begin{equation*}
        d =
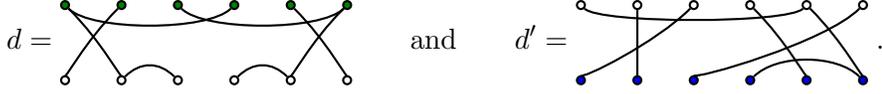
        \begin{tikzpicture}[scale = 0.5,thick, baseline={(0,-1ex/2)}]
        \node[vertex, fill=Green] (G-1) at (0.0, 1) [shape = circle, draw] {};
        \node[vertex, fill=Green] (G-2) at (1.5, 1) [shape = circle, draw] {};
        \node[vertex, fill=Green] (G-3) at (3.0, 1) [shape = circle, draw] {};
        \node[vertex, fill=Green] (G-4) at (4.5, 1) [shape = circle, draw] {};
        \node[vertex, fill=Green] (G-5) at (6.0, 1) [shape = circle, draw] {};
        \node[vertex, fill=Green] (G-6) at (7.5, 1) [shape = circle, draw] {};
        \node[vertex] (G--1) at (0.0, -1) [shape = circle, draw] {};
        \node[vertex] (G--2) at (1.5, -1) [shape = circle, draw] {};
        \node[vertex] (G--3) at (3.0, -1) [shape = circle, draw] {};
        \node[vertex] (G--4) at (4.5, -1) [shape = circle, draw] {};
        \node[vertex] (G--5) at (6.0, -1) [shape = circle, draw] {};
        \node[vertex] (G--6) at (7.5, -1) [shape = circle, draw] {};
        \draw[] (G-2) .. controls +(-1, -1) and +(0,0) .. (G--1);
        \draw[] (G-1) .. controls +(1, -1) and +(0, 0) .. (G--2);
        \draw[] (G--2) .. controls +(0.5, 0.5) and +(-0.5, 0.5) .. (G--3);
        \draw[] (G-1) .. controls +(0.7, -0.7) and +(-0.7, -0.7) .. (G-4);
        \draw[] (G-3) .. controls +(0.7, -0.7) and +(-0.7, -0.7) .. (G-6);
        \draw[] (G--4) .. controls +(0.5, 0.5) and +(-0.5, 0.5) .. (G--5);
        \draw[] (G-6) .. controls +(-1, -1) and +(0, 0) .. (G--5);
        \draw[] (G-5) .. controls +(1, -1) and +(0, 0) .. (G--6);
        \end{tikzpicture}
        \qquad\text{and}\qquad
        d' =
        \begin{tikzpicture}[scale = 0.5,thick, baseline={(0,-1ex/2)}]
        \node[vertex] (G-1) at (0.0, 1) [shape = circle, draw] {};
        \node[vertex] (G-2) at (1.5, 1) [shape = circle, draw] {};
        \node[vertex] (G-3) at (3.0, 1) [shape = circle, draw] {};
        \node[vertex] (G-4) at (4.5, 1) [shape = circle, draw] {};
        \node[vertex] (G-5) at (6.0, 1) [shape = circle, draw] {};
        \node[vertex] (G-6) at (7.5, 1) [shape = circle, draw] {};
        \node[vertex, fill=Blue] (G--1) at (0.0, -1) [shape = circle, draw] {};
        \node[vertex, fill=Blue] (G--2) at (1.5, -1) [shape = circle, draw] {};
        \node[vertex, fill=Blue] (G--3) at (3.0, -1) [shape = circle, draw] {};
        \node[vertex, fill=Blue] (G--4) at (4.5, -1) [shape = circle, draw] {};
        \node[vertex, fill=Blue] (G--5) at (6.0, -1) [shape = circle, draw] {};
        \node[vertex, fill=Blue] (G--6) at (7.5, -1) [shape = circle, draw] {};
        \draw[] (G-1) .. controls +(0.5, -0.5) and +(-0.5, -0.5) .. (G-5);
        \draw[] (G--4) .. controls +(0.7, 0.7) and +(-0.7, 0.7) .. (G--6);
        \draw[] (G-2) .. controls +(0, -1) and +(0, 0) .. (G--2);
        \draw[] (G-3) .. controls +(-1, -1) and +(0, 0) .. (G--1);
        \draw[] (G-4) .. controls +(1, -1) and +(0, 0) .. (G--5);
        \draw[] (G-5) .. controls +(1, -1) and +(0, 0) .. (G--6);
        \draw[] (G-6) .. controls +(-1, -1) and +(0, 0) .. (G--3);
        \end{tikzpicture}\; .
    \end{equation*}
    The product $d d'$ is the set partition whose blocks correspond to the
    connected components of the diagram obtained by stacking the diagrams of
    $d$ and $d'$:
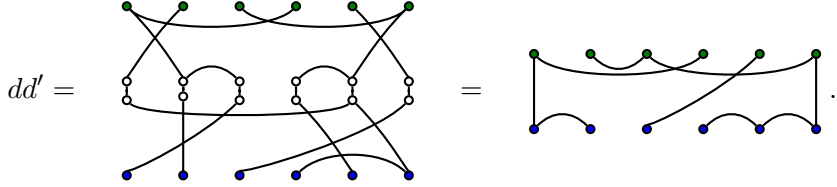
    \begin{equation*}
        d d' = \hskip .2in
        \raisebox{.25in}{\begin{tikzpicture}[scale = 0.5,thick, baseline={(0,-1ex/2)}]
        \node[vertex] (G--1) at (0.0, -1) [shape = circle, draw] {};
        \node[vertex] (G--2) at (1.5, -1) [shape = circle, draw] {};
        \node[vertex] (G--3) at (3.0, -1) [shape = circle, draw] {};
        \node[vertex] (G--4) at (4.5, -1) [shape = circle, draw] {};
        \node[vertex] (G--5) at (6.0, -1) [shape = circle, draw] {};
        \node[vertex] (G--6) at (7.5, -1) [shape = circle, draw] {};
        \node[vertex, fill=Green] (G-1) at (0.0, 1) [shape = circle, draw] {};
        \node[vertex, fill=Green] (G-2) at (1.5, 1) [shape = circle, draw] {};
        \node[vertex, fill=Green] (G-3) at (3.0, 1) [shape = circle, draw] {};
        \node[vertex, fill=Green] (G-4) at (4.5, 1) [shape = circle, draw] {};
        \node[vertex, fill=Green] (G-5) at (6.0, 1) [shape = circle, draw] {};
        \node[vertex, fill=Green] (G-6) at (7.5, 1) [shape = circle, draw] {};
        \node[vertex, fill=Blue] (G--1p) at (0.0, -3.5) [shape = circle, draw] {};
        \node[vertex, fill=Blue] (G--2p) at (1.5, -3.5) [shape = circle, draw] {};
        \node[vertex, fill=Blue] (G--3p) at (3.0, -3.5) [shape = circle, draw] {};
        \node[vertex, fill=Blue] (G--4p) at (4.5, -3.5) [shape = circle, draw] {};
        \node[vertex, fill=Blue] (G--5p) at (6.0, -3.5) [shape = circle, draw] {};
        \node[vertex, fill=Blue] (G--6p) at (7.5, -3.5) [shape = circle, draw] {};
        \node[vertex] (G-1p) at (0.0, -1.5) [shape = circle, draw] {};
        \node[vertex] (G-2p) at (1.5, -1.4) [shape = circle, draw] {};
        \node[vertex] (G-3p) at (3.0, -1.5) [shape = circle, draw] {};
        \node[vertex] (G-4p) at (4.5, -1.5) [shape = circle, draw] {};
        \node[vertex] (G-5p) at (6.0, -1.5) [shape = circle, draw] {};
        \node[vertex] (G-6p) at (7.5, -1.5) [shape = circle, draw] {};
        \draw[] (G-2) .. controls +(-1, -1) and +(0,0) .. (G--1);
        \draw[] (G-1) .. controls +(1, -1) and +(0, 0) .. (G--2);
        \draw[] (G--2) .. controls +(0.5, 0.5) and +(-0.5, 0.5) .. (G--3);
        \draw[] (G-1) .. controls +(0.7, -0.7) and +(-0.7, -0.7) .. (G-4);
        \draw[] (G-3) .. controls +(0.7, -0.7) and +(-0.7, -0.7) .. (G-6);
        \draw[] (G--4) .. controls +(0.5, 0.5) and +(-0.5, 0.5) .. (G--5);
        \draw[] (G-6) .. controls +(-1, -1) and +(0, 0) .. (G--5);
        \draw[] (G-5) .. controls +(1, -1) and +(0, 0) .. (G--6);
        \draw[] (G-1p) .. controls +(0.5, -0.5) and +(-0.5, -0.5) .. (G-5p);
        \draw[] (G--4p) .. controls +(0.7, 0.7) and +(-0.7, 0.7) .. (G--6p);
        \draw[] (G-2p) .. controls +(0, -1) and +(0, 0) .. (G--2p);
        \draw[] (G-3p) .. controls +(-1, -1) and +(0, 0) .. (G--1p);
        \draw[] (G-4p) .. controls +(1, -1) and +(0, 0) .. (G--5p);
        \draw[] (G-5p) .. controls +(1, -1) and +(0, 0) .. (G--6p);
        \draw[] (G-6p) .. controls +(-1, -1) and +(0, 0) .. (G--3p);
        \draw[] (G--1) .. controls +(0, -.5) and +(0, +.5) .. (G-1p);
        \draw[] (G--2) .. controls +(0, -.5) and +(0, +.5) .. (G-2p);
        \draw[] (G--3) .. controls +(0, -.5) and +(0, +.5) .. (G-3p);
        \draw[] (G--4) .. controls +(0, -.5) and +(0, +.5) .. (G-4p);
        \draw[] (G--5) .. controls +(0, -.5) and +(0, +.5) .. (G-5p);
        \draw[] (G--6) .. controls +(0, -.5) and +(0, +.5) .. (G-6p);
        \end{tikzpicture}}\hskip .2in
        =\hskip .2in
        \begin{tikzpicture}[scale = 0.5,thick, baseline={(0,-1ex/2)}]
        \node[vertex, fill=Blue] (G--1) at (0.0, -1) [shape = circle, draw] {};
        \node[vertex, fill=Blue] (G--2) at (1.5, -1) [shape = circle, draw] {};
        \node[vertex, fill=Blue] (G--3) at (3.0, -1) [shape = circle, draw] {};
        \node[vertex, fill=Blue] (G--4) at (4.5, -1) [shape = circle, draw] {};
        \node[vertex, fill=Blue] (G--5) at (6.0, -1) [shape = circle, draw] {};
        \node[vertex, fill=Blue] (G--6) at (7.5, -1) [shape = circle, draw] {};
        \node[vertex, fill=Green] (G-1) at (0.0, 1) [shape = circle, draw] {};
        \node[vertex, fill=Green] (G-2) at (1.5, 1) [shape = circle, draw] {};
        \node[vertex, fill=Green] (G-3) at (3.0, 1) [shape = circle, draw] {};
        \node[vertex, fill=Green] (G-4) at (4.5, 1) [shape = circle, draw] {};
        \node[vertex, fill=Green] (G-5) at (6.0, 1) [shape = circle, draw] {};
        \node[vertex, fill=Green] (G-6) at (7.5, 1) [shape = circle, draw] {};
        \draw[] (G-1) .. controls +(0, -1) and +(0,0) .. (G--1);
        \draw[] (G-1) .. controls +(0.7, -0.7) and +(-0.7, -0.7) .. (G-4);
        \draw[] (G-2) .. controls +(0.5, -0.5) and +(-0.5, -0.5) .. (G-3);
        \draw[] (G-3) .. controls +(0.7, -0.7) and +(-0.7, -0.7) .. (G-6);
        \draw[] (G--1) .. controls +(0.5, 0.5) and +(-0.5, 0.5) .. (G--2);
        \draw[] (G--4) .. controls +(0.5, 0.5) and +(-0.5, 0.5) .. (G--5);
        \draw[] (G--5) .. controls +(0.5, 0.5) and +(-0.5, 0.5) .. (G--6);
        \draw[] (G--3) .. controls +(0, 0) and +(-1, -1) .. (G-5);
        \draw[] (G-6) .. controls +(0, -1) and +(0, 0) .. (G--6);
        \end{tikzpicture}\; .
    \end{equation*}
\end{example}

This multiplication of diagrams is associative and the product of two uniform
block permutations is a uniform block permutations, and hence makes
$\mathcal{U}_k$ into a finite monoid whose identity element is
\begin{equation*}
    \big\{ \{1, \bar{1}\}, \{2, \bar{2}\}, \ldots, \{k, \bar{k}\} \big\}
    \hskip .2in = \hskip .2in
    \begin{tikzpicture}[scale = 0.5,thick, baseline={(0,-1ex/2)}]
        \node[vertex] (G-1) at (1.0, 1) [shape = circle, draw] {};
        \node[vertex] (G-2) at (2.0, 1) [shape = circle, draw] {};
        \node[vertex] (G-3) at (3.0, 1) [] {};
        \node[vertex] (G-4) at (4.0, 1) [] {};
        \node[vertex] (G-5) at (5.0, 1) [shape = circle, draw] {};
        \node[vertex] (G--1) at (1.0, -1) [shape = circle, draw] {};
        \node[vertex] (G--2) at (2.0, -1) [shape = circle, draw] {};
        \node[vertex] (G--3) at (3.0, -1) [] {};
        \node[vertex] (G--4) at (4.0, -1) [] {};
        \node[vertex] (G--5) at (5.0, -1) [shape = circle, draw] {};
        \draw[] (G-1) to (G--1);
        \draw[] (G-2) to (G--2);
        \draw[] (G-5) to (G--5);
        \node (DOTS) at (3.5, 0) [] {$\cdots$};
    \end{tikzpicture}\; .
\end{equation*}
Since connected vertices in the top row of $d$ remain connected in $d d'$,
it follows that the set partition $\mathsf{top}(d d')$ is coarser than or equal to $\mathsf{top}(d)$.
Similarly, the set partition $\mathsf{bot}(d d')$ is coarser than or equal to $\mathsf{bot}(d')$.
Furthermore, any block of $d d'$ contains at least one block of $\mathsf{top}(d)$
and at least one block of $\mathsf{bot}(d')$.
If $n(d)$ is the \defn{number of blocks} in a diagram $d$, then
for all $d, d' \in \mathcal{U}_k$,
\begin{equation*}
    n(dd') \leqslant \mathsf{min}\{n(d), n(d')\}.
\end{equation*}

\begin{remark}[Diagram multiplication and composition of bijections]
    \label{remark:composition-and-product}
    As explained in Section~\ref{sssection:bijection},
    it is often useful to think of diagrams $d$ as bijections
    $d : \mathsf{top}(d) \to \mathsf{bot}(d)$ that preserve block-size,
    and so we highlight some important nuances of this approach.

    If $d : \mathsf{top}(d) \to \mathsf{bot}(d)$
    and $d' : \mathsf{top}(d') \to \mathsf{bot}(d')$
    satisfy $\mathsf{top}(d') = \mathsf{bot}(d)$,
    then the composition of $d$ and $d'$ is defined,
    and the resulting bijection is precisely the one
    associated with the product $d d'$.
    In particular, in this case, $d d'$ maps a block $B$ to the block $d'(d(B))$.

    The inverse of a bijection $d : \mathsf{top}(d) \to \mathsf{bot}(d)$
    is obtained by reflecting the diagram of $d$ across a horizontal line,
    which we denote by $\widetilde{d}$ (\cf Section~\ref{sec:inverse-monoid}).
    Note that $d \widetilde{d}$ is the identity mapping on $\mathsf{top}(d)$
    and $\widetilde{d} d $ is the identity mapping on $\mathsf{bot}(d)$,
    which are not necessarily equal to the identity element of $\mathcal{U}_k$.
    However, they are idempotents of $\mathcal{U}_k$ (\cf Lemma~\ref{lem:idempotents}).
\end{remark}

\subsection{Permutation-idempotent and idempotent-permutation decompositions}
\label{section.factorizable}

We prove that every uniform block permutation
can be factored as a product of a permutation and an idempotent,
and also as a product of an idempotent and a permutation.
We begin by embedding the symmetric group $\mathfrak{S}_k$ in $\mathcal{U}_k$,
then we characterize the idempotents in $\mathcal{U}_k$,
and finally we prove the existence of the factorizations.

\subsubsection{Permutations}
Let $\mathfrak{S}_k$ denote the symmetric group consisting of
the permutations of the set $[k]$. We identify each permutation $\sigma \in
\mathfrak{S}_k$ with the uniform block permutation
$\{\{1, \overline{\sigma(1)}\}, \ldots, \{k, \overline{\sigma(k)}\}\}$,
which we also denote by $\sigma$.
Note that the diagram representing $\sigma$ is the diagram
with an edge connecting $i$ and $\overline{\sigma(i)}$.
(Observe that under this identification,
the product of two permutations $\sigma_1 \sigma_2$ maps $i \in [k]$ to $\sigma_2(\sigma_1(i))$
instead of $\sigma_1(\sigma_2(i))$; \cf Remark~\ref{remark:composition-and-product}.)
For instance, if $s_i$ is the permutation that swaps $i$ and $i+1$ and fixes
all other elements of $[k]$, then
\begin{equation*}
    s_i
    \hskip .1in = \hskip .1in
    \big\{\{1,{\overline 1}\},\ldots,\{i,{\overline {i+1}}\}, \{i+1,{\overline i}\},\ldots, \{k,{\overline k}\}\big\}
    \hskip .1in = \hskip .1in
    \begin{tikzpicture}[scale = 0.5,thick, baseline={(0,-1ex/2)}]
    \node[vertex] (G--6) at (7.5, -1) [shape = circle, draw] {};
    \node[vertex] (G-6) at (7.5, 1) [shape = circle, draw] {};
    \node[vertex] (G--5) at (5.5, -1) [shape = circle, draw] {};
    \node[vertex] (G-5) at (5.5, 1) [shape = circle, draw] {};
    \node[vertex] (G--4) at (4.5, -1) [shape = circle, draw] {};
    \node[vertex] (G-3) at (3.0, 1) [shape = circle, draw] {};
    \node[vertex] (G--3) at (3.0, -1) [shape = circle, draw] {};
    \node[vertex] (G-4) at (4.5, 1) [shape = circle, draw] {};
    \node[vertex] (G--2) at (2.0, -1) [shape = circle, draw] {};
    \node[vertex] (G-2) at (2.0, 1) [shape = circle, draw] {};
    \node[vertex] (G--1) at (0.0, -1) [shape = circle, draw] {};
    \node[vertex] (G-1) at (0.0, 1) [shape = circle, draw] {};
    \node[] at (1.0,0) {...};
    \node[] at (6.5,0) {...};
    \draw[] (G-6) .. controls +(0, -1) and +(0, 1) .. (G--6);
    \draw[] (G-5) .. controls +(0, -1) and +(0, 1) .. (G--5);
    \draw[] (G-3) .. controls +(0.75, -1) and +(-0.75, 1) .. (G--4);
    \draw[] (G-4) .. controls +(-0.75, -1) and +(0.75, 1) .. (G--3);
    \draw[] (G-2) .. controls +(0, -1) and +(0, 1) .. (G--2);
    \draw[] (G-1) .. controls +(0, -1) and +(0, 1) .. (G--1);
    \end{tikzpicture}.
\end{equation*}

\subsubsection{Idempotents}
\label{section.idempotents}

For every set partition $\pi$ of $[k]$ we define the following element of $\mathcal{U}_k$:
\[
    e_\pi = \{A\cup \bar{A}: A\in \pi\} \in \mathcal{U}_k,
\]
where $\bar{A} = \{\bar{i} : i \in A\}$.
For example,
\begin{equation*}
e_{2|7|14|36|589}
\hskip .1in = \hskip .1in
\begin{tikzpicture}[scale = 0.5,thick, baseline={(0,-1ex/2)}]
\node[vertex] (G--9) at (12.0, -1) [shape = circle, draw] {};
\node[vertex] (G--8) at (10.5, -1) [shape = circle, draw] {};
\node[vertex] (G--5) at (6.0, -1) [shape = circle, draw] {};
\node[vertex] (G-5) at (6.0, 1) [shape = circle, draw] {};
\node[vertex] (G-8) at (10.5, 1) [shape = circle, draw] {};
\node[vertex] (G-9) at (12.0, 1) [shape = circle, draw] {};
\node[vertex] (G--7) at (9.0, -1) [shape = circle, draw] {};
\node[vertex] (G-7) at (9.0, 1) [shape = circle, draw] {};
\node[vertex] (G--6) at (7.5, -1) [shape = circle, draw] {};
\node[vertex] (G--3) at (3.0, -1) [shape = circle, draw] {};
\node[vertex] (G-3) at (3.0, 1) [shape = circle, draw] {};
\node[vertex] (G-6) at (7.5, 1) [shape = circle, draw] {};
\node[vertex] (G--4) at (4.5, -1) [shape = circle, draw] {};
\node[vertex] (G--1) at (0.0, -1) [shape = circle, draw] {};
\node[vertex] (G-1) at (0.0, 1) [shape = circle, draw] {};
\node[vertex] (G-4) at (4.5, 1) [shape = circle, draw] {};
\node[vertex] (G--2) at (1.5, -1) [shape = circle, draw] {};
\node[vertex] (G-2) at (1.5, 1) [shape = circle, draw] {};
\draw[] (G-5) .. controls +(0.7, -0.7) and +(-0.7, -0.7) .. (G-8);
\draw[] (G-8) .. controls +(0.5, -0.5) and +(-0.5, -0.5) .. (G-9);
\draw[] (G-9) .. controls +(0, -1) and +(0, 1) .. (G--9);
\draw[] (G--9) .. controls +(-0.5, 0.5) and +(0.5, 0.5) .. (G--8);
\draw[] (G--8) .. controls +(-0.7, 0.7) and +(0.7, 0.7) .. (G--5);
\draw[] (G-7) .. controls +(0, -1) and +(0, 1) .. (G--7);
\draw[] (G-3) .. controls +(0.7, -0.7) and +(-0.7, -0.7) .. (G-6);
\draw[] (G--6) .. controls +(-0.7, 0.7) and +(0.7, 0.7) .. (G--3);
\draw[] (G--3) .. controls +(0, 1) and +(0, -1) .. (G-3);
\draw[] (G-1) .. controls +(0.7, -0.7) and +(-0.7, -0.7) .. (G-4);
\draw[] (G--4) .. controls +(-0.7, 0.7) and +(0.7, 0.7) .. (G--1);
\draw[] (G--1) .. controls +(0, 1) and +(0, -1) .. (G-1);
\draw[] (G-2) .. controls +(0, -1) and +(0, 1) .. (G--2);
\end{tikzpicture}\;.
\end{equation*}

It is not hard to see that $e_\pi$ is an idempotent, and the
next result proves that all the idempotents in $\mathcal{U}_k$
are of this form.

\begin{lemma} \label{lem:idempotents}
The set $E(\mathcal{U}_k) = \{e_\pi : \pi \vdash [k] \}$ is a complete set of idempotents in $\mathcal{U}_k$.

Furthermore, if $\Pi_k$ is the lattice of set partitions of $[k]$ viewed as a monoid with the join operation $\vee$,
then the map $\pi \mapsto e_\pi$ is monoid isomorphism from $\Pi_k$ to $E(\mathcal{U}_k)$. Thus,
\begin{equation*}
    e_\pi e_\gamma = e_{\pi \vee \gamma}.
\end{equation*}
\end{lemma}

\begin{proof}
    Suppose $d \in \mathcal{U}_k$ is an idempotent.
    We will prove that $d = e_{\mathsf{top}(d)}$, where
    \begin{equation*}
        e_{\mathsf{top}(d)} = \left\{
            \big(d_i \cap [k]\big)
            \cup
            \overline{\big(d_i \cap [k]\big)}
            :
            d_i \text{~is a block of~} d
        \right\}.
    \end{equation*}
    We need to prove $d_i = (d_i \cap [k]) \cup \overline{(d_i \cap [k])}$
    or equivalently, $\overline{d_i \cap [k]} =  d_i \cap [\overline{k}]$.

    Let us first show that it suffices to prove $\mathsf{bot}(d) = \mathsf{top}(d)$.
    Let $\delta : \mathsf{top}(d) \to \mathsf{bot}(d)$
    be the bijection associated with $d$ that maps
    $d_i \cap [k]$ to $\overline{d_i} \cap [k]$.
    Then $\delta$ is a permutation of $\mathsf{top}(d)$ and
    $\delta \circ \delta$ is the bijection associated with $dd$.
    Since $dd = d$, it follows that $\delta \circ \delta = \delta$,
    and thus $\delta$ is the identity mapping.
    Hence, $\overline{d_i \cap [k]}
    = \overline{\delta(d_i \cap [k])}
    = \overline{\overline{d_i} \cap [k]}
    = d_i \cap [\overline{k}]$.

    We now prove $\mathsf{bot}(d) = \mathsf{top}(d)$; more specifically,
    if $d_i \cap [k] \in \mathsf{top}(d)$, then $d_i \cap [k] \in \mathsf{bot}(d)$.
    Consider the three-row diagram constructed by stacking $d$ on top of
    a second copy of $d$, which we denote by $d'$.
    We will refer to the three vertices occurring in a column by $v_a$, $v_a'$
    and $v_a''$, with $v_a$ in the top row, $v_a'$ in the middle row, and
    $v_a''$ in the bottom row.

    View $d_i$ as a connected component of $d$ and let $d_i'$ be the corresponding
    component of $d'$.
    Write $\mathsf{top}(d_i) = \{v_1, \ldots, v_\ell\}$ and
    $\mathsf{top}(d_i') = \{v_1', \ldots, v_\ell'\}$.
    To prove $d_i \cap [k] \in \mathsf{bot}(d)$,
    it suffices to prove $v_1'', \ldots, v_\ell''$ belong to the same connected
    component of $d'$.

    For each $v_a''$, pick $u_a'$ such that $u_a'$ and $v_a''$ are in the same connected component of $d'$.
    Then $u_a$ and $v_a'$ are in the same connected component of $d$.
    Since $v_1', \ldots, v_\ell' \in d_i'$, and $d_i'$ is a connected component,
    it follows that $u_1, \ldots, u_\ell$ are in the same connected component of $dd'$.
    Since $d$ is idempotent, the connected components of $dd'$ and $d$
    coincide (up to relabelling).
    Thus, $u_1, \ldots, u_\ell$ are in the same connected component of $d$.
    Consequently, $u_1', \ldots, u_\ell'$ are in the same connected component of $d'$,
    and thus so are $v_1'', \ldots, v_\ell''$.
\end{proof}

\begin{example} There are 5 idempotents of $\mathcal{U}_3$ corresponding to the 5 set partitions of $[3]$.
    These are depicted below:
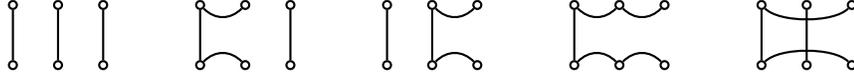
\begin{equation*}
    \begin{tikzpicture}[partition-diagram, idempotent]
    \node[vertex] (G--3) at (3.0, -1) [shape = circle, draw] {};
    \node[vertex] (G-3) at (3.0, 1) [shape = circle, draw] {};
    \node[vertex] (G--2) at (1.5, -1) [shape = circle, draw] {};
    \node[vertex] (G-2) at (1.5, 1) [shape = circle, draw] {};
    \node[vertex] (G--1) at (0.0, -1) [shape = circle, draw] {};
    \node[vertex] (G-1) at (0.0, 1) [shape = circle, draw] {};
    \draw[] (G-3) .. controls +(0, -1) and +(0, 1) .. (G--3);
    \draw[] (G-2) .. controls +(0, -1) and +(0, 1) .. (G--2);
    \draw[] (G-1) .. controls +(0, -1) and +(0, 1) .. (G--1);
    \end{tikzpicture}
    \qquad\quad
    \begin{tikzpicture}[partition-diagram, idempotent]
    \node[vertex] (G--3) at (3.0, -1) [shape = circle, draw] {};
    \node[vertex] (G-3) at (3.0, 1) [shape = circle, draw] {};
    \node[vertex] (G--2) at (1.5, -1) [shape = circle, draw] {};
    \node[vertex] (G--1) at (0.0, -1) [shape = circle, draw] {};
    \node[vertex] (G-1) at (0.0, 1) [shape = circle, draw] {};
    \node[vertex] (G-2) at (1.5, 1) [shape = circle, draw] {};
    \draw[] (G-3) .. controls +(0, -1) and +(0, 1) .. (G--3);
    \draw[] (G-1) .. controls +(0.5, -0.5) and +(-0.5, -0.5) .. (G-2);
    \draw[] (G--2) .. controls +(-0.5, 0.5) and +(0.5, 0.5) .. (G--1);
    \draw[] (G--1) .. controls +(0, 1) and +(0, -1) .. (G-1);
    \end{tikzpicture}
    \qquad\quad
    \begin{tikzpicture}[partition-diagram, idempotent]
    \node[vertex] (G--3) at (3.0, -1) [shape = circle, draw] {};
    \node[vertex] (G--2) at (1.5, -1) [shape = circle, draw] {};
    \node[vertex] (G-2) at (1.5, 1) [shape = circle, draw] {};
    \node[vertex] (G-3) at (3.0, 1) [shape = circle, draw] {};
    \node[vertex] (G--1) at (0.0, -1) [shape = circle, draw] {};
    \node[vertex] (G-1) at (0.0, 1) [shape = circle, draw] {};
    \draw[] (G-2) .. controls +(0.5, -0.5) and +(-0.5, -0.5) .. (G-3);
    \draw[] (G--3) .. controls +(-0.5, 0.5) and +(0.5, 0.5) .. (G--2);
    \draw[] (G--2) .. controls +(0, 1) and +(0, -1) .. (G-2);
    \draw[] (G-1) .. controls +(0, -1) and +(0, 1) .. (G--1);
    \end{tikzpicture}
    \qquad\quad
    \begin{tikzpicture}[partition-diagram]
    \node[vertex] (G--3) at (3.0, -1) [shape = circle, draw] {};
    \node[vertex] (G--2) at (1.5, -1) [shape = circle, draw] {};
    \node[vertex] (G--1) at (0.0, -1) [shape = circle, draw] {};
    \node[vertex] (G-1) at (0.0, 1) [shape = circle, draw] {};
    \node[vertex] (G-2) at (1.5, 1) [shape = circle, draw] {};
    \node[vertex] (G-3) at (3.0, 1) [shape = circle, draw] {};
    \draw[] (G-1) .. controls +(0.5, -0.5) and +(-0.5, -0.5) .. (G-2);
    \draw[] (G-2) .. controls +(0.5, -0.5) and +(-0.5, -0.5) .. (G-3);
    \draw[] (G--3) .. controls +(-0.5, 0.5) and +(0.5, 0.5) .. (G--2);
    \draw[] (G--2) .. controls +(-0.5, 0.5) and +(0.5, 0.5) .. (G--1);
    \draw[] (G--1) .. controls +(0, 1) and +(0, -1) .. (G-1);
    \end{tikzpicture}
    \qquad\quad
    \begin{tikzpicture}[partition-diagram, idempotent]
    \node[vertex] (G--3) at (3.0, -1) [shape = circle, draw] {};
    \node[vertex] (G--1) at (0.0, -1) [shape = circle, draw] {};
    \node[vertex] (G-1) at (0.0, 1) [shape = circle, draw] {};
    \node[vertex] (G-3) at (3.0, 1) [shape = circle, draw] {};
    \node[vertex] (G--2) at (1.5, -1) [shape = circle, draw] {};
    \node[vertex] (G-2) at (1.5, 1) [shape = circle, draw] {};
    \draw[] (G-1) .. controls +(0.6, -0.6) and +(-0.6, -0.6) .. (G-3);
    \draw[] (G--3) .. controls +(-0.6, 0.6) and +(0.6, 0.6) .. (G--1);
    \draw[] (G--1) .. controls +(0, 1) and +(0, -1) .. (G-1);
    \draw[] (G-2) .. controls +(0, -1) and +(0, 1) .. (G--2);
    \end{tikzpicture}
\end{equation*}
\end{example}

\subsubsection{Permutation-idempotent and idempotent-permutation decompositions}
We now prove that every uniform block permutation can be factored as the product of
a permutation and an idempotent; for example,
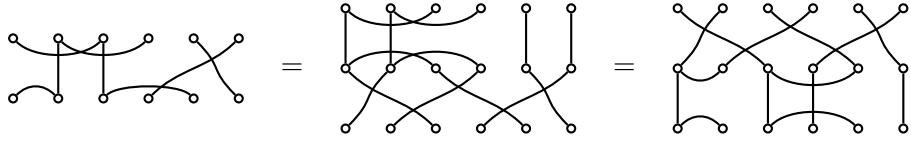
\begin{equation*}
    \begin{tikzpicture}[partition-diagram]
        \node[vertex] (G--1) at (0.0, -1) [shape = circle, draw] {};
        \node[vertex] (G--2) at (1.5, -1) [shape = circle, draw] {};
        \node[vertex] (G--3) at (3.0, -1) [shape = circle, draw] {};
        \node[vertex] (G--4) at (4.5, -1) [shape = circle, draw] {};
        \node[vertex] (G--5) at (6.0, -1) [shape = circle, draw] {};
        \node[vertex] (G--6) at (7.5, -1) [shape = circle, draw] {};
        \node[vertex] (G-1) at (0.0, 1) [shape = circle, draw] {};
        \node[vertex] (G-2) at (1.5, 1) [shape = circle, draw] {};
        \node[vertex] (G-3) at (3.0, 1) [shape = circle, draw] {};
        \node[vertex] (G-4) at (4.5, 1) [shape = circle, draw] {};
        \node[vertex] (G-5) at (6.0, 1) [shape = circle, draw] {};
        \node[vertex] (G-6) at (7.5, 1) [shape = circle, draw] {};
        \draw[] (G-2) -- (G--2);
        \draw[] (G-3) -- (G--3);
        \draw[] (G--1) .. controls +(0.5, 0.5) and +(-0.5, 0.5) .. (G--2);
        \draw[] (G-1) .. controls +(0.7, -0.7) and +(-0.7, -0.7) .. (G-3);
        \draw[] (G-2) .. controls +(0.7, -0.7) and +(-0.7, -0.7) .. (G-4);
        \draw[] (G--3) .. controls +(0.5, 0.5) and +(-0.5, 0.5) .. (G--5);
        \draw[] (G-6) .. controls +(-1, -1) and +(1, 1) .. (G--4);
        \draw[] (G-5) .. controls +(1, -1) and +(-1, 1) .. (G--6);
    \end{tikzpicture}
    \quad = \quad
    \begin{tikzpicture}[partition-diagram]
        \node[vertex] (G--1p) at (0.0, -2) [shape = circle, draw] {};
        \node[vertex] (G--2p) at (1.5, -2) [shape = circle, draw] {};
        \node[vertex] (G--3p) at (3.0, -2) [shape = circle, draw] {};
        \node[vertex] (G--4p) at (4.5, -2) [shape = circle, draw] {};
        \node[vertex] (G--5p) at (6.0, -2) [shape = circle, draw] {};
        \node[vertex] (G--6p) at (7.5, -2) [shape = circle, draw] {};
        \node[vertex] (G--1) at (0.0, 0) [shape = circle, draw] {};
        \node[vertex] (G--2) at (1.5, 0) [shape = circle, draw] {};
        \node[vertex] (G--3) at (3.0, 0) [shape = circle, draw] {};
        \node[vertex] (G--4) at (4.5, 0) [shape = circle, draw] {};
        \node[vertex] (G--5) at (6.0, 0) [shape = circle, draw] {};
        \node[vertex] (G--6) at (7.5, 0) [shape = circle, draw] {};
        \node[vertex] (G-1) at (0.0, 2) [shape = circle, draw] {};
        \node[vertex] (G-2) at (1.5, 2) [shape = circle, draw] {};
        \node[vertex] (G-3) at (3.0, 2) [shape = circle, draw] {};
        \node[vertex] (G-4) at (4.5, 2) [shape = circle, draw] {};
        \node[vertex] (G-5) at (6.0, 2) [shape = circle, draw] {};
        \node[vertex] (G-6) at (7.5, 2) [shape = circle, draw] {};
        \draw[] (G--1) .. controls +(0, 1) and +(0, -1) .. (G-1);
        \draw[] (G-1) .. controls +(0.7, -0.7) and +(-0.7, -0.7) .. (G-3);
        \draw[] (G--1) .. controls +(0.7, 0.7) and +(-0.7, 0.7) .. (G--3);
        \draw[] (G--2) .. controls +(0, 1) and +(0, -1) .. (G-2);
        \draw[] (G-2) .. controls +(0.7, -0.7) and +(-0.7, -0.7) .. (G-4);
        \draw[] (G--2) .. controls +(0.7, 0.7) and +(-0.7, 0.7) .. (G--4);
        \draw[] (G--5) .. controls +(0, 1) and +(0, -1) .. (G-5);
        \draw[] (G--6) .. controls +(0, 1) and +(0, -1) .. (G-6);
        \draw[] (G--1) .. controls +( 1, -1) and +(-1,  1) .. (G--3p);
        \draw[] (G--2) .. controls +(-1, -1) and +( 1,  1) .. (G--1p);
        \draw[] (G--3) .. controls +( 1, -1) and +(-1,  1) .. (G--5p);
        \draw[] (G--4) .. controls +(-1, -1) and +( 1,  1) .. (G--2p);
        \draw[] (G--5) .. controls +( 1, -1) and +(-1,  1) .. (G--6p);
        \draw[] (G--6) .. controls +(-1, -1) and +( 1,  1) .. (G--4p);
    \end{tikzpicture}
    \quad = \quad
    \begin{tikzpicture}[partition-diagram]
        \node[vertex] (G--1p) at (0.0, -2) [shape = circle, draw] {};
        \node[vertex] (G--2p) at (1.5, -2) [shape = circle, draw] {};
        \node[vertex] (G--3p) at (3.0, -2) [shape = circle, draw] {};
        \node[vertex] (G--4p) at (4.5, -2) [shape = circle, draw] {};
        \node[vertex] (G--5p) at (6.0, -2) [shape = circle, draw] {};
        \node[vertex] (G--6p) at (7.5, -2) [shape = circle, draw] {};
        \node[vertex] (G--1) at (0.0, 0) [shape = circle, draw] {};
        \node[vertex] (G--2) at (1.5, 0) [shape = circle, draw] {};
        \node[vertex] (G--3) at (3.0, 0) [shape = circle, draw] {};
        \node[vertex] (G--4) at (4.5, 0) [shape = circle, draw] {};
        \node[vertex] (G--5) at (6.0, 0) [shape = circle, draw] {};
        \node[vertex] (G--6) at (7.5, 0) [shape = circle, draw] {};
        \node[vertex] (G-1) at (0.0, 2) [shape = circle, draw] {};
        \node[vertex] (G-2) at (1.5, 2) [shape = circle, draw] {};
        \node[vertex] (G-3) at (3.0, 2) [shape = circle, draw] {};
        \node[vertex] (G-4) at (4.5, 2) [shape = circle, draw] {};
        \node[vertex] (G-5) at (6.0, 2) [shape = circle, draw] {};
        \node[vertex] (G-6) at (7.5, 2) [shape = circle, draw] {};
        \draw[] (G-1) .. controls +( 1, -1) and +(-1,  1) .. (G--3);
        \draw[] (G-2) .. controls +(-1, -1) and +( 1,  1) .. (G--1);
        \draw[] (G-3) .. controls +( 1, -1) and +(-1,  1) .. (G--5);
        \draw[] (G-4) .. controls +(-1, -1) and +( 1,  1) .. (G--2);
        \draw[] (G-5) .. controls +( 1, -1) and +(-1,  1) .. (G--6);
        \draw[] (G-6) .. controls +(-1, -1) and +( 1,  1) .. (G--4);
        \draw[] (G--1p) .. controls +(0, 1) and +(0, -1) .. (G--1);
        \draw[] (G--1) .. controls +(0.5, -0.5) and +(-0.5, -0.5) .. (G--2);
        \draw[] (G--1p) .. controls +(0.5, 0.5) and +(-0.5, 0.5) .. (G--2p);
        \draw[] (G--3p) .. controls +(0, 1) and +(0, -1) .. (G--3);
        \draw[] (G--3) .. controls +(0.7, -0.7) and +(-0.7, -0.7) .. (G--5);
        \draw[] (G--3p) .. controls +(0.7, 0.7) and +(-0.7, 0.7) .. (G--5p);
        \draw[] (G--4p) .. controls +(0, 1) and +(0, -1) .. (G--4);
        \draw[] (G--6p) .. controls +(0, 1) and +(0, -1) .. (G--6);
    \end{tikzpicture}
\end{equation*}
It turns out that the idempotents in the above decomposition are determined by
$d$: they are $e_{\mathsf{top}(d)}$ and $e_{\mathsf{bot}(d)}$, respectively.
However, the permutation is \emph{not} unique.

\begin{prop}
    \label{prop:standard-decomposition}
    For every $d \in \mathcal{U}_k$ and every $\sigma \in \mathfrak{S}_k$
    satisfying $\sigma(B \cap [k]) = \overline{B} \cap [k]$ for all blocks
    $B$ of $d$, we have
    \begin{equation*}
        d = e_{\mathsf{top}(d)} \, \sigma = \sigma e \,_{\mathsf{bot}(d)}.
    \end{equation*}
    Consequently,
    \begin{equation*}
        \mathcal{U}_k
        = E(\mathcal{U}_k) \, \mathfrak{S}_k
        = \mathfrak{S}_k \, E(\mathcal{U}_k).
    \end{equation*}
\end{prop}

\begin{remark}
    A monoid $M$ is said to be \defn{factorizable} if $M = G E$, where $G$ is
    a subgroup of $M$ and $E$ is a set of idempotents in $M$.
    Thus, Proposition~\ref{prop:standard-decomposition} states that
    $\mathcal{U}_k$ is factorizable.
\end{remark}

\begin{proof}[Proof of Proposition~\ref{prop:standard-decomposition}]
    Recall from Section~\ref{sssection:bijection} that every uniform block
    permutation $d \in \mathcal{U}_k$ is associated with the size-preserving
    bijection defined by
    \begin{eqnarray*}
        \mathsf{top}(d) & \longrightarrow & \mathsf{bot}(d) \\
        B \cap [k] & \longmapsto & \overline{B} \cap [k]
    \end{eqnarray*}
    where $B$ ranges over all blocks of $d$.
    If $\sigma$ is any permutation in $\mathfrak{S}_k$ satisfying
    $\sigma(B \cap [k]) = \overline{B} \cap [k]$ for all blocks $B$,
    then $\sigma^{-1} \circ d$ maps each block $B \cap [k]$ to itself,
    and so it is the bijection associated with the idempotent $e_{\mathsf{top}(d)}$.
    Thus, in $\mathcal{U}_k$ we have $d \sigma^{-1} = e_{\mathsf{top}(d)}$.
\end{proof}

\subsubsection{Properties of idempotents}

The following lemma collects some useful properties of the idempotents in
$\mathcal{U}_k$ that we use throughout the paper.
They can be proved directly from the definition of the product of two diagrams.

\begin{lemma}\label{lem:eproperties}
Let $\pi\vdash [k]$, $\tau\in \mathfrak{S}_k$ and $d \in \mathcal{U}_k$.
\begin{enumerate}[label=(\alph*)]
\item $\widetilde{e}_\pi = e_\pi$.
\item $\mathsf{top}(\tau e_\pi) = \tau^{-1}(\pi)$ and $\mathsf{bot}(e_\pi\tau) = \tau(\pi)$.
\item $\tau e_\pi \tau^{-1} = e_{\tau^{-1}(\pi)}$; consequently, $e_\pi \tau = \tau e_{\tau(\pi)}$ and $\tau e_\pi  =  e_{\tau^{-1}(\pi)}\tau$.
\item $e_{\mathsf{top}(d)} d = d$ and $de_{\mathsf{bot}(d)} = d$.
\item $\mathsf{bot}(d e_\pi)$ and $\mathsf{top}(e_\pi d)$ are coarser than or equal to $\pi$.
\end{enumerate}
\end{lemma}

\subsection{Presentation of $\mathcal{U}_k$}
\label{sec:presentation}
We recall here a known presentation of $\mathcal{U}_k$; see
\cite{FitzGerald.2003, Kosuda.2000, Kosuda.2001}.
For $1 \leqslant i < k$,
set $s_i = \{\{1,{\overline 1}\},\ldots,\{i,{\overline {i+1}}\}, \{i+1,{\overline i}\},\ldots, \{k,{\overline k}\}\}$,
which corresponds to the permutation in $\mathfrak{S}_k$ that swaps $i$ and $i+1$,
and $b_i = \{\{1,{\overline 1}\},\ldots,\{i,i+1,{\overline i},{\overline {i+1}}\}, \ldots, \{k,{\overline k}\}\}$.
As diagrams
\begin{equation*}
s_i =
\begin{tikzpicture}[scale = 0.5,thick, baseline={(0,-1ex/2)}]
\node[vertex] (G--6) at (7.5, -1) [shape = circle, draw] {};
\node[vertex] (G-6) at (7.5, 1) [shape = circle, draw] {};
\node[vertex] (G--5) at (5.5, -1) [shape = circle, draw] {};
\node[vertex] (G-5) at (5.5, 1) [shape = circle, draw] {};
\node[vertex] (G--4) at (4.5, -1) [shape = circle, draw] {};
\node[vertex] (G-3) at (3.0, 1) [shape = circle, draw] {};
\node[vertex] (G--3) at (3.0, -1) [shape = circle, draw] {};
\node[vertex] (G-4) at (4.5, 1) [shape = circle, draw] {};
\node[vertex] (G--2) at (2.0, -1) [shape = circle, draw] {};
\node[vertex] (G-2) at (2.0, 1) [shape = circle, draw] {};
\node[vertex] (G--1) at (0.0, -1) [shape = circle, draw] {};
\node[vertex] (G-1) at (0.0, 1) [shape = circle, draw] {};
\node[] at (1.0,0) {...};
\node[] at (6.5,0) {...};
\draw[] (G-6) .. controls +(0, -1) and +(0, 1) .. (G--6);
\draw[] (G-5) .. controls +(0, -1) and +(0, 1) .. (G--5);
\draw[] (G-3) .. controls +(0.75, -1) and +(-0.75, 1) .. (G--4);
\draw[] (G-4) .. controls +(-0.75, -1) and +(0.75, 1) .. (G--3);
\draw[] (G-2) .. controls +(0, -1) and +(0, 1) .. (G--2);
\draw[] (G-1) .. controls +(0, -1) and +(0, 1) .. (G--1);
\end{tikzpicture}
\qquad\qquad
b_i =
\begin{tikzpicture}[scale = 0.5,thick, baseline={(0,-1ex/2)}]
\node[vertex] (G--6) at (7.5, -1) [shape = circle, draw] {};
\node[vertex] (G-6) at (7.5, 1) [shape = circle, draw] {};
\node[vertex] (G-5) at (5.5, 1) [shape = circle, draw] {};
\node[vertex] (G--5) at (5.5, -1) [shape = circle, draw] {};
\node[vertex] (G--4) at (4.5, -1) [shape = circle, draw] {};
\node[vertex] (G-4) at (4.5, 1) [shape = circle, draw] {};
\node[vertex] (G--3) at (3.0, -1) [shape = circle, draw] {};
\node[vertex] (G-3) at (3.0, 1) [shape = circle, draw] {};
\node[vertex] (G--2) at (2, -1) [shape = circle, draw] {};
\node[vertex] (G--1) at (0.0, -1) [shape = circle, draw] {};
\node[vertex] (G-1) at (0.0, 1) [shape = circle, draw] {};
\node[vertex] (G-2) at (2, 1) [shape = circle, draw] {};
\draw[] (G-6) .. controls +(0, -1) and +(0, 1) .. (G--6);
\node[] at (1.0,0) {...};
\node[] at (6.5,0) {...};
\draw[] (G-3) .. controls +(0, -1) and +(0, 1) .. (G--3);
\draw[] (G-3) .. controls +(0.5, -0.5) and +(-0.5, -0.5) .. (G-4);
\draw[] (G-5) .. controls +(0, -1) and +(0, 1) .. (G--5);
\draw[] (G--4) .. controls +(-0.5, 0.5) and +(0.5, 0.5) .. (G--3);
\draw[] (G--1) .. controls +(0, 1) and +(0, -1) .. (G-1);
\draw[] (G--2) .. controls +(0, 1) and +(0, -1) .. (G-2);
\end{tikzpicture}
\end{equation*}
Then $s_i, b_i$ for $1\leqslant i \leqslant k-1$ generate $\mathcal{U}_k$ subject to the following relations:
\begin{align*}
   & (1)\  s_i^2 = 1, \quad 1\leqslant i \leqslant k-1&   &(2)\  b_i^2 = b_i, \quad 1\leqslant i \leqslant k-1\\
        &(3)\   s_is_{i+1}s_i = s_{i+1}s_is_{i+1}, \quad 1\leqslant i \leqslant k-2&   &(4)\  s_i b_{i+1}s_i = s_{i+1}b_i s_{i+1} , \quad 1\leqslant i \leqslant k-2 \\
    &(5)\  s_i s_j = s_js_i, \quad |i-j|>1 &   & (6)\  b_is_j = s_j b_i, \quad |i-j|>1 \\
    &(7)\  b_i s_i =s_i b_i = b_i , \quad 1\leqslant i \leqslant k-1 &  &(8)\  b_i b_j = b_j b_i, \quad 1\leqslant i \leqslant k-1.
\end{align*}

\subsection{$\mathcal{U}_k$ is an inverse monoid}
\label{sec:inverse-monoid}

We prove that $\mathcal{U}_k$ is an inverse monoid,
which will allow us to make use of known results for this class of monoids
(see \cite[Chapter 3]{Steinberg.2016}).

A monoid $M$ is called an \defn{inverse monoid} if for every $x \in M$,
there exists a unique element $x^\ast \in M$,
called the \defn{generalized inverse} of $x$,
satisfying $x x^\ast x = x$ and $x^\ast x x^\ast = x^\ast$.

Given a set partition $d$ of $[k] \cap [\bar{k}]$,
let $\widetilde{d}$ denote the set partition whose diagram is obtained by
reflecting the diagram of $d$ across a horizontal line.
Note that if $d$ is uniform, then $\widetilde{d}$ is also uniform.
If $d$ is a permutation in $\mathfrak{S}_k$,
then $\widetilde{d}$ is precisely the inverse of the permutation.
Furthermore, it can be verified directly using diagrams that
\begin{equation*}
    d\widetilde{d}d = d,
    \qquad\quad
    \widetilde{d}d\widetilde{d} = \widetilde{d},
    \qquad\quad
    \widetilde{dd'} = \widetilde{d'}\widetilde{d},
    \qquad\quad
    \tilde{\tilde{d}} = d.
\end{equation*}
Using the fact that the idempotents of $\mathcal{U}_k$ commute
(Lemma~\ref{lem:idempotents}), one can prove that $\widetilde{d}$ is the unique
element satisfying
$d\widetilde{d}d = d$ and $\widetilde{d}d\widetilde{d} = \widetilde{d}$;
see the proof of \cite[Theorem~3.2]{Steinberg.2016}.
\begin{prop} \leavevmode
    \begin{enumerate}
        \item
            $\mathcal{U}_k$ is an inverse monoid,
            where the generalized inverse of $d \in \mathcal{U}_k$ is $\widetilde{d}$.

        \item
            $E(\mathcal{U}_k)$ is a commutative inverse monoid
            that is generated by
            $(i+1,j) b_i (i+1,j)$ for $1\leqslant i < j \leqslant k$,
            where $(i+1,j)$ is the transposition in $\mathfrak{S}_k$ that swaps $i+1$ and $j$.
    \end{enumerate}
\end{prop}

\section{Irreducible representations of $\mathcal{U}_k$}
\label{section.irreps of Uk}

We will develop the representation theory of $\mathcal{U}_k$
using results from the representation theory of finite monoids
as presented in the excellent book by Steinberg~\cite{Steinberg.2016}.
We begin with a very brief overview to guide our development.

Let $M$ be a finite monoid.
Given an idempotent $e \in M$, there is a unique largest subgroup of $M$ that
contains $e$, which is called the maximal subgroup of $M$ at the
idempotent $e$ and denoted by $G_e$.
The irreducible (complex) representations of $M$ (\ie, the simple
$\CC{M}$-modules) are determined by the irreducible representations of the
maximal subgroups $G_e$. We describe the maximal subgroups of $\mathcal{U}_k$
in Section~\ref{sec:maximalsubgroups}.
Two maximal subgroups $G_e$ and $G_f$ of $M$ are isomorphic if the idempotents
$e$ and $f$ are \emph{$\mathscr{J}$-equivalent}.
This equivalence relation is defined in Section~\ref{section.J classes}, where
we determine the $\mathscr{J}$-classes of $\mathcal{U}_k$.
In Section~\ref{sec:irrepsGe} we describe the irreducible representations
of the maximal subgroups of $\mathcal{U}_k$.
The construction makes use of an auxiliary representation called the
Sch\"utzenberger representation that we describe in
Section~\ref{sec:Schutzenberger-representations}.
In Section~\ref{sec:irreps-of-UBP-algebra} we construct all the
irreducible representations of $\mathcal{U}_k$.
Finally, in Section~\ref{sec:irreps-of-UBP-algebra-2} we give a tableau model
for the irreducible $\mathcal{U}_k$-representations.

\subsection{Maximal subgroups of $\mathcal{U}_k$}
\label{sec:maximalsubgroups}

As explained above,
the representation theory of $\mathcal{U}_k$
can be expressed in terms of the representation theory of
its maximal subgroups, and the subsequent results will describe their structure.
The next result identifies the maximal subgroups of $\mathcal{U}_k$.
Recall that every idempotent of $\mathcal{U}_k$ is of the form
$e_\pi$, where $\pi$ is a set partition of $[k]$ (\cf Lemma~\ref{lem:idempotents}).

\begin{lemma}
    \label{lemma:maximal-subgroups}
    Let $e_\pi \in \mathcal{U}_k$ be the idempotent corresponding to a set partition
    $\pi \vdash [k]$.
    The maximal subgroup of $\mathcal{U}_k$ at the idempotent $e_\pi$ is
    \[
        G_{e_\pi} = \{d \in \mathcal{U}_k: d \widetilde{d} = \widetilde{d} d = e_\pi\}
                  = \{d \in \mathcal{U}_k: \mathsf{top}(d) =\mathsf{bot}(d) = \pi\}.
    \]
\end{lemma}

\begin{proof}
    Let $G_{e_\pi}$ be the maximal subgroup of $\mathcal{U}_k$ associated with $e_\pi$.
    Since $\mathcal{U}_k$ is an inverse monoid, $G_{e_\pi}$ consists of all
    elements $d$ of $\mathcal{U}_k$ such that $\widetilde{d}d = d\widetilde{d} = e_\pi$ \cite[Corollary 3.6]{Steinberg.2016}.

    Suppose $\mathsf{top}(d) = \mathsf{bot}(d) = \pi$.
    By Proposition \ref{prop:standard-decomposition},
    there exists $\sigma \in \mathfrak{S}_k$ satisfying $d = \sigma e_\pi = e_\pi \sigma$
    and $\sigma(\pi) = \pi$. Thus,
    \begin{align*}
        \widetilde{d} d = (e_\pi \sigma^{-1}) (\sigma e_\pi) = e_\pi
        \quad\text{and}\quad
        d \widetilde{d} = (\sigma e_\pi) (e_\pi \sigma^{-1}) = e_{\sigma^{-1}(\pi)} = e_{\pi}.
    \end{align*}

    Conversely, suppose that $\widetilde{d}d = d\widetilde{d} = e_\pi$.
    Write $d = \sigma e_{\mathsf{bot}(d)}$ with $\sigma \in \mathfrak{S}_k$.
    Then
    \begin{align*}
        e_\pi & = \widetilde{d}d = \big(e_{\mathsf{bot}(d)} \sigma^{-1}\big) \big(\sigma e_{\mathsf{bot}(d)}\big) = e_{\mathsf{bot}(d)}, \\
        e_\pi &= d \widetilde{d} = \big(\sigma e_{\mathsf{bot}(d)}\big) \big(e_{\mathsf{bot}(d)} \sigma^{-1}\big) = e_{\sigma^{-1}(\mathsf{bot}(d))},
    \end{align*}
    which imply that $\pi = \mathsf{bot}(d)$; and $\pi = \sigma^{-1}(\mathsf{bot}(d)) = \mathsf{top}(d)$.
\end{proof}

Next, we prove that each maximal subgroup is a direct product of symmetric groups.
For any set $B$, let $\mathfrak{S}_B$ denote the permutation group of the elements in $B$.

\begin{prop}\label{prop:structGe}
    Let $\pi = \{\pi_1, \ldots, \pi_\ell\}\vdash [k]$.
    \begin{enumerate}
        \item
            $G_{e_\pi} = \mathfrak{S}_{\pi^{(1)}} \times \mathfrak{S}_{\pi^{(2)}}
            \times \cdots \times \mathfrak{S}_{\pi^{(k)}}$,
            where $\pi^{(i)}$ is the set of blocks of $\pi$ of size $i$.

        \item
            Let $B_i=\{j : |\pi_j| = i \}$. If $d \in G_{e_{\pi}}$, then there
            exists $\tau \in \mathfrak{S}_{B_1}\times \cdots \times
            \mathfrak{S}_{B_k}$ such that
            \begin{equation*}
                d = \big\{
                    \pi_i \cup \overline{\pi_{\tau(i)}} : 1 \leqslant i \leqslant \ell
                \big\}.
            \end{equation*}
    \end{enumerate}
\end{prop}

\begin{proof}
If $d\in G_{e_\pi}$ then $\mathsf{top}(d) = \mathsf{bot}(d) = \pi$. Then $d$ is a bijection from the
blocks of $\pi$ to the blocks of $\pi$ such that blocks of the same size map to blocks of the same size.  This means that if
we consider only the blocks of size $i$ in $d$, there is a $\tau^{(i)} \in \mathfrak{S}_{B_i}$ that describes the bijection for these blocks.
Since this holds for any size $i$, the permutation $\tau : = \tau^{(1)} \times \tau^{(2)}\times \cdots \times \tau^{(k)} \in \mathfrak{S}_{B_1}
\times \mathfrak{S}_{B_2}\times \cdots \times \mathfrak{S}_{B_k}$ describes the bijection for all the blocks of all sizes.  Since $\pi_i \mapsto \pi_{\tau(i)}$,
the corresponding set partition has blocks $\pi_i \cup \overline{\pi}_{\tau(i)}$.
\end{proof}

\begin{cor} \label{cor:whatisGe}
For $\pi \vdash [k]$ with $\mathsf{type}(\pi)= (1^{a_1}2^{a_2}\ldots k^{a_k})$, we have
\[
    G_{e_\pi} \simeq \mathfrak{S}_{a_1} \times \mathfrak{S}_{a_2} \times \cdots \times \mathfrak{S}_{a_k}.
\]
\end{cor}

It follows from Corollary~\ref{cor:whatisGe} that
$G_{e_\pi}$ and $G_{e_\gamma}$ are isomorphic if $\mathsf{type}(\pi)
= \mathsf{type}(\gamma)$. Explicitly, $\sigma^{-1} G_{e_\pi} \sigma
= G_{e_\gamma}$ for any permutation $\sigma \in \mathfrak{S}_k$
satisfying $\sigma^{-1} e_\pi \sigma = e_\gamma$.

\begin{cor}\label{cor:isomGe}
If $\pi$ and $\gamma$ are two set partitions of $[k]$ satisfying $\mathsf{type}(\pi) = \mathsf{type}(\gamma)$, then $G_{e_\pi}$
is isomorphic to $G_{e_\gamma}$. In particular, there exists a $\sigma \in \mathfrak{S}_k$ such that $\sigma(\pi) = \gamma$ and
\[
    \sigma^{-1} G_{e_\pi} \sigma = G_{e_\gamma}.
\]
\end{cor}

\subsection{$\mathscr{J}$-classes}
\label{section.J classes}

Let $x$ and $y$ be elements of a monoid $M$.  We say that $x$ and $y$ are \defn{$\mathscr{J}$-equivalent} if $MxM = MyM$.
This is an equivalence relation; hence, it partitions the monoid $M$ into classes which are called the \defn{$\mathscr{J}$-classes} of $M$. We denote by \defn{$J_x$} the $\mathscr{J}$-class containing $x$.
In the next proposition we give a characterization for the $\mathscr{J}$-classes of $\mathcal{U}_k$ and show that they are indexed by partitions of $k$.

\begin{prop}\label{prop:charJclass}
Let $k$ be a nonnegative integer. 
\begin{enumerate}
    \item[(a)] Every $\mathscr{J}$-class of $\mathcal{U}_k$ contains an idempotent. 
    \item[(b)] Two elements $d, d' \in \mathcal{U}_k$ are in the same $\mathscr{J}$-class if and only if $\mathsf{type}(\mathsf{top}(d)) = \mathsf{type}(\mathsf{top}(d'))$.
    \item[(c)] The $\mathscr{J}$-classes are in bijection with partitions $\lambda$ of $k$. In particular, if $d\in \mathcal{U}_k$ and 
    $\mathsf{type}(\mathsf{top}(d)) = \lambda$, then 
    \[ {\defncolor J_\lambda} := J_d = \{ d' : \mathsf{type}(\mathsf{top}(d')) = \lambda\}. \]
    \item[(d)] If $\pi\vdash [k]$ and $\mathsf{type}(\pi) = \lambda$, then 
    \[ J_{\lambda} = \{\sigma e_\pi \tau : \sigma, \tau\in \mathfrak{S}_k\}.\]
\end{enumerate}
\end{prop}

\begin{proof}
(a) Every $d\in \mathcal{U}_k$ can be written using the permutation-idempotent
representation as $d = \sigma e_\pi$ for some $\sigma \in \mathfrak{S}_k$,
where $\pi = \mathsf{bot}(d)$. Then $\mathcal{U}_k d \mathcal{U}_k
= \mathcal{U}_k \sigma e_\pi \mathcal{U}_k = \mathcal{U}_k e_\pi \mathcal{U}_k$,
where the last equality follows because $\sigma$ is invertible in $\mathcal{U}_k$.
Thus, $J_d$ contains $e_\pi$.

(b)  Using the permutation-idempotent representation of $d$ and $d'$ we know that $d= \sigma e_\pi$ and $d' = \sigma' e_\gamma$ for some 
$\pi, \gamma \vdash [k]$ and $\sigma, \sigma' \in \mathfrak{S}_k$.
Furthermore, by Lemma~\ref{lem:eproperties}, $\mathsf{type}(\mathsf{top}(d)) = \mathsf{type}(\pi)$ and
$\mathsf{type}(\mathsf{top}(d')) = \mathsf{type}(\gamma)$.  Therefore, it suffices to prove that $e_\pi$ and $e_\gamma$ are in the same 
$\mathscr{J}$-class if and only if $\mathsf{type}(\pi) = \mathsf{type}(\gamma)$.  

If $\mathsf{type}(\pi) = \mathsf{type}(\gamma)$, then there exists a $\tau\in \mathfrak{S}_k$ such that $\gamma = \tau(\pi)$.  
Also, $\tau^{-1} e_\pi \tau= e_{\tau(\pi)}$. Hence, 
     $\mathcal{U}_k e_\gamma \mathcal{U}_k =
     \mathcal{U}_k \tau^{-1}e_\pi\tau \mathcal{U}_k = \mathcal{U}_k e_\pi \mathcal{U}_k$,
since $\mathcal{U}_k \tau^{-1} = \mathcal{U}_k = \tau \mathcal{U}_k$ due to the fact that $\tau$ is invertible in $\mathcal{U}_k$.
This implies that any two idempotents $e_\pi$ and $e_\gamma$ such that $\mathsf{type}(\pi) = \mathsf{type}(\gamma)$ are in the same $\mathscr{J}$-class.

Observe that $\mathcal{U}_k e_\pi \mathcal{U}_k$ contains elements $d$ such that $\mathsf{top}(d)$ and $\mathsf{bot}(d)$ are equal to or coarser than
$\tau(\pi)$ for some $\tau\in \mathfrak{S}_k$.
Assume that $\mathsf{type}(\gamma) \neq \mathsf{type}(\pi)$ with
$\gamma, \pi\vdash [k]$ and $\mathsf{type}(\gamma) = (1^{b_1} 2^{b_2}\ldots k^{b_k})$ and $\mathsf{type}(\pi) = (1^{a_1}2^{a_2}\ldots k^{a_k})$. 
Then for some $i$,  $a_i\neq b_i$.  Without loss of generality assume that $a_1=b_1, \ldots, a_{i-1}=b_{i-1}$ and $a_i<b_i$ for some $i$. 
This means that $\gamma$ has more blocks of size
$i$ than $\pi$.  Therefore, it is not possible for $\gamma$ to be coarser than $\tau(\pi)$ for any $\tau\in \mathfrak{S}_k$.  Hence,
$e_\gamma\notin \mathcal{U}_k e_\pi \mathcal{U}_k$.

(c) This is a direct consequence of (b).

(d) Multiplying $e_\pi$ on the left by a permutation results in a diagram whose top is a permutation of $\pi$ and hence has the same type as $\pi$.  
Similarly, multiplying $e_\pi$ on the right by a permutation results in a diagram with bottom that has the same type as $\pi$.
Hence $\{\sigma e_\pi \tau : \sigma, \tau\in \mathfrak{S}_k\} \subseteq J_\lambda$. Conversely, suppose $d \in J_\lambda$. We may write
$d= e_{\mathsf{top}(d)} \sigma$ for some $\sigma \in \mathfrak{S}_k$. Since $\mathsf{type}(\mathsf{top}(d)) = \mathsf{type}(\pi)$, there exists a 
$\tau \in \mathfrak{S}_k$ such that $\tau(\pi) = \mathsf{top}(d)$. Hence $d = e_{\tau(\pi)} \sigma = \tau^{-1} e_\pi \tau \sigma$, proving that
$J_\lambda \subseteq \{\sigma e_\pi \tau : \sigma, \tau\in \mathfrak{S}_k\}$.
\end{proof}

\begin{example}
There are three $\mathscr{J}$-classes for $\mathcal{U}_3$:
\begin{gather*}
\begin{split}
J_{(3)} &=
\left\{
\begin{tikzpicture}[scale=0.8, partition-diagram, idempotent]
    \node[vertex] (G--3) at (3.0, -1) [shape = circle, draw] {};
    \node[vertex] (G--2) at (1.5, -1) [shape = circle, draw] {};
    \node[vertex] (G--1) at (0.0, -1) [shape = circle, draw] {};
    \node[vertex] (G-1) at (0.0, 1) [shape = circle, draw] {};
    \node[vertex] (G-2) at (1.5, 1) [shape = circle, draw] {};
    \node[vertex] (G-3) at (3.0, 1) [shape = circle, draw] {};
    \draw[] (G-1) .. controls +(0.5, -0.5) and +(-0.5, -0.5) .. (G-2);
    \draw[] (G-2) .. controls +(0.5, -0.5) and +(-0.5, -0.5) .. (G-3);
    \draw[] (G--3) .. controls +(-0.5, 0.5) and +(0.5, 0.5) .. (G--2);
    \draw[] (G--2) .. controls +(-0.5, 0.5) and +(0.5, 0.5) .. (G--1);
    \draw[] (G--1) .. controls +(0, 1) and +(0, -1) .. (G-1);
    \end{tikzpicture}\right\},
    \\
   J_{(1,1,1)} &= \left\{
        \begin{gathered}
    \begin{tikzpicture}[scale=0.8, partition-diagram]
            \node[vertex] (G--3) at (3.0, -1) [shape = circle, draw] {};
            \node[vertex] (G-1) at (0.0, 1) [shape = circle, draw] {};
            \node[vertex] (G--2) at (1.5, -1) [shape = circle, draw] {};
            \node[vertex] (G-2) at (1.5, 1) [shape = circle, draw] {};
            \node[vertex] (G--1) at (0.0, -1) [shape = circle, draw] {};
            \node[vertex] (G-3) at (3.0, 1) [shape = circle, draw] {};
            \draw[] (G-1) .. controls +(1, -1) and +(-1, 1) .. (G--3);
            \draw[] (G-2) .. controls +(0, -1) and +(0, 1) .. (G--2);
            \draw[] (G-3) .. controls +(-1, -1) and +(1, 1) .. (G--1);
            \end{tikzpicture},
            \quad
    \begin{tikzpicture}[scale=0.8, partition-diagram]
            \node[vertex] (G--3) at (3.0, -1) [shape = circle, draw] {};
            \node[vertex] (G-3) at (3.0, 1) [shape = circle, draw] {};
            \node[vertex] (G--2) at (1.5, -1) [shape = circle, draw] {};
            \node[vertex] (G-2) at (1.5, 1) [shape = circle, draw] {};
            \node[vertex] (G--1) at (0.0, -1) [shape = circle, draw] {};
            \node[vertex] (G-1) at (0.0, 1) [shape = circle, draw] {};
            \draw[] (G-3) .. controls +(0, -1) and +(0, 1) .. (G--3);
            \draw[] (G-2) .. controls +(0, -1) and +(0, 1) .. (G--2);
            \draw[] (G-1) .. controls +(0, -1) and +(0, 1) .. (G--1);
            \end{tikzpicture},
            \quad
    \begin{tikzpicture}[scale=0.8, partition-diagram]
            \node[vertex] (G--3) at (3.0, -1) [shape = circle, draw] {};
            \node[vertex] (G-1) at (0.0, 1) [shape = circle, draw] {};
            \node[vertex] (G--2) at (1.5, -1) [shape = circle, draw] {};
            \node[vertex] (G-3) at (3.0, 1) [shape = circle, draw] {};
            \node[vertex] (G--1) at (0.0, -1) [shape = circle, draw] {};
            \node[vertex] (G-2) at (1.5, 1) [shape = circle, draw] {};
            \draw[] (G-1) .. controls +(1, -1) and +(-1, 1) .. (G--3);
            \draw[] (G-3) .. controls +(-0.75, -1) and +(0.75, 1) .. (G--2);
            \draw[] (G-2) .. controls +(-0.75, -1) and +(0.75, 1) .. (G--1);
            \end{tikzpicture},
            \quad
    \begin{tikzpicture}[scale=0.8, partition-diagram]
            \node[vertex] (G--3) at (3.0, -1) [shape = circle, draw] {};
            \node[vertex] (G-3) at (3.0, 1) [shape = circle, draw] {};
            \node[vertex] (G--2) at (1.5, -1) [shape = circle, draw] {};
            \node[vertex] (G-1) at (0.0, 1) [shape = circle, draw] {};
            \node[vertex] (G--1) at (0.0, -1) [shape = circle, draw] {};
            \node[vertex] (G-2) at (1.5, 1) [shape = circle, draw] {};
            \draw[] (G-3) .. controls +(0, -1) and +(0, 1) .. (G--3);
            \draw[] (G-1) .. controls +(0.75, -1) and +(-0.75, 1) .. (G--2);
            \draw[] (G-2) .. controls +(-0.75, -1) and +(0.75, 1) .. (G--1);
            \end{tikzpicture},
            \quad
    \begin{tikzpicture}[scale=0.8, partition-diagram]
            \node[vertex] (G--3) at (3.0, -1) [shape = circle, draw] {};
            \node[vertex] (G-2) at (1.5, 1) [shape = circle, draw] {};
            \node[vertex] (G--2) at (1.5, -1) [shape = circle, draw] {};
            \node[vertex] (G-3) at (3.0, 1) [shape = circle, draw] {};
            \node[vertex] (G--1) at (0.0, -1) [shape = circle, draw] {};
            \node[vertex] (G-1) at (0.0, 1) [shape = circle, draw] {};
            \draw[] (G-2) .. controls +(0.75, -1) and +(-0.75, 1) .. (G--3);
            \draw[] (G-3) .. controls +(-0.75, -1) and +(0.75, 1) .. (G--2);
            \draw[] (G-1) .. controls +(0, -1) and +(0, 1) .. (G--1);
            \end{tikzpicture},
            \quad
    \begin{tikzpicture}[scale=0.8, partition-diagram]
            \node[vertex] (G--3) at (3.0, -1) [shape = circle, draw] {};
            \node[vertex] (G-2) at (1.5, 1) [shape = circle, draw] {};
            \node[vertex] (G--2) at (1.5, -1) [shape = circle, draw] {};
            \node[vertex] (G-1) at (0.0, 1) [shape = circle, draw] {};
            \node[vertex] (G--1) at (0.0, -1) [shape = circle, draw] {};
            \node[vertex] (G-3) at (3.0, 1) [shape = circle, draw] {};
            \draw[] (G-2) .. controls +(0.75, -1) and +(-0.75, 1) .. (G--3);
            \draw[] (G-1) .. controls +(0.75, -1) and +(-0.75, 1) .. (G--2);
            \draw[] (G-3) .. controls +(-1, -1) and +(1, 1) .. (G--1);
            \end{tikzpicture}
        \end{gathered}
    \right\},
    \\
  J_{(2,1)} &=  \left\{
    \begin{gathered}
    \begin{tikzpicture}[scale=0.8, partition-diagram]
    \node[vertex] (G--3) at (3.0, -1) [shape = circle, draw] {};
    \node[vertex] (G--2) at (1.5, -1) [shape = circle, draw] {};
    \node[vertex] (G-2) at (1.5, 1) [shape = circle, draw] {};
    \node[vertex] (G-3) at (3.0, 1) [shape = circle, draw] {};
    \node[vertex] (G--1) at (0.0, -1) [shape = circle, draw] {};
    \node[vertex] (G-1) at (0.0, 1) [shape = circle, draw] {};
    \draw[] (G-2) .. controls +(0.5, -0.5) and +(-0.5, -0.5) .. (G-3);
    \draw[] (G--3) .. controls +(-0.5, 0.5) and +(0.5, 0.5) .. (G--2);
    \draw[] (G--2) .. controls +(0, 1) and +(0, -1) .. (G-2);
    \draw[] (G-1) .. controls +(0, -1) and +(0, 1) .. (G--1);
    \end{tikzpicture},
    \
    \begin{tikzpicture}[scale=0.8, partition-diagram]
    \node[vertex] (G--3) at (3.0, -1) [shape = circle, draw] {};
    \node[vertex] (G--2) at (1.5, -1) [shape = circle, draw] {};
    \node[vertex] (G-1) at (0.0, 1) [shape = circle, draw] {};
    \node[vertex] (G-3) at (3.0, 1) [shape = circle, draw] {};
    \node[vertex] (G--1) at (0.0, -1) [shape = circle, draw] {};
    \node[vertex] (G-2) at (1.5, 1) [shape = circle, draw] {};
    \draw[] (G-1) .. controls +(0.6, -0.6) and +(-0.6, -0.6) .. (G-3);
    \draw[] (G-3) .. controls +(0, -1) and +(0, 1) .. (G--3);
    \draw[] (G--3) .. controls +(-0.5, 0.5) and +(0.5, 0.5) .. (G--2);
    \draw[] (G-2) .. controls +(-0.75, -1) and +(0.75, 1) .. (G--1);
    \end{tikzpicture},
    \
    \begin{tikzpicture}[scale=0.8, partition-diagram]
    \node[vertex] (G--3) at (3.0, -1) [shape = circle, draw] {};
    \node[vertex] (G--1) at (0.0, -1) [shape = circle, draw] {};
    \node[vertex] (G-1) at (0.0, 1) [shape = circle, draw] {};
    \node[vertex] (G-3) at (3.0, 1) [shape = circle, draw] {};
    \node[vertex] (G--2) at (1.5, -1) [shape = circle, draw] {};
    \node[vertex] (G-2) at (1.5, 1) [shape = circle, draw] {};
    \draw[] (G-1) .. controls +(0.6, -0.6) and +(-0.6, -0.6) .. (G-3);
    \draw[] (G--3) .. controls +(-0.6, 0.6) and +(0.6, 0.6) .. (G--1);
    \draw[] (G--1) .. controls +(0, 1) and +(0, -1) .. (G-1);
    \draw[] (G-2) .. controls +(0, -1) and +(0, 1) .. (G--2);
    \end{tikzpicture},
    \
    \begin{tikzpicture}[scale=0.8, partition-diagram]
    \node[vertex] (G--3) at (3.0, -1) [shape = circle, draw] {};
    \node[vertex] (G--1) at (0.0, -1) [shape = circle, draw] {};
    \node[vertex] (G-2) at (1.5, 1) [shape = circle, draw] {};
    \node[vertex] (G-3) at (3.0, 1) [shape = circle, draw] {};
    \node[vertex] (G--2) at (1.5, -1) [shape = circle, draw] {};
    \node[vertex] (G-1) at (0.0, 1) [shape = circle, draw] {};
    \draw[] (G-2) .. controls +(0.5, -0.5) and +(-0.5, -0.5) .. (G-3);
    \draw[] (G-3) .. controls +(0, -1) and +(0, 1) .. (G--3);
    \draw[] (G--3) .. controls +(-0.6, 0.6) and +(0.6, 0.6) .. (G--1);
    \draw[] (G-1) .. controls +(0.75, -1) and +(-0.75, 1) .. (G--2);
    \end{tikzpicture},
    \
    \begin{tikzpicture}[scale=0.8, partition-diagram]
    \node[vertex] (G--3) at (3.0, -1) [shape = circle, draw] {};
    \node[vertex] (G--1) at (0.0, -1) [shape = circle, draw] {};
    \node[vertex] (G-1) at (0.0, 1) [shape = circle, draw] {};
    \node[vertex] (G-2) at (1.5, 1) [shape = circle, draw] {};
    \node[vertex] (G--2) at (1.5, -1) [shape = circle, draw] {};
    \node[vertex] (G-3) at (3.0, 1) [shape = circle, draw] {};
    \draw[] (G-1) .. controls +(0.5, -0.5) and +(-0.5, -0.5) .. (G-2);
    \draw[] (G--3) .. controls +(-0.6, 0.6) and +(0.6, 0.6) .. (G--1);
    \draw[] (G--1) .. controls +(0, 1) and +(0, -1) .. (G-1);
    \draw[] (G-3) .. controls +(-0.75, -1) and +(0.75, 1) .. (G--2);
    \end{tikzpicture},
    \
    \begin{tikzpicture}[scale=0.8, partition-diagram]
    \node[vertex] (G--3) at (3.0, -1) [shape = circle, draw] {};
    \node[vertex] (G--2) at (1.5, -1) [shape = circle, draw] {};
    \node[vertex] (G-1) at (0.0, 1) [shape = circle, draw] {};
    \node[vertex] (G-2) at (1.5, 1) [shape = circle, draw] {};
    \node[vertex] (G--1) at (0.0, -1) [shape = circle, draw] {};
    \node[vertex] (G-3) at (3.0, 1) [shape = circle, draw] {};
    \draw[] (G-1) .. controls +(0.5, -0.5) and +(-0.5, -0.5) .. (G-2);
    \draw[] (G--3) .. controls +(-0.5, 0.5) and +(0.5, 0.5) .. (G--2);
    \draw[] (G--2) -- (G-2);
    \draw[] (G-3) .. controls +(-1, -1) and +(1, 1) .. (G--1);
    \end{tikzpicture},
    \
    \begin{tikzpicture}[scale=0.8, partition-diagram]
    \node[vertex] (G--3) at (3.0, -1) [shape = circle, draw] {};
    \node[vertex] (G-2) at (1.5, 1) [shape = circle, draw] {};
    \node[vertex] (G--2) at (1.5, -1) [shape = circle, draw] {};
    \node[vertex] (G--1) at (0.0, -1) [shape = circle, draw] {};
    \node[vertex] (G-1) at (0.0, 1) [shape = circle, draw] {};
    \node[vertex] (G-3) at (3.0, 1) [shape = circle, draw] {};
    \draw[] (G-2) .. controls +(0.75, -1) and +(-0.75, 1) .. (G--3);
    \draw[] (G-1) .. controls +(0.6, -0.6) and +(-0.6, -0.6) .. (G-3);
    \draw[] (G--2) .. controls +(-0.5, 0.5) and +(0.5, 0.5) .. (G--1);
    \draw[] (G--1) .. controls +(0, 1) and +(0, -1) .. (G-1);
    \end{tikzpicture},
    \
    \begin{tikzpicture}[scale=0.8, partition-diagram]
    \node[vertex] (G--3) at (3.0, -1) [shape = circle, draw] {};
    \node[vertex] (G-1) at (0.0, 1) [shape = circle, draw] {};
    \node[vertex] (G--2) at (1.5, -1) [shape = circle, draw] {};
    \node[vertex] (G--1) at (0.0, -1) [shape = circle, draw] {};
    \node[vertex] (G-2) at (1.5, 1) [shape = circle, draw] {};
    \node[vertex] (G-3) at (3.0, 1) [shape = circle, draw] {};
    \draw[] (G-1) .. controls +(1, -1) and +(-1, 1) .. (G--3);
    \draw[] (G-2) .. controls +(0.5, -0.5) and +(-0.5, -0.5) .. (G-3);
    \draw[] (G--2) .. controls +(-0.5, 0.5) and +(0.5, 0.5) .. (G--1);
    \draw[] (G--2) -- (G-2);
    \end{tikzpicture},
    \
    \begin{tikzpicture}[scale=0.8, partition-diagram]
    \node[vertex] (G--3) at (3.0, -1) [shape = circle, draw] {};
    \node[vertex] (G-3) at (3.0, 1) [shape = circle, draw] {};
    \node[vertex] (G--2) at (1.5, -1) [shape = circle, draw] {};
    \node[vertex] (G--1) at (0.0, -1) [shape = circle, draw] {};
    \node[vertex] (G-1) at (0.0, 1) [shape = circle, draw] {};
    \node[vertex] (G-2) at (1.5, 1) [shape = circle, draw] {};
    \draw[] (G-3) .. controls +(0, -1) and +(0, 1) .. (G--3);
    \draw[] (G-1) .. controls +(0.5, -0.5) and +(-0.5, -0.5) .. (G-2);
    \draw[] (G--2) .. controls +(-0.5, 0.5) and +(0.5, 0.5) .. (G--1);
    \draw[] (G--1) .. controls +(0, 1) and +(0, -1) .. (G-1);
    \end{tikzpicture}
    \end{gathered}
    \right\}.
\end{split}
\end{gather*}
\end{example}

When constructing irreducible representations of $\mathcal{U}_k$, we need only one maximal subgroup for each $\mathscr{J}$-class.
It is useful to make this choice standard.
Recall that the $\mathscr{J}$-classes are indexed by partitions of $k$. Hence, if $\lambda = (1^{a_1}2^{a_2}\ldots k^{a_k})$, we define the
\defn{representative set partition} associated to $\lambda$ as
\begin{equation}
\label{equation.canonical}
    \pi_\lambda = \{\{1\}, \{2\}, \ldots, \{a_1\}, \{a_1+1, a_1+2\}, \ldots, \{a_1+2a_2-1, a_1+2a_2\}, \ldots \}.
\end{equation}
This is the set partition that uses $1, \ldots, a_1$ for blocks of size one, $a_1+1, \ldots, a_1+2a_2$ for blocks of size two, etc..
For example if $k=11$ and $\lambda = (1^42^23^1)$, then
\[
    \pi_\lambda = \{\{1\}, \{2\}, \{3\}, \{4\},\{5,6\},\{7,8\},\{9,10,11\}\}.
\]
In this case, we write $G_\lambda := G_{e_{\pi_\lambda}}$ and call it the
\defn{representative maximal subgroup} associated to $\lambda$.
Note that $\{G_\lambda : \lambda \vdash k\}$ is a set of maximal subgroups of
$\mathcal{U}_k$ with each subgroup associated with a distinct
$\mathscr{J}$-class of $\mathcal{U}_k$.

\subsection{Irreducible representations of the maximal subgroups}
\label{sec:irrepsGe}
We now describe the irreducible representations of the maximal subgroups $G_{e_\pi}$.
From Proposition~\ref{prop:structGe}, we know that
$G_{e_\pi} = \mathfrak{S}_{\pi^{(1)}} \times \mathfrak{S}_{\pi^{(2)}}
\times \cdots \times \mathfrak{S}_{\pi^{(k)}}$,
where $\pi^{(i)}$ is the set of blocks of $\pi$ of size $i$.
Hence, each irreducible representation of $G_{e_\pi}$
is isomorphic to a tensor product of the form
$V_1 \otimes V_2 \otimes \cdots \otimes V_k$,
where $V_i$ is an irreducible representation of $\mathfrak{S}_{\pi^{(i)}}$
for all $1 \leqslant i \leqslant k$.

Recall that the irreducible representations of the symmetric group
$\mathfrak{S}_{\pi^{(i)}}$ are indexed by partitions of $|\pi^{(i)}|$
and admit the following combinatorial description.
As a vector space, the irreducible representation indexed by the partition
$\lambda^{(i)}$ is spanned by the set of standard tableaux with entries the
blocks in $\pi^{(i)}$. The action of $\mathfrak{S}_{\pi^{(i)}}$ is given by
permuting the entries of the tableaux; however, the result may not be
a standard tableaux, in which case one uses the Garnir relations to express the
result as a linear combination of standard tableaux. For details, see
\cite[Chapter~2]{Sagan.2001}.

For a sequence of partitions
$\vec{\lambda} = (\lambda^{(1)}, \lambda^{(2)}, \ldots, \lambda^{(k)})$
with $\lambda^{(i)} \vdash |\pi^{(i)}|$,
let
\begin{equation*}
    {\defncolor V_{G_{e_\pi}}^{\vec{\lambda}}} =
    V_{\mathfrak{S}_{\pi^{(1)}}}^{\lambda^{(1)}}
    \otimes \cdots \otimes
    V_{\mathfrak{S}_{\pi^{(k)}}}^{\lambda^{(k)}},
\end{equation*}
where $V_{\mathfrak{S}_{\pi^{(i)}}}^{\lambda^{(i)}}$ is the irreducible
representation of $\mathfrak{S}_{\pi^{(i)}}$ indexed by $\lambda^{(i)}$.
By the above discussion, this is an irreducible representation of $G_{e_\pi}$
and all the irreducible representations of $G_{e_\pi}$ are of this form.
In particular, the irreducible representations of $G_{e_\pi}$ are indexed by
$k$-tuples of partitions
$\vec{\lambda} = (\lambda^{(1)}, \lambda^{(2)}, \ldots, \lambda^{(k)})$ such that
$\lambda^{(i)} \vdash |\pi^{(i)}|$ and $\sum_{i=1}^k i |\pi^{(i)}| = k$.
This implies that $\vec{\lambda}\in I_k$ with $I_k$ as defined in Equation~\eqref{equation.Ik}.

The combinatorial descriptions of each of the irreducible representations in
the tensor product above combine to give a combinatorial description of the
irreducible representations of $G_{e_\pi}$. In order to state it, we need the
following definitions.

\begin{definition}
Let $\pi \vdash [k]$.
A \defn{$\pi$-tableau} $\mathbf{T}$ of shape
$\vec{\lambda} = (\lambda^{(1)}, \lambda^{(2)}, \ldots, \lambda^{(k)})$
is a
$k$-tuple of tableaux $(T^{(1)}, T^{(2)}, \ldots, T^{(k)})$,
where  $T^{(i)}$ is a standard tableau of shape $\lambda^{(i)} = |\lambda^{(i)}|$
filled with the blocks in $\pi$ of size $i$.
\end{definition}

As a vector space, the irreducible $G_{e_\pi}$-representation
$V^{\vec{\lambda}}_{G_{e_\pi}}$ is spanned by the set of $\pi$-tableaux
of shape $\vec{\lambda}$.
By Proposition~\ref{prop:structGe},
every $\tau\in G_{e_\pi}$ can be expressed uniquely as $\tau = \tau^{(1)}
\tau^{(2)} \cdots  \tau^{(k)}$ with $\tau^{(i)} \in \mathfrak{S}_{\pi^{(i)}}$.
Then the action of $\tau$ on a $\pi$-tableau is given by
\begin{equation*}
    \tau \cdot \mathbf{T} = ( \tau^{(1)}\cdot T^{(1)}, \ldots, \tau^{(k)}\cdot T^{(k)}),
\end{equation*}
where $\tau^{(i)} \cdot T^{(i)}$ is obtained by permuting the entries of the tableau $T^{(i)}$.
The result may not be a $\pi$-tableau since it may not be standard,
in which case we use the Garnir straightening relations on each component
to express the result as a linear combination of standard $\pi$-tableaux
(for details see \cite[Section 2.6]{Sagan.2001} or \cite{CarreLascouxLeclerc}).

\begin{example} \label{ex:maximalsubgp}
Let $\pi = \{ \{1\}, \{2\}, \{3,4\}, \{5, 6\}\}$, so that $\mathsf{type}(\pi) = (1^2 2^2)$. In this case,
\begin{gather*}
G_{e_\pi} =
    \left\{\begin{tikzpicture}[partition-diagram]
    \node[vertex] (G--3) at (2.4, -1) [shape = circle, draw] {};
    \node[vertex] (G-1) at (0.0, 1) [shape = circle, draw] {};
    \node[vertex] (G--2) at (1.2, -1) [shape = circle, draw] {};
    \node[vertex] (G-2) at (1.2, 1) [shape = circle, draw] {};
    \node[vertex] (G--1) at (0.0, -1) [shape = circle, draw] {};
    \node[vertex] (G-3) at (2.4, 1) [shape = circle, draw] {};
    \node[vertex] (G-4) at (3.6, 1) [shape = circle, draw] {};
    \node[vertex] (G--4) at (3.6,-1) [shape = circle, draw] {};
    \node[vertex] (G-5) at (4.8, 1) [shape = circle, draw] {};
    \node[vertex] (G--5) at (4.8,-1) [shape = circle, draw] {};
    \node[vertex] (G-6) at (6.0, 1) [shape = circle, draw] {};
    \node[vertex] (G--6) at (6.0,-1) [shape = circle, draw] {};
    \draw[] (G-1) .. controls +(0, -1) and +(0, 1) .. (G--1);
    \draw[] (G-2) .. controls +(0, -1) and +(0, 1) .. (G--2);
    \draw[] (G-3) .. controls +(0, -1) and +(0, 1) .. (G--3);
    \draw[] (G-5) .. controls +(0, -1) and +(0, 1) .. (G--5);
    \draw[] (G-3) .. controls +(.5, -.5) and +(-.5, -.5) .. (G-4);
    \draw[] (G-5) .. controls +(.5, -.5) and +(-.5, -.5) .. (G-6);
    \draw[] (G--3) .. controls +(.5, .5) and +(-.5, .5) .. (G--4);
    \draw[] (G--5) .. controls +(.5, .5) and +(-.5, .5) .. (G--6);
    \end{tikzpicture},\quad
    \,
    \begin{tikzpicture}[partition-diagram]
    \node[vertex] (G--3) at (2.4, -1) [shape = circle, draw] {};
    \node[vertex] (G-1) at (0.0, 1) [shape = circle, draw] {};
    \node[vertex] (G--2) at (1.2, -1) [shape = circle, draw] {};
    \node[vertex] (G-2) at (1.2, 1) [shape = circle, draw] {};
    \node[vertex] (G--1) at (0.0, -1) [shape = circle, draw] {};
    \node[vertex] (G-3) at (2.4, 1) [shape = circle, draw] {};
    \node[vertex] (G-4) at (3.6, 1) [shape = circle, draw] {};
    \node[vertex] (G--4) at (3.6,-1) [shape = circle, draw] {};
    \node[vertex] (G-5) at (4.8, 1) [shape = circle, draw] {};
    \node[vertex] (G--5) at (4.8,-1) [shape = circle, draw] {};
    \node[vertex] (G-6) at (6.0, 1) [shape = circle, draw] {};
    \node[vertex] (G--6) at (6.0,-1) [shape = circle, draw] {};
    \draw[] (G-1) .. controls +(0, -1) and +(0, 1) .. (G--1);
    \draw[] (G-2) .. controls +(0, -1) and +(0, 1) .. (G--2);
    \draw[] (G-4) .. controls +(0, -1) and +(0, 1) .. (G--5);
    \draw[] (G-5) .. controls +(0, -1) and +(0, 1) .. (G--4);
    \draw[] (G-3) .. controls +(.5, -.5) and +(-.5, -.5) .. (G-4);
    \draw[] (G-5) .. controls +(.5, -.5) and +(-.5, -.5) .. (G-6);
    \draw[] (G--3) .. controls +(.5, .5) and +(-.5, .5) .. (G--4);
    \draw[] (G--5) .. controls +(.5, .5) and +(-.5, .5) .. (G--6);
    \end{tikzpicture},\quad
    \,
    \begin{tikzpicture}[partition-diagram]
    \node[vertex] (G--3) at (2.4, -1) [shape = circle, draw] {};
    \node[vertex] (G-1) at (0.0, 1) [shape = circle, draw] {};
    \node[vertex] (G--2) at (1.2, -1) [shape = circle, draw] {};
    \node[vertex] (G-2) at (1.2, 1) [shape = circle, draw] {};
    \node[vertex] (G--1) at (0.0, -1) [shape = circle, draw] {};
    \node[vertex] (G-3) at (2.4, 1) [shape = circle, draw] {};
    \node[vertex] (G-4) at (3.6, 1) [shape = circle, draw] {};
    \node[vertex] (G--4) at (3.6,-1) [shape = circle, draw] {};
    \node[vertex] (G-5) at (4.8, 1) [shape = circle, draw] {};
    \node[vertex] (G--5) at (4.8,-1) [shape = circle, draw] {};
    \node[vertex] (G-6) at (6.0, 1) [shape = circle, draw] {};
    \node[vertex] (G--6) at (6.0,-1) [shape = circle, draw] {};
    \draw[] (G-1) .. controls +(0, -1) and +(0, 1) .. (G--2);
    \draw[] (G-2) .. controls +(0, -1) and +(0, 1) .. (G--1);
    \draw[] (G-3) .. controls +(0, -1) and +(0, 1) .. (G--3);
    \draw[] (G-5) .. controls +(0, -1) and +(0, 1) .. (G--5);
    \draw[] (G-3) .. controls +(.5, -.5) and +(-.5, -.5) .. (G-4);
    \draw[] (G-5) .. controls +(.5, -.5) and +(-.5, -.5) .. (G-6);
    \draw[] (G--3) .. controls +(.5, .5) and +(-.5, .5) .. (G--4);
    \draw[] (G--5) .. controls +(.5, .5) and +(-.5, .5) .. (G--6);
    \end{tikzpicture},\quad
    \,
    \begin{tikzpicture}[partition-diagram]
    \node[vertex] (G--3) at (2.4, -1) [shape = circle, draw] {};
    \node[vertex] (G-1) at (0.0, 1) [shape = circle, draw] {};
    \node[vertex] (G--2) at (1.2, -1) [shape = circle, draw] {};
    \node[vertex] (G-2) at (1.2, 1) [shape = circle, draw] {};
    \node[vertex] (G--1) at (0.0, -1) [shape = circle, draw] {};
    \node[vertex] (G-3) at (2.4, 1) [shape = circle, draw] {};
    \node[vertex] (G-4) at (3.6, 1) [shape = circle, draw] {};
    \node[vertex] (G--4) at (3.6,-1) [shape = circle, draw] {};
    \node[vertex] (G-5) at (4.8, 1) [shape = circle, draw] {};
    \node[vertex] (G--5) at (4.8,-1) [shape = circle, draw] {};
    \node[vertex] (G-6) at (6.0, 1) [shape = circle, draw] {};
    \node[vertex] (G--6) at (6.0,-1) [shape = circle, draw] {};
    \draw[] (G-1) .. controls +(0, -1) and +(0, 1) .. (G--2);
    \draw[] (G-2) .. controls +(0, -1) and +(0, 1) .. (G--1);
    \draw[] (G-4) .. controls +(0, -1) and +(0, 1) .. (G--5);
    \draw[] (G-5) .. controls +(0, -1) and +(0, 1) .. (G--4);
    \draw[] (G-3) .. controls +(.5, -.5) and +(-.5, -.5) .. (G-4);
    \draw[] (G-5) .. controls +(.5, -.5) and +(-.5, -.5) .. (G-6);
    \draw[] (G--3) .. controls +(.5, .5) and +(-.5, .5) .. (G--4);
    \draw[] (G--5) .. controls +(.5, .5) and +(-.5, .5) .. (G--6);
    \end{tikzpicture}    \right\}.
\end{gather*}
There are four irreducible representations of $G_{e_\pi}$, which are all one dimensional:
\begin{align*}
    V^{((1,1), (1,1))}_{G_{e_\pi}} \text{ with basis } & \left\{ \left(\raisebox{-.1in}{\tikztableausmall{2,1}, \tikztableausmall{56,34}} \right)
\right\},  \\
    V^{((2), (1,1))}_{G_{e_\pi}} \text{ with basis } & \left\{ \left(\raisebox{-.1in}{\tikztableausmall{{1,2}}, \tikztableausmall{56,34}} \right)
\right\} ,\\
    V^{((1,1), (2))}_{G_{e_\pi}} \text{ with basis } & \left\{ \left(\raisebox{-.1in}{\tikztableausmall{2,1}, \tikztableausmall{{34,56}}} \right)
\right\},\\
    V^{((2), (2))}_{G_{e_\pi}} \text{ with basis } & \left\{ \left(\raisebox{-.1in}{\tikztableausmall{{1,2}}, \tikztableausmall{{34,56}}} \right)
\right\}.
\end{align*}
\end{example}

\subsection{$\mathscr{L}$-classes and Sch\"utzenberger representations}
\label{sec:Schutzenberger-representations}

For each idempotent $e \in \mathcal{U}_k$, we define a $(\mathcal{U}_k, G_e)$-bimodule
$\CC{L_e}$ that is known in the semigroup theory literature as
the \defn{left Sch\"utzenberger representation associated with $e$}.
The  Sch\"utzenberger representations will be used in Section~\ref{sec:irreps-of-UBP-algebra} to construct
irreducible $\mathcal{U}_k$-representations from irreducible $G_e$-representations.
As a vector space, $\CC{L_e}$ is spanned by the elements of the $\mathscr{L}$-class
of $e$, so we begin by studying the $\mathscr{L}$-classes of $\mathcal{U}_k$.

Let $x$ and $y$ be elements of a monoid $M$.  We say that $x$ and $y$ are \defn{$\mathscr{L}$-equivalent} if $Mx = My$.
This is an equivalence relation; hence, it partitions $M$ into classes which are called the \defn{$\mathscr{L}$-classes} of $M$.
The $\mathscr{L}$-class of an element $x$ is denoted by \defn{$L_x$}.

\begin{prop} Let $k$ be a nonnegative integer.
\begin{enumerate}
    \item[(a)] Two elements $d_1,d_2\in \mathcal{U}_k$ are in the same $\mathscr{L}$-class if and only if $\mathsf{bot}(d_1) = \mathsf{bot}(d_2)$.
    \item[(b)] Every $\mathscr{L}$-class in $\mathcal{U}_k$ contains a unique idempotent.
    \item[(c)] The $\mathscr{L}$-classes of $\mathcal{U}_k$ are in bijection with the set partitions $\pi$ of $[k]$. More precisely,
    $${\defncolor L_\pi} := L_{e_\pi} = \left\{ d \in \mathcal{U}_k : \mathsf{bot}(d) = \pi\right\}.$$
    \item[(d)] For every $\lambda \vdash k$, the $\mathscr{J}$-class $J_\lambda$
        is a disjoint union of $\mathscr{L}$-classes.
        More precisely, $$J_\lambda = \biguplus_{\pi: \mathsf{type}(\pi) = \lambda} L_\pi.$$
\end{enumerate}
\end{prop}

\begin{proof}
(a) By Proposition~\ref{prop:standard-decomposition}, every element $d \in \mathcal{U}_k$ can be written as
$d = \sigma e_{\mathsf{bot}(d)}$ for some permutation $\sigma \in \mathfrak{S}_k$. Thus,
$\mathcal{U}_k d = \mathcal{U}_k \sigma e_{\mathsf{bot}(d)} = \mathcal{U}_k e_{\mathsf{bot}(d)}$. The last equality follows since $\sigma$
is invertible in $\mathcal{U}_k$. Therefore, if $m \in \mathcal{U}_k e_{\mathsf{bot}(d)}$, then $\mathsf{bot}(m)$ is coarser than or equal to
$\mathsf{bot}(d)$. Hence, $d_1,d_2$ are in the same $\mathscr{L}$-class if and only if $\mathsf{bot}(d_1)= \mathsf{bot}(d_2)$.

(b) From part (a) we have that an $\mathscr{L}$-class $L$ contains elements that have the same set partition $\pi$ as the bottom row. Since there is a
unique idempotent, namely $e_\pi$, that has bottom row $\pi$, the result follows.

(c) Since every $\mathscr{L}$-class contains a unique idempotent, the $\mathscr{L}$-classes are in bijection with the set partitions of $[k]$.

(d) By Proposition \ref{prop:charJclass} (d),  $J_\lambda =\{ \sigma e_\pi \tau : \sigma, \tau \in \mathfrak{S}_k \text{ and } \mathsf{type}(\pi) = \lambda\}$, while
$L_\pi = \{\sigma e_\pi : \sigma \in \mathfrak{S}_k \}$. Thus, the result follows.
\end{proof}

\begin{example}
There are five $\mathscr{L}$-classes for $\mathcal{U}_3$:
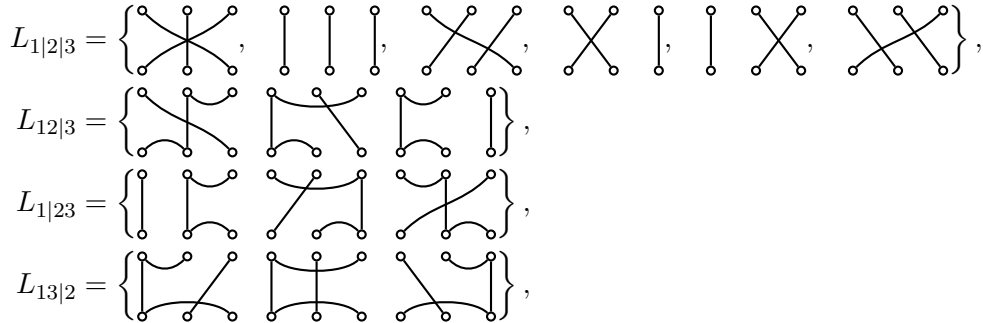
\begin{align*}
  L_{1|2|3} &=  \left\{\begin{tikzpicture}[partition-diagram]
    \node[vertex] (G--3) at (3.0, -1) [shape = circle, draw] {};
    \node[vertex] (G-1) at (0.0, 1) [shape = circle, draw] {};
    \node[vertex] (G--2) at (1.5, -1) [shape = circle, draw] {};
    \node[vertex] (G-2) at (1.5, 1) [shape = circle, draw] {};
    \node[vertex] (G--1) at (0.0, -1) [shape = circle, draw] {};
    \node[vertex] (G-3) at (3.0, 1) [shape = circle, draw] {};
    \draw[] (G-1) .. controls +(1, -1) and +(-1, 1) .. (G--3);
    \draw[] (G-2) .. controls +(0, -1) and +(0, 1) .. (G--2);
    \draw[] (G-3) .. controls +(-1, -1) and +(1, 1) .. (G--1);
    \end{tikzpicture},
    \quad
    \begin{tikzpicture}[partition-diagram, idempotent]
    \node[vertex] (G--3) at (3.0, -1) [shape = circle, draw] {};
    \node[vertex] (G-3) at (3.0, 1) [shape = circle, draw] {};
    \node[vertex] (G--2) at (1.5, -1) [shape = circle, draw] {};
    \node[vertex] (G-2) at (1.5, 1) [shape = circle, draw] {};
    \node[vertex] (G--1) at (0.0, -1) [shape = circle, draw] {};
    \node[vertex] (G-1) at (0.0, 1) [shape = circle, draw] {};
    \draw[] (G-3) .. controls +(0, -1) and +(0, 1) .. (G--3);
    \draw[] (G-2) .. controls +(0, -1) and +(0, 1) .. (G--2);
    \draw[] (G-1) .. controls +(0, -1) and +(0, 1) .. (G--1);
    \end{tikzpicture},
    \quad
    \begin{tikzpicture}[partition-diagram]
    \node[vertex] (G--3) at (3.0, -1) [shape = circle, draw] {};
    \node[vertex] (G-1) at (0.0, 1) [shape = circle, draw] {};
    \node[vertex] (G--2) at (1.5, -1) [shape = circle, draw] {};
    \node[vertex] (G-3) at (3.0, 1) [shape = circle, draw] {};
    \node[vertex] (G--1) at (0.0, -1) [shape = circle, draw] {};
    \node[vertex] (G-2) at (1.5, 1) [shape = circle, draw] {};
    \draw[] (G-1) .. controls +(1, -1) and +(-1, 1) .. (G--3);
    \draw[] (G-3) .. controls +(-0.75, -1) and +(0.75, 1) .. (G--2);
    \draw[] (G-2) .. controls +(-0.75, -1) and +(0.75, 1) .. (G--1);
    \end{tikzpicture},
    \quad
    \begin{tikzpicture}[partition-diagram]
    \node[vertex] (G--3) at (3.0, -1) [shape = circle, draw] {};
    \node[vertex] (G-3) at (3.0, 1) [shape = circle, draw] {};
    \node[vertex] (G--2) at (1.5, -1) [shape = circle, draw] {};
    \node[vertex] (G-1) at (0.0, 1) [shape = circle, draw] {};
    \node[vertex] (G--1) at (0.0, -1) [shape = circle, draw] {};
    \node[vertex] (G-2) at (1.5, 1) [shape = circle, draw] {};
    \draw[] (G-3) .. controls +(0, -1) and +(0, 1) .. (G--3);
    \draw[] (G-1) .. controls +(0.75, -1) and +(-0.75, 1) .. (G--2);
    \draw[] (G-2) .. controls +(-0.75, -1) and +(0.75, 1) .. (G--1);
    \end{tikzpicture},
    \quad
    \begin{tikzpicture}[partition-diagram]
    \node[vertex] (G--3) at (3.0, -1) [shape = circle, draw] {};
    \node[vertex] (G-2) at (1.5, 1) [shape = circle, draw] {};
    \node[vertex] (G--2) at (1.5, -1) [shape = circle, draw] {};
    \node[vertex] (G-3) at (3.0, 1) [shape = circle, draw] {};
    \node[vertex] (G--1) at (0.0, -1) [shape = circle, draw] {};
    \node[vertex] (G-1) at (0.0, 1) [shape = circle, draw] {};
    \draw[] (G-2) .. controls +(0.75, -1) and +(-0.75, 1) .. (G--3);
    \draw[] (G-3) .. controls +(-0.75, -1) and +(0.75, 1) .. (G--2);
    \draw[] (G-1) .. controls +(0, -1) and +(0, 1) .. (G--1);
    \end{tikzpicture},
    \quad
    \begin{tikzpicture}[partition-diagram]
    \node[vertex] (G--3) at (3.0, -1) [shape = circle, draw] {};
    \node[vertex] (G-2) at (1.5, 1) [shape = circle, draw] {};
    \node[vertex] (G--2) at (1.5, -1) [shape = circle, draw] {};
    \node[vertex] (G-1) at (0.0, 1) [shape = circle, draw] {};
    \node[vertex] (G--1) at (0.0, -1) [shape = circle, draw] {};
    \node[vertex] (G-3) at (3.0, 1) [shape = circle, draw] {};
    \draw[] (G-2) .. controls +(0.75, -1) and +(-0.75, 1) .. (G--3);
    \draw[] (G-1) .. controls +(0.75, -1) and +(-0.75, 1) .. (G--2);
    \draw[] (G-3) .. controls +(-1, -1) and +(1, 1) .. (G--1);
    \end{tikzpicture}\right\},
    \\   L_{12|3} &=
    \left\{\begin{tikzpicture}[partition-diagram]
    \node[vertex] (G--3) at (3.0, -1) [shape = circle, draw] {};
    \node[vertex] (G-1) at (0.0, 1) [shape = circle, draw] {};
    \node[vertex] (G--2) at (1.5, -1) [shape = circle, draw] {};
    \node[vertex] (G--1) at (0.0, -1) [shape = circle, draw] {};
    \node[vertex] (G-2) at (1.5, 1) [shape = circle, draw] {};
    \node[vertex] (G-3) at (3.0, 1) [shape = circle, draw] {};
    \draw[] (G-1) .. controls +(1, -1) and +(-1, 1) .. (G--3);
    \draw[] (G-2) .. controls +(0.5, -0.5) and +(-0.5, -0.5) .. (G-3);
    \draw[] (G--2) .. controls +(-0.5, 0.5) and +(0.5, 0.5) .. (G--1);
    \draw[] (G--2) -- (G-2);
    \end{tikzpicture}
    \quad
    \begin{tikzpicture}[partition-diagram]
    \node[vertex] (G--3) at (3.0, -1) [shape = circle, draw] {};
    \node[vertex] (G-2) at (1.5, 1) [shape = circle, draw] {};
    \node[vertex] (G--2) at (1.5, -1) [shape = circle, draw] {};
    \node[vertex] (G--1) at (0.0, -1) [shape = circle, draw] {};
    \node[vertex] (G-1) at (0.0, 1) [shape = circle, draw] {};
    \node[vertex] (G-3) at (3.0, 1) [shape = circle, draw] {};
    \draw[] (G-2) .. controls +(0.75, -1) and +(-0.75, 1) .. (G--3);
    \draw[] (G-1) .. controls +(0.6, -0.6) and +(-0.6, -0.6) .. (G-3);
    \draw[] (G--2) .. controls +(-0.5, 0.5) and +(0.5, 0.5) .. (G--1);
    \draw[] (G--1) .. controls +(0, 1) and +(0, -1) .. (G-1);
    \end{tikzpicture}
    \quad
    \begin{tikzpicture}[partition-diagram, idempotent]
    \node[vertex] (G--3) at (3.0, -1) [shape = circle, draw] {};
    \node[vertex] (G-3) at (3.0, 1) [shape = circle, draw] {};
    \node[vertex] (G--2) at (1.5, -1) [shape = circle, draw] {};
    \node[vertex] (G--1) at (0.0, -1) [shape = circle, draw] {};
    \node[vertex] (G-1) at (0.0, 1) [shape = circle, draw] {};
    \node[vertex] (G-2) at (1.5, 1) [shape = circle, draw] {};
    \draw[] (G-3) .. controls +(0, -1) and +(0, 1) .. (G--3);
    \draw[] (G-1) .. controls +(0.5, -0.5) and +(-0.5, -0.5) .. (G-2);
    \draw[] (G--2) .. controls +(-0.5, 0.5) and +(0.5, 0.5) .. (G--1);
    \draw[] (G--1) .. controls +(0, 1) and +(0, -1) .. (G-1);
    \end{tikzpicture}\right\},
    \\   L_{1|23} &=
    \left\{\begin{tikzpicture}[partition-diagram, idempotent]
    \node[vertex] (G--3) at (3.0, -1) [shape = circle, draw] {};
    \node[vertex] (G--2) at (1.5, -1) [shape = circle, draw] {};
    \node[vertex] (G-2) at (1.5, 1) [shape = circle, draw] {};
    \node[vertex] (G-3) at (3.0, 1) [shape = circle, draw] {};
    \node[vertex] (G--1) at (0.0, -1) [shape = circle, draw] {};
    \node[vertex] (G-1) at (0.0, 1) [shape = circle, draw] {};
    \draw[] (G-2) .. controls +(0.5, -0.5) and +(-0.5, -0.5) .. (G-3);
    \draw[] (G--3) .. controls +(-0.5, 0.5) and +(0.5, 0.5) .. (G--2);
    \draw[] (G--2) .. controls +(0, 1) and +(0, -1) .. (G-2);
    \draw[] (G-1) .. controls +(0, -1) and +(0, 1) .. (G--1);
    \end{tikzpicture}
    \quad
    \begin{tikzpicture}[partition-diagram]
    \node[vertex] (G--3) at (3.0, -1) [shape = circle, draw] {};
    \node[vertex] (G--2) at (1.5, -1) [shape = circle, draw] {};
    \node[vertex] (G-1) at (0.0, 1) [shape = circle, draw] {};
    \node[vertex] (G-3) at (3.0, 1) [shape = circle, draw] {};
    \node[vertex] (G--1) at (0.0, -1) [shape = circle, draw] {};
    \node[vertex] (G-2) at (1.5, 1) [shape = circle, draw] {};
    \draw[] (G-1) .. controls +(0.6, -0.6) and +(-0.6, -0.6) .. (G-3);
    \draw[] (G-3) .. controls +(0, -1) and +(0, 1) .. (G--3);
    \draw[] (G--3) .. controls +(-0.5, 0.5) and +(0.5, 0.5) .. (G--2);
    \draw[] (G-2) .. controls +(-0.75, -1) and +(0.75, 1) .. (G--1);
    \end{tikzpicture}
    \quad
    \begin{tikzpicture}[partition-diagram]
    \node[vertex] (G--3) at (3.0, -1) [shape = circle, draw] {};
    \node[vertex] (G--2) at (1.5, -1) [shape = circle, draw] {};
    \node[vertex] (G-1) at (0.0, 1) [shape = circle, draw] {};
    \node[vertex] (G-2) at (1.5, 1) [shape = circle, draw] {};
    \node[vertex] (G--1) at (0.0, -1) [shape = circle, draw] {};
    \node[vertex] (G-3) at (3.0, 1) [shape = circle, draw] {};
    \draw[] (G-1) .. controls +(0.5, -0.5) and +(-0.5, -0.5) .. (G-2);
    \draw[] (G--3) .. controls +(-0.5, 0.5) and +(0.5, 0.5) .. (G--2);
    \draw[] (G--2) -- (G-2);
    \draw[] (G-3) .. controls +(-1, -1) and +(1, 1) .. (G--1);
    \end{tikzpicture}\right\},
    \\  L_{13|2} &=
    \left\{\begin{tikzpicture}[partition-diagram]
    \node[vertex] (G--3) at (3.0, -1) [shape = circle, draw] {};
    \node[vertex] (G--1) at (0.0, -1) [shape = circle, draw] {};
    \node[vertex] (G-1) at (0.0, 1) [shape = circle, draw] {};
    \node[vertex] (G-2) at (1.5, 1) [shape = circle, draw] {};
    \node[vertex] (G--2) at (1.5, -1) [shape = circle, draw] {};
    \node[vertex] (G-3) at (3.0, 1) [shape = circle, draw] {};
    \draw[] (G-1) .. controls +(0.5, -0.5) and +(-0.5, -0.5) .. (G-2);
    \draw[] (G--3) .. controls +(-0.6, 0.6) and +(0.6, 0.6) .. (G--1);
    \draw[] (G--1) .. controls +(0, 1) and +(0, -1) .. (G-1);
    \draw[] (G-3) .. controls +(-0.75, -1) and +(0.75, 1) .. (G--2);
    \end{tikzpicture}
    \quad
    \begin{tikzpicture}[partition-diagram, idempotent]
    \node[vertex] (G--3) at (3.0, -1) [shape = circle, draw] {};
    \node[vertex] (G--1) at (0.0, -1) [shape = circle, draw] {};
    \node[vertex] (G-1) at (0.0, 1) [shape = circle, draw] {};
    \node[vertex] (G-3) at (3.0, 1) [shape = circle, draw] {};
    \node[vertex] (G--2) at (1.5, -1) [shape = circle, draw] {};
    \node[vertex] (G-2) at (1.5, 1) [shape = circle, draw] {};
    \draw[] (G-1) .. controls +(0.6, -0.6) and +(-0.6, -0.6) .. (G-3);
    \draw[] (G--3) .. controls +(-0.6, 0.6) and +(0.6, 0.6) .. (G--1);
    \draw[] (G--1) .. controls +(0, 1) and +(0, -1) .. (G-1);
    \draw[] (G-2) .. controls +(0, -1) and +(0, 1) .. (G--2);
    \end{tikzpicture}
    \quad
    \begin{tikzpicture}[partition-diagram]
    \node[vertex] (G--3) at (3.0, -1) [shape = circle, draw] {};
    \node[vertex] (G--1) at (0.0, -1) [shape = circle, draw] {};
    \node[vertex] (G-2) at (1.5, 1) [shape = circle, draw] {};
    \node[vertex] (G-3) at (3.0, 1) [shape = circle, draw] {};
    \node[vertex] (G--2) at (1.5, -1) [shape = circle, draw] {};
    \node[vertex] (G-1) at (0.0, 1) [shape = circle, draw] {};
    \draw[] (G-2) .. controls +(0.5, -0.5) and +(-0.5, -0.5) .. (G-3);
    \draw[] (G-3) .. controls +(0, -1) and +(0, 1) .. (G--3);
    \draw[] (G--3) .. controls +(-0.6, 0.6) and +(0.6, 0.6) .. (G--1);
    \draw[] (G-1) .. controls +(0.75, -1) and +(-0.75, 1) .. (G--2);
    \end{tikzpicture}\right\},
    \\  L_{123} &=
    \left\{\begin{tikzpicture}[partition-diagram, idempotent]
    \node[vertex] (G--3) at (3.0, -1) [shape = circle, draw] {};
    \node[vertex] (G--2) at (1.5, -1) [shape = circle, draw] {};
    \node[vertex] (G--1) at (0.0, -1) [shape = circle, draw] {};
    \node[vertex] (G-1) at (0.0, 1) [shape = circle, draw] {};
    \node[vertex] (G-2) at (1.5, 1) [shape = circle, draw] {};
    \node[vertex] (G-3) at (3.0, 1) [shape = circle, draw] {};
    \draw[] (G-1) .. controls +(0.5, -0.5) and +(-0.5, -0.5) .. (G-2);
    \draw[] (G-2) .. controls +(0.5, -0.5) and +(-0.5, -0.5) .. (G-3);
    \draw[] (G--3) .. controls +(-0.5, 0.5) and +(0.5, 0.5) .. (G--2);
    \draw[] (G--2) .. controls +(-0.5, 0.5) and +(0.5, 0.5) .. (G--1);
    \draw[] (G--1) .. controls +(0, 1) and +(0, -1) .. (G-1);
    \end{tikzpicture}\right\}.
\end{align*}
\end{example}

For any $\pi \vdash [k]$, let $\CC{L_\pi}$ be the vector space with basis the
elements of the $\mathscr{L}$-class $L_\pi$.
It has a \defn{left $\mathcal{U}_k$-action} defined by
\begin{equation}
    \label{eq:odot}
    m \odot \ell=\begin{cases}
    m\ell, & \text{if } m\ell \in L_\pi,  \\
    0, & \text{else,}  \end{cases}
\end{equation}
for all $m\in \mathcal{U}_k$ and $\ell \in L_\pi$, which is then extended $\CC$-linearly to all of $\CC{L_\pi}$.
The nonzero products $m \odot \ell$ can be characterized as follows.

\begin{lemma}\label{lem:UkactsL} Let $m\in \mathcal{U}_k$ and $\ell \in L_\pi$. Then
$m\odot \ell \neq 0$ if and only if $\mathsf{bot}(m)$ is finer than $\mathsf{top}(\ell)$.
\end{lemma}

The \defn{right $G_{e_\pi}$-action} on $\CC{L_\pi}$ is extended $\CC$-linearly from the
action of $G_{e_\pi}$ on $L_\pi$ given by right multiplication.
Although it is true for any finite monoid that the maximal subgroup of an
idempotent $e$ acts by right multiplication on the $\mathscr{L}$-class of $e$
\cite[Proposition~1.10]{Steinberg.2016},
below we provide a proof that is specific to $\mathcal{U}_k$
and we identify the orbits of this action.

\begin{prop}\label{prop:orbitreps}
Let $\pi$ be a set partition of $[k]$.
\begin{enumerate}[label=(\alph*)]
    \item
        $G_{e_\pi}$ acts by right multiplication on $L_\pi$
        and this right action is free. In other words,
        \begin{itemize}
            \item
                if $\ell \in L_\pi$ and $g \in G_{e_\pi}$, then $\ell g \in L_\pi$; and
            \item
                if $\ell \in L_\pi$ and $\ell g = \ell h$ for some $g, h \in G_{e_\pi}$, then $g = h$.
        \end{itemize}

    \item
        $d_1, d_2\in L_\pi$ are in the same $G_{e_\pi}$-orbit if and only if $\mathsf{top}(d_1) = \mathsf{top}(d_2)$.

    \item
        For every set partition $\gamma\vdash [k]$ such that $\mathsf{type}(\gamma) = \mathsf{type}(\pi)$,
        $$L_{\pi}^\gamma = \{d \in \mathcal{U}_k : \mathsf{top}(d) = \gamma \text{ and } \mathsf{bot}(d) = \pi\}$$
        is an orbit for the right $G_{e_\pi}$-action on $L_\pi$, and all the orbits are of this form.
        Thus, the $G_{e_\pi}$-orbits in $L_\pi$ are in bijection with the set partitions $\gamma$ of $\mathsf{type}(\pi)$.
\end{enumerate}
\end{prop}

\begin{proof}
    (a) Let $\ell \in L_\pi$ and $g, h \in G_{e_\pi}$, and think of them as bijections
    as in Section~\ref{sssection:bijection}.
    Since $\mathsf{bot}(\ell) = \pi = \mathsf{top}(g)$,
    we have that $\ell g$ is the composition of $\ell$ and $g$ (see Remark~\ref{remark:composition-and-product}).
    Hence, $\mathsf{bot}(\ell g) = \mathsf{bot}(g) = \pi$ and so $\ell g \in L_\pi$.
    Similarly, $\ell h$ is the composition of $\ell$ and $h$. Thus,
    if $\ell g = \ell h$, then $g = h$ since $\ell$ is a bijection.

    (b) If $d_2 = d_1 g$ for some $g \in G_{e_\pi}$, then $\mathsf{top}(d_1)
    = \mathsf{top}(d_1 g) = \mathsf{top}(d_2)$ since multiplying on the right
    by an element in $G_{e_\pi}$ has no effect on the top row of the diagram of $d_1$.

    Conversely, if $d_1, d_2 \in L_\pi$ and $\mathsf{top}(d_1)
    = \mathsf{top}(d_2)$, then both $d_1$ and $d_2$ are
    bijections from $\mathsf{top}(d_1)$ to $\pi$.
    Since $\widetilde{d}_1 d_2$ is a size-preserving bijection
    from $\pi$ to itself, it is equal to an element $g \in G_{e_\pi}$.
    By composing $\widetilde{d}_1 d_2 = g$ on the right with $d_1$,
    we conclude that $d_2 = d_1 g$.

    (c) This follows directly from (b).
\end{proof}

\begin{example}
Let $\pi = 12|34$.
The $\mathscr{L}$-class $L_\pi$ contains 6 elements all with the same bottom row.
The group $G_{e_\pi}$ contains two elements,
the identity and the permutation of the two blocks.
Hence we obtain that $L_\pi$ decomposes into three $G_{e_\pi}$-orbits:
\begin{gather*}
    L_{12|34}^{13|24} =
    \left\{\begin{tikzpicture}[partition-diagram]
    \node[vertex] (G--4) at (4.5, -1) [shape = circle, draw] {};
    \node[vertex] (G--3) at (3.0, -1) [shape = circle, draw] {};
    \node[vertex] (G-1) at (0.0, 1) [shape = circle, draw] {};
    \node[vertex] (G-3) at (3.0, 1) [shape = circle, draw] {};
    \node[vertex] (G--2) at (1.5, -1) [shape = circle, draw] {};
    \node[vertex] (G--1) at (0.0, -1) [shape = circle, draw] {};
    \node[vertex] (G-2) at (1.5, 1) [shape = circle, draw] {};
    \node[vertex] (G-4) at (4.5, 1) [shape = circle, draw] {};
    \draw[] (G-1) .. controls +(0.6, -0.6) and +(-0.6, -0.6) .. (G-3);
    \draw[] (G--4) .. controls +(-0.5, 0.5) and +(0.5, 0.5) .. (G--3);
    \draw[] (G-2) .. controls +(0.6, -0.6) and +(-0.6, -0.6) .. (G-4);
    \draw[] (G--2) .. controls +(-0.5, 0.5) and +(0.5, 0.5) .. (G--1);
    \draw[] (G--2) -- (G-2);
    \draw[] (G--3) -- (G-3);
    \end{tikzpicture},
    \quad
    \begin{tikzpicture}[partition-diagram, representative]
    \node[vertex] (G--4) at (4.5, -1) [shape = circle, draw] {};
    \node[vertex] (G--3) at (3.0, -1) [shape = circle, draw] {};
    \node[vertex] (G-2) at (1.5, 1) [shape = circle, draw] {};
    \node[vertex] (G-4) at (4.5, 1) [shape = circle, draw] {};
    \node[vertex] (G--2) at (1.5, -1) [shape = circle, draw] {};
    \node[vertex] (G--1) at (0.0, -1) [shape = circle, draw] {};
    \node[vertex] (G-1) at (0.0, 1) [shape = circle, draw] {};
    \node[vertex] (G-3) at (3.0, 1) [shape = circle, draw] {};
    \draw[] (G-2) .. controls +(0.6, -0.6) and +(-0.6, -0.6) .. (G-4);
    \draw[] (G-4) .. controls +(0, -1) and +(0, 1) .. (G--4);
    \draw[] (G--4) .. controls +(-0.5, 0.5) and +(0.5, 0.5) .. (G--3);
    \draw[] (G-1) .. controls +(0.6, -0.6) and +(-0.6, -0.6) .. (G-3);
    \draw[] (G--2) .. controls +(-0.5, 0.5) and +(0.5, 0.5) .. (G--1);
    \draw[] (G--1) .. controls +(0, 1) and +(0, -1) .. (G-1);
    \end{tikzpicture}\right\},
    \\
    L_{12|34}^{23|14} =
    \left\{\begin{tikzpicture}[partition-diagram, representative]
    \node[vertex] (G--4) at (4.5, -1) [shape = circle, draw] {};
    \node[vertex] (G--3) at (3.0, -1) [shape = circle, draw] {};
    \node[vertex] (G-1) at (0.0, 1) [shape = circle, draw] {};
    \node[vertex] (G-4) at (4.5, 1) [shape = circle, draw] {};
    \node[vertex] (G--2) at (1.5, -1) [shape = circle, draw] {};
    \node[vertex] (G--1) at (0.0, -1) [shape = circle, draw] {};
    \node[vertex] (G-2) at (1.5, 1) [shape = circle, draw] {};
    \node[vertex] (G-3) at (3.0, 1) [shape = circle, draw] {};
    \draw[] (G-1) .. controls +(0.7, -0.7) and +(-0.7, -0.7) .. (G-4);
    \draw[] (G-4) .. controls +(0, -1) and +(0, 1) .. (G--4);
    \draw[] (G--4) .. controls +(-0.5, 0.5) and +(0.5, 0.5) .. (G--3);
    \draw[] (G-2) .. controls +(0.5, -0.5) and +(-0.5, -0.5) .. (G-3);
    \draw[] (G--2) .. controls +(-0.5, 0.5) and +(0.5, 0.5) .. (G--1);
    \draw[] (G--2) -- (G-2);
    \end{tikzpicture},
    \quad
    \begin{tikzpicture}[partition-diagram]
    \node[vertex] (G--4) at (4.5, -1) [shape = circle, draw] {};
    \node[vertex] (G--3) at (3.0, -1) [shape = circle, draw] {};
    \node[vertex] (G-2) at (1.5, 1) [shape = circle, draw] {};
    \node[vertex] (G-3) at (3.0, 1) [shape = circle, draw] {};
    \node[vertex] (G--2) at (1.5, -1) [shape = circle, draw] {};
    \node[vertex] (G--1) at (0.0, -1) [shape = circle, draw] {};
    \node[vertex] (G-1) at (0.0, 1) [shape = circle, draw] {};
    \node[vertex] (G-4) at (4.5, 1) [shape = circle, draw] {};
    \draw[] (G-2) .. controls +(0.5, -0.5) and +(-0.5, -0.5) .. (G-3);
    \draw[] (G--4) .. controls +(-0.5, 0.5) and +(0.5, 0.5) .. (G--3);
    \draw[] (G-1) .. controls +(0.7, -0.7) and +(-0.7, -0.7) .. (G-4);
    \draw[] (G--2) .. controls +(-0.5, 0.5) and +(0.5, 0.5) .. (G--1);
    \draw[] (G--1) .. controls +(0, 1) and +(0, -1) .. (G-1);
    \draw[] (G--3) .. controls +(0, 1) and +(0, -1) .. (G-3);
    \end{tikzpicture}\right\},
    \\
    L_{12|34}^{12|34} =
    \left\{\begin{tikzpicture}[partition-diagram, representative]    \node[vertex] (G--4) at (4.5, -1) [shape = circle, draw] {};     \node[vertex] (G--3) at (3.0, -1) [shape = circle, draw] {};
    \node[vertex] (G-3) at (3.0, 1) [shape = circle, draw] {};
    \node[vertex] (G-4) at (4.5, 1) [shape = circle, draw] {};
    \node[vertex] (G--2) at (1.5, -1) [shape = circle, draw] {};
    \node[vertex] (G--1) at (0.0, -1) [shape = circle, draw] {};
    \node[vertex] (G-1) at (0.0, 1) [shape = circle, draw] {};
    \node[vertex] (G-2) at (1.5, 1) [shape = circle, draw] {};
    \draw[] (G-3) .. controls +(0.5, -0.5) and +(-0.5, -0.5) .. (G-4);
    \draw[] (G--4) .. controls +(-0.5, 0.5) and +(0.5, 0.5) .. (G--3);
    \draw[] (G--3) .. controls +(0, 1) and +(0, -1) .. (G-3);
    \draw[] (G-1) .. controls +(0.5, -0.5) and +(-0.5, -0.5) .. (G-2);
    \draw[] (G--2) .. controls +(-0.5, 0.5) and +(0.5, 0.5) .. (G--1);
    \draw[] (G--1) .. controls +(0, 1) and +(0, -1) .. (G-1);
    \end{tikzpicture},
    \quad
    \begin{tikzpicture}[partition-diagram]
    \node[vertex] (G--4) at (4.5, -1) [shape = circle, draw] {};
    \node[vertex] (G--3) at (3.0, -1) [shape = circle, draw] {};
    \node[vertex] (G-1) at (0.0, 1) [shape = circle, draw] {};
    \node[vertex] (G-2) at (1.5, 1) [shape = circle, draw] {};
    \node[vertex] (G--2) at (1.5, -1) [shape = circle, draw] {};
    \node[vertex] (G--1) at (0.0, -1) [shape = circle, draw] {};
    \node[vertex] (G-3) at (3.0, 1) [shape = circle, draw] {};
    \node[vertex] (G-4) at (4.5, 1) [shape = circle, draw] {};
    \draw[] (G-1) .. controls +(0.5, -0.5) and +(-0.5, -0.5) .. (G-2);
    \draw[] (G--4) .. controls +(-0.5, 0.5) and +(0.5, 0.5) .. (G--3);
    \draw[] (G-3) .. controls +(0.5, -0.5) and +(-0.5, -0.5) .. (G-4);
    \draw[] (G--2) .. controls +(-0.5, 0.5) and +(0.5, 0.5) .. (G--1);
    \draw[] (G--2) .. controls +(1, 1) and +(-1, -1) .. (G-3);
    \draw[] (G--3) .. controls +(-1, 1) and +(1, -1) .. (G-2);
    \end{tikzpicture}\right\}.
\end{gather*}
\end{example}

We will now choose \defn{orbit representatives} of the right $G_{e_\pi}$-action of $L_\pi$.
Let $L_\pi^\gamma$ be an orbit, where $\gamma \vdash [k]$ and $\mathsf{type}(\gamma) = \mathsf{type}(\pi)$.
Think of the elements of $L_\pi^\gamma$ as
bijections $\ell \colon \gamma \rightarrow \pi$.
Assume $\pi = \{\pi_1 <\cdots <\pi_r\}$ and $\gamma=\{\gamma_1 < \cdots <\gamma_r\}$
are ordered using the graded last letter order.
Let $\ell_\pi^\gamma \colon \gamma \rightarrow \pi$ be the bijection that sends
$\gamma_i$ to $\pi_i$ for all $i$. If $\gamma = \pi$,
then this is the identity bijection and we have $\ell_\pi^\pi = e_\pi$.

\begin{example}
The orbit representative for $L_{12|34}^{13|24}$ is $\ell_{12|34}^{13|24}=$
    \begin{tikzpicture}[partition-diagram, representative]
    \node[vertex] (G--4) at (4.5, -1) [shape = circle, draw] {};
    \node[vertex] (G--3) at (3.0, -1) [shape = circle, draw] {};
    \node[vertex] (G-2) at (1.5, 1) [shape = circle, draw] {};
    \node[vertex] (G-4) at (4.5, 1) [shape = circle, draw] {};
    \node[vertex] (G--2) at (1.5, -1) [shape = circle, draw] {};
    \node[vertex] (G--1) at (0.0, -1) [shape = circle, draw] {};
    \node[vertex] (G-1) at (0.0, 1) [shape = circle, draw] {};
    \node[vertex] (G-3) at (3.0, 1) [shape = circle, draw] {};
    \draw[] (G-2) .. controls +(0.6, -0.6) and +(-0.6, -0.6) .. (G-4);
    \draw[] (G-4) .. controls +(0, -1) and +(0, 1) .. (G--4);
    \draw[] (G--4) .. controls +(-0.5, 0.5) and +(0.5, 0.5) .. (G--3);
    \draw[] (G-1) .. controls +(0.6, -0.6) and +(-0.6, -0.6) .. (G-3);
    \draw[] (G--2) .. controls +(-0.5, 0.5) and +(0.5, 0.5) .. (G--1);
    \draw[] (G--1) .. controls +(0, 1) and +(0, -1) .. (G-1);
   \end{tikzpicture}.
\end{example}

The next result describes the relationship between the actions of
$\mathcal{U}_k$ and $G_{e_\pi}$ on $\CC{L_\pi}$.
\begin{prop}\label{prop:GeactiononL}
Let $m\in \mathcal{U}_k$ and $d\in L_\pi^\gamma$. If $m \odot d \neq 0$,
then there exists a unique $g \in G_{e_\pi}$  such that $md = \ell_\pi^{\gamma'} g$,
where $\gamma' = \sigma^{-1}(\gamma)$ for any $\sigma \in \mathfrak{S}_k$ such that $m = \sigma e_{\mathsf{bot}(m)}$.
\end{prop}
\begin{proof}
Let $m\in \mathcal{U}_k$ and write $m = \sigma e_{\mathsf{bot}(m)}$ for some $\sigma \in \mathfrak{S}_k$.
Let $d\in L_\pi^\gamma$ be such that $m \odot d \neq 0$.
Then $md \in L_\pi$, and so $\mathsf{bot}(m)$ is finer than $\mathsf{top}(d) = \gamma$.
Since $\mathsf{bot}(m)$ is finer than $\gamma$, we have $e_{\mathsf{bot}(m)} e_\gamma = e_\gamma$ by Lemma \ref{lem:idempotents}.
By Lemma \ref{lem:eproperties} we have $e_\gamma d = d$, which means that $md = \sigma e_{\mathsf{bot}(m)} d = \sigma d$.
Since $\mathsf{top}(\sigma d) = \sigma^{-1}(\gamma)$, we have $\sigma d \in L_\pi^{\gamma'}$, where $\gamma' = \sigma^{-1}(\gamma)$.
By Proposition~\ref{prop:orbitreps},
the right $G_{e_\pi}$-action on $L_\pi$ is free and $L_\pi^{\gamma'}$ is an orbit for this action.
This means that there is a unique $g\in G_{e_\pi}$ such that $m d = \ell_\pi^{\gamma'} g$.
\end{proof}

\begin{example} 
Let $\pi = 12|34$ and $\gamma= 13|24$. The following diagram equation is an example of
Proposition~\ref{prop:GeactiononL}, where the left hand side product is $md$ and the right hand side is $\ell_\pi^{\gamma'} g$ with
$\gamma' = 23|14$ and $g = 12\o3\o4 |34\o1\o2 \in G_{(2,2)}$.
\begin{align*}
    \begin{tikzpicture}[partition-diagram]
    \node[vertex] (G--4) at (4.5, -1) [shape = circle, draw] {};
    \node[vertex] (G--2) at (1.5, -1) [shape = circle, draw] {};
    \node[vertex] (G-2) at (1.5, 1) [shape = circle, draw] {};
    \node[vertex] (G-3) at (3.0, 1) [shape = circle, draw] {};
    \node[vertex] (G--3) at (3.0, -1) [shape = circle, draw] {};
    \node[vertex] (G--1) at (0, -1) [shape = circle, draw] {};
    \node[vertex] (G-1) at (0.0, 1) [shape = circle, draw] {};
    \node[vertex] (G-4) at (4.5, 1) [shape = circle, draw] {};
    \draw[] (G-2) .. controls +(0.5, -0.5) and +(-0.5, -0.5) .. (G-3);
    \draw[] (G--4) .. controls +(-0.6, 0.6) and +(0.6, 0.6) .. (G--2);
    \draw[] (G--2) .. controls +(0, 1) and +(0, -1) .. (G-2);
    \draw[] (G-1) .. controls +(0.7, -0.7) and +(-0.7, -0.7) .. (G-4);
    \draw[] (G--3) .. controls +(-0.6, 0.6) and +(0.6, 0.6) .. (G--1);
    \draw[] (G--1) .. controls +(0, 1) and +(0, -1) .. (G-1);
    \end{tikzpicture} \quad
    \begin{tikzpicture}[partition-diagram, representative]
    \node[vertex] (G--4) at (4.5, -1) [shape = circle, draw] {};
    \node[vertex] (G--3) at (3.0, -1) [shape = circle, draw] {};
    \node[vertex] (G-2) at (1.5, 1) [shape = circle, draw] {};
    \node[vertex] (G-4) at (4.5, 1) [shape = circle, draw] {};
    \node[vertex] (G--2) at (1.5, -1) [shape = circle, draw] {};
    \node[vertex] (G--1) at (0.0, -1) [shape = circle, draw] {};
    \node[vertex] (G-1) at (0.0, 1) [shape = circle, draw] {};
    \node[vertex] (G-3) at (3.0, 1) [shape = circle, draw] {};
    \draw[] (G-2) .. controls +(0.6, -0.6) and +(-0.6, -0.6) .. (G-4);
    \draw[] (G-4) .. controls +(0, -1) and +(0, 1) .. (G--4);
    \draw[] (G--4) .. controls +(-0.5, 0.5) and +(0.5, 0.5) .. (G--3);
    \draw[] (G-1) .. controls +(0.6, -0.6) and +(-0.6, -0.6) .. (G-3);
    \draw[] (G--2) .. controls +(-0.5, 0.5) and +(0.5, 0.5) .. (G--1);
    \draw[] (G--1) .. controls +(0, 1) and +(0, -1) .. (G-1);
    \end{tikzpicture}
    =
    \begin{tikzpicture}[partition-diagram]
    \node[vertex] (G--4) at (4.5, -1) [shape = circle, draw] {};
    \node[vertex] (G--3) at (3.0, -1) [shape = circle, draw] {};
    \node[vertex] (G-2) at (1.5, 1) [shape = circle, draw] {};
    \node[vertex] (G-3) at (3.0, 1) [shape = circle, draw] {};
    \node[vertex] (G--2) at (1.5, -1) [shape = circle, draw] {};
    \node[vertex] (G--1) at (0.0, -1) [shape = circle, draw] {};
    \node[vertex] (G-1) at (0.0, 1) [shape = circle, draw] {};
    \node[vertex] (G-4) at (4.5, 1) [shape = circle, draw] {};
    \draw[] (G-2) .. controls +(0.5, -0.5) and +(-0.5, -0.5) .. (G-3);
    \draw[] (G--4) .. controls +(-0.5, 0.5) and +(0.5, 0.5) .. (G--3);
    \draw[] (G--3) -- (G-3);
    \draw[] (G-1) .. controls +(0.7, -0.7) and +(-0.7, -0.7) .. (G-4);
    \draw[] (G--2) .. controls +(-0.5, 0.5) and +(0.5, 0.5) .. (G--1);
    \draw[] (G--1) .. controls +(0, 1) and +(0, -1) .. (G-1);
    \end{tikzpicture}
    =
    \begin{tikzpicture}[partition-diagram, representative]
    \node[vertex] (G--4) at (4.5, -1) [shape = circle, draw] {};
    \node[vertex] (G--3) at (3.0, -1) [shape = circle, draw] {};
    \node[vertex] (G-1) at (0.0, 1) [shape = circle, draw] {};
    \node[vertex] (G-4) at (4.5, 1) [shape = circle, draw] {};
    \node[vertex] (G--2) at (1.5, -1) [shape = circle, draw] {};
    \node[vertex] (G--1) at (0.0, -1) [shape = circle, draw] {};
    \node[vertex] (G-2) at (1.5, 1) [shape = circle, draw] {};
    \node[vertex] (G-3) at (3.0, 1) [shape = circle, draw] {};
    \draw[] (G-1) .. controls +(0.7, -0.7) and +(-0.7, -0.7) .. (G-4);
    \draw[] (G-4) .. controls +(0, -1) and +(0, 1) .. (G--4);
    \draw[] (G--4) .. controls +(-0.5, 0.5) and +(0.5, 0.5) .. (G--3);
    \draw[] (G-2) .. controls +(0.5, -0.5) and +(-0.5, -0.5) .. (G-3);
    \draw[] (G--2) .. controls +(-0.5, 0.5) and +(0.5, 0.5) .. (G--1);
    \draw[] (G--2) -- (G-2);
    \end{tikzpicture}  \quad
    \begin{tikzpicture}[partition-diagram]
    \node[vertex] (G--4) at (4.5, -1) [shape = circle, draw] {};
    \node[vertex] (G--3) at (3.0, -1) [shape = circle, draw] {};
    \node[vertex] (G-1) at (0.0, 1) [shape = circle, draw] {};
    \node[vertex] (G-2) at (1.5, 1) [shape = circle, draw] {};
    \node[vertex] (G--2) at (1.5, -1) [shape = circle, draw] {};
    \node[vertex] (G--1) at (0.0, -1) [shape = circle, draw] {};
    \node[vertex] (G-3) at (3.0, 1) [shape = circle, draw] {};
    \node[vertex] (G-4) at (4.5, 1) [shape = circle, draw] {};
    \draw[] (G-1) .. controls +(0.5, -0.5) and +(-0.5, -0.5) .. (G-2);
    \draw[] (G--4) .. controls +(-0.5, 0.5) and +(0.5, 0.5) .. (G--3);
    \draw[] (G-3) .. controls +(0.5, -0.5) and +(-0.5, -0.5) .. (G-4);
    \draw[] (G--2) .. controls +(-0.5, 0.5) and +(0.5, 0.5) .. (G--1);
    \draw[] (G--2) .. controls +(1, 1) and +(-1, -1) .. (G-3);
    \draw[] (G--3) .. controls +(-1, 1) and +(1, -1) .. (G-2);
    \end{tikzpicture}.
\end{align*}
\end{example}

\subsection{Irreducible representations of $\mathcal{U}_k$}
\label{sec:irreps-of-UBP-algebra}
In this section, we explain how each irreducible representation of
$\mathcal{U}_k$ is obtained by inflating an irreducible representation of
one of its maximal subgroups.
In Section~\ref{sec:irreps-of-UBP-algebra-2}, we will describe a tableau model
for these representations.

We begin by identifying a natural indexing set of the isomorphism classes of
irreducible representations of $\mathcal{U}_k$.

\begin{prop}
    The isomorphism classes of the irreducible representations of $\mathcal{U}_k$ are indexed by $I_k$ as defined in~\eqref{equation.Ik}.
\end{prop}

\begin{proof}
    For any finite monoid $M$, let $\operatorname{Irr}_\mathbb{C}(M)$ be
    the set of isomorphism classes of irreducible representations of $M$ over $\CC$.
    By \cite[Corollary 5.6]{Steinberg.2016}, there is a bijection between
    $\operatorname{Irr}_\mathbb{C}(M)$
    and
    $\bigcup_{e} \operatorname{Irr}_\mathbb{C}(G_e)$,
    where the idempotents $e$ are chosen one from each $\mathscr{J}$-class of $M$.

    Recall from Section~\ref{section.J classes} that
    $\{e_{\pi_\lambda} : \lambda \vdash k\}$ is a set of representative
    idempotents for the $\mathscr{J}$-classes of $\mathcal{U}_k$,
    and that the associated maximal subgroup $G_\lambda$ is isomorphic to
    $\mathfrak{S}_{a_1} \times \mathfrak{S}_{a_2} \times \cdots \times \mathfrak{S}_{a_k}$
    if $\lambda = (1^{a_1}2^{a_2} \ldots k^{a_k})$ (Corollary~\ref{cor:whatisGe}).
    Hence, the isomorphism classes of irreducible representations of
    $G_\lambda$ are indexed by sequences of partitions $(\lambda^{(1)},
    \ldots, \lambda^{(k)})$ such that $\lambda^{(i)}\vdash a_i$
    and $\sum_{i=1}^k ia_i = k$ (\cf Section~\ref{sec:irrepsGe}).
\end{proof}

For $\vec{\lambda} = (\lambda^{(1)}, \lambda^{(2)}, \ldots, \lambda^{(k)})$ with $|\lambda^{(i)}|=a_i$,
we define ${\defncolor \vtype(\vec{\lambda})} = (1^{a_1} 2^{a_2} \ldots k^{a_k})$.
Let $\vec{\lambda} \in I_k$ and write $\lambda = \vtype(\vec{\lambda})$.
Let $V_{G_\lambda}^{\vec{\lambda}}$ be the irreducible representation of $G_\lambda$ indexed by $\vec{\lambda}$.
By~\cite[Theorem 5.5]{Steinberg.2016},
\begin{equation*}
    W_{\mathcal{U}_k}^{\vec{\lambda}} = \mathsf{Ind}_{G_\lambda}^{\mathcal{U}_k}\big(V_{G_\lambda}^{\vec{\lambda}}\big) 
    \Big/ \mathsf{rad}\Big(\mathsf{Ind}_{G_\lambda}^{\mathcal{U}_k}\big(V_{G_\lambda}^{\vec{\lambda}}\big)\Big)
\end{equation*}
is an irreducible representation of $\mathcal{U}_k$.
Since $\mathcal{U}_k$ is a finite inverse monoid, the monoid algebra
$\mathbb{C}\mathcal{U}_k$ is semisimple~\cite[Corollary~9.4]{Steinberg.2016},
from which it follows that $\mathsf{rad}(\mathsf{Ind}_{G_\lambda}^{\mathcal{U}_k}(V_{G_\lambda}^{\vec{\lambda}})) = 0$.
Thus,
\begin{equation*}
    W_{\mathcal{U}_k}^{\vec{\lambda}} = \mathsf{Ind}_{G_\lambda}^{\mathcal{U}_k}\big(V_{G_\lambda}^{\vec{\lambda}}\big) 
    = \mathbb{C}L_\lambda \otimes_{\mathbb{C}G_\lambda} V_{G_\lambda}^{\vec{\lambda}},
\end{equation*}
where $\CC L_\lambda$ is the left Sch\"utzenberger representation associated with the idempotent $e_{\pi_\lambda}$
(\cf Section~\ref{sec:Schutzenberger-representations}).
Since $\CC L_\lambda$ is a $(\mathcal{U}_k, G_\lambda)$-bimodule, the tensor product
$\mathbb{C}L_\lambda \otimes_{\mathbb{C}G_\lambda} V_{G_\lambda}^{\vec{\lambda}}$ is a left
$\mathcal{U}_k$-module, where for all $d \in \mathcal{U}_k$, $\ell \in L_\lambda$ and $v \in V_{G_\lambda}^{\vec{\lambda}}$:
\begin{equation}\label{leftaction}
	d \cdot (\ell \otimes v)  = (d \odot \ell) \otimes v.
\end{equation}
Notice that the tensor product is over $\mathbb{C}G_\lambda$, which is the case throughout this section.

We now describe a basis of $W_{\mathcal{U}_k}^{\vec{\lambda}}$.
In Section~\ref{sec:irrepsGe}, we found that a basis of
$V_{G_\lambda}^{\vec{\lambda}}$ is given by the $\pi$-tableaux of shape $\vec{\lambda}$.
To obtain a basis of $W_{\mathcal{U}_k}^{\vec{\lambda}}$, it suffices to tensor this basis with
the orbit representatives of the right $G_\lambda$-action on $L_\lambda$,
as we prove next.

\begin{prop} \label{prop:irrepbasis}
Let $\vec{\lambda} \in I_k$, $\lambda = \vtype(\vec{\lambda})$ and $\pi = \pi_\lambda$.
Let
$\{\ell_\pi^\gamma : \gamma \vdash [k], \mathsf{type}(\gamma) = \lambda\}$
be the orbit representatives of the right $G_\lambda$-action on $L_\lambda$ as defined in
Section~\ref{sec:Schutzenberger-representations},
and let $\mathcal{B}_{\vec{\lambda}}(G_\lambda)$ be a basis for the irreducible
$G_\lambda$-representation $V_{G_\lambda}^{\vec{\lambda}}$ indexed by $\vec{\lambda}$.
Then a basis for the irreducible $\mathcal{U}_k$-representation $W_{\mathcal{U}_k}^{\vec{\lambda}}$ is
\[
    \mathcal{B}_{\vec{\lambda}}(\mathcal{U}_k) := \left\{ \ell_\pi^\gamma \otimes \mathbf{T} : \gamma \vdash [k], \mathsf{type}(\gamma) = \lambda \text{ and }
    \mathbf{T} \in \mathcal{B}_{\vec{\lambda}}(G_\lambda) \right\}.
\]
\end{prop}

\begin{proof}
    Since $\mathcal{B}_{\vec{\lambda}}(G_\lambda)$ is a basis of $V_{G_\lambda}^{\vec{\lambda}}$,
    it follows that $W_{\mathcal{U}_k}^{\vec{\lambda}}$ is spanned by $\ell \otimes \mathbf{T}$
    with $\ell \in L_\lambda$ and $\mathbf{T} \in \mathcal{B}_{\vec{\lambda}}(G_{\lambda})$.
    By Proposition \ref{prop:GeactiononL}, if $d\odot \ell\neq 0$, then there is a
    unique $g \in G_{\lambda}$ and $\gamma \vdash [k]$ satisfying $\mathsf{type}(\gamma) = \lambda$ and
    $d\odot \ell = \ell_\pi^\gamma g$. Thus,
    \begin{equation*}
        d \cdot \big(\ell \otimes \mathbf{T}\big)
        = \ell_\pi^{\gamma} g \otimes \mathbf{T}
        = \ell_\pi^{\gamma} \otimes g \mathbf{T},
    \end{equation*}
    which proves that $W_{\mathcal{U}_k}^{\vec{\lambda}}$ is spanned by elements of the form $\ell_\pi^\gamma \otimes \mathbf{T}$.

Furthermore, since $\{\ell_\pi^\gamma : \gamma \vdash [k], \mathsf{type}(\gamma) = \lambda\}$
is a basis of $\mathbb{C}L_\pi$ as a right $G_\lambda$-module
and $\mathcal{B}_{\vec{\lambda}}(G_\lambda)$ is a $\CC$-basis
for the irreducible $G_\lambda$-representation $V^{\vec{\lambda}}_{G_\lambda}$, then
$\mathcal{B}_{\vec{\lambda}}(\mathcal{U}_k)$ is linearly independent
as a vector space over $\mathbb{C}$.
\end{proof}

As a consequence of identifying that the basis is indexed by a pair
consisting of a set partition $\gamma \vdash [k]$ such that
$\mathsf{type}(\gamma) = \vtype(\vec{\lambda})$ and a $\pi$-tableau of
shape $\vec{\lambda}$, we have the following formula for the dimension of
the irreducible representation of $\mathcal{U}_k$.

\begin{cor}\label{cor:dimension}
Let $\vec{\lambda} = (\lambda^{(1)}, \lambda^{(2)}, \ldots, \lambda^{(k)}) \in I_k$
and $\lambda = \vtype(\vec{\lambda})$, then
\[
\dim W_{\mathcal{U}_k}^{\vec{\lambda}} = \mathsf{sp}_k(\lambda)
f^{\lambda^{(1)}} f^{\lambda^{(2)}}\cdots f^{\lambda^{(k)}},
\] where $\mathsf{sp}_k(\lambda)$ is equal to the number
of set partitions of type $\lambda$ (see Equation~\eqref{eq:nosetpartitions})
and $f^\lambda$ is equal to the number of standard tableaux of shape $\lambda$.
\end{cor}

\begin{example}
There are five irreducible $\mathcal{U}_3$-representations. We give their bases below:
\begin{align*}
W_{\mathcal{U}_3}^{((3))} & = \operatorname{span}\left\{ \ell_{1|2|3}^{1|2|3}\otimes \left(
\raisebox{-.05in}{\tikztableausmall{{1,2,3}}} \right)
\right\}, \\
W_{\mathcal{U}_3}^{((2,1))} & = \operatorname{span}\left\{ \ell_{1|2|3}^{1|2|3}\otimes \left(
\raisebox{-.13in}{\tikztableausmall{{3},{1,2}}} \right),
\ell_{1|2|3}^{1|2|3}\otimes \left(
\raisebox{-.13in}{\tikztableausmall{{2},{1,3}}} \right)
\right\},  \\
W_{\mathcal{U}_3}^{((1,1,1))} & = \operatorname{span}\left\{ \ell_{1|2|3}^{1|2|3}\otimes \left(
\raisebox{-.22in}{\tikztableausmall{{3},{2},{1}}} \right)
\right\}, \\
W_{\mathcal{U}_3}^{((1), (1))} & = \operatorname{span}\left\{ \ell_{1|23}^{1|23}\otimes  \left(
\raisebox{-.05in}{\tikztableausmall{{1}}},\raisebox{-.05in}{\tikztableausmall{{23}}} \right),
\ell_{1|23}^{2|13}\otimes \left(
\raisebox{-.05in}{\tikztableausmall{{1}}},\raisebox{-.05in}{\tikztableausmall{{23}}} \right),\ell_{1|23}^{3|12}\otimes \left(
\raisebox{-.05in}{\tikztableausmall{{1}}},\raisebox{-.05in}{\tikztableausmall{{23}}} \right)\right\}, \\
W_{\mathcal{U}_3}^{(\epart, \epart, (1))} & = \operatorname{span}\left\{ \ell_{123}^{123}\otimes \left(
\epart, \epart,\raisebox{-.05in}{\tikztableausmall{{123}}} \right)\right\}.
\end{align*}
\end{example}

\begin{example}
\label{big-example-of-action}
\def\aa{{\textsl a}}
\def\bb{{\textsl b}}
\def\cc{{\textsl c}}
\def\dd{{\textsl d}}
\def\ee{{\textsl e}}
\def\ff{{\textsl f}}
\def\gg{{\textsl g}}
\def\hh{{\textsl h}}
In this example, we illustrate the action of
an element in $\mathcal{U}_k$ on a basis element.
To demonstrate with an example that is
sufficiently large, take $k=17$ and represent the labels $10$ through $17$ by the letters
$\aa$ through $\hh$.

Let $\vec{\lambda} = ((2,1), (2,2), (1,1)) \in I_{17}$ so that $\lambda = \vtype(\vec{\lambda}) = (1^3 \, 2^4 \, 3^2)$.
Choose our basis element to be
$\ell^\gamma_\pi \otimes {\mathbf T} \in W^{\vec{\lambda}}_{\mathcal{U}_{17}}$,
where
\begin{equation*}
    \ell^\gamma_\pi = 
        2\, \o1 \mid 7\, \o2 \mid g\, \o3 \mid
        1\,3\,\o4\,\o5 \mid 5\,\bb\, \o6\,\o7 \mid 6\,\dd\,\o8\,\o9 \mid
        9\,e\,\o\aa\,\o\bb \mid 4\,\aa\,\cc\,\o\cc\,\o\dd\,\o\ee \mid
        8\,\ff\,\hh\,\o\ff\,\o\gg\,\o\hh
\end{equation*}
and
\[
{\mathbf T} =
\left( \quad\raisebox{-.2in}{\tikztableau{{3},{1,2}}}\, , \quad
\raisebox{-.2in}{\tikztableau{{67,\aa\bb},{45, 89}}}\, , \quad
\raisebox{-.2in}{\tikztableau{{\ff\gg\hh},{\cc\dd\ee}}}\quad\right).
\]

Now any element $d \in \mathcal{U}_{17}$ such that the number of
blocks in $d \ell^\gamma_\pi$ is smaller than the number of blocks in
$\ell^\gamma_\pi$ will act as $0$.

As an example then, let us consider the action of an element $d$ such that
the number of blocks in $d \ell^\gamma_\pi$ is equal to the number of blocks
of $\ell^\gamma_\pi$.  That is, $\mathsf{bot}(d)$ must be finer than
$\gamma$.  Let
\begin{equation*}
    d =   2\,\o8 \mid
            8\,\o2 \mid 9\,\o\gg \mid \aa\,\o\dd \mid \bb\,\o7 \mid \cc\,\o6 \mid \ee\,\o\aa \mid
            \ff\,\o3 \mid \hh\,\o1 \mid
            1\,4\,\o5\, \o\bb \mid 6\,7\,\o9\,\o\ee \mid 3\,\dd\,\o4\, \o\cc \mid 5\,\gg\,\o\ff\,\o\hh.
\end{equation*}
Then the action of $d$ on $\ell_\pi^\gamma \otimes \mathbf{T}$ is
$\ell^{\gamma'}_\pi \otimes g \cdot {\mathbf T}$, where
\begin{equation*}
    \gamma' =
                8 \mid
                9 \mid
                \bb \mid
                1\,4 \mid
                6\,7 \mid
                \aa\,\cc \mid
                \ff\,\hh \mid
                3\,\dd\,\ee \mid
                2\,5\,\gg
 \end{equation*}
and
\begin{equation*}
    g = 1\, \o1 \mid
        2\, \o3 \mid
        3\, \o2 \mid
        4\,5\,\o6\,\o7 \mid
        6\,7\,\o\aa\,\o\bb \mid
        8\,9\,\o8\,\o9 \mid
        \aa\,\bb\,\o4\,\o5 \mid
        \cc\,\dd\,\ee\,\o\cc\,\o\dd\,\o\ee \mid
        \ff\,\gg\,\hh\,\o\ff\,\o\gg\,\o\hh
     \in G_{e_\pi},
\end{equation*}
since $d \ell_\pi^\gamma = \ell_\pi^{\gamma'} g$. We must then
apply some straightening relations on $g \cdot \mathbf{T}$ to express it as
a linear combination of the basis elements.
\end{example}

\subsection{Tableau model for the irreducible $\mathcal{U}_k$-representations}
\label{sec:irreps-of-UBP-algebra-2}
We prove that the basis of $W_{\mathcal{U}_k}^{\vec{\lambda}}$ in
Proposition~\ref{prop:irrepbasis} is in bijection with certain sequences of
set-valued tableaux and we describe the action of $\mathcal{U}_k$ directly in
terms of these sequences.

\begin{definition}
    A \defn{uniform tableau} of shape $\vec{\lambda} = (\lambda^{(1)}, \ldots, \lambda^{(k)}) \in I_k$
    is a sequence of tableaux $\mathbf{S} = (S^{(1)}, \ldots, S^{(k)})$ such that:
    \begin{enumerate}
        \item $S^{(i)}$ is a tableau of shape $\lambda^{(i)}$ filled with subsets of $[k]$ of size $i$;
        \item $S^{(i)}$ is standard, \ie, increasing along rows and columns in the last letter order; and
        \item the subsets appearing in $\mathbf{S}$ form a set partition of $[k]$.
    \end{enumerate}
    We define \defn{$\mathcal{T}_{\vec{\lambda}}$} to be the set of uniform tableaux of shape $\vec{\lambda}$.
\end{definition}

\begin{example} Here are the elements in $\mathcal{T}_{((1), (1,1))}$:
$$\left(\raisebox{-.1in}{\tikztableausmall{1}, \tikztableausmall{45,23}}\right),
\left(\raisebox{-.1in}{\tikztableausmall{1}, \tikztableausmall{35,24}}\right),
\left(\raisebox{-.1in}{\tikztableausmall{1}, \tikztableausmall{25,34}}\right),
\left(\raisebox{-.1in}{\tikztableausmall{2}, \tikztableausmall{45,13}}\right),
\left(\raisebox{-.1in}{\tikztableausmall{2}, \tikztableausmall{35,14}}\right),$$
$$\left(\raisebox{-.1in}{\tikztableausmall{2}, \tikztableausmall{15,34}}\right),
\left(\raisebox{-.1in}{\tikztableausmall{3}, \tikztableausmall{45,12}}\right),
\left(\raisebox{-.1in}{\tikztableausmall{3}, \tikztableausmall{25,14}}\right),
\left(\raisebox{-.1in}{\tikztableausmall{3}, \tikztableausmall{15,24}}\right),
\left(\raisebox{-.1in}{\tikztableausmall{4}, \tikztableausmall{35,12}}\right),$$
$$\left(\raisebox{-.1in}{\tikztableausmall{4}, \tikztableausmall{25,13}}\right),
\left(\raisebox{-.1in}{\tikztableausmall{4}, \tikztableausmall{15,23}}\right),
\left(\raisebox{-.1in}{\tikztableausmall{5}, \tikztableausmall{34,12}}\right),
\left(\raisebox{-.1in}{\tikztableausmall{5}, \tikztableausmall{24,13}}\right),
\left(\raisebox{-.1in}{\tikztableausmall{5}, \tikztableausmall{14,23}}\right)~.$$
\end{example}

We now define an action of $\mathcal{U}_k$ on the vector space  $\mathbb{C}\mathcal{T}_{\vec{\lambda}}$ consisting of formal linear
combinations of the uniform tableaux in $\mathcal{T}_{\vec{\lambda}}$ with complex coefficients and then show that it is isomorphic
to the irreducible representation $W^{\vec{\lambda}}_{\mathcal{U}_k}$.

Recall that $\mathcal{U}_k$ is generated by $s_i$ and $b_i$,
where $1\leqslant i \leqslant k-1$, as described in Section~\ref{sec:presentation}. For $\mathbf{S}\in \mathcal{T}_{\vec{\lambda}}$,
let
\begin{equation}
    \label{bi-action-on-S}
    b_i \mathbf{S} = \begin{cases} \mathbf{S}, & \text{if $i$ and $i+1$ are in the same cell in $\mathbf{S}$,} \\
    0, & \text{otherwise,}\end{cases}
\end{equation}
and let $s_i\mathbf{S}$ be obtained from $\mathbf{S}$ by interchanging $i$ and $i+1$. It is possible that
$s_i \mathbf{S}$ is not standard, in which case we apply the Garnir straightening relations
(we illustrate this in Example~\ref{big-example-of-action2} and refer the reader
to \cite[Section 2.6]{Sagan.2001} or \cite{CarreLascouxLeclerc} for details)
to obtain a linear combination of elements in $\mathcal{T}_{\vec{\lambda}}$.
It is straightforward to verify that the relations in Section~\ref{sec:presentation} hold
so that $\mathbb{C}\mathcal{T}_{\vec{\lambda}}$ is a representation of $\mathcal{U}_k$.

For the next result, we remind the reader that the blocks of set partitions are
ordered using the graded last letter order and that
the elements $\ell_\pi^\gamma \otimes \mathbf{T}$, where $\mathsf{type}(\gamma)
= \lambda$ and $\mathbf{T} \in \mathcal{B}_{\vec{\lambda}}(G_\lambda)$,
form a basis of $W^{\vec{\lambda}}_{\mathcal{U}_k}$ (Proposition~\ref{prop:irrepbasis}).
\begin{theorem}\label{th:tableauxmodule}
Let $\vec{\lambda}\in I_k$ and write $\lambda = \vtype(\vec{\lambda})$ and $\pi = \pi_\lambda$.
For $\ell_{\pi}^\gamma \otimes \mathbf{T} \in \mathcal{B}_{\vec{\lambda}}(\mathcal{U}_k)$, let
$\rho(\ell_{\pi}^\gamma \otimes \mathbf{T})$ be the sequence of tableaux obtained
from $\mathbf{T}$ by replacing the block $\pi_i\in \pi$ with the block $\gamma_i\in \gamma$.
Then $\rho$ extends linearly to an isomorphism of representations
$\rho \colon W^{\vec{\lambda}}_{\mathcal{U}_k}\rightarrow \CC\mathcal{T}_{\vec{\lambda}}$.
\end{theorem}

\begin{proof}
Let $\mathbf{S} = \rho(\ell_\pi^\gamma \otimes \mathbf{T})$.
If we view $\ell_\pi^\gamma$ as a bijection from $\gamma$ to $\pi$,
then $\mathbf{S}$ is obtained by applying the inverse bijection to the entries in $\mathbf{T}$.
In particular, $\rho$ is invertible.

Notice that $\mathbf{S}\in \mathcal{T}_{\vec{\lambda}}$:
first of all, $\mathbf{S}$ has the same shape $\vec{\lambda}$ as $\mathbf{T}$;
its entries are the blocks of $\gamma$;
each $S^{(i)}$ is filled with blocks of the same size $i$ since $\ell_\pi^\gamma$ preserves the sizes of the blocks;
finally, since the blocks are ordered using graded last letter order,
each $S^{(i)}$ is standard since we have replaced the entries of $\mathbf{T}$ with blocks in the same order as those of $\pi$.

Now we show that the action of $\mathcal{U}_k$ commutes with $\rho$.
It suffices to show this for the generators $s_i$ and $b_i$.
By definition,
\begin{equation*}
    b_i \cdot \left(\ell_\pi^\gamma \otimes \mathbf{T}\right)
    = (b_i \odot \ell_\pi^\gamma) \otimes \mathbf{T}
    =
    \begin{cases}
        b_i \ell_\pi^\gamma \otimes \mathbf{T}, & \text{if $b_i \ell_\pi^\gamma \in L_\pi$}, \\
        0, & \text{otherwise}.
    \end{cases}
\end{equation*}
Note that $b_i \ell_\pi^\gamma \in L_\pi$ if and only if $b_i \ell_\pi^\gamma = \ell_\pi^\gamma$,
or equivalently, if and only if $i$ and $i+1$ are in the same block of $\gamma$. Thus,
\begin{equation*}
    b_i \cdot \left(\ell_\pi^\gamma \otimes \mathbf{T}\right)
    =
    \begin{cases}
        \ell_\pi^\gamma \otimes \mathbf{T}, & \text{if $i$ and $i+1$ are in the same block of $\gamma$,} \\
        0, & \text{otherwise.}
    \end{cases}
\end{equation*}
Comparing with Equation~\eqref{bi-action-on-S},
it follows that $\rho(b_i \cdot (\ell_\pi^\gamma \otimes \mathbf{T})) = b_i \cdot \rho(\ell_\pi^\gamma \otimes \mathbf{T})$.

Next, we consider the action of $s_i$. Tracing through the definitions, we have
\begin{equation*}
    s_i \cdot \left(\ell_\pi^\gamma \otimes \mathbf{T}\right)
    = (s_i \odot \ell_\pi^\gamma) \otimes \mathbf{T}
    = s_i \ell_\pi^\gamma \otimes \mathbf{T},
\end{equation*}
where the last equality follows from the observation that $s_i \ell_\pi^\gamma \in L_\pi$
because $\mathsf{bot}(s_i \ell_\pi^\gamma) = \pi$.

To describe $s_i \ell_\pi^\gamma$ explicitly, write
$\gamma = \{\gamma_1, \ldots, \gamma_\ell\}$ and $\pi = \{\pi_1, \ldots, \pi_\ell\}$
with the blocks order using graded last letter order,
and recall that $\ell_\pi^\gamma$ is the bijection that maps $\gamma_h$ to $\pi_h$.
If exchanging $i$ and $i+1$ in $\gamma$ does not change the order of the
blocks (\ie, $s_i(\gamma_1) < \cdots < s_i(\gamma_\ell)$ in
graded last letter order), then $s_i \ell_\pi^\gamma = \ell_\pi^{s_i(\gamma)}$
so that $s_i \cdot (\ell_\pi^\gamma \otimes \mathbf{T})
= \ell_\pi^{s_i(\gamma)} \otimes \mathbf{T}$.
Its image under $\rho$ is obtained from $\mathbf{T}$ by replacing each block
$\pi_j$ appearing in $\mathbf{T}$ with $s_i(\gamma_j)$,
which is precisely the definition of the action of $s_i$.
Thus, $\rho(s_i \cdot (\ell_\pi^\gamma \otimes \mathbf{T})) = s_i \cdot \rho(\ell_\pi^\gamma \otimes \mathbf{T})$.

Otherwise, there exist blocks $\gamma_{j}$ and $\gamma_{j+1}$ with
$\max(\gamma_{j}) = i$, $\max(\gamma_{j+1}) = i+1$,
$|\gamma_j| = |\gamma_{j+1}|$ and
\begin{equation} \label{eq:sigamma}
    s_i(\gamma) = \big\{
        s_i(\gamma_1), s_i(\gamma_2), \ldots,
        s_i(\gamma_{j-1}), s_i(\gamma_{j+1}), s_i(\gamma_{j}), s_i(\gamma_{j+2}), \ldots, s_i(\gamma_\ell)
        \big\},
\end{equation}
where the blocks are listed in graded last letter order.
Then $s_i \ell_\pi^\gamma = \ell_\pi^{s_i(\gamma)} g$, where
\[
g = \begin{pmatrix} \pi_1&\cdots&\pi_j&\pi_{j+1}&\cdots&\pi_{\ell}\\
\pi_1&\cdots&\pi_{j+1}&\pi_j&\cdots&\pi_{\ell}\end{pmatrix}
\]
is the permutation in $G_\lambda$ that exchanges $\pi_{j}$ and $\pi_{j+1}$.
Therefore, the image of
$s_i \cdot (\ell_\pi^\gamma \otimes \mathbf{T}) = \ell_\pi^{s_i(\gamma)} \otimes g \cdot \mathbf{T}$
under $\rho$ is obtained from $\mathbf{T}$ by exchanging $\pi_j$ and $\pi_{j+1}$
and then each $\pi_{h}$ is replaced with the \emph{block in position} $h$ of $s_i(\gamma)$
(as listed in Equation~\eqref{eq:sigamma}).
Thus, $\rho(s_i \cdot (\ell_\pi^\gamma \otimes \mathbf{T}))$ is
again obtained from $\rho(\ell_\pi^\gamma \otimes \mathbf{T})$
by interchanging $i$ and $i+1$.
\end{proof}

\begin{example}
Under the bijection $\rho$ described in Theorem~\ref{th:tableauxmodule} the basis elements of $W^{((1),(1))}_{\mathcal{U}_3}$ correspond to the tableaux in
$\mathcal{T}_{\vec{\lambda}}$ as follows:
\begin{align*}
&\ell_{1|23}^{1|23}\otimes  \left(
\raisebox{-.05in}{\tikztableausmall{{1}}},\raisebox{-.05in}{\tikztableausmall{{23}}} \right) \mapsto \left(\raisebox{-.05in}{\tikztableausmall{{1}}},\raisebox{-.05in}{\tikztableausmall{{23}}} \right),\\
& \ell_{1|23}^{2|13}\otimes \left(
\raisebox{-.05in}{\tikztableausmall{{1}}},\raisebox{-.05in}{\tikztableausmall{{23}}} \right) \mapsto
\left(\raisebox{-.05in}{\tikztableausmall{{2}}},\raisebox{-.05in}{\tikztableausmall{{13}}} \right),\\
&\ell_{1|23}^{3|12}\otimes \left(
\raisebox{-.05in}{\tikztableausmall{{1}}},\raisebox{-.05in}{\tikztableausmall{{23}}} \right) \mapsto
\left(\raisebox{-.05in}{\tikztableausmall{{3}}},\raisebox{-.05in}{\tikztableausmall{{12}}} \right).
\end{align*}
\end{example}

\begin{example}
\label{big-example-of-action2}
\def\aa{{\textsl a}}
\def\bb{{\textsl b}}
\def\cc{{\textsl c}}
\def\dd{{\textsl d}}
\def\ee{{\textsl e}}
\def\ff{{\textsl f}}
\def\gg{{\textsl g}}
\def\hh{{\textsl h}}
Let $\vec{\lambda} = ((2,1), (2,2), (1,1)) \in I_{17}$ so that $\lambda = \vtype(\vec{\lambda}) = (1^3 \, 2^4 \, 3^2)$.
As in Example \ref{big-example-of-action}, we represent $10$ through $17$ by the letters $\aa$ through $\hh$.
Consider
\[
\mathbf{S} = \left( \quad\raisebox{-.2in}{\tikztableau{{\gg},{2,7}}}\,, \quad
\raisebox{-.2in}{\tikztableau{{5\bb,9\ee},{13, 6\dd}}}\,, \quad
\raisebox{-.2in}{\tikztableau{{8\ff\hh},{4\aa\cc}}}\quad\right)
\]
which is the image under $\rho$ of the basis element in Example~\ref{big-example-of-action}.
Consider the action of
\begin{equation*}
    d =
            2\,\o8 \mid
            8\,\o2 \mid 9\,\o\gg \mid \aa\,\o\dd \mid \bb\,\o7 \mid \cc\,\o6 \mid \ee\,\o\aa \mid
            \ff\,\o3 \mid \hh\,\o1 \mid
            1\,4\,\o5\, \o\bb \mid 6\,7\,\o9\,\o\ee \mid 3\,\dd\,\o4\, \o\cc \mid 5\,\gg\,\o\ff\,\o\hh
\end{equation*}
on the uniform tableau $\mathbf{S}$.
Since $\mathsf{bot}(d)$ is finer than
the set partition of the entries of $\mathbf{S}$, the
result is non-zero and is equal to
\[
\left( \quad\raisebox{-.2in}{\tikztableau{{9},{8,\bb}}}\, , \quad
\raisebox{-.2in}{\tikztableau{{14,67},{\ff\hh, \aa\cc}}}\, , \quad
\raisebox{-.2in}{\tikztableau{{25\gg},{3\dd\ee}}}\quad\right)~.
\]
This is not a basis element because the middle tableau
is not standard with respect to the graded last letter order.
We then
apply some straightening relations to express it as
a linear combination of the basis elements.
The interested reader may then compute that the action of $d$ on $\mathbf{S}$
is equal to the following
linear combination:
\[
\left( \quad\raisebox{-.2in}{\tikztableau{{9},{8,\bb}}}\,, \quad
\raisebox{-.2in}{\tikztableau{{\aa\cc,\ff\hh},{14,67}}}\,, \quad
\raisebox{-.2in}{\tikztableau{{25\gg},{3\dd\ee}}}\quad\right)
-
\left( \quad\raisebox{-.2in}{\tikztableau{{9},{8,\bb}}}\,, \quad
\raisebox{-.2in}{\tikztableau{{67,\ff\hh},{14,\aa\cc}}}\,, \quad
\raisebox{-.2in}{\tikztableau{{25\gg},{3\dd\ee}}}\quad\right)~.
\]
\end{example}

\section{The characters of $\mathcal{U}_k$}
\label{section.characters}
The last two sections of this paper are devoted to a careful analysis of the
characters of $\mathcal{U}_k$.  This development will allow
us to give an expression of the character values in terms of symmetric
functions in Section \ref{sec:charsSym} and make explicit the connection between
plethysm and the restriction of $\mathcal{U}_k$-modules to the symmetric
group $\mathfrak{S}_k \subseteq \mathcal{U}_k$.

In this section, we describe the characters for the irreducible
$\mathcal{U}_k$-representations that were presented in the previous section.
In general, the characters of finite monoids were studied by McAlister \cite{McAlister}.
Here we use the notation described in~\cite[Chapter 7]{Steinberg.2016}.

\subsection{Generalized conjugacy classes}
Let $M$ be a finite monoid. For every $m \in M$, the subsemigroup of $M$ generated
by $m$ contains a unique idempotent that we denote \defn{$m^\omega$} (see~\cite[Corollary 1.2]{Steinberg.2016}). One
can think of $\omega$ as representing the smallest positive integer such that $m^\omega$ is an idempotent.
Two elements $m$ and $n$ in $M$ are \defn{conjugate} if there exist $x, x' \in M$ such that
$xx'x = x$, $x'xx' = x'$, $x'x = m^{\omega}$,
$xx' = n^{\omega}$ and $x m^{\omega +1}x' = n^{\omega +1}$.
This is an equivalence relation whose equivalence classes are called the
\defn{generalized conjugacy classes} of $M$.
Notice that $m$ and $m^{\omega + 1}$ are conjugate for all $m\in M$
\cite[Chapter 7]{Steinberg.2016}.

By \cite[Proposition~7.4]{Steinberg.2016},
there is a bijection between the generalized conjugacy classes of $M$
and the union of the sets of conjugacy classes of the maximal subgroups
$G_{e_1}, \ldots, G_{e_s}$, where $e_1, \ldots, e_s$ are idempotents chosen
one from each $\mathscr{J}$-class of $M$ that contains an idempotent.
The bijection is obtained by intersecting a generalized conjugacy
class of $M$ with the conjugacy classes of $G_{e_i}$: exactly one of these
intersections is nonempty.
In particular, to select a set of representatives of the generalized conjugacy
classes of $M$, it suffices to take one element from each of the conjugacy
classes of the maximal subgroups $G_{e_1}, \ldots, G_{e_s}$.

We now apply the above to $\mathcal{U}_k$.
Since $\mathcal{U}_k$ is an inverse monoid, every
$x\in \mathcal{U}_k$ has a (unique) generalized inverse $\widetilde{x}$
satisfying $x\widetilde{x}x = x$ and $\widetilde{x}x\widetilde{x}
= \widetilde{x}$ (\cf Section~\ref{sec:inverse-monoid}).
Therefore, two elements $c$ and $d$ are conjugate in $\mathcal{U}_k$ if and only if there
exists $x \in \mathcal{U}_k$ such that
$\widetilde{x}x = c^{\omega}$,
$x \widetilde{x} = d^{\omega}$
and $x c^{\omega +1} \widetilde{x} = d^{\omega + 1}$.

We next define a notion of ``cycle type'' for the elements of $\mathcal{U}_k$,
which will allow us to determine whether two elements are conjugate in $\mathcal{U}_k$.
First, let $d$ be an element of a maximal subgroup $G_{e_\pi}$.
Then $d$ is a permutation of the blocks of $\pi$ that maps blocks of size $i$
to blocks of size $i$ for every $1 \leqslant i \leqslant k$.
Letting $d^{(i)}$ denote the restriction of $d$ to the blocks of size $i$ of $\pi$,
we define the \defn{cycle type} of $d$ to be $\mathsf{cycletype}(d) = (\mu^{(1)}, \mu^{(2)}, \ldots, \mu^{(k)})$,
where $\mu^{(i)}$ is the cycle type of the permutation $d^{(i)}$.
For an arbitrary element $x \in \mathcal{U}_k$, we define its
\defn{cycle type} to be the cycle type of $x^{\omega +1} \in G_{x^\omega}$.
In other words, $\mathsf{cycletype}(x) = \mathsf{cycletype}(x^{\omega + 1})$.

\begin{prop}\label{prop:conjugacyclasses}
    Two elements $c, d \in \mathcal{U}_k$ are conjugate if and only if
    \[
        \mathsf{cycletype}(c) = \mathsf{cycletype}(d).
    \]
\end{prop}

\begin{proof}
    Suppose $\mathsf{cycletype}(c) = \mathsf{cycletype}(d) = (\mu^{(1)}, \mu^{(2)}, \ldots, \mu^{(k)})$.
    Hence by definition
    $\mathsf{cycletype}(c^{\omega+1}) = \mathsf{cycletype}(d^{\omega +1})$.
    Then $c^\omega = e_\pi$ and $d^{\omega} = e_\gamma$ for some set partitions $\pi$ and $\gamma$.
    Moreover, by the definition of cycle type, $\pi$ and $\gamma$ must have
    type $(1^{a_1} 2^{a_2}\ldots k^{a_k})$, where $a_i = |\mu^{(i)}|$ for all $1 \leqslant i \leqslant k$.
    By Corollary~\ref{cor:isomGe}, there exists a permutation $\sigma \in \mathfrak{S}_k$ such that
    $\widetilde{\sigma} G_{e_\pi} \sigma = G_{e_\gamma}$; note that $\widetilde{\sigma} = \sigma^{-1}$ for permutations.
    Thus, $\widetilde{\sigma} c^{\omega+1} \sigma$ and $d^{\omega+1}$
    both belong to $G_{e_\gamma}$ and they both have the same cycle type.
    Hence, they are conjugate in $G_{e_\gamma}$, which implies they are conjugate
    in $\mathcal{U}_k$: explicitly, there exists $y \in G_{e_\gamma}$ such that
    $\widetilde{y} (\widetilde{\sigma} c^{\omega+1} \sigma) y = d^{\omega + 1}$,
    and so the element $x = \widetilde{\sigma y}$ satisfies
    $x \widetilde{x} = \widetilde{y} y = e_\gamma = d^\omega$,
    $\widetilde{x} x = \sigma e_\gamma \widetilde{\sigma} = e_\pi = c^\omega$
    and $x c^{\omega+1} \widetilde{x} = d^{\omega+1}$.

    Conversely, suppose $c$ and $d$ are conjugate in $\mathcal{U}_k$.
    Then there exists $x \in \mathcal{U}_k$ such that
        $\widetilde{x} x = c^\omega$,
        $x \widetilde{x} = d^\omega$,
        and $x c^{\omega + 1} \widetilde{x} = d^{\omega + 1}$.
    Let $\pi = \mathsf{bot}(x)$ and $\gamma = \mathsf{top}(x)$
    so that $c^\omega = e_\pi$ and $d^\omega = e_\gamma$.
    Then $\mathsf{type(\pi)} = \mathsf{type(\gamma)}$
    because $x$ is a bijection from $\pi$ to $\gamma$ that preserves block sizes.
    By Corollary~\ref{cor:isomGe}, there exists a permutation $\sigma \in \mathfrak{S}_k$ such that
    $\widetilde{\sigma} G_{e_\pi} \sigma = G_{e_\gamma}$.
    Thus, $\widetilde{\sigma} c^{\omega + 1} \sigma$ and $d^{\omega + 1}$
    both belong to $G_{e_\gamma}$ and they are conjugate in $\mathcal{U}_k$.
    It follows that they are conjugate in $G_{e_\gamma}$
    and so they have the same cycle type since $G_{e_\gamma}$ is a group of
    permutations.
    Hence,
    $\mathsf{cycletype}(c^{\omega+1})
    = \mathsf{cycletype}(\widetilde{\sigma} c^{\omega+1} \sigma)
    = \mathsf{cycletype}(d^{\omega+1})$.
\end{proof}

The above gives a straightforward algorithm for computing the
cycle type of any $d\in \mathcal{U}_k$.  To find the cycle type of
$d$, we compute $d^{\omega +1}$ and then find the cycle type of
$d^{\omega +1} \in G_{d^{\omega}}$.  For $\vec{\mu} \in I_k$, define
\[
{\defncolor C_{\vec{\mu}}} = \{ x \in \mathcal{U}_k : \mathsf{cycletype}(x) = \vec{\mu} \}~.
\]

\begin{example} The generalized conjugacy classes in $\mathcal{U}_3$
are listed below:
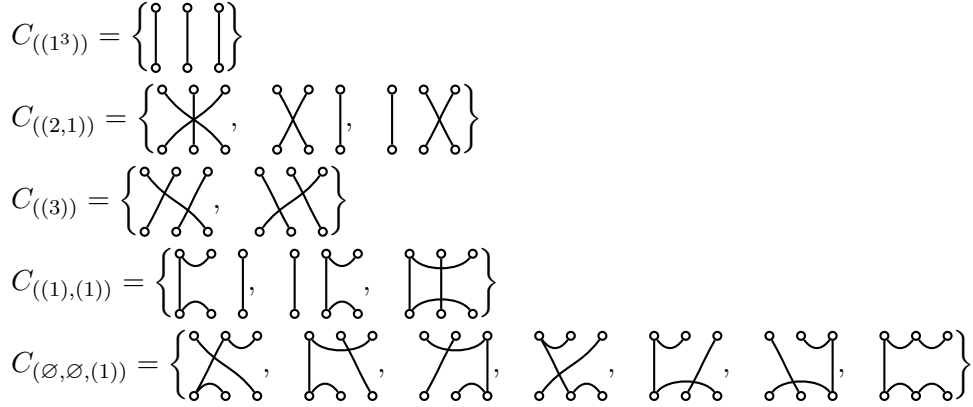
\begin{align*}
  &C_{((1^3))} =  \left\{\begin{tikzpicture}[partition-diagram, xscale=0.75]
    \node[vertex] (G--3) at (2.8, -1) [shape = circle, draw] {};
    \node[vertex] (G-3) at (2.8, 1) [shape = circle, draw] {};
    \node[vertex] (G--2) at (1.4, -1) [shape = circle, draw] {};
    \node[vertex] (G-2) at (1.4, 1) [shape = circle, draw] {};
    \node[vertex] (G--1) at (0.0, -1) [shape = circle, draw] {};
    \node[vertex] (G-1) at (0.0, 1) [shape = circle, draw] {};
    \draw[] (G-3) .. controls +(0, -1) and +(0, 1) .. (G--3);
    \draw[] (G-2) .. controls +(0, -1) and +(0, 1) .. (G--2);
    \draw[] (G-1) .. controls +(0, -1) and +(0, 1) .. (G--1);
    \end{tikzpicture}\right\}
\\
    &C_{((2,1))} = \left\{
    \begin{tikzpicture}[partition-diagram, xscale=0.75]
    \node[vertex] (G--3) at (2.8, -1) [shape = circle, draw] {};
    \node[vertex] (G-1) at (0.0, 1) [shape = circle, draw] {};
    \node[vertex] (G--2) at (1.4, -1) [shape = circle, draw] {};
    \node[vertex] (G-2) at (1.4, 1) [shape = circle, draw] {};
    \node[vertex] (G--1) at (0.0, -1) [shape = circle, draw] {};
    \node[vertex] (G-3) at (2.8, 1) [shape = circle, draw] {};
    \draw[] (G-1) .. controls +(1, -1) and +(-1, 1) .. (G--3);
    \draw[] (G-2) .. controls +(0, -1) and +(0, 1) .. (G--2);
    \draw[] (G-3) .. controls +(-1, -1) and +(1, 1) .. (G--1);
    \end{tikzpicture},
    \quad
    \begin{tikzpicture}[partition-diagram, xscale=0.75]
    \node[vertex] (G--3) at (2.8, -1) [shape = circle, draw] {};
    \node[vertex] (G-3) at (2.8, 1) [shape = circle, draw] {};
    \node[vertex] (G--2) at (1.4, -1) [shape = circle, draw] {};
    \node[vertex] (G-1) at (0.0, 1) [shape = circle, draw] {};
    \node[vertex] (G--1) at (0.0, -1) [shape = circle, draw] {};
    \node[vertex] (G-2) at (1.4, 1) [shape = circle, draw] {};
    \draw[] (G-3) .. controls +(0, -1) and +(0, 1) .. (G--3);
    \draw[] (G-1) .. controls +(0.75, -1) and +(-0.75, 1) .. (G--2);
    \draw[] (G-2) .. controls +(-0.75, -1) and +(0.75, 1) .. (G--1);
    \end{tikzpicture},
    \quad
    \begin{tikzpicture}[partition-diagram, xscale=0.75]
    \node[vertex] (G--3) at (2.8, -1) [shape = circle, draw] {};
    \node[vertex] (G-2) at (1.4, 1) [shape = circle, draw] {};
    \node[vertex] (G--2) at (1.4, -1) [shape = circle, draw] {};
    \node[vertex] (G-3) at (2.8, 1) [shape = circle, draw] {};
    \node[vertex] (G--1) at (0.0, -1) [shape = circle, draw] {};
    \node[vertex] (G-1) at (0.0, 1) [shape = circle, draw] {};
    \draw[] (G-2) .. controls +(0.75, -1) and +(-0.75, 1) .. (G--3);
    \draw[] (G-3) .. controls +(-0.75, -1) and +(0.75, 1) .. (G--2);
    \draw[] (G-1) .. controls +(0, -1) and +(0, 1) .. (G--1);
    \end{tikzpicture}\right\}
    \\
    &C_{((3))}= \left\{
   \begin{tikzpicture}[partition-diagram, xscale=0.75]
    \node[vertex] (G--3) at (2.8, -1) [shape = circle, draw] {};
    \node[vertex] (G-1) at (0.0, 1) [shape = circle, draw] {};
    \node[vertex] (G--2) at (1.4, -1) [shape = circle, draw] {};
    \node[vertex] (G-3) at (2.8, 1) [shape = circle, draw] {};
    \node[vertex] (G--1) at (0.0, -1) [shape = circle, draw] {};
    \node[vertex] (G-2) at (1.4, 1) [shape = circle, draw] {};
    \draw[] (G-1) .. controls +(1, -1) and +(-1, 1) .. (G--3);
    \draw[] (G-3) .. controls +(-0.75, -1) and +(0.75, 1) .. (G--2);
    \draw[] (G-2) .. controls +(-0.75, -1) and +(0.75, 1) .. (G--1);
    \end{tikzpicture},
    \quad
    \begin{tikzpicture}[partition-diagram, xscale=0.75]
    \node[vertex] (G--3) at (2.8, -1) [shape = circle, draw] {};
    \node[vertex] (G-2) at (1.4, 1) [shape = circle, draw] {};
    \node[vertex] (G--2) at (1.4, -1) [shape = circle, draw] {};
    \node[vertex] (G-1) at (0.0, 1) [shape = circle, draw] {};
    \node[vertex] (G--1) at (0.0, -1) [shape = circle, draw] {};
    \node[vertex] (G-3) at (2.8, 1) [shape = circle, draw] {};
    \draw[] (G-2) .. controls +(0.75, -1) and +(-0.75, 1) .. (G--3);
    \draw[] (G-1) .. controls +(0.75, -1) and +(-0.75, 1) .. (G--2);
    \draw[] (G-3) .. controls +(-1, -1) and +(1, 1) .. (G--1);
    \end{tikzpicture}
    \right\}
    \\
    &C_{((1),(1))} = \left\{
    \begin{tikzpicture}[partition-diagram, idempotent, xscale=0.75]
    \node[vertex] (G--3) at (2.8, -1) [shape = circle, draw] {};
    \node[vertex] (G-3) at (2.8, 1) [shape = circle, draw] {};
    \node[vertex] (G--2) at (1.4, -1) [shape = circle, draw] {};
    \node[vertex] (G--1) at (0.0, -1) [shape = circle, draw] {};
    \node[vertex] (G-1) at (0.0, 1) [shape = circle, draw] {};
    \node[vertex] (G-2) at (1.4, 1) [shape = circle, draw] {};
    \draw[] (G-3) .. controls +(0, -1) and +(0, 1) .. (G--3);
    \draw[] (G-1) .. controls +(0.5, -0.5) and +(-0.5, -0.5) .. (G-2);
    \draw[] (G--2) .. controls +(-0.5, 0.5) and +(0.5, 0.5) .. (G--1);
    \draw[] (G--1) .. controls +(0, 1) and +(0, -1) .. (G-1);
    \end{tikzpicture},
    \quad
    \begin{tikzpicture}[partition-diagram, idempotent, xscale=0.75]
    \node[vertex] (G--3) at (2.8, -1) [shape = circle, draw] {};
    \node[vertex] (G--2) at (1.4, -1) [shape = circle, draw] {};
    \node[vertex] (G-2) at (1.4, 1) [shape = circle, draw] {};
    \node[vertex] (G-3) at (2.8, 1) [shape = circle, draw] {};
    \node[vertex] (G--1) at (0.0, -1) [shape = circle, draw] {};
    \node[vertex] (G-1) at (0.0, 1) [shape = circle, draw] {};
    \draw[] (G-2) .. controls +(0.5, -0.5) and +(-0.5, -0.5) .. (G-3);
    \draw[] (G--3) .. controls +(-0.5, 0.5) and +(0.5, 0.5) .. (G--2);
    \draw[] (G--2) .. controls +(0, 1) and +(0, -1) .. (G-2);
    \draw[] (G-1) .. controls +(0, -1) and +(0, 1) .. (G--1);
    \end{tikzpicture},
    \quad
    \begin{tikzpicture}[partition-diagram, idempotent, xscale=0.75]
    \node[vertex] (G--3) at (2.8, -1) [shape = circle, draw] {};
    \node[vertex] (G--1) at (0.0, -1) [shape = circle, draw] {};
    \node[vertex] (G-1) at (0.0, 1) [shape = circle, draw] {};
    \node[vertex] (G-3) at (2.8, 1) [shape = circle, draw] {};
    \node[vertex] (G--2) at (1.4, -1) [shape = circle, draw] {};
    \node[vertex] (G-2) at (1.4, 1) [shape = circle, draw] {};
    \draw[] (G-1) .. controls +(0.6, -0.6) and +(-0.6, -0.6) .. (G-3);
    \draw[] (G--3) .. controls +(-0.6, 0.6) and +(0.6, 0.6) .. (G--1);
    \draw[] (G--1) .. controls +(0, 1) and +(0, -1) .. (G-1);
    \draw[] (G-2) .. controls +(0, -1) and +(0, 1) .. (G--2);
    \end{tikzpicture}
    \right\}
    \\
    &C_{(\epart,\epart,(1))} = \left\{
    \begin{tikzpicture}[partition-diagram, xscale=0.75]
    \node[vertex] (G--3) at (2.8, -1) [shape = circle, draw] {};
    \node[vertex] (G-1) at (0.0, 1) [shape = circle, draw] {};
    \node[vertex] (G--2) at (1.4, -1) [shape = circle, draw] {};
    \node[vertex] (G--1) at (0.0, -1) [shape = circle, draw] {};
    \node[vertex] (G-2) at (1.4, 1) [shape = circle, draw] {};
    \node[vertex] (G-3) at (2.8, 1) [shape = circle, draw] {};
    \draw[] (G-1) .. controls +(1, -1) and +(-1, 1) .. (G--3);
    \draw[] (G-2) .. controls +(0.5, -0.5) and +(-0.5, -0.5) .. (G-3);
    \draw[] (G--2) .. controls +(-0.5, 0.5) and +(0.5, 0.5) .. (G--1);
    \draw[] (G--1) .. controls +(0.75, 1) and +(-0.75, -1) .. (G-2);
    \end{tikzpicture},
    \quad
    \begin{tikzpicture}[partition-diagram, xscale=0.75]
    \node[vertex] (G--3) at (2.8, -1) [shape = circle, draw] {};
    \node[vertex] (G-2) at (1.4, 1) [shape = circle, draw] {};
    \node[vertex] (G--2) at (1.4, -1) [shape = circle, draw] {};
    \node[vertex] (G--1) at (0.0, -1) [shape = circle, draw] {};
    \node[vertex] (G-1) at (0.0, 1) [shape = circle, draw] {};
    \node[vertex] (G-3) at (2.8, 1) [shape = circle, draw] {};
    \draw[] (G-2) .. controls +(0.75, -1) and +(-0.75, 1) .. (G--3);
    \draw[] (G-1) .. controls +(0.6, -0.6) and +(-0.6, -0.6) .. (G-3);
    \draw[] (G--2) .. controls +(-0.5, 0.5) and +(0.5, 0.5) .. (G--1);
    \draw[] (G--1) .. controls +(0, 1) and +(0, -1) .. (G-1);
    \end{tikzpicture},
    \quad
     \begin{tikzpicture}[partition-diagram, xscale=0.75]
    \node[vertex] (G--3) at (2.8, -1) [shape = circle, draw] {};
    \node[vertex] (G--2) at (1.4, -1) [shape = circle, draw] {};
    \node[vertex] (G-1) at (0.0, 1) [shape = circle, draw] {};
    \node[vertex] (G-3) at (2.8, 1) [shape = circle, draw] {};
    \node[vertex] (G--1) at (0.0, -1) [shape = circle, draw] {};
    \node[vertex] (G-2) at (1.4, 1) [shape = circle, draw] {};
    \draw[] (G-1) .. controls +(0.6, -0.6) and +(-0.6, -0.6) .. (G-3);
    \draw[] (G-3) .. controls +(0, -1) and +(0, 1) .. (G--3);
    \draw[] (G--3) .. controls +(-0.5, 0.5) and +(0.5, 0.5) .. (G--2);
    \draw[] (G-2) .. controls +(-0.75, -1) and +(0.75, 1) .. (G--1);
    \end{tikzpicture},
    \quad
    \begin{tikzpicture}[partition-diagram, xscale=0.75]
    \node[vertex] (G--3) at (2.8, -1) [shape = circle, draw] {};
    \node[vertex] (G--2) at (1.4, -1) [shape = circle, draw] {};
    \node[vertex] (G-1) at (0.0, 1) [shape = circle, draw] {};
    \node[vertex] (G-2) at (1.4, 1) [shape = circle, draw] {};
    \node[vertex] (G--1) at (0.0, -1) [shape = circle, draw] {};
    \node[vertex] (G-3) at (2.8, 1) [shape = circle, draw] {};
    \draw[] (G-1) .. controls +(0.5, -0.5) and +(-0.5, -0.5) .. (G-2);
    \draw[] (G--3) .. controls +(-0.5, 0.5) and +(0.5, 0.5) .. (G--2);
    \draw[] (G--2) .. controls +(-0.75, 1) and +(0.75, -1) .. (G-1);
    \draw[] (G-3) .. controls +(-1, -1) and +(1, 1) .. (G--1);
    \end{tikzpicture},
    \quad
    \begin{tikzpicture}[partition-diagram, xscale=0.75]
    \node[vertex] (G--3) at (2.8, -1) [shape = circle, draw] {};
    \node[vertex] (G--1) at (0.0, -1) [shape = circle, draw] {};
    \node[vertex] (G-1) at (0.0, 1) [shape = circle, draw] {};
    \node[vertex] (G-2) at (1.4, 1) [shape = circle, draw] {};
    \node[vertex] (G--2) at (1.4, -1) [shape = circle, draw] {};
    \node[vertex] (G-3) at (2.8, 1) [shape = circle, draw] {};
    \draw[] (G-1) .. controls +(0.5, -0.5) and +(-0.5, -0.5) .. (G-2);
    \draw[] (G--3) .. controls +(-0.6, 0.6) and +(0.6, 0.6) .. (G--1);
    \draw[] (G--1) .. controls +(0, 1) and +(0, -1) .. (G-1);
    \draw[] (G-3) .. controls +(-0.75, -1) and +(0.75, 1) .. (G--2);
    \end{tikzpicture},
    \quad
    \begin{tikzpicture}[partition-diagram, xscale=0.75]
    \node[vertex] (G--3) at (2.8, -1) [shape = circle, draw] {};
    \node[vertex] (G--1) at (0.0, -1) [shape = circle, draw] {};
    \node[vertex] (G-2) at (1.4, 1) [shape = circle, draw] {};
    \node[vertex] (G-3) at (2.8, 1) [shape = circle, draw] {};
    \node[vertex] (G--2) at (1.4, -1) [shape = circle, draw] {};
    \node[vertex] (G-1) at (0.0, 1) [shape = circle, draw] {};
    \draw[] (G-2) .. controls +(0.5, -0.5) and +(-0.5, -0.5) .. (G-3);
    \draw[] (G-3) .. controls +(0, -1) and +(0, 1) .. (G--3);
    \draw[] (G--3) .. controls +(-0.6, 0.6) and +(0.6, 0.6) .. (G--1);
    \draw[] (G-1) .. controls +(0.75, -1) and +(-0.75, 1) .. (G--2);
    \end{tikzpicture},
    \quad
    \begin{tikzpicture}[partition-diagram, idempotent, xscale=0.75]
    \node[vertex] (G--3) at (2.8, -1) [shape = circle, draw] {};
    \node[vertex] (G--2) at (1.4, -1) [shape = circle, draw] {};
    \node[vertex] (G--1) at (0.0, -1) [shape = circle, draw] {};
    \node[vertex] (G-1) at (0.0, 1) [shape = circle, draw] {};
    \node[vertex] (G-2) at (1.4, 1) [shape = circle, draw] {};
    \node[vertex] (G-3) at (2.8, 1) [shape = circle, draw] {};
    \draw[] (G-1) .. controls +(0.5, -0.5) and +(-0.5, -0.5) .. (G-2);
    \draw[] (G-2) .. controls +(0.5, -0.5) and +(-0.5, -0.5) .. (G-3);
    \draw[] (G--3) .. controls +(-0.5, 0.5) and +(0.5, 0.5) .. (G--2);
    \draw[] (G--2) .. controls +(-0.5, 0.5) and +(0.5, 0.5) .. (G--1);
    \draw[] (G--1) .. controls +(0, 1) and +(0, -1) .. (G-1);
    \end{tikzpicture}\right\}
\end{align*}
\end{example}

\begin{example}
\def\o{\overline}
Consider the element
$x \in \mathcal{U}_{10}$ with diagram
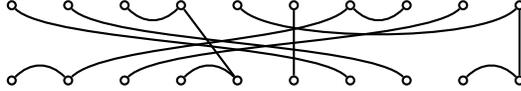
\begin{center}
\begin{tikzpicture}[scale = 0.5,thick, baseline={(0,-1ex/2)}]
\node[vertex] (G--10) at (13.5, -1) [shape = circle, draw] {};
\node[vertex] (G--9) at (12.0, -1) [shape = circle, draw] {};
\node[vertex] (G-5) at (6.0, 1) [shape = circle, draw] {};
\node[vertex] (G-10) at (13.5, 1) [shape = circle, draw] {};
\node[vertex] (G--8) at (10.5, -1) [shape = circle, draw] {};
\node[vertex] (G-2) at (1.5, 1) [shape = circle, draw] {};
\node[vertex] (G--7) at (9.0, -1) [shape = circle, draw] {};
\node[vertex] (G-1) at (0.0, 1) [shape = circle, draw] {};
\node[vertex] (G--6) at (7.5, -1) [shape = circle, draw] {};
\node[vertex] (G-6) at (7.5, 1) [shape = circle, draw] {};
\node[vertex] (G--5) at (6.0, -1) [shape = circle, draw] {};
\node[vertex] (G--4) at (4.5, -1) [shape = circle, draw] {};
\node[vertex] (G-3) at (3.0, 1) [shape = circle, draw] {};
\node[vertex] (G-4) at (4.5, 1) [shape = circle, draw] {};
\node[vertex] (G--3) at (3.0, -1) [shape = circle, draw] {};
\node[vertex] (G-9) at (12.0, 1) [shape = circle, draw] {};
\node[vertex] (G--2) at (1.5, -1) [shape = circle, draw] {};
\node[vertex] (G--1) at (0.0, -1) [shape = circle, draw] {};
\node[vertex] (G-7) at (9.0, 1) [shape = circle, draw] {};
\node[vertex] (G-8) at (10.5, 1) [shape = circle, draw] {};
\draw[] (G-5) .. controls +(1, -1) and +(-1, -1) .. (G-10);
\draw[] (G-10) .. controls +(0, -1) and +(0, 1) .. (G--10);
\draw[] (G--10) .. controls +(-0.5, 0.5) and +(0.5, 0.5) .. (G--9);
\draw[] (G-2) .. controls +(1, -1) and +(-1, 1) .. (G--8);
\draw[] (G-1) .. controls +(1, -1) and +(-1, 1) .. (G--7);
\draw[] (G-6) .. controls +(0, -1) and +(0, 1) .. (G--6);
\draw[] (G-3) .. controls +(0.5, -0.5) and +(-0.5, -0.5) .. (G-4);
\draw[] (G-4) .. controls +(0.75, -1) and +(-0.75, 1) .. (G--5);
\draw[] (G--5) .. controls +(-0.5, 0.5) and +(0.5, 0.5) .. (G--4);
\draw[] (G-9) .. controls +(-1, -1) and +(1, 1) .. (G--3);
\draw[] (G-7) .. controls +(0.5, -0.5) and +(-0.5, -0.5) .. (G-8);
\draw[] (G-7) .. controls +(-1, -1) and +(1, 1) .. (G--2);
\draw[] (G--2) .. controls +(-0.5, 0.5) and +(0.5, 0.5) .. (G--1);
\end{tikzpicture}.
\end{center}
We can then check that $4$ is the smallest integer such that
$x^4$ is idempotent and $x^4 = e_{\pi}$,
where $\pi = \{\{6\},\{1,2\},\{7,8\},\{3,4,5,9,10\}\}$.
The element $x^4$
has the following diagram:
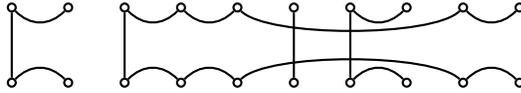
\begin{center}
\begin{tikzpicture}[scale = 0.5,thick, baseline={(0,-1ex/2)}]
\node[vertex] (G--10) at (13.5, -1) [shape = circle, draw] {};
\node[vertex] (G--9) at (12.0, -1) [shape = circle, draw] {};
\node[vertex] (G--5) at (6.0, -1) [shape = circle, draw] {};
\node[vertex] (G--4) at (4.5, -1) [shape = circle, draw] {};
\node[vertex] (G--3) at (3.0, -1) [shape = circle, draw] {};
\node[vertex] (G-3) at (3.0, 1) [shape = circle, draw] {};
\node[vertex] (G-4) at (4.5, 1) [shape = circle, draw] {};
\node[vertex] (G-5) at (6.0, 1) [shape = circle, draw] {};
\node[vertex] (G-9) at (12.0, 1) [shape = circle, draw] {};
\node[vertex] (G-10) at (13.5, 1) [shape = circle, draw] {};
\node[vertex] (G--8) at (10.5, -1) [shape = circle, draw] {};
\node[vertex] (G--7) at (9.0, -1) [shape = circle, draw] {};
\node[vertex] (G-7) at (9.0, 1) [shape = circle, draw] {};
\node[vertex] (G-8) at (10.5, 1) [shape = circle, draw] {};
\node[vertex] (G--6) at (7.5, -1) [shape = circle, draw] {};
\node[vertex] (G-6) at (7.5, 1) [shape = circle, draw] {};
\node[vertex] (G--2) at (1.5, -1) [shape = circle, draw] {};
\node[vertex] (G--1) at (0.0, -1) [shape = circle, draw] {};
\node[vertex] (G-1) at (0.0, 1) [shape = circle, draw] {};
\node[vertex] (G-2) at (1.5, 1) [shape = circle, draw] {};
\draw[] (G-3) .. controls +(0.5, -0.5) and +(-0.5, -0.5) .. (G-4);
\draw[] (G-4) .. controls +(0.5, -0.5) and +(-0.5, -0.5) .. (G-5);
\draw[] (G-5) .. controls +(0.8, -0.8) and +(-0.8, -0.8) .. (G-9);
\draw[] (G-9) .. controls +(0.5, -0.5) and +(-0.5, -0.5) .. (G-10);
\draw[] (G--10) .. controls +(-0.5, 0.5) and +(0.5, 0.5) .. (G--9);
\draw[] (G--9) .. controls +(-0.8, 0.8) and +(0.8, 0.8) .. (G--5);
\draw[] (G--5) .. controls +(-0.5, 0.5) and +(0.5, 0.5) .. (G--4);
\draw[] (G--4) .. controls +(-0.5, 0.5) and +(0.5, 0.5) .. (G--3);
\draw[] (G--3) .. controls +(0, 1) and +(0, -1) .. (G-3);
\draw[] (G-7) .. controls +(0.5, -0.5) and +(-0.5, -0.5) .. (G-8);
\draw[] (G--8) .. controls +(-0.5, 0.5) and +(0.5, 0.5) .. (G--7);
\draw[] (G--7) .. controls +(0, 1) and +(0, -1) .. (G-7);
\draw[] (G-6) .. controls +(0, -1) and +(0, 1) .. (G--6);
\draw[] (G-1) .. controls +(0.5, -0.5) and +(-0.5, -0.5) .. (G-2);
\draw[] (G--2) .. controls +(-0.5, 0.5) and +(0.5, 0.5) .. (G--1);
\draw[] (G--1) .. controls +(0, 1) and +(0, -1) .. (G-1);
\end{tikzpicture}.
\end{center}
\vskip .1in
We deduce that $\vtype(\mathsf{cycletype}(x)) = \mathsf{type}(\pi) = (5,2,2,1)$.
Since $x^5 = x e_\pi = e_\pi x$ is
\[
    \big\{
        \{6,\o{6}\},
        \{1,2,\o{7},\o{8}\},
        \{7,8,\o{1},\o{2}\},
        \{2,3,4,9,10,\o{2},\o{3},\o{4},\o{9},\o{10}\}
    \big\}~,
\]
we conclude that $\mathsf{cycletype}(x) = \mathsf{cycletype}(x^5) = ((1), (2), \epart, \epart, (1))$
because $x$ acts on $e_\pi$ by permuting the two sets of size $2$ (in a cycle)
and fixing the sets of size $1$ and $5$.
\end{example}

\subsection{Representative conjugacy class element}
\label{sec:conj-class-reps}
In this section, we describe for each $\vec{\mu}\in I_k$ a representative conjugacy class element, denoted $d_{\vec{\mu}}$, contained in the
generalized conjugacy class $C_{\vec{\mu}}$ consisting of all elements of cycle type $\vec{\mu}$.
As shown in the previous section, we can choose $d_{\vec{\mu}}$ to be an
element of a representative maximal subgroup of $\mathcal{U}_k$.

For any set $A = \{x_1, \ldots, x_{|A|}\}$ and $\mu=(\mu_1,\ldots,\mu_\ell)
\vdash |A|$, we define the \defn{representative element of cycle type $\mu$}
in $\mathfrak{S}_A$ to be
\[
    d_\mu^A = (x_1 x_2 \cdots x_{\mu_1})(x_{\mu_1 +1} x_{\mu_1 +2} \cdots x_{\mu_1+\mu_2}) \cdots (x_{|A|-\mu_\ell +1} x_{|A|-\mu_\ell +2}\cdots x_{|A|}),
\]
where $d_\mu^A$ is expressed in cycle notation.
If $A$ is clear from the context, we write $d_\mu = d_\mu^A$.

Let $\vec{\mu} = (\mu^{(1)}, \ldots, \mu^{(k)})  \in I_k$ and $\mu = \vtype(\vec{\mu})$.
Recall that $G_\mu = \mathfrak{S}_{\pi^{(1)}} \times \cdots \times \mathfrak{S}_{\pi^{(k)}}$,
where $\pi^{(i)}$ is the set of blocks of size $i$ in $\pi_\mu$.
Since $d_{\mu^{(i)}}^{\pi^{(i)}} \in \mathfrak{S}_{\pi^{(i)}}$ for all $1 \leqslant i \leqslant k$, we define
\[
    {\defncolor d_{\vec{\mu}}} = d_{\mu^{(1)}}^{\pi^{(1)}} \, d_{\mu^{(2)}}^{\pi^{(2)}} \, \cdots \, d_{\mu^{(k)}}^{\pi^{(k)}}
    \in G_\mu.
\] 

\begin{prop}
 Let $\vec{\mu}\in I_k$. Then $\mathsf{cycletype}(d_{\vec{\mu}}) = \vec{\mu}$.
\end{prop}

\begin{example}\label{ex:canonicalconjugacyrep}
Let $k=34$ and $\vec{\mu} = ( (2,1), \epart, (3,1,1), (2,2))$.
Then $d_{\vec{\mu}}$ is represented by the following diagram:
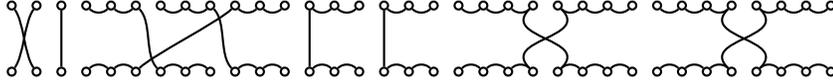
\begin{center}
\begin{tikzpicture}[scale = 0.22,thick, baseline={(0,-1ex/2)}]
\node[vertex] (G--34) at (49.5, -2) [shape = circle, draw] {};
\node[vertex] (G--33) at (48.0, -2) [shape = circle, draw] {};
\node[vertex] (G--32) at (46.5, -2) [shape = circle, draw] {};
\node[vertex] (G--31) at (45.0, -2) [shape = circle, draw] {};
\node[vertex] (G-27) at (39.0, 2) [shape = circle, draw] {};
\node[vertex] (G-28) at (40.5, 2) [shape = circle, draw] {};
\node[vertex] (G-29) at (42.0, 2) [shape = circle, draw] {};
\node[vertex] (G-30) at (43.5, 2) [shape = circle, draw] {};
\node[vertex] (G--30) at (43.5, -2) [shape = circle, draw] {};
\node[vertex] (G--29) at (42.0, -2) [shape = circle, draw] {};
\node[vertex] (G--28) at (40.5, -2) [shape = circle, draw] {};
\node[vertex] (G--27) at (39.0, -2) [shape = circle, draw] {};
\node[vertex] (G-31) at (45.0, 2) [shape = circle, draw] {};
\node[vertex] (G-32) at (46.5, 2) [shape = circle, draw] {};
\node[vertex] (G-33) at (48.0, 2) [shape = circle, draw] {};
\node[vertex] (G-34) at (49.5, 2) [shape = circle, draw] {};
\node[vertex] (G--26) at (37.5, -2) [shape = circle, draw] {};
\node[vertex] (G--25) at (36.0, -2) [shape = circle, draw] {};
\node[vertex] (G--24) at (34.5, -2) [shape = circle, draw] {};
\node[vertex] (G--23) at (33.0, -2) [shape = circle, draw] {};
\node[vertex] (G-19) at (27.0, 2) [shape = circle, draw] {};
\node[vertex] (G-20) at (28.5, 2) [shape = circle, draw] {};
\node[vertex] (G-21) at (30.0, 2) [shape = circle, draw] {};
\node[vertex] (G-22) at (31.5, 2) [shape = circle, draw] {};
\node[vertex] (G--22) at (31.5, -2) [shape = circle, draw] {};
\node[vertex] (G--21) at (30.0, -2) [shape = circle, draw] {};
\node[vertex] (G--20) at (28.5, -2) [shape = circle, draw] {};
\node[vertex] (G--19) at (27.0, -2) [shape = circle, draw] {};
\node[vertex] (G-23) at (33.0, 2) [shape = circle, draw] {};
\node[vertex] (G-24) at (34.5, 2) [shape = circle, draw] {};
\node[vertex] (G-25) at (36.0, 2) [shape = circle, draw] {};
\node[vertex] (G-26) at (37.5, 2) [shape = circle, draw] {};
\node[vertex] (G--18) at (25.5, -2) [shape = circle, draw] {};
\node[vertex] (G--17) at (24.0, -2) [shape = circle, draw] {};
\node[vertex] (G--16) at (22.5, -2) [shape = circle, draw] {};
\node[vertex] (G-16) at (22.5, 2) [shape = circle, draw] {};
\node[vertex] (G-17) at (24.0, 2) [shape = circle, draw] {};
\node[vertex] (G-18) at (25.5, 2) [shape = circle, draw] {};
\node[vertex] (G--15) at (21.0, -2) [shape = circle, draw] {};
\node[vertex] (G--14) at (19.5, -2) [shape = circle, draw] {};
\node[vertex] (G--13) at (18.0, -2) [shape = circle, draw] {};
\node[vertex] (G-13) at (18.0, 2) [shape = circle, draw] {};
\node[vertex] (G-14) at (19.5, 2) [shape = circle, draw] {};
\node[vertex] (G-15) at (21.0, 2) [shape = circle, draw] {};
\node[vertex] (G--12) at (16.5, -2) [shape = circle, draw] {};
\node[vertex] (G--11) at (15.0, -2) [shape = circle, draw] {};
\node[vertex] (G--10) at (13.5, -2) [shape = circle, draw] {};
\node[vertex] (G-4) at (4.5, 2) [shape = circle, draw] {};
\node[vertex] (G-5) at (6.0, 2) [shape = circle, draw] {};
\node[vertex] (G-6) at (7.5, 2) [shape = circle, draw] {};
\node[vertex] (G--9) at (12.0, -2) [shape = circle, draw] {};
\node[vertex] (G--8) at (10.5, -2) [shape = circle, draw] {};
\node[vertex] (G--7) at (9.0, -2) [shape = circle, draw] {};
\node[vertex] (G-10) at (13.5, 2) [shape = circle, draw] {};
\node[vertex] (G-11) at (15.0, 2) [shape = circle, draw] {};
\node[vertex] (G-12) at (16.5, 2) [shape = circle, draw] {};
\node[vertex] (G--6) at (7.5, -2) [shape = circle, draw] {};
\node[vertex] (G--5) at (6.0, -2) [shape = circle, draw] {};
\node[vertex] (G--4) at (4.5, -2) [shape = circle, draw] {};
\node[vertex] (G-7) at (9.0, 2) [shape = circle, draw] {};
\node[vertex] (G-8) at (10.5, 2) [shape = circle, draw] {};
\node[vertex] (G-9) at (12.0, 2) [shape = circle, draw] {};
\node[vertex] (G--3) at (3.0, -2) [shape = circle, draw] {};
\node[vertex] (G-3) at (3.0, 2) [shape = circle, draw] {};
\node[vertex] (G--2) at (1.5, -2) [shape = circle, draw] {};
\node[vertex] (G-1) at (0.0, 2) [shape = circle, draw] {};
\node[vertex] (G--1) at (0.0, -2) [shape = circle, draw] {};
\node[vertex] (G-2) at (1.5, 2) [shape = circle, draw] {};
\draw[] (G-27) .. controls +(0.5, -0.5) and +(-0.5, -0.5) .. (G-28);
\draw[] (G-28) .. controls +(0.5, -0.5) and +(-0.5, -0.5) .. (G-29);
\draw[] (G-29) .. controls +(0.5, -0.5) and +(-0.5, -0.5) .. (G-30);
\draw[] (G-30) .. controls +(-2, -2) and +(2, 2) .. (G--31);
\draw[] (G--34) .. controls +(-0.5, 0.5) and +(0.5, 0.5) .. (G--33);
\draw[] (G--33) .. controls +(-0.5, 0.5) and +(0.5, 0.5) .. (G--32);
\draw[] (G--32) .. controls +(-0.5, 0.5) and +(0.5, 0.5) .. (G--31);
\draw[] (G--30) .. controls +(-2, 2) and +(2, -2) .. (G-31);
\draw[] (G-31) .. controls +(0.5, -0.5) and +(-0.5, -0.5) .. (G-32);
\draw[] (G-32) .. controls +(0.5, -0.5) and +(-0.5, -0.5) .. (G-33);
\draw[] (G-33) .. controls +(0.5, -0.5) and +(-0.5, -0.5) .. (G-34);
\draw[] (G--30) .. controls +(-0.5, 0.5) and +(0.5, 0.5) .. (G--29);
\draw[] (G--29) .. controls +(-0.5, 0.5) and +(0.5, 0.5) .. (G--28);
\draw[] (G--28) .. controls +(-0.5, 0.5) and +(0.5, 0.5) .. (G--27);
\draw[] (G-19) .. controls +(0.5, -0.5) and +(-0.5, -0.5) .. (G-20);
\draw[] (G-20) .. controls +(0.5, -0.5) and +(-0.5, -0.5) .. (G-21);
\draw[] (G-21) .. controls +(0.5, -0.5) and +(-0.5, -0.5) .. (G-22);
\draw[] (G-22) .. controls +(-2, -2) and +(2, 2) .. (G--23);
\draw[] (G--26) .. controls +(-0.5, 0.5) and +(0.5, 0.5) .. (G--25);
\draw[] (G--25) .. controls +(-0.5, 0.5) and +(0.5, 0.5) .. (G--24);
\draw[] (G--24) .. controls +(-0.5, 0.5) and +(0.5, 0.5) .. (G--23);
\draw[] (G-23) .. controls +(0.5, -0.5) and +(-0.5, -0.5) .. (G-24);
\draw[] (G-24) .. controls +(0.5, -0.5) and +(-0.5, -0.5) .. (G-25);
\draw[] (G-25) .. controls +(0.5, -0.5) and +(-0.5, -0.5) .. (G-26);
\draw[] (G--22) .. controls +(-0.5, 0.5) and +(0.5, 0.5) .. (G--21);
\draw[] (G--21) .. controls +(-0.5, 0.5) and +(0.5, 0.5) .. (G--20);
\draw[] (G--20) .. controls +(-0.5, 0.5) and +(0.5, 0.5) .. (G--19);
\draw[] (G--22) .. controls +(-2, 2) and +(2, -2) .. (G-23);
\draw[] (G-16) .. controls +(0.5, -0.5) and +(-0.5, -0.5) .. (G-17);
\draw[] (G-17) .. controls +(0.5, -0.5) and +(-0.5, -0.5) .. (G-18);
\draw[] (G--18) .. controls +(-0.5, 0.5) and +(0.5, 0.5) .. (G--17);
\draw[] (G--17) .. controls +(-0.5, 0.5) and +(0.5, 0.5) .. (G--16);
\draw[] (G--16) .. controls +(0, 1) and +(0, -1) .. (G-16);
\draw[] (G-13) .. controls +(0.5, -0.5) and +(-0.5, -0.5) .. (G-14);
\draw[] (G-14) .. controls +(0.5, -0.5) and +(-0.5, -0.5) .. (G-15);
\draw[] (G--15) .. controls +(-0.5, 0.5) and +(0.5, 0.5) .. (G--14);
\draw[] (G--14) .. controls +(-0.5, 0.5) and +(0.5, 0.5) .. (G--13);
\draw[] (G--13) .. controls +(0, 1) and +(0, -1) .. (G-13);
\draw[] (G-4) .. controls +(0.5, -0.5) and +(-0.5, -0.5) .. (G-5);
\draw[] (G-5) .. controls +(0.5, -0.5) and +(-0.5, -0.5) .. (G-6);
\draw[] (G--12) .. controls +(-0.5, 0.5) and +(0.5, 0.5) .. (G--11);
\draw[] (G--11) .. controls +(-0.5, 0.5) and +(0.5, 0.5) .. (G--10);
\draw[] (G--7) .. controls +(-1, 1) and +(1, -1) .. (G-6);
\draw[] (G-10) .. controls +(0.5, -0.5) and +(-0.5, -0.5) .. (G-11);
\draw[] (G-11) .. controls +(0.5, -0.5) and +(-0.5, -0.5) .. (G-12);
\draw[] (G-10) .. controls +(-1, -1) and +(1, 1) .. (G--6);
\draw[] (G--9) .. controls +(-0.5, 0.5) and +(0.5, 0.5) .. (G--8);
\draw[] (G--8) .. controls +(-0.5, 0.5) and +(0.5, 0.5) .. (G--7);
\draw[] (G--10) .. controls +(-1, 1) and +(1, -1) .. (G-9);
\draw[] (G-7) .. controls +(0.5, -0.5) and +(-0.5, -0.5) .. (G-8);
\draw[] (G-8) .. controls +(0.5, -0.5) and +(-0.5, -0.5) .. (G-9);
\draw[] (G--6) .. controls +(-0.5, 0.5) and +(0.5, 0.5) .. (G--5);
\draw[] (G--5) .. controls +(-0.5, 0.5) and +(0.5, 0.5) .. (G--4);
\draw[] (G-3) .. controls +(0, -1) and +(0, 1) .. (G--3);
\draw[] (G-1) .. controls +(0.75, -1) and +(-0.75, 1) .. (G--2);
\draw[] (G-2) .. controls +(-0.75, -1) and +(0.75, 1) .. (G--1);
\end{tikzpicture}\; .
\end{center}
\end{example}

\subsection{Characters of $\mathcal{U}_k$}
\label{sec:UBP-characters}

We give a formula for the characters of the irreducible representations of
$\mathcal{U}_k$ in terms of the characters of the irreducible representations
of the representative maximal subgroups $G_\lambda$.

For $\vec{\lambda} \in I_k$, let
$W^{\vec{\lambda}}_{\mathcal{U}_k}$ denote the irreducible representation
of $\mathcal{U}_k$ defined in Section~\ref{sec:irreps-of-UBP-algebra}.
Its \defn{character} is the function
$\chi^{\vec{\lambda}}_{\mathcal{U}_k} \colon \mathcal{U}_k \rightarrow {\mathbb C}$
defined by
\begin{equation*}
    {\defncolor \chi^{\vec{\lambda}}_{\mathcal{U}_k}(d)} = \operatorname{trace}\left(\rho_{\mathcal{B}}(d)\right),
\end{equation*}
where $\mathcal{B}$ is any basis of $W^{\vec{\lambda}}_{\mathcal{U}_k}$
and $\rho_{\mathcal{B}}(d)$ denotes the matrix representing the action of
$d$ with respect to $\mathcal{B}$.
Since trace is unchanged under change of basis,
this definition does not depend on the chosen basis.

As in the case of group characters, monoid characters are constant on
generalized conjugacy classes~\cite[Proposition~7.9]{Steinberg.2016}.
Therefore, we need only determine the value of
$\chi^{\vec{\lambda}}_{\mathcal{U}_k}$
on the representative conjugacy class elements
$d_{\vec{\mu}}$ defined in Section~\ref{sec:conj-class-reps}.
We begin by expressing this in terms of the characters
$\chi^{\vec{\lambda}}_{G_\lambda}$
of the irreducible representations $V^{\vec{\lambda}}_{G_\lambda}$.
Recall the refinement order on set partitions defined in Section~\ref{section.set partitions}.

\begin{prop}
\label{prop:char}
Let $\vec{\lambda} \in I_k$ and $\lambda = \vtype(\vec{\lambda})$.
If $d_{\vec{\mu}} \in \mathcal{U}_k$ is an element of the generalized conjugacy
class $C_{\vec{\mu}} \subseteq \mathcal{U}_k$ indexed by $\vec{\mu} \in I_k$, then
\[
    \chi^{\vec{\lambda}}_{\mathcal{U}_k}(d_{\vec{\mu}})
    = \sum_{d \in C(d_{\vec{\mu}}; \lambda)} \chi^{\vec{\lambda}}_{G_\lambda}(\sigma_d),
\]
where
\begin{itemize}
    \item
        $C(d_{\vec{\mu}}; \lambda)
        = \left\{
        d : d \geqslant d_{\vec{\mu}}, \,
        \mathsf{top}(d) = \mathsf{bot}(d), \, \mathsf{type}(\mathsf{top}(d))=\lambda
        \right\}$, and
    \item
        $\sigma_d$ is the unique element in $G_\lambda$ satisfying
        $d \ell_{\pi_\lambda}^{\mathsf{top}(d)} = \ell_{\pi_\lambda}^{\mathsf{top}(d)} \sigma_d$.
\end{itemize}
\end{prop}

\begin{proof}
We compute the trace of $d_{\vec{\mu}}$ acting on the basis
$\{ \ell_{\pi_\lambda}^\gamma \otimes \mathbf{T} :
    \mathsf{type}(\gamma) = \lambda, \,
    \mathbf{T} \in \mathcal{B}_{\vec{\lambda}}(G_\lambda) \}$
of Proposition~\ref{prop:irrepbasis}.
Recall that the action of $d_{\vec{\mu}}$ on $\ell_{\pi_\lambda}^\gamma \otimes \mathbf{T}$
is given by
\begin{equation*}
    d_{\vec{\mu}} \cdot (\ell_{\pi_\lambda}^\gamma \otimes \mathbf{T})
    = (d_{\vec{\mu}} \odot \ell_{\pi_\lambda}^\gamma) \otimes \mathbf{T}
    =
    \begin{cases}
        d_{\vec{\mu}} \ell_{\pi_\lambda}^\gamma \otimes \mathbf{T}, & \text{if~} \mathsf{bot}(d_{\vec{\mu}} \ell_{\pi_\lambda}^\gamma) = \pi_\lambda, \\
        0, & \text{otherwise.}
    \end{cases}
\end{equation*}
Writing $\mu = \vtype(\vec{\mu})$ and noting that $d_{\vec{\mu}}$ is
a permutation of the blocks of $\pi_\mu$ (since $d_{\vec{\mu}} \in G_\mu$),
it follows that $\mathsf{bot}(d_{\vec{\mu}} \ell_{\pi_\lambda}^\gamma) = \pi_\lambda$
if and only if $\pi_\mu = \mathsf{bot}(d_{\vec{\mu}})$ is finer than $\gamma$.
Thus, if $\pi_\mu$ is not finer than $\gamma$, then there is no contribution to
the trace.

Suppose that $\pi_\mu$ is finer than $\gamma$. Then we can merge blocks of
$d_{\vec{\mu}}$ to obtain a diagram $d$ such that $\mathsf{bot}(d) = \gamma$.
Notice that the diagram $d$ satisfies
$d_{\vec{\mu}} \ell_{\pi_\lambda}^\gamma = d \ell_{\pi_\lambda}^\gamma$.
By Proposition~\ref{prop:GeactiononL}, there is a unique $\sigma_{d} \in G_\lambda$
satisfying $d \ell_{\pi_\lambda}^\gamma = \ell_{\pi_\lambda}^{\mathsf{top}(d)} \sigma_{d}$.
Therefore, $d_{\vec{\mu}} \cdot (\ell_{\pi_\lambda}^\gamma \otimes \mathbf{T})
= \ell_{\pi_\lambda}^{\mathsf{top}(d)} \otimes \sigma_{d} \mathbf{T}$,
which contributes to the computation of the trace only when $\mathsf{top}(d) = \gamma$.
In this case,
$d_{\vec{\mu}}$ maps $\ell_{\pi_\lambda}^\gamma \otimes \mathbf{T}$
to $\ell_{\pi_\lambda}^\gamma \otimes \sigma_d \mathbf{T}$ and so
the trace is equal to the trace of the mapping
$\mathbf{T} \mapsto \sigma_d \mathbf{T}$,
which is precisely $\chi^{\vec{\lambda}}_{G_\lambda}(\sigma_{d})$
because $\mathcal{B}_{\vec{\lambda}}(G_\lambda)$
is a basis of $V^{\vec{\lambda}}_{G_\lambda}$.
\end{proof}

\begin{example} \label{ex:character}
Let $\vec{\lambda} = (\epart, (1,1))$ and $\vec{\mu} = ((2,2))$.
Then $\lambda = (2, 2)$,
$d_{\vec{\mu}} =$
\begin{tikzpicture}[partition-diagram]
\node[vertex] (G--3) at (3.0, -1) [shape = circle, draw] {};
\node[vertex] (G-3) at (3.0, 1) [shape = circle, draw] {};
\node[vertex] (G--4) at (4.5, -1) [shape = circle, draw] {};
\node[vertex] (G--1) at (0.0, -1) [shape = circle, draw] {};
\node[vertex] (G-1) at (0.0, 1) [shape = circle, draw] {};
\node[vertex] (G-4) at (4.5, 1) [shape = circle, draw] {};
\node[vertex] (G--2) at (1.5, -1) [shape = circle, draw] {};
\node[vertex] (G-2) at (1.5, 1) [shape = circle, draw] {};
\draw[] (G-2) .. controls +(-1, -1) and +(1,1) .. (G--1);
\draw[] (G-1) .. controls +(1, -1) and +(-1, 1) .. (G--2);
\draw[] (G-4) .. controls +(-1, -1) and +(1, 1) .. (G--3);
\draw[] (G-3) .. controls +(1, -1) and +(-1,1) .. (G--4);
\end{tikzpicture},
and
\begin{equation*}
    C(d_{\vec{\mu}}; \lambda) =
    \left\{
    \begin{tikzpicture}[partition-diagram]
    \node[vertex] (G--3) at (3.0, -1) [shape = circle, draw] {};
    \node[vertex] (G-3) at (3.0, 1) [shape = circle, draw] {};
    \node[vertex] (G--4) at (4.5, -1) [shape = circle, draw] {};
    \node[vertex] (G--1) at (0.0, -1) [shape = circle, draw] {};
    \node[vertex] (G-1) at (0.0, 1) [shape = circle, draw] {};
    \node[vertex] (G-4) at (4.5, 1) [shape = circle, draw] {};
    \node[vertex] (G--2) at (1.5, -1) [shape = circle, draw] {};
    \node[vertex] (G-2) at (1.5, 1) [shape = circle, draw] {};
    \draw[] (G-1) -- (G--1);
    \draw[] (G-3) -- (G--3);
    \draw[] (G-1) .. controls +(.5,-.5) and +(-.5,-.5) .. (G-2);
    \draw[] (G-3) .. controls +(.5,-.5) and +(-.5,-.5) .. (G-4);
    \draw[] (G--1) .. controls +(.5, .5) and +(-.5, .5) .. (G--2);
    \draw[] (G--3) .. controls +(.5, .5) and +(-.5, .5) .. (G--4);
    \end{tikzpicture}, \quad
    \begin{tikzpicture}[partition-diagram]
    \node[vertex] (G--3) at (3.0, -1) [shape = circle, draw] {};
    \node[vertex] (G-3) at (3.0, 1) [shape = circle, draw] {};
    \node[vertex] (G--4) at (4.5, -1) [shape = circle, draw] {};
    \node[vertex] (G--1) at (0.0, -1) [shape = circle, draw] {};
    \node[vertex] (G-1) at (0.0, 1) [shape = circle, draw] {};
    \node[vertex] (G-4) at (4.5, 1) [shape = circle, draw] {};
    \node[vertex] (G--2) at (1.5, -1) [shape = circle, draw] {};
    \node[vertex] (G-2) at (1.5, 1) [shape = circle, draw] {};
    \draw[] (G-1) .. controls +(1, -1) and +(-1, 1) .. (G--2);
    \draw[] (G-4) .. controls +(-1, -1) and +(1, 1) .. (G--3);
    \draw[] (G-1) .. controls +(.6,-.6) and +(-.6,-.6) .. (G-3);
    \draw[] (G--1) .. controls +(.6, .6) and +(-.6, .6) .. (G--3);
    \draw[] (G-2) .. controls +(.6,-.6) and +(-.6,-.6) .. (G-4);
    \draw[] (G--2) .. controls +(.6, .6) and +(-.6, .6) .. (G--4);
    \end{tikzpicture},\quad
    \begin{tikzpicture}[partition-diagram]
    \node[vertex] (G--3) at (3.0, -1) [shape = circle, draw] {};
    \node[vertex] (G-3) at (3.0, 1) [shape = circle, draw] {};
    \node[vertex] (G--4) at (4.5, -1) [shape = circle, draw] {};
    \node[vertex] (G--1) at (0.0, -1) [shape = circle, draw] {};
    \node[vertex] (G-1) at (0.0, 1) [shape = circle, draw] {};
    \node[vertex] (G-4) at (4.5, 1) [shape = circle, draw] {};
    \node[vertex] (G--2) at (1.5, -1) [shape = circle, draw] {};
    \node[vertex] (G-2) at (1.5, 1) [shape = circle, draw] {};
    \draw[] (G-2) .. controls +(-1, -1) and +(1,1) .. (G--1);
    \draw[] (G-4) .. controls +(-1, -1) and +(1, 1) .. (G--3);
    \draw[] (G-1) .. controls +(.7,-.7) and +(-.7,-.7) .. (G-4);
    \draw[] (G--1) .. controls +(.7, .7) and +(-.7, .7) .. (G--4);
    \draw[] (G-2) .. controls +(.5,-.5) and +(-.5,-.5) .. (G-3);
    \draw[] (G--2) .. controls +(.5, .5) and +(-.5, .5) .. (G--3);
    \end{tikzpicture}
    \right\}.
\end{equation*}
For each $d \in C(d_{\vec{\mu}}; \lambda)$, there is a unique $\sigma_d \in
G_\lambda$ satisfying
$d \ell_{\pi_\lambda}^{\mathsf{top}(d)} = \ell_{\pi_\lambda}^{\mathsf{top}(d)} \sigma_d$,
which we illustrate below:
\begin{align*}
\begin{tikzpicture}[partition-diagram]
\node[vertex] (G--3) at (3.0, -1) [shape = circle, draw] {};
\node[vertex] (G-3) at (3.0, 1) [shape = circle, draw] {};
\node[vertex] (G--4) at (4.5, -1) [shape = circle, draw] {};
\node[vertex] (G--1) at (0.0, -1) [shape = circle, draw] {};
\node[vertex] (G-1) at (0.0, 1) [shape = circle, draw] {};
\node[vertex] (G-4) at (4.5, 1) [shape = circle, draw] {};
\node[vertex] (G--2) at (1.5, -1) [shape = circle, draw] {};
\node[vertex] (G-2) at (1.5, 1) [shape = circle, draw] {};
\draw[] (G-1) -- (G--1);
\draw[] (G-3) -- (G--3);
\draw[] (G-1) .. controls +(.5,-.5) and +(-.5,-.5) .. (G-2);
\draw[] (G-3) .. controls +(.5,-.5) and +(-.5,-.5) .. (G-4);
\draw[] (G--1) .. controls +(.5, .5) and +(-.5, .5) .. (G--2);
\draw[] (G--3) .. controls +(.5, .5) and +(-.5, .5) .. (G--4);
\end{tikzpicture} \quad \cdot \quad
\begin{tikzpicture}[partition-diagram]
\node[vertex] (G--3) at (3.0, -1) [shape = circle, draw] {};
\node[vertex] (G-3) at (3.0, 1) [shape = circle, draw] {};
\node[vertex] (G--4) at (4.5, -1) [shape = circle, draw] {};
\node[vertex] (G--1) at (0.0, -1) [shape = circle, draw] {};
\node[vertex] (G-1) at (0.0, 1) [shape = circle, draw] {};
\node[vertex] (G-4) at (4.5, 1) [shape = circle, draw] {};
\node[vertex] (G--2) at (1.5, -1) [shape = circle, draw] {};
\node[vertex] (G-2) at (1.5, 1) [shape = circle, draw] {};
\draw[] (G-1) -- (G--1);
\draw[] (G-3) -- (G--3);
\draw[] (G-1) .. controls +(.5,-.5) and +(-.5,-.5) .. (G-2);
\draw[] (G-3) .. controls +(.5,-.5) and +(-.5,-.5) .. (G-4);
\draw[] (G--1) .. controls +(.5, .5) and +(-.5, .5) .. (G--2);
\draw[] (G--3) .. controls +(.5, .5) and +(-.5, .5) .. (G--4);
\end{tikzpicture}  & \quad = \quad
\begin{tikzpicture}[partition-diagram]
\node[vertex] (G--3) at (3.0, -1) [shape = circle, draw] {};
\node[vertex] (G-3) at (3.0, 1) [shape = circle, draw] {};
\node[vertex] (G--4) at (4.5, -1) [shape = circle, draw] {};
\node[vertex] (G--1) at (0.0, -1) [shape = circle, draw] {};
\node[vertex] (G-1) at (0.0, 1) [shape = circle, draw] {};
\node[vertex] (G-4) at (4.5, 1) [shape = circle, draw] {};
\node[vertex] (G--2) at (1.5, -1) [shape = circle, draw] {};
\node[vertex] (G-2) at (1.5, 1) [shape = circle, draw] {};
\draw[] (G-1) -- (G--1);
\draw[] (G-3) -- (G--3);
\draw[] (G-1) .. controls +(.5,-.5) and +(-.5,-.5) .. (G-2);
\draw[] (G-3) .. controls +(.5,-.5) and +(-.5,-.5) .. (G-4);
\draw[] (G--1) .. controls +(.5, .5) and +(-.5, .5) .. (G--2);
\draw[] (G--3) .. controls +(.5, .5) and +(-.5, .5) .. (G--4);
\end{tikzpicture} \quad \cdot \quad
\begin{tikzpicture}[partition-diagram]
\node[vertex] (G--3) at (3.0, -1) [shape = circle, draw] {};
\node[vertex] (G-3) at (3.0, 1) [shape = circle, draw] {};
\node[vertex] (G--4) at (4.5, -1) [shape = circle, draw] {};
\node[vertex] (G--1) at (0.0, -1) [shape = circle, draw] {};
\node[vertex] (G-1) at (0.0, 1) [shape = circle, draw] {};
\node[vertex] (G-4) at (4.5, 1) [shape = circle, draw] {};
\node[vertex] (G--2) at (1.5, -1) [shape = circle, draw] {};
\node[vertex] (G-2) at (1.5, 1) [shape = circle, draw] {};
\draw[] (G-1) -- (G--1);
\draw[] (G-3) -- (G--3);
\draw[] (G-1) .. controls +(.5,-.5) and +(-.5,-.5) .. (G-2);
\draw[] (G-3) .. controls +(.5,-.5) and +(-.5,-.5) .. (G-4);
\draw[] (G--1) .. controls +(.5, .5) and +(-.5, .5) .. (G--2);
\draw[] (G--3) .. controls +(.5, .5) and +(-.5, .5) .. (G--4);
\end{tikzpicture}\\
\begin{tikzpicture}[partition-diagram]
\node[vertex] (G--3) at (3.0, -1) [shape = circle, draw] {};
\node[vertex] (G-3) at (3.0, 1) [shape = circle, draw] {};
\node[vertex] (G--4) at (4.5, -1) [shape = circle, draw] {};
\node[vertex] (G--1) at (0.0, -1) [shape = circle, draw] {};
\node[vertex] (G-1) at (0.0, 1) [shape = circle, draw] {};
\node[vertex] (G-4) at (4.5, 1) [shape = circle, draw] {};
\node[vertex] (G--2) at (1.5, -1) [shape = circle, draw] {};
\node[vertex] (G-2) at (1.5, 1) [shape = circle, draw] {};
\draw[] (G-1) .. controls +(1, -1) and +(-1, 1) .. (G--2);
\draw[] (G-4) .. controls +(-1, -1) and +(1, 1) .. (G--3);
\draw[] (G-1) .. controls +(.6,-.6) and +(-.6,-.6) .. (G-3);
\draw[] (G--1) .. controls +(.6, .6) and +(-.6, .6) .. (G--3);
\draw[] (G-2) .. controls +(.6,-.6) and +(-.6,-.6) .. (G-4);
\draw[] (G--2) .. controls +(.6, .6) and +(-.6, .6) .. (G--4);
\end{tikzpicture}  \quad \cdot \quad
\begin{tikzpicture}[partition-diagram]
\node[vertex] (G--3) at (3.0, -1) [shape = circle, draw] {};
\node[vertex] (G-3) at (3.0, 1) [shape = circle, draw] {};
\node[vertex] (G--4) at (4.5, -1) [shape = circle, draw] {};
\node[vertex] (G--1) at (0.0, -1) [shape = circle, draw] {};
\node[vertex] (G-1) at (0.0, 1) [shape = circle, draw] {};
\node[vertex] (G-4) at (4.5, 1) [shape = circle, draw] {};
\node[vertex] (G--2) at (1.5, -1) [shape = circle, draw] {};
\node[vertex] (G-2) at (1.5, 1) [shape = circle, draw] {};
\draw[] (G-1) -- (G--1);
\draw[] (G-4) -- (G--4);
\draw[] (G-1) .. controls +(.6,-.6) and +(-.6,-.6) .. (G-3);
\draw[] (G--1) .. controls +(.6, .6) and +(-.6, .6) .. (G--2);
\draw[] (G-2) .. controls +(.6,-.6) and +(-.6,-.6) .. (G-4);
\draw[] (G--3) .. controls +(.6, .6) and +(-.6, .6) .. (G--4);
\end{tikzpicture}
&\quad = \quad
\begin{tikzpicture}[partition-diagram]
\node[vertex] (G--3) at (3.0, -1) [shape = circle, draw] {};
\node[vertex] (G-3) at (3.0, 1) [shape = circle, draw] {};
\node[vertex] (G--4) at (4.5, -1) [shape = circle, draw] {};
\node[vertex] (G--1) at (0.0, -1) [shape = circle, draw] {};
\node[vertex] (G-1) at (0.0, 1) [shape = circle, draw] {};
\node[vertex] (G-4) at (4.5, 1) [shape = circle, draw] {};
\node[vertex] (G--2) at (1.5, -1) [shape = circle, draw] {};
\node[vertex] (G-2) at (1.5, 1) [shape = circle, draw] {};
\draw[] (G-1) -- (G--1);
\draw[] (G-4) -- (G--4);
\draw[] (G-1) .. controls +(.6,-.6) and +(-.6,-.6) .. (G-3);
\draw[] (G--1) .. controls +(.6, .6) and +(-.6, .6) .. (G--2);
\draw[] (G-2) .. controls +(.6,-.6) and +(-.6,-.6) .. (G-4);
\draw[] (G--3) .. controls +(.6, .6) and +(-.6, .6) .. (G--4);
\end{tikzpicture}\quad \cdot \quad
\begin{tikzpicture}[partition-diagram]
\node[vertex] (G--3) at (3.0, -1) [shape = circle, draw] {};
\node[vertex] (G-3) at (3.0, 1) [shape = circle, draw] {};
\node[vertex] (G--4) at (4.5, -1) [shape = circle, draw] {};
\node[vertex] (G--1) at (0.0, -1) [shape = circle, draw] {};
\node[vertex] (G-1) at (0.0, 1) [shape = circle, draw] {};
\node[vertex] (G-4) at (4.5, 1) [shape = circle, draw] {};
\node[vertex] (G--2) at (1.5, -1) [shape = circle, draw] {};
\node[vertex] (G-2) at (1.5, 1) [shape = circle, draw] {};
\draw[] (G-3) .. controls +(-1, -1) and +(1, 1) .. (G--2);
\draw[] (G-2) .. controls +(1, -1) and +(-1, 1) .. (G--3);
\draw[] (G-1) .. controls +(.6,-.6) and +(-.6,-.6) .. (G-2);
\draw[] (G--1) .. controls +(.6, .6) and +(-.6, .6) .. (G--2);
\draw[] (G-3) .. controls +(.6,-.6) and +(-.6,-.6) .. (G-4);
\draw[] (G--3) .. controls +(.6, .6) and +(-.6, .6) .. (G--4);
\end{tikzpicture}\\
\begin{tikzpicture}[partition-diagram]
\node[vertex] (G--3) at (3.0, -1) [shape = circle, draw] {};
\node[vertex] (G-3) at (3.0, 1) [shape = circle, draw] {};
\node[vertex] (G--4) at (4.5, -1) [shape = circle, draw] {};
\node[vertex] (G--1) at (0.0, -1) [shape = circle, draw] {};
\node[vertex] (G-1) at (0.0, 1) [shape = circle, draw] {};
\node[vertex] (G-4) at (4.5, 1) [shape = circle, draw] {};
\node[vertex] (G--2) at (1.5, -1) [shape = circle, draw] {};
\node[vertex] (G-2) at (1.5, 1) [shape = circle, draw] {};
\draw[] (G-2) .. controls +(-1, -1) and +(1,1) .. (G--1);
\draw[] (G-4) .. controls +(-1, -1) and +(1, 1) .. (G--3);
\draw[] (G-1) .. controls +(.7,-.7) and +(-.7,-.7) .. (G-4);
\draw[] (G--1) .. controls +(.7, .7) and +(-.7, .7) .. (G--4);
\draw[] (G-2) .. controls +(.5,-.5) and +(-.5,-.5) .. (G-3);
\draw[] (G--2) .. controls +(.5, .5) and +(-.5, .5) .. (G--3);
\end{tikzpicture} \quad \cdot \quad
\begin{tikzpicture}[partition-diagram]
\node[vertex] (G--3) at (3.0, -1) [shape = circle, draw] {};
\node[vertex] (G-3) at (3.0, 1) [shape = circle, draw] {};
\node[vertex] (G--4) at (4.5, -1) [shape = circle, draw] {};
\node[vertex] (G--1) at (0.0, -1) [shape = circle, draw] {};
\node[vertex] (G-1) at (0.0, 1) [shape = circle, draw] {};
\node[vertex] (G-4) at (4.5, 1) [shape = circle, draw] {};
\node[vertex] (G--2) at (1.5, -1) [shape = circle, draw] {};
\node[vertex] (G-2) at (1.5, 1) [shape = circle, draw] {};
\draw[] (G-2) -- (G--2);
\draw[] (G-4) -- (G--4);
\draw[] (G-1) .. controls +(.8,-.8) and +(-.8,-.8) .. (G-4);
\draw[] (G--1) .. controls +(.5, .5) and +(-.5, .5) .. (G--2);
\draw[] (G-2) .. controls +(.5,-.5) and +(-.5,-.5) .. (G-3);
\draw[] (G--3) .. controls +(.5, .5) and +(-.5, .5) .. (G--4);
\end{tikzpicture}
&\quad = \quad
\begin{tikzpicture}[partition-diagram]
\node[vertex] (G--3) at (3.0, -1) [shape = circle, draw] {};
\node[vertex] (G-3) at (3.0, 1) [shape = circle, draw] {};
\node[vertex] (G--4) at (4.5, -1) [shape = circle, draw] {};
\node[vertex] (G--1) at (0.0, -1) [shape = circle, draw] {};
\node[vertex] (G-1) at (0.0, 1) [shape = circle, draw] {};
\node[vertex] (G-4) at (4.5, 1) [shape = circle, draw] {};
\node[vertex] (G--2) at (1.5, -1) [shape = circle, draw] {};
\node[vertex] (G-2) at (1.5, 1) [shape = circle, draw] {};
\draw[] (G-2) -- (G--2);
\draw[] (G-4) -- (G--4);
\draw[] (G-1) .. controls +(.8,-.8) and +(-.8,-.8) .. (G-4);
\draw[] (G--1) .. controls +(.5, .5) and +(-.5, .5) .. (G--2);
\draw[] (G-2) .. controls +(.5,-.5) and +(-.5,-.5) .. (G-3);
\draw[] (G--3) .. controls +(.5, .5) and +(-.5, .5) .. (G--4);
\end{tikzpicture} \quad \cdot \quad
\begin{tikzpicture}[partition-diagram]
\node[vertex] (G--3) at (3.0, -1) [shape = circle, draw] {};
\node[vertex] (G-3) at (3.0, 1) [shape = circle, draw] {};
\node[vertex] (G--4) at (4.5, -1) [shape = circle, draw] {};
\node[vertex] (G--1) at (0.0, -1) [shape = circle, draw] {};
\node[vertex] (G-1) at (0.0, 1) [shape = circle, draw] {};
\node[vertex] (G-4) at (4.5, 1) [shape = circle, draw] {};
\node[vertex] (G--2) at (1.5, -1) [shape = circle, draw] {};
\node[vertex] (G-2) at (1.5, 1) [shape = circle, draw] {};
\draw[] (G-3) .. controls +(-1, -1) and +(1, 1) .. (G--2);
\draw[] (G-2) .. controls +(1, -1) and +(-1, 1) .. (G--3);
\draw[] (G-1) .. controls +(.6,-.6) and +(-.6,-.6) .. (G-2);
\draw[] (G--1) .. controls +(.6, .6) and +(-.6, .6) .. (G--2);
\draw[] (G-3) .. controls +(.6,-.6) and +(-.6,-.6) .. (G-4);
\draw[] (G--3) .. controls +(.6, .6) and +(-.6, .6) .. (G--4);
\end{tikzpicture}
\end{align*}
Therefore,
\begin{equation*}
\chi^{\vec{\lambda}}_{\mathcal{U}_4}\left(\  \begin{tikzpicture}[partition-diagram]
\node[vertex] (G--3) at (3.0, -1) [shape = circle, draw] {};
\node[vertex] (G-3) at (3.0, 1) [shape = circle, draw] {};
\node[vertex] (G--4) at (4.5, -1) [shape = circle, draw] {};
\node[vertex] (G--1) at (0.0, -1) [shape = circle, draw] {};
\node[vertex] (G-1) at (0.0, 1) [shape = circle, draw] {};
\node[vertex] (G-4) at (4.5, 1) [shape = circle, draw] {};
\node[vertex] (G--2) at (1.5, -1) [shape = circle, draw] {};
\node[vertex] (G-2) at (1.5, 1) [shape = circle, draw] {};
\draw[] (G-2) .. controls +(-1, -1) and +(1,1) .. (G--1);
\draw[] (G-1) .. controls +(1, -1) and +(-1, 1) .. (G--2);
\draw[] (G-4) .. controls +(-1, -1) and +(1, 1) .. (G--3);
\draw[] (G-3) .. controls +(1, -1) and +(-1,1) .. (G--4);
\end{tikzpicture}\ \right)  = \chi^{\vec{\lambda}}_{G_\lambda}\left(\
\begin{tikzpicture}[partition-diagram]
\node[vertex] (G--3) at (3.0, -1) [shape = circle, draw] {};
\node[vertex] (G-3) at (3.0, 1) [shape = circle, draw] {};
\node[vertex] (G--4) at (4.5, -1) [shape = circle, draw] {};
\node[vertex] (G--1) at (0.0, -1) [shape = circle, draw] {};
\node[vertex] (G-1) at (0.0, 1) [shape = circle, draw] {};
\node[vertex] (G-4) at (4.5, 1) [shape = circle, draw] {};
\node[vertex] (G--2) at (1.5, -1) [shape = circle, draw] {};
\node[vertex] (G-2) at (1.5, 1) [shape = circle, draw] {};
\draw[] (G-1) -- (G--1);
\draw[] (G-3) -- (G--3);
\draw[] (G-1) .. controls +(.5,-.5) and +(-.5,-.5) .. (G-2);
\draw[] (G-3) .. controls +(.5,-.5) and +(-.5,-.5) .. (G-4);
\draw[] (G--1) .. controls +(.5, .5) and +(-.5, .5) .. (G--2);
\draw[] (G--3) .. controls +(.5, .5) and +(-.5, .5) .. (G--4);
\end{tikzpicture}\
\right) + 2
\chi^{\vec{\lambda}}_{G_\lambda}\left(\
\begin{tikzpicture}[partition-diagram]
\node[vertex] (G--3) at (3.0, -1) [shape = circle, draw] {};
\node[vertex] (G-3) at (3.0, 1) [shape = circle, draw] {};
\node[vertex] (G--4) at (4.5, -1) [shape = circle, draw] {};
\node[vertex] (G--1) at (0.0, -1) [shape = circle, draw] {};
\node[vertex] (G-1) at (0.0, 1) [shape = circle, draw] {};
\node[vertex] (G-4) at (4.5, 1) [shape = circle, draw] {};
\node[vertex] (G--2) at (1.5, -1) [shape = circle, draw] {};
\node[vertex] (G-2) at (1.5, 1) [shape = circle, draw] {};
\draw[] (G-3) .. controls +(-1, -1) and +(1, 1) .. (G--2);
\draw[] (G-2) .. controls +(1, -1) and +(-1, 1) .. (G--3);
\draw[] (G-1) .. controls +(.6,-.6) and +(-.6,-.6) .. (G-2);
\draw[] (G--1) .. controls +(.6, .6) and +(-.6, .6) .. (G--2);
\draw[] (G-3) .. controls +(.6,-.6) and +(-.6,-.6) .. (G-4);
\draw[] (G--3) .. controls +(.6, .6) and +(-.6, .6) .. (G--4);
\end{tikzpicture}\
\right) = -1.
\end{equation*}
\end{example}

\subsection{Reformulation of the characters of $\mathcal{U}_k$}
\label{sec:zeeformula}

As observed in Example~\ref{ex:character}, there can be elements in $C(d_{\vec{\mu}}; \lambda)$ that have the same cycle type in
$G_\lambda$.
For $\vec{\nu}$ with $\vtype(\vec{\nu}) = \lambda$, let
\[
    C(d_{\vec{\mu}}; \vec{\nu}) = \{ d\in C(d_{\vec{\mu}}; \lambda)\, :\, \mathsf{cycletype}(d)= \vec{\nu} \}
\]
and define $b_{\vec{\mu}}^{\vec{\nu}}= | C(d_{\vec{\mu}}; \vec{\nu})|$.
Then rewriting the formula in Proposition~\ref{prop:char}, we have
\begin{equation}\label{Eq:X=CB}
    \chi^{\vec{\lambda}}_{\mathcal{U}_k}(d_{\vec{\mu}}) = \sum_{\substack{\vec{\nu}\in I_k\\ \vtype(\vec{\nu}) = \lambda}}
    b_{\vec{\mu}}^{\vec{\nu}} \; \chi^{\vec{\lambda}}_{G_\lambda}(d_{\vec{\nu}}).
\end{equation}
Our next goal is to find an explicit formula for the multiplicities $b_{\vec{\mu}}^{\vec{\nu}}$.

\begin{example}\label{ex:Cnuset}
Consider $\vec{\mu} = ((4,2), (2)) \in I_{10}$, so that
\begin{center}
$d_{\vec{\mu}} = $
\begin{tikzpicture}[partition-diagram]
\node[vertex] (G--10) at (13.5, -1) [shape = circle, draw] {};
\node[vertex] (G--9) at (12.0, -1) [shape = circle, draw] {};
\node[vertex] (G-7) at (9.0, 1) [shape = circle, draw] {};
\node[vertex] (G-8) at (10.5, 1) [shape = circle, draw] {};
\node[vertex] (G--8) at (10.5, -1) [shape = circle, draw] {};
\node[vertex] (G--7) at (9.0, -1) [shape = circle, draw] {};
\node[vertex] (G-9) at (12.0, 1) [shape = circle, draw] {};
\node[vertex] (G-10) at (13.5, 1) [shape = circle, draw] {};
\node[vertex] (G--6) at (7.5, -1) [shape = circle, draw] {};
\node[vertex] (G-5) at (6.0, 1) [shape = circle, draw] {};
\node[vertex] (G--5) at (6.0, -1) [shape = circle, draw] {};
\node[vertex] (G-6) at (7.5, 1) [shape = circle, draw] {};
\node[vertex] (G--4) at (4.5, -1) [shape = circle, draw] {};
\node[vertex] (G-3) at (3.0, 1) [shape = circle, draw] {};
\node[vertex] (G--3) at (3.0, -1) [shape = circle, draw] {};
\node[vertex] (G-2) at (1.5, 1) [shape = circle, draw] {};
\node[vertex] (G--2) at (1.5, -1) [shape = circle, draw] {};
\node[vertex] (G-1) at (0.0, 1) [shape = circle, draw] {};
\node[vertex] (G--1) at (0.0, -1) [shape = circle, draw] {};
\node[vertex] (G-4) at (4.5, 1) [shape = circle, draw] {};
\draw[] (G-7) .. controls +(0.5, -0.5) and +(-0.5, -0.5) .. (G-8);
\draw[] (G--10) .. controls +(-0.5, 0.5) and +(0.5, 0.5) .. (G--9);
\draw[] (G-9) .. controls +(0.5, -0.5) and +(-0.5, -0.5) .. (G-10);
\draw[] (G--8) .. controls +(-0.5, 0.5) and +(0.5, 0.5) .. (G--7);
\draw[] (G-9) .. controls +(-1, -1) and +(1, 1) .. (G--8);
\draw[] (G-8) .. controls +(1, -1) and +(-1, 1) .. (G--9);
\draw[] (G-5) .. controls +(0.75, -1) and +(-0.75, 1) .. (G--6);
\draw[] (G-6) .. controls +(-0.75, -1) and +(0.75, 1) .. (G--5);
\draw[] (G-3) .. controls +(0.75, -1) and +(-0.75, 1) .. (G--4);
\draw[] (G-2) .. controls +(0.75, -1) and +(-0.75, 1) .. (G--3);
\draw[] (G-1) .. controls +(0.75, -1) and +(-0.75, 1) .. (G--2);
\draw[] (G-4) .. controls +(-1, -1) and +(1, 1) .. (G--1);
\end{tikzpicture}~.
\end{center}
We compute $C(d_{\vec{\mu}}; \vec{\nu})$ for three different choices of $\vec{\nu}$.
These examples will then be used later to demonstrate how
the algebraic formulas capture the enumeration of these sets.

First consider $\vec{\nu} = (\epart, (2,1), \epart, (1)) \in I_{10}$.
Starting with the above diagram for $d_{\vec{\mu}}$, we see that there
are two ways of creating a block of size $4$, either by adding edges
to the first cycle or the last cycle.  Therefore,
$C(d_{\vec{\mu}}; \vec{\nu})$ consists of the following two elements
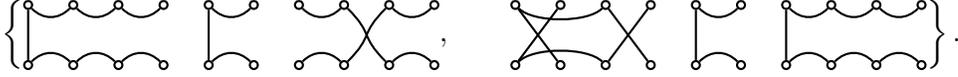
\[
\left\{
\begin{tikzpicture}[partition-diagram]
\node[vertex] (G--10) at (13.5, -1) [shape = circle, draw] {};
\node[vertex] (G--9) at (12.0, -1) [shape = circle, draw] {};
\node[vertex] (G-7) at (9.0, 1) [shape = circle, draw] {};
\node[vertex] (G-8) at (10.5, 1) [shape = circle, draw] {};
\node[vertex] (G--8) at (10.5, -1) [shape = circle, draw] {};
\node[vertex] (G--7) at (9.0, -1) [shape = circle, draw] {};
\node[vertex] (G-9) at (12.0, 1) [shape = circle, draw] {};
\node[vertex] (G-10) at (13.5, 1) [shape = circle, draw] {};
\node[vertex] (G--6) at (7.5, -1) [shape = circle, draw] {};
\node[vertex] (G-5) at (6.0, 1) [shape = circle, draw] {};
\node[vertex] (G--5) at (6.0, -1) [shape = circle, draw] {};
\node[vertex] (G-6) at (7.5, 1) [shape = circle, draw] {};
\node[vertex] (G--4) at (4.5, -1) [shape = circle, draw] {};
\node[vertex] (G-3) at (3.0, 1) [shape = circle, draw] {};
\node[vertex] (G--3) at (3.0, -1) [shape = circle, draw] {};
\node[vertex] (G-2) at (1.5, 1) [shape = circle, draw] {};
\node[vertex] (G--2) at (1.5, -1) [shape = circle, draw] {};
\node[vertex] (G-1) at (0.0, 1) [shape = circle, draw] {};
\node[vertex] (G--1) at (0.0, -1) [shape = circle, draw] {};
\node[vertex] (G-4) at (4.5, 1) [shape = circle, draw] {};
\draw[] (G-7) .. controls +(0.5, -0.5) and +(-0.5, -0.5) .. (G-8);
\draw[] (G--10) .. controls +(-0.5, 0.5) and +(0.5, 0.5) .. (G--9);
\draw[] (G-9) .. controls +(0.5, -0.5) and +(-0.5, -0.5) .. (G-10);
\draw[] (G--8) .. controls +(-0.5, 0.5) and +(0.5, 0.5) .. (G--7);
\draw[] (G-9) .. controls +(-1, -1) and +(1, 1) .. (G--8);
\draw[] (G-8) .. controls +(1, -1) and +(-1, 1) .. (G--9);
\draw[] (G-1) .. controls +(0.5, -0.5) and +(-0.5, -0.5) .. (G-2);
\draw[] (G-2) .. controls +(0.5, -0.5) and +(-0.5, -0.5) .. (G-3);
\draw[] (G-3) .. controls +(0.5, -0.5) and +(-0.5, -0.5) .. (G-4);
\draw[] (G--2) .. controls +(-0.5, 0.5) and +(0.5, 0.5) .. (G--1);
\draw[] (G--3) .. controls +(-0.5, 0.5) and +(0.5, 0.5) .. (G--2);
\draw[] (G--4) .. controls +(-0.5, 0.5) and +(0.5, 0.5) .. (G--3);
\draw[] (G-5) .. controls +(0.5, -0.5) and +(-0.5, -0.5) .. (G-6);
\draw[] (G--6) .. controls +(-0.5, 0.5) and +(0.5, 0.5) .. (G--5);
\draw[] (G-1) -- (G--1);
\draw[] (G-5) -- (G--5);
\end{tikzpicture},\qquad
\begin{tikzpicture}[partition-diagram]
\node[vertex] (G--10) at (13.5, -1) [shape = circle, draw] {};
\node[vertex] (G--9) at (12.0, -1) [shape = circle, draw] {};
\node[vertex] (G-7) at (9.0, 1) [shape = circle, draw] {};
\node[vertex] (G-8) at (10.5, 1) [shape = circle, draw] {};
\node[vertex] (G--8) at (10.5, -1) [shape = circle, draw] {};
\node[vertex] (G--7) at (9.0, -1) [shape = circle, draw] {};
\node[vertex] (G-9) at (12.0, 1) [shape = circle, draw] {};
\node[vertex] (G-10) at (13.5, 1) [shape = circle, draw] {};
\node[vertex] (G--6) at (7.5, -1) [shape = circle, draw] {};
\node[vertex] (G-5) at (6.0, 1) [shape = circle, draw] {};
\node[vertex] (G--5) at (6.0, -1) [shape = circle, draw] {};
\node[vertex] (G-6) at (7.5, 1) [shape = circle, draw] {};
\node[vertex] (G--4) at (4.5, -1) [shape = circle, draw] {};
\node[vertex] (G-3) at (3.0, 1) [shape = circle, draw] {};
\node[vertex] (G--3) at (3.0, -1) [shape = circle, draw] {};
\node[vertex] (G-2) at (1.5, 1) [shape = circle, draw] {};
\node[vertex] (G--2) at (1.5, -1) [shape = circle, draw] {};
\node[vertex] (G-1) at (0.0, 1) [shape = circle, draw] {};
\node[vertex] (G--1) at (0.0, -1) [shape = circle, draw] {};
\node[vertex] (G-4) at (4.5, 1) [shape = circle, draw] {};
\draw[] (G-7) .. controls +(0.5, -0.5) and +(-0.5, -0.5) .. (G-8);
\draw[] (G-8) .. controls +(0.5, -0.5) and +(-0.5, -0.5) .. (G-9);
\draw[] (G--10) .. controls +(-0.5, 0.5) and +(0.5, 0.5) .. (G--9);
\draw[] (G-9) .. controls +(0.5, -0.5) and +(-0.5, -0.5) .. (G-10);
\draw[] (G--8) .. controls +(-0.5, 0.5) and +(0.5, 0.5) .. (G--7);
\draw[] (G--9) .. controls +(-0.5, 0.5) and +(0.5, 0.5) .. (G--8);
\draw[] (G-7) -- (G--7);
\draw[] (G-1) .. controls +(.6,-.6) and +(-.6,-.6) .. (G-3);
\draw[] (G--1) .. controls +(.6, .6) and +(-.6, .6) .. (G--3);
\draw[] (G-3) .. controls +(0.75, -1) and +(-0.75, 1) .. (G--4);
\draw[] (G-1) .. controls +(0.75, -1) and +(-0.75, 1) .. (G--2);
\draw[] (G-2) .. controls +(-0.75, -1) and +(0.75, 1) .. (G--1);
\draw[] (G-4) .. controls +(-0.75, -1) and +(0.75, 1) .. (G--3);
\draw[] (G-5) .. controls +(0.5, -0.5) and +(-0.5, -0.5) .. (G-6);
\draw[] (G--6) .. controls +(-0.5, 0.5) and +(0.5, 0.5) .. (G--5);
\draw[] (G-5) -- (G--5);
\end{tikzpicture}
\right\}.
\]

Next consider $\vec{\nu} = (\epart, (2), (2)) \in I_{10}$.  To make a cycle
of length $2$ with three vertices, the cycle of length two in the middle
of the diagram can either connect to the first or the last cycle and it
can connect in two ways.  Therefore, there are 4 elements in the set
$C(d_{\vec{\mu}}; \vec{\nu})$
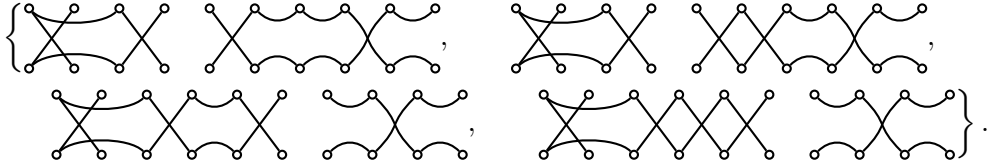
\[
\left\{
\begin{tikzpicture}[partition-diagram]
\node[vertex] (G--10) at (13.5, -1) [shape = circle, draw] {};
\node[vertex] (G--9) at (12.0, -1) [shape = circle, draw] {};
\node[vertex] (G-7) at (9.0, 1) [shape = circle, draw] {};
\node[vertex] (G-8) at (10.5, 1) [shape = circle, draw] {};
\node[vertex] (G--8) at (10.5, -1) [shape = circle, draw] {};
\node[vertex] (G--7) at (9.0, -1) [shape = circle, draw] {};
\node[vertex] (G-9) at (12.0, 1) [shape = circle, draw] {};
\node[vertex] (G-10) at (13.5, 1) [shape = circle, draw] {};
\node[vertex] (G--6) at (7.5, -1) [shape = circle, draw] {};
\node[vertex] (G-5) at (6.0, 1) [shape = circle, draw] {};
\node[vertex] (G--5) at (6.0, -1) [shape = circle, draw] {};
\node[vertex] (G-6) at (7.5, 1) [shape = circle, draw] {};
\node[vertex] (G--4) at (4.5, -1) [shape = circle, draw] {};
\node[vertex] (G-3) at (3.0, 1) [shape = circle, draw] {};
\node[vertex] (G--3) at (3.0, -1) [shape = circle, draw] {};
\node[vertex] (G-2) at (1.5, 1) [shape = circle, draw] {};
\node[vertex] (G--2) at (1.5, -1) [shape = circle, draw] {};
\node[vertex] (G-1) at (0.0, 1) [shape = circle, draw] {};
\node[vertex] (G--1) at (0.0, -1) [shape = circle, draw] {};
\node[vertex] (G-4) at (4.5, 1) [shape = circle, draw] {};
\draw[] (G-7) .. controls +(0.5, -0.5) and +(-0.5, -0.5) .. (G-8);
\draw[] (G--10) .. controls +(-0.5, 0.5) and +(0.5, 0.5) .. (G--9);
\draw[] (G-9) .. controls +(0.5, -0.5) and +(-0.5, -0.5) .. (G-10);
\draw[] (G--8) .. controls +(-0.5, 0.5) and +(0.5, 0.5) .. (G--7);
\draw[] (G-9) .. controls +(-1, -1) and +(1, 1) .. (G--8);
\draw[] (G-8) .. controls +(1, -1) and +(-1, 1) .. (G--9);
\draw[] (G-1) .. controls +(.6,-.6) and +(-.6,-.6) .. (G-3);
\draw[] (G--1) .. controls +(.6, .6) and +(-.6, .6) .. (G--3);
\draw[] (G-3) .. controls +(0.75, -1) and +(-0.75, 1) .. (G--4);
\draw[] (G-1) .. controls +(0.75, -1) and +(-0.75, 1) .. (G--2);
\draw[] (G-2) .. controls +(-0.75, -1) and +(0.75, 1) .. (G--1);
\draw[] (G-4) .. controls +(-0.75, -1) and +(0.75, 1) .. (G--3);
\draw[] (G-5) .. controls +(0.75, -1) and +(-0.75, 1) .. (G--6);
\draw[] (G-6) .. controls +(-0.75, -1) and +(0.75, 1) .. (G--5);
\draw[] (G-6) .. controls +(0.5, -0.5) and +(-0.5, -0.5) .. (G-7);
\draw[] (G--7) .. controls +(-0.5, 0.5) and +(0.5, 0.5) .. (G--6);
\end{tikzpicture},\qquad
\begin{tikzpicture}[partition-diagram]
\node[vertex] (G--10) at (13.5, -1) [shape = circle, draw] {};
\node[vertex] (G--9) at (12.0, -1) [shape = circle, draw] {};
\node[vertex] (G-7) at (9.0, 1) [shape = circle, draw] {};
\node[vertex] (G-8) at (10.5, 1) [shape = circle, draw] {};
\node[vertex] (G--8) at (10.5, -1) [shape = circle, draw] {};
\node[vertex] (G--7) at (9.0, -1) [shape = circle, draw] {};
\node[vertex] (G-9) at (12.0, 1) [shape = circle, draw] {};
\node[vertex] (G-10) at (13.5, 1) [shape = circle, draw] {};
\node[vertex] (G--6) at (7.5, -1) [shape = circle, draw] {};
\node[vertex] (G-5) at (6.0, 1) [shape = circle, draw] {};
\node[vertex] (G--5) at (6.0, -1) [shape = circle, draw] {};
\node[vertex] (G-6) at (7.5, 1) [shape = circle, draw] {};
\node[vertex] (G--4) at (4.5, -1) [shape = circle, draw] {};
\node[vertex] (G-3) at (3.0, 1) [shape = circle, draw] {};
\node[vertex] (G--3) at (3.0, -1) [shape = circle, draw] {};
\node[vertex] (G-2) at (1.5, 1) [shape = circle, draw] {};
\node[vertex] (G--2) at (1.5, -1) [shape = circle, draw] {};
\node[vertex] (G-1) at (0.0, 1) [shape = circle, draw] {};
\node[vertex] (G--1) at (0.0, -1) [shape = circle, draw] {};
\node[vertex] (G-4) at (4.5, 1) [shape = circle, draw] {};
\draw[] (G-7) .. controls +(0.5, -0.5) and +(-0.5, -0.5) .. (G-8);
\draw[] (G--10) .. controls +(-0.5, 0.5) and +(0.5, 0.5) .. (G--9);
\draw[] (G-9) .. controls +(0.5, -0.5) and +(-0.5, -0.5) .. (G-10);
\draw[] (G--8) .. controls +(-0.5, 0.5) and +(0.5, 0.5) .. (G--7);
\draw[] (G-9) .. controls +(-1, -1) and +(1, 1) .. (G--8);
\draw[] (G-8) .. controls +(1, -1) and +(-1, 1) .. (G--9);
\draw[] (G-1) .. controls +(.6,-.6) and +(-.6,-.6) .. (G-3);
\draw[] (G--1) .. controls +(.6, .6) and +(-.6, .6) .. (G--3);
\draw[] (G-3) .. controls +(0.75, -1) and +(-0.75, 1) .. (G--4);
\draw[] (G-4) .. controls +(-0.75, -1) and +(0.75, 1) .. (G--3);
\draw[] (G-1) .. controls +(0.75, -1) and +(-0.75, 1) .. (G--2);
\draw[] (G-2) .. controls +(-0.75, -1) and +(0.75, 1) .. (G--1);
\draw[] (G-5) .. controls +(0.75, -1) and +(-0.75, 1) .. (G--6);
\draw[] (G-6) .. controls +(-0.75, -1) and +(0.75, 1) .. (G--5);
\draw[] (G-6) .. controls +(0.75, -1) and +(-0.75, 1) .. (G--7);
\draw[] (G-7) .. controls +(-0.75, -1) and +(0.75, 1) .. (G--6);
\end{tikzpicture}, \qquad
\right.\hskip .09in
\]
\[
\left.\hskip .09in
\begin{tikzpicture}[partition-diagram]
\node[vertex] (G--10) at (13.5, -1) [shape = circle, draw] {};
\node[vertex] (G--9) at (12.0, -1) [shape = circle, draw] {};
\node[vertex] (G-7) at (9.0, 1) [shape = circle, draw] {};
\node[vertex] (G-8) at (10.5, 1) [shape = circle, draw] {};
\node[vertex] (G--8) at (10.5, -1) [shape = circle, draw] {};
\node[vertex] (G--7) at (9.0, -1) [shape = circle, draw] {};
\node[vertex] (G-9) at (12.0, 1) [shape = circle, draw] {};
\node[vertex] (G-10) at (13.5, 1) [shape = circle, draw] {};
\node[vertex] (G--6) at (7.5, -1) [shape = circle, draw] {};
\node[vertex] (G-5) at (6.0, 1) [shape = circle, draw] {};
\node[vertex] (G--5) at (6.0, -1) [shape = circle, draw] {};
\node[vertex] (G-6) at (7.5, 1) [shape = circle, draw] {};
\node[vertex] (G--4) at (4.5, -1) [shape = circle, draw] {};
\node[vertex] (G-3) at (3.0, 1) [shape = circle, draw] {};
\node[vertex] (G--3) at (3.0, -1) [shape = circle, draw] {};
\node[vertex] (G-2) at (1.5, 1) [shape = circle, draw] {};
\node[vertex] (G--2) at (1.5, -1) [shape = circle, draw] {};
\node[vertex] (G-1) at (0.0, 1) [shape = circle, draw] {};
\node[vertex] (G--1) at (0.0, -1) [shape = circle, draw] {};
\node[vertex] (G-4) at (4.5, 1) [shape = circle, draw] {};
\draw[] (G-7) .. controls +(0.5, -0.5) and +(-0.5, -0.5) .. (G-8);
\draw[] (G--10) .. controls +(-0.5, 0.5) and +(0.5, 0.5) .. (G--9);
\draw[] (G-9) .. controls +(0.5, -0.5) and +(-0.5, -0.5) .. (G-10);
\draw[] (G--8) .. controls +(-0.5, 0.5) and +(0.5, 0.5) .. (G--7);
\draw[] (G-9) .. controls +(-1, -1) and +(1, 1) .. (G--8);
\draw[] (G-8) .. controls +(1, -1) and +(-1, 1) .. (G--9);
\draw[] (G-1) .. controls +(.6,-.6) and +(-.6,-.6) .. (G-3);
\draw[] (G--1) .. controls +(.6, .6) and +(-.6, .6) .. (G--3);
\draw[] (G-3) .. controls +(0.75, -1) and +(-0.75, 1) .. (G--4);
\draw[] (G-1) .. controls +(0.75, -1) and +(-0.75, 1) .. (G--2);
\draw[] (G-2) .. controls +(-0.75, -1) and +(0.75, 1) .. (G--1);
\draw[] (G-4) .. controls +(-0.75, -1) and +(0.75, 1) .. (G--3);
\draw[] (G-5) .. controls +(0.75, -1) and +(-0.75, 1) .. (G--6);
\draw[] (G-6) .. controls +(-0.75, -1) and +(0.75, 1) .. (G--5);
\draw[] (G--5) .. controls +(-0.5, 0.5) and +(0.5, 0.5) .. (G--4);
\draw[] (G-4) .. controls +(0.5, -0.5) and +(-0.5, -0.5) .. (G-5);
\end{tikzpicture}, \qquad
\begin{tikzpicture}[partition-diagram]
\node[vertex] (G--10) at (13.5, -1) [shape = circle, draw] {};
\node[vertex] (G--9) at (12.0, -1) [shape = circle, draw] {};
\node[vertex] (G-7) at (9.0, 1) [shape = circle, draw] {};
\node[vertex] (G-8) at (10.5, 1) [shape = circle, draw] {};
\node[vertex] (G--8) at (10.5, -1) [shape = circle, draw] {};
\node[vertex] (G--7) at (9.0, -1) [shape = circle, draw] {};
\node[vertex] (G-9) at (12.0, 1) [shape = circle, draw] {};
\node[vertex] (G-10) at (13.5, 1) [shape = circle, draw] {};
\node[vertex] (G--6) at (7.5, -1) [shape = circle, draw] {};
\node[vertex] (G-5) at (6.0, 1) [shape = circle, draw] {};
\node[vertex] (G--5) at (6.0, -1) [shape = circle, draw] {};
\node[vertex] (G-6) at (7.5, 1) [shape = circle, draw] {};
\node[vertex] (G--4) at (4.5, -1) [shape = circle, draw] {};
\node[vertex] (G-3) at (3.0, 1) [shape = circle, draw] {};
\node[vertex] (G--3) at (3.0, -1) [shape = circle, draw] {};
\node[vertex] (G-2) at (1.5, 1) [shape = circle, draw] {};
\node[vertex] (G--2) at (1.5, -1) [shape = circle, draw] {};
\node[vertex] (G-1) at (0.0, 1) [shape = circle, draw] {};
\node[vertex] (G--1) at (0.0, -1) [shape = circle, draw] {};
\node[vertex] (G-4) at (4.5, 1) [shape = circle, draw] {};
\draw[] (G-7) .. controls +(0.5, -0.5) and +(-0.5, -0.5) .. (G-8);
\draw[] (G--10) .. controls +(-0.5, 0.5) and +(0.5, 0.5) .. (G--9);
\draw[] (G-9) .. controls +(0.5, -0.5) and +(-0.5, -0.5) .. (G-10);
\draw[] (G--8) .. controls +(-0.5, 0.5) and +(0.5, 0.5) .. (G--7);
\draw[] (G-9) .. controls +(-1, -1) and +(1, 1) .. (G--8);
\draw[] (G-8) .. controls +(1, -1) and +(-1, 1) .. (G--9);
\draw[] (G-1) .. controls +(.6,-.6) and +(-.6,-.6) .. (G-3);
\draw[] (G--1) .. controls +(.6, .6) and +(-.6, .6) .. (G--3);
\draw[] (G-3) .. controls +(0.75, -1) and +(-0.75, 1) .. (G--4);
\draw[] (G-1) .. controls +(0.75, -1) and +(-0.75, 1) .. (G--2);
\draw[] (G-2) .. controls +(-0.75, -1) and +(0.75, 1) .. (G--1);
\draw[] (G-4) .. controls +(-0.75, -1) and +(0.75, 1) .. (G--3);
\draw[] (G-5) .. controls +(0.75, -1) and +(-0.75, 1) .. (G--6);
\draw[] (G-6) .. controls +(-0.75, -1) and +(0.75, 1) .. (G--5);
\draw[] (G-4) .. controls +(0.75, -1) and +(-0.75, 1) .. (G--5);
\draw[] (G-5) .. controls +(-0.75, -1) and +(0.75, 1) .. (G--4);
\end{tikzpicture}
\right\}.
\]

Finally, consider $\vec{\nu} = (\epart, (2,2,1))$.  There is
precisely one way to add edges to $d_{\vec{\mu}}$ to obtain an
element of cycle type $\vec{\nu}$ and so $C(d_{\vec{\mu}}; \vec{\nu})$
is the following singleton set
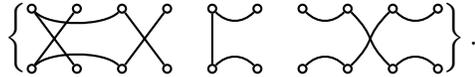
\[
\left\{
\begin{tikzpicture}[partition-diagram]
\node[vertex] (G--10) at (13.5, -1) [shape = circle, draw] {};
\node[vertex] (G--9) at (12.0, -1) [shape = circle, draw] {};
\node[vertex] (G-7) at (9.0, 1) [shape = circle, draw] {};
\node[vertex] (G-8) at (10.5, 1) [shape = circle, draw] {};
\node[vertex] (G--8) at (10.5, -1) [shape = circle, draw] {};
\node[vertex] (G--7) at (9.0, -1) [shape = circle, draw] {};
\node[vertex] (G-9) at (12.0, 1) [shape = circle, draw] {};
\node[vertex] (G-10) at (13.5, 1) [shape = circle, draw] {};
\node[vertex] (G--6) at (7.5, -1) [shape = circle, draw] {};
\node[vertex] (G-5) at (6.0, 1) [shape = circle, draw] {};
\node[vertex] (G--5) at (6.0, -1) [shape = circle, draw] {};
\node[vertex] (G-6) at (7.5, 1) [shape = circle, draw] {};
\node[vertex] (G--4) at (4.5, -1) [shape = circle, draw] {};
\node[vertex] (G-3) at (3.0, 1) [shape = circle, draw] {};
\node[vertex] (G--3) at (3.0, -1) [shape = circle, draw] {};
\node[vertex] (G-2) at (1.5, 1) [shape = circle, draw] {};
\node[vertex] (G--2) at (1.5, -1) [shape = circle, draw] {};
\node[vertex] (G-1) at (0.0, 1) [shape = circle, draw] {};
\node[vertex] (G--1) at (0.0, -1) [shape = circle, draw] {};
\node[vertex] (G-4) at (4.5, 1) [shape = circle, draw] {};
\draw[] (G-7) .. controls +(0.5, -0.5) and +(-0.5, -0.5) .. (G-8);
\draw[] (G--10) .. controls +(-0.5, 0.5) and +(0.5, 0.5) .. (G--9);
\draw[] (G-9) .. controls +(0.5, -0.5) and +(-0.5, -0.5) .. (G-10);
\draw[] (G--8) .. controls +(-0.5, 0.5) and +(0.5, 0.5) .. (G--7);
\draw[] (G-9) .. controls +(-1, -1) and +(1, 1) .. (G--8);
\draw[] (G-8) .. controls +(1, -1) and +(-1, 1) .. (G--9);
\draw[] (G-5) .. controls +(0.5, -0.5) and +(-0.5, -0.5) .. (G-6);
\draw[] (G--6) .. controls +(-0.5, 0.5) and +(0.5, 0.5) .. (G--5);
\draw[] (G-5) -- (G--5);
\draw[] (G-1) .. controls +(.6,-.6) and +(-.6,-.6) .. (G-3);
\draw[] (G--1) .. controls +(.6, .6) and +(-.6, .6) .. (G--3);
\draw[] (G-3) .. controls +(0.75, -1) and +(-0.75, 1) .. (G--4);
\draw[] (G-1) .. controls +(0.75, -1) and +(-0.75, 1) .. (G--2);
\draw[] (G-2) .. controls +(-0.75, -1) and +(0.75, 1) .. (G--1);
\draw[] (G-4) .. controls +(-0.75, -1) and +(0.75, 1) .. (G--3);
\end{tikzpicture}
\right\}.
\]
\end{example}

Let $A_1, \ldots, A_r$ be the blocks of a set partition of $[k]$ so that $|A_i| = |A_j|$ for all $1\leqslant i, j\leqslant r$. Suppose that $d$ consists of
the blocks $\{A_i, \overline{A_{i+1}}\}$ for $1\leqslant i \leqslant r-1$ and $\{A_r, \overline{A_1}\}$. If $|A_i|=q$ and $A_i =\{(i-1)q +1, (i-1)q +2,\ldots, iq\}$ 
for all $i$, then we refer to this particular cycle as the \defn{canonical $r$-cycle}.  Notice that $d$ is an $r$-cycle in the sense that if we
ignore bars and present the permutation as a two line array, then $d$ has the form
\[
    \left( \begin{array}{cccc}
    A_1 & A_2 & \cdots & A_r\\
    A_2 & A_3 & \cdots & A_1
    \end{array}\right).
\]
In general, a permutation of the blocks $A_1, A_2, \ldots, A_r$ is called an
\defn{$r$-cycle} if there exists a $\sigma\in \mathfrak{S}_r$ such that $A_i \mapsto A_{\sigma(i)}$ and $\sigma$ is an $r$-cycle in $\mathfrak{S}_r$.

\begin{lemma}
\label{lem:onecycle}
Let $d\in \mathcal{U}_k$ be the canonical $r$-cycle permuting blocks $A_1, \ldots, A_r$ with $|A_i| = \frac{k}{r}$ for all $i$.
If we take unions of blocks of $d$ to get a new diagram $d'$ satisfying $\mathsf{top}(d') = \mathsf{bot}(d')$, then $d'$ is an $s$-cycle where $s$
divides $r$ and the blocks of $\mathsf{top}(d')$ are all of size $\frac{k}{s}$.
\end{lemma}

\begin{proof}
Suppose $d'$ is a uniform block permutation obtained by taking the union of blocks of $d$ and $\mathsf{top}(d') = \mathsf{bot}(d')=\{B_1, \ldots, B_s\}$,
where $s\leqslant r$. Without loss of generality, suppose that $B_1$ is the block containing $A_1$ and $B_1$ is the union of $r_1$ blocks. Then
$B_1 = A_1\cup A_{i_2} \cup \cdots \cup A_{i_{r_1}}$ with $1<i_2< \cdots < i_{r_1}$. If two indices $1,i_2,\ldots,i_r$ are adjacent, then by the assumption
that $d$ is the canonical cycle and $\mathsf{top}(d')= \mathsf{bot}(d')$, $B_1$ would be the union of all $A_i$. In this case $d'$ is a cycle of length $s=1$.
Otherwise, since $d$ is the canonical $r$-cycle, we have that $A_2\cup A_{i_2+1}\cup \cdots \cup A_{i_{r_1}+1}$ is the image block of $B_1$ in
$\mathsf{bot}(d')$ which is the union of $r_1$ blocks, we call this block $B_2$.
By a similar argument, there are blocks $B_3, B_4, \ldots, B_{i_2-1}$ all of size $r_1$ containing $A_3, A_4, \ldots, A_{i_2-1}$,
respectively. Notice that $B_{i_2-1}$ contains $A_r$ since $d$ maps $A_{i_2-1}$ to $A_{i_2}$ and $A_r$ to $A_1$, so that in $d'$ $A_{i_2}$ and $A_1$
are in the same block, i.e. $B_1$, and this block connects to the block that contains $A_{i_2-1}$ and $A_r$.  Similar arguments show that $A_{r-1}$
is contained in $B_{i_2-2}$, etc. So, $s=i_2-1$ is the total number of blocks in $\mathsf{top}(d')$.  In addition, the blocks $B_1, B_2, \ldots, B_{i_2-1}$
contain blocks $A_1, A_2, \ldots, A_{(i_2-1)r_1}$ and they are all of the same size.  This means that $r = (i_2-1)r_1$, hence
$|B_i| = r_1|A_i|= \frac{r|A_i|}{s} = \frac{k}{s}$.  By the description above we see that $d' = \left( \begin{array}{cccc} B_1 & B_2 & \ldots & B_s\\
B_2 & B_3 & \ldots & B_1\end{array} \right)$ is an $s$-cycle.
\end{proof}

There are two takeaways from the proof of Lemma \ref{lem:onecycle}. The first is that taking a union of blocks of the canonical $r$-cycle results in an
$s$-cycle. The second take away is that if we label the blocks of $d'$ as in the proof of Lemma~\ref{lem:onecycle}, then $B_i = \cup_j A_j$ where $j$ is 
congruent to $i$ mod $s$. In addition, the proof implies that the union of blocks of the canonical $r$-cycle results in an $s$-cycle, where $s$ divides $r$ 
and there is only one possible way to get an $s$-cycle.

\begin{example}\label{ex:unions1cycle} 
In this example we use squares for the vertices to indicate that the vertices represent sets $A_1, A_2, \ldots, A_6$ all of the same size and 
containing consecutive integers. Let 
\begin{center}
$d = \left( \begin{array}{cccccc}
    A_1 & A_2 & A_3 & A_4 & A_5 & A_6\\
    A_2 & A_3 & A_4 & A_5 & A_6 & A_1
    \end{array}\right)= $
\begin{tikzpicture}[partition-diagram]
\node[vertex] (G--6) at (7.5, -1) [shape = rectangle, draw] {};
\node[vertex] (G-5) at (6.0, 1) [shape = rectangle, draw] {};
\node[vertex] (G--5) at (6.0, -1) [shape = rectangle, draw] {};
\node[vertex] (G-6) at (7.5, 1) [shape = rectangle, draw] {};
\node[vertex] (G--4) at (4.5, -1) [shape = rectangle, draw] {};
\node[vertex] (G-3) at (3.0, 1) [shape = rectangle, draw] {};
\node[vertex] (G--3) at (3.0, -1) [shape = rectangle, draw] {};
\node[vertex] (G-2) at (1.5, 1) [shape = rectangle, draw] {};
\node[vertex] (G--2) at (1.5, -1) [shape = rectangle, draw] {};
\node[vertex] (G-1) at (0.0, 1) [shape = rectangle, draw] {};
\node[vertex] (G--1) at (0.0, -1) [shape = rectangle, draw] {};
\node[vertex] (G-4) at (4.5, 1) [shape = rectangle, draw] {};
\draw[] (G-5) .. controls +(0.75, -1) and +(-0.75, 1) .. (G--6);
\draw[] (G-6) .. controls +(-0.75, -1) and +(0.75, 1) .. (G--1);
\draw[] (G-3) .. controls +(0.75, -1) and +(-0.75, 1) .. (G--4);
\draw[] (G-2) .. controls +(0.75, -1) and +(-0.75, 1) .. (G--3);
\draw[] (G-1) .. controls +(0.75, -1) and +(-0.75, 1) .. (G--2);
\draw[] (G-4) .. controls +(.75, -1) and +(-.75, 1) .. (G--5);
\end{tikzpicture}~.
\end{center}
Then there are three possible ways to take unions of the blocks of $d$ to get diagrams $d'$ with $\mathsf{top}(d') = \mathsf{bot}(d')$.  
We can take the union of all the blocks to get one single block on top and on the bottom, i.e., $B = \cup_{i=1}^6 A_i$, which yields
\begin{center}
$d' = \left( \begin{array}{c}
    B\\ B
    \end{array}\right)= $
\begin{tikzpicture}[partition-diagram]
\node[vertex] (G--6) at (7.5, -1) [shape = rectangle, draw] {};
\node[vertex] (G-5) at (6.0, 1) [shape = rectangle, draw] {};
\node[vertex] (G--5) at (6.0, -1) [shape = rectangle, draw] {};
\node[vertex] (G-6) at (7.5, 1) [shape = rectangle, draw] {};
\node[vertex] (G--4) at (4.5, -1) [shape = rectangle, draw] {};
\node[vertex] (G-3) at (3.0, 1) [shape = rectangle, draw] {};
\node[vertex] (G--3) at (3.0, -1) [shape = rectangle, draw] {};
\node[vertex] (G-2) at (1.5, 1) [shape = rectangle, draw] {};
\node[vertex] (G--2) at (1.5, -1) [shape = rectangle, draw] {};
\node[vertex] (G-1) at (0.0, 1) [shape = rectangle, draw] {};
\node[vertex] (G--1) at (0.0, -1) [shape = rectangle, draw] {};
\node[vertex] (G-4) at (4.5, 1) [shape = rectangle, draw] {};
\draw[] (G-1) -- (G--1);
\draw[] (G-1) .. controls +(.5, -.5) and +(-.5, -.5) .. (G-2);
\draw[] (G-2) .. controls +(.5, -.5) and +(-.5, -.5) .. (G-3);
\draw[] (G-3) .. controls +(.5, -.5) and +(-.5, -.5) .. (G-4);
\draw[] (G-4) .. controls +(.5, -.5) and +(-.5, -.5) .. (G-5);
\draw[] (G-5) .. controls +(.5, -.5) and +(-.5, -.5) .. (G-6);
\draw[] (G--6) .. controls +(-0.5, .5) and +(0.5, .5) .. (G--5);
\draw[] (G--5) .. controls +(-0.5, .5) and +(0.5, .5) .. (G--4);
\draw[] (G--4) .. controls +(-0.5, .5) and +(0.5, .5) .. (G--3);
\draw[] (G--3) .. controls +(-0.5, .5) and +(0.5, .5) .. (G--2);
\draw[] (G--2) .. controls +(-0.5, .5) and +(0.5, .5) .. (G--1);
\end{tikzpicture}~.
\end{center}
Alternatively, we can get a two cycle by taking the following unions $B_1 = A_1\cup A_3 \cup A_5$ and $B_2 = A_2 \cup A_4 \cup A_6$, which gives
\begin{center}
$d' = \left( \begin{array}{cc}
   \textcolor{red}{B_1} & \textcolor{blue}{B_2}\\ \textcolor{blue}{B_2} & \textcolor{red}{B_1}
    \end{array}\right)= $
\begin{tikzpicture}[partition-diagram]
\node[vertex] (G--6) at (7.5, -1) [shape = rectangle, draw] {};
\node[vertex] (G-5) at (6.0, 1) [shape = rectangle, draw] {};
\node[vertex] (G--5) at (6.0, -1) [shape = rectangle, draw] {};
\node[vertex] (G-6) at (7.5, 1) [shape = rectangle, draw] {};
\node[vertex] (G--4) at (4.5, -1) [shape = rectangle, draw] {};
\node[vertex] (G-3) at (3.0, 1) [shape = rectangle, draw] {};
\node[vertex] (G--3) at (3.0, -1) [shape = rectangle, draw] {};
\node[vertex] (G-2) at (1.5, 1) [shape = rectangle, draw] {};
\node[vertex] (G--2) at (1.5, -1) [shape = rectangle, draw] {};
\node[vertex] (G-1) at (0.0, 1) [shape = rectangle, draw] {};
\node[vertex] (G--1) at (0.0, -1) [shape = rectangle, draw] {};
\node[vertex] (G-4) at (4.5, 1) [shape = rectangle, draw] {};
\draw[line width=0.5mm] (G-1) .. controls +(.7, -.7) and +(-.7, -.7) .. (G-3)[color=red];
\draw[line width=0.5mm] (G-3) .. controls +(.7, -.7) and +(-.7, -.7) .. (G-5)[color=red];
\draw[line width=0.5mm] (G-2) .. controls +(.7, -.7) and +(-.7, -.7) .. (G-4)[color=blue];
\draw[line width=0.5mm] (G-4) .. controls +(.7, -.7) and +(-.7, -.7) .. (G-6)[color=blue];
\draw[line width=0.5mm] (G--6) .. controls +(-0.7, .7) and +(0.7, .7) .. (G--4)[color=blue];
\draw[line width=0.5mm] (G--4) .. controls +(-0.7, .7) and +(0.7, .7) .. (G--2)[color=blue];
\draw[line width=0.5mm] (G--5) .. controls +(-0.7, .7) and +(0.7, .7) .. (G--3)[color=red];
\draw[line width=0.5mm] (G--3) .. controls +(-0.7, .7) and +(0.7, .7) .. (G--1)[color=red];
\draw[] (G-1) .. controls +(0.75, -1) and +(-0.75, 1) .. (G--2);
\draw[] (G-2) .. controls +(-0.75, -1) and +(0.75, 1) .. (G--1);
\end{tikzpicture}~.
\end{center}
Finally, we can get a three cycle by taking the following unions $B_1 = A_1\cup A_4$ and $B_2 = A_2 \cup A_5$ and $B_3 = A_3 \cup A_6$ to obtain
\begin{center}
$d' = \left( \begin{array}{ccc}
    \textcolor{red}{B_1} & \textcolor{blue}{B_2} & \textcolor{darkgreen}{B_3}\\ \textcolor{blue}{B_2} & \textcolor{darkgreen}{B_3} & \textcolor{red}{B_1}
    \end{array}\right)= $
\begin{tikzpicture}[partition-diagram]
\node[vertex] (G--6) at (7.5, -1) [shape = rectangle, draw] {};
\node[vertex] (G-5) at (6.0, 1) [shape = rectangle, draw] {};
\node[vertex] (G--5) at (6.0, -1) [shape = rectangle, draw] {};
\node[vertex] (G-6) at (7.5, 1) [shape = rectangle, draw] {};
\node[vertex] (G--4) at (4.5, -1) [shape = rectangle, draw] {};
\node[vertex] (G-3) at (3.0, 1) [shape = rectangle, draw] {};
\node[vertex] (G--3) at (3.0, -1) [shape = rectangle, draw] {};
\node[vertex] (G-2) at (1.5, 1) [shape = rectangle, draw] {};
\node[vertex] (G--2) at (1.5, -1) [shape = rectangle, draw] {};
\node[vertex] (G-1) at (0.0, 1) [shape = rectangle, draw] {};
\node[vertex] (G--1) at (0.0, -1) [shape = rectangle, draw] {};
\node[vertex] (G-4) at (4.5, 1) [shape = rectangle, draw] {};
\draw[line width=0.5mm] (G-1) .. controls +(.8, -.8) and +(-.8, -.8) .. (G-4)[color=red];
\draw[line width=0.5mm] (G-2) .. controls +(.8, -.8) and +(-.8, -.8) .. (G-5)[color=blue];
\draw[line width=0.5mm] (G-3) .. controls +(.8, -.8) and +(-.8, -.8) .. (G-6)[color=darkgreen];
\draw[line width=0.5mm] (G--6) .. controls +(-0.8, .8) and +(0.8, .8) .. (G--3)[color=darkgreen];
\draw[line width=0.5mm] (G--5) .. controls +(-0.8, .8) and +(0.8, .8) .. (G--2)[color=blue];
\draw[line width=0.5mm] (G--4) .. controls +(-0.8, .8) and +(0.8, .8) .. (G--1)[color=red];
\draw[] (G-1) .. controls +(0.75, -1) and +(-0.75, 1) .. (G--2);
\draw[] (G-2) .. controls +(0.75, -1) and +(-0.75, 1) .. (G--3);
\draw[] (G-3) .. controls +(-0.75, -1) and +(0.75, 1) .. (G--1);
\end{tikzpicture}~.
\end{center}
\end{example}

\begin{lemma}
\label{lem:manycycles}
Let $r, t$ be positive integers.
Let $d\in \mathcal{U}_k$ consist of $t$ $r$-cycles so that the $r$ blocks in each cycle are labeled $B_1^{(i)},\ldots, B_r^{(i)}$ for $1\leqslant i\leqslant t$
and the $r$-cycles have the form:
\[
    \left( \begin{array}{cccc}
    B_1^{(i)} & B_2^{(i)} & \cdots & B_r^{(i)}\\
    B_2^{(i)} & B_3^{(i)} & \cdots & B_1^{(i)}
    \end{array}\right),
\]
as obtained in the proof of Lemma \ref{lem:onecycle}.
There are $r^{t-1}$ ways to take unions of the blocks of $d$ to get a diagram $d'$ that is an $r$-cycle and $\mathsf{top}(d') = \mathsf{bot}(d')$.
Moreover, the blocks of $d'$ consist of disjoint unions of the blocks of $d$, where there is exactly one block from each $r$-cycle in $d$ and hence all blocks of $d'$ have the same size.
\end{lemma}

\begin{proof}
The proof is by induction on $t$. The lemma is trivially true when $t=1$. Suppose that $d$ consists of two $r$-cycles, the first permutes blocks
$B_1^{(1)}, \ldots, B_r^{(1)}$ and the other permutes blocks $B_{1}^{(2)}, \ldots, B_{r}^{(2)}$. Notice that there are exactly $r$ ways to form an $r$-cycle
if the union consists of one block in the first cycle with one block in the second cycle. Since $B_1^{(1)}\cup B_{j}^{(2)}$, for any $1\leqslant j \leqslant r$,
can be a block in the union, by the
structure of $d$ and the requirement that $\mathsf{top}(d') = \mathsf{bot}(d')$ we obtain that $B_2^{(1)}\cup B_{j+1}^{(2)}$ is a block and in general
$B_i^{(1)} \cup B_{j+i-1}^{(2)}$ are the blocks permuted by $d'$, where $j+i-1$ is taken mod $r$.  This is an $r$-cycle because there are $r$ blocks and
 we have that $d'$ maps $B_i^{(1)} \cup B_{j+i-1}^{(2)}$ to the block $B_{i+1}^{(1)} \cup B_{j+i}^{(2)}$ and, if $i = r$, then it maps to $B_1^{(1)}\cup B_{j}^{(2)}$.

To see that these are the only ways to obtain an $r$ cycle, we argue by contradiction. If the unions contain more blocks from each $r$-cycle, $d'$ has
fewer than $r$ blocks; and if the union contains more blocks from only one cycle, then the cycles are not all of the same size. In either case we
cannot get an $r$-cycle of blocks all of the same size.

Now if we have $t\geqslant 2$ $r$-cycles, we know by induction that there are $r^{t-2}$ ways in which the first $t-1$ form an $r$-cycle. Now for each way of
forming this union there are $r$ ways to take the union with the last $r$-cycle.
\end{proof}
In our next Proposition we want to count the number of elements in $C(d_{\vec{\mu}}; \vec{\nu})$ for any $\vec{\mu}, \vec{\nu}\in I_k$.
Here, the parts of both $\vec{\mu}$ and $\vec{\nu}$ represent the sizes of cycles.  We think of each component in $\vec{\mu} = (\mu^{(1)}, \ldots, \mu^{(k)})$ in exponential notation with $m_i(\mu^{(a)})$ the number of cycles
of length $i$ in $d_{\vec{\mu}}$ permuting blocks of size $a$.  For any partition $\mu$ of $k$, let 
$\mathfrak{m}(\mu) = (m_1(\mu), m_2(\mu), \ldots, m_k(\mu))$, where $m_i(\mu)$ is the multiplicity of $i$ in $\mu$. 
For $\vec{\mu}\in I_k$, define $\mathfrak{m}(\vec{\mu}) :=(\mathfrak{m}(\mu^{(1)}), \mathfrak{m}(\mu^{(2)}), \ldots, \mathfrak{m}(\mu^{(k)}))$, where 
we think of $\mathfrak{m}(\mu^{(i)})$ as a sequence of length $k$ by adding trailing zeros if necessary.
Given two vectors $\vec{v} = (v_1, \ldots, v_k)$ and $\vec{u}= (u_1, \ldots, u_k)$ we define ${\vec{v}\choose \vec{u}}
= {v_1\choose u_1}{v_2\choose u_2}\ldots{v_k \choose u_k}$. Given any vector partition $\vec{\tau} = (\tau^{(1)}, \ldots,\tau^{(k)})$ we define
\[
    { \mathfrak{m}(\vec{\mu}) \choose \mathfrak{m}(\vec{\tau})}:= {\mathfrak{m}(\mu^{(1)})
    \choose \mathfrak{m}(\tau^{(1)})} {\mathfrak{m}(\mu^{(2)})
    \choose \mathfrak{m}(\tau^{(2)})} \cdots{\mathfrak{m}(\mu^{(k)})
    \choose \mathfrak{m}(\tau^{(k)})}.
\]
In a similar way, the multinomial coefficient generalizes to any vectors. If $\vec{u}_1, \vec{u}_2, \ldots, \vec{u}_\ell$ are vectors such that
$\vec{u}_1+\vec{u}_2+\cdots +\vec{u}_\ell = \vec{v}$, then
\[
    {\vec{v} \choose {\vec{u}_1, \vec{u}_2, \ldots, \vec{u}_\ell}}={\vec{v} \choose \vec{u}_1}{\vec{v}-\vec{u}_1 \choose
    \vec{u}_2} {\vec{v}-\vec{u}_1-\vec{u}_2 \choose \vec{u}_3}\cdots {\vec{u}_\ell \choose \vec{u}_{\ell}}.
\]

For any vector partition $\vec{\mu} \in I_k$, let $\ell(\vec{\mu}) = \sum_{a=1}^k \ell(\mu^{(a)})$
be the sum of the length of  the partitions in $\vec{\mu}$ and $\mathfrak{m}(\vec{\mu})! :=
\prod_{i=1}^k m_1(\mu^{(i)})! m_2(\mu^{(i)})! \cdots m_k(\mu^{(i)})!$
be the product of the factorial of the multiplicities of the parts of $\vec{\mu}$.
Furthermore, for a partition $\mu$, we define $\mathfrak{m}(\mu)! = m_1(\mu)! m_2(\mu)! \cdots m_k(\mu)!$.

\begin{prop}\label{prop:B}  For $\vec{\mu}, \vec{\nu} \in I_k$,
\begin{equation}\label{eq:B}
b_{\vec{\mu}}^{\vec{\nu}} = \frac{1}{\mathfrak{m}(\vec{\nu})!}\sum_{\vec{\tau}(\bullet, \bullet)}{\mathfrak{m}(\vec{\mu})\choose \mathfrak{m}(\vec{\tau}(1,1)),
\mathfrak{m}(\vec{\tau}(2,1)), \ldots, \mathfrak{m}(\vec{\tau}(k,k))}
\prod_{\substack{1\leqslant j\leqslant k\\
    1 \leqslant i \leqslant \ell(\nu^{(j)})}} (\nu_i^{(j)})^{\ell(\vec{\tau}(i,j))-1},
\end{equation}
where the sum is over all sequences $\vec{\tau}(\bullet, \bullet)$ of $\vec{\tau}(i,j)$'s such that for each $\nu_i^{(j)} \neq 0$,
$\vec{\tau}(i,j) \in I_j$ and $\vec{\mu} = \biguplus_{i,j} \nu_i^{(j)} \vec{\tau}(i,j)$.
\end{prop}

\begin{proof}
Recall that $b_{\vec{\mu}}^{\vec{\nu}}$ is the number of ways
to take the union of the cycles in $d_{\vec{\mu}}$ so that the result $d'$
satisfies $\mathsf{top}(d') = \mathsf{bot}(d')$ and $\mathsf{cycletype}(d') = \vec{\nu}$.

By Lemma \ref{lem:onecycle} and Lemma \ref{lem:manycycles}
we know that every $r$ cycle in $d_{\vec{\mu}}$ can contribute to exactly one
cycle of length $s$, where $s$ divides $r$.  Fix $1\leqslant j\leqslant k$ and $1\leqslant i\leqslant \ell(\nu^{(j)})$, and
let $C_{i,j}$ be a cycle of length $\nu_i^{(j)}$ in $d'$, an element in $C(d_{\vec{\mu}}; \vec{\nu})$.
The cycle $C_{i,j}$ is obtained as the union of cycles in $d_{\vec{\mu}}$.
We record which cycles we use to form $C_{i,j}$ in a vector partition
$\vec{\tau}(i,j)$ that satisfies the following conditions: 
\begin{enumerate}
    \item For every part $\vec{\tau}(i,j)_t^{(a)}$ of $\vec{\tau}(i,j)$ there exists a nonzero $\mu_{t'}^{(a)}$ such that 
    $\mu_{t'}^{(a)} = \nu_i^{(j)}\vec{\tau}(i,j)_t^{(a)}$, where $1\leqslant a\leqslant k$, $1\leqslant t \leqslant \ell(\vec{\tau}(i,j)^{(a)})$ and 
    $1\leqslant t' \leqslant \ell(\mu^{(a)})$.
    \item The total number of cycles from $d_{\vec{\mu}}$ used to construct $C_{i,j}$
    is $\ell(\vec{\tau}(i,j))$.
    \item The blocks that make up $C_{i,j}$ have $\sum_{a=1}^j |\vec{\tau}(i,j)^{(a)}|\cdot a = j$ elements
    and so $\vec{\tau}(i,j) \in I_j$.
\end{enumerate}
Condition (1) simply says that if a cycle $\mu_{t'}^{(a)}$ is used in the union that gives $C_{i,j}$, then we get a part in $\vec{\tau}(i,j)$.  
Condition (2) is a consequence of condition (1).  Condition (3) follows because $C_{i,j}$ is a cycle that permutes blocks of size $j$, 
$\vec{\tau}(i,j)_t^{(a)}$ represents the number of blocks in a cycle of length $\mu_{t'}^{(a)}$ that were unioned to get a cycle of length $\nu_i^{(j)}$.  

To count all possible ways to union the cycles of $d_{\vec{\mu}}$ to form a diagram $d'\in C(d_{\vec{\mu}}; \vec{\nu})$ we order the cycle lengths: 
$\nu_i^{(j)} < \nu_{i'}^{(j')}$ if $j< j'$, or if $j=j'$ and $i<i'$. If it is possible to union the blocks of $d_{\vec{\mu}}$ to get
$d'$, then there exists a sequence of $\vec{\tau}(i,j)$ such that
$\vec{\mu} = \biguplus_{i,j} \nu_i^{(j)} \vec{\tau}(i,j)$ and
each $\vec{\tau}(i,j) \in I_j$.  For each such $\vec{\tau}(\bullet, \bullet)$ there are
${\mathfrak{m}(\vec{\mu})\choose \mathfrak{m}(\vec{\tau}(1,1)), \mathfrak{m}(\vec{\tau}(2,1)), \ldots, \mathfrak{m}(\vec{\tau}(k,k))}$
ways to choose the cycles from $d_{\vec{\mu}}$ to form the vector partitions $\vec{\tau}(i,j)$. Once the cycles are chosen from $d_{\vec{\mu}}$, we use Lemma \ref{lem:onecycle} to get cycles of length $\nu_i^{(j)}$ and then by
Lemma \ref{lem:manycycles} there are $(\nu_i^{(j)})^{\ell(\vec{\tau}(i,j))-1}$ ways to union the cycles indexed by $\vec{\tau}(i,j)$ since by condition (2) 
there are a total of $\ell(\vec{\tau}(i,j))$ such cycles. If the cycles were ordered as described above, there are
\begin{equation}\label{eq:orderedcycles}
    \sum_{\vec{\tau}(\bullet, \bullet)}{\mathfrak{m}(\vec{\mu})\choose \mathfrak{m}(\vec{\tau}(1,1)), \mathfrak{m}(\vec{\tau}(2,1)), \ldots,
    \mathfrak{m}(\vec{\tau}(k,k))}\prod_{\substack{1\leqslant j\leqslant k\\
    1 \leqslant i \leqslant \ell(\nu^{(j)})}} (\nu_i^{(j)})^{\ell(\vec{\tau}(i,j))-1}
\end{equation}
ways to obtain elements $d'\in C(d_{\vec{\mu}}; \vec{\nu})$.

However, we want to count the number of ways of obtaining elements of cycle type $\vec{\nu}$,
where the elements of the same length are indistinguishable.
In the case when $\nu_i^{(j)} = \nu_{i'}^{(j)}$ the order in which the cycles in
$d_{\vec{\mu}}$ are chosen to form the cycles of this length is
interchangeable.  Therefore, we need to divide by the ways to permute the
cycles of a fixed length. That is we divide Equation~\eqref{eq:orderedcycles} by $\mathfrak{m}(\vec{\nu})!$
to enumerate $b_{\vec{\mu}}^{\vec{\nu}}$.
\end{proof}

\begin{example}  In Example \ref{ex:Cnuset} we gave three examples of
constructing the elements of $C( d_{\vec{\mu}}; \vec{\nu})$ with $\vec{\mu} = ((4,2), (2))$.
Here, we show
how Proposition~\ref{prop:B} enumerates these sets.

First, we choose $\vec{\nu} = (\epart, (2,1), \epart, (1)) \in I_{10}$
and we compute that there are two possible
$\vec{\tau}(\bullet, \bullet) = (\vec{\tau}(1,2), \vec{\tau}(2,2), \vec{\tau}(1,4))$
satisfying
\[
    \big((4, 2), (2)\big) =
    2 \cdot \vec{\tau}(1,2)~\biguplus~1 \cdot \vec{\tau}(2,2)~\biguplus~1 \cdot \vec{\tau}(1,4) \,;
\]
namely,
\begin{equation*}
    \vec{\tau}(\bullet, \bullet) = \Big(\big((2), \epart\big), \big((2), \epart\big), \big(\epart, (2)\big)\Big)
    \quad\text{and}\quad
    \vec{\tau}(\bullet, \bullet) = \Big(\big(\epart, (1)\big), \big((2), \epart \big), \big((4), \epart\big)\Big).
\end{equation*}
Note that $\mathfrak{m}(\vec{\nu})! = 1$ and in both cases
the summands in Equation~\eqref{eq:orderedcycles} are equal to $1$.
Hence $b_{\vec{\mu}}^{\vec{\nu}} = 2$.

Now consider $\vec{\nu} = (\epart, (2), (2)) \in I_{10}$.
There are again two possible $\vec{\tau}(\bullet, \bullet) = (\vec{\tau}(1,2), \vec{\tau}(1,3))$
such that
\[
\big((4, 2), (2)\big) = 2 \cdot \vec{\tau}(1,2)~\biguplus~2 \cdot \vec{\tau}(1,3)
\]
and this time they are
\begin{equation*}
    \vec{\tau}(\bullet, \bullet) = \Big(\big((2), \epart\big), \big((1), (1)\big)\Big)
    \quad\text{and}\quad
    \vec{\tau}(\bullet, \bullet) = \Big(\big(\epart, (1)\big), \big((2, 1), \epart\big)\Big).
\end{equation*}
We have again that $\mathfrak{m}(\vec{\nu})! = 1$,
however this time we see that the summands for Equation~\eqref{eq:orderedcycles}
are both equal to $2$, hence $b_{\vec{\mu}}^{\vec{\nu}} = 4$.

Finally,
$\vec{\nu} = (\epart, (2,2,1)) \in I_{10}$, there are two
$\vec{\tau}(\bullet, \bullet) = (\vec{\tau}(1,2), \vec{\tau}(2,2), \vec{\tau}(3,2))$ satisfying
\[
\big((4, 2), (2)\big) =
2 \cdot \vec{\tau}(1,2)~\biguplus~2 \cdot \vec{\tau}(2,2)~\biguplus~1 \cdot \vec{\tau}(3,2) \,;
\]
namely,
\begin{equation*}
    \vec{\tau}(\bullet, \bullet) = \Big(\big((2), \epart\big), \big(\epart, (1)\big), \big((2), \epart\big)\Big)
    \quad\text{and}\quad
    \vec{\tau}(\bullet, \bullet) = \Big(\big(\epart, (1)\big), \big((2), \epart \big), \big((2), \epart\big)\Big).
\end{equation*}
The summands of Equation~\eqref{eq:orderedcycles} are both $1$ for these $\vec{\tau}(\bullet, \bullet)$,
but we now have $\mathfrak{m}(\vec{\nu})! = 2$,
hence Equation~\eqref{eq:B} says that $b_{\vec{\mu}}^{\vec{\nu}} = 1$.
\end{example}

For any partition $\mu = (1^{a_1} 2^{a_2} \ldots r^{a_r})$ we define
\[
    z_\mu = 1^{a_1} a_1! 2^{a_2} a_2! \ldots r^{a_r}a_r!.
\]
Notice that we can rewrite this as follows:
\[
    z_\mu = \mathfrak{m}(\mu)!\prod_{1\leqslant i\leqslant \ell(\mu)}\mu_i ~,
\]
which is the product of the parts of $\mu$ times the factorials of the multiplicities.
Then for $\vec{\mu} \in I_k$ we set
\[
    \mathbf{z}_{\vec{\mu}} = z_{\mu^{(1)}}z_{\mu^{(2)}}\cdots z_{\mu^{(k)}}.
\]

\begin{cor}
For $\vec{\mu}, \vec{\nu} \in I_k$,
\begin{equation}\label{eq:BinZ}
    b_{\vec{\mu}}^{\vec{\nu}} = \frac{1}{\mathbf{z}_{\vec{\nu}}} \sum_{\vec{\tau}(\bullet, \bullet)}\frac{\mathbf{z}_{\vec{\mu}}}{\prod_{i,j}\mathbf{z}_{\vec{\tau}(i,j)}},
\end{equation}
where the sum is over all $\vec{\tau}(\bullet, \bullet)$
such that for each $1\leqslant j\leqslant k$ and $1 \leqslant i \leqslant \ell(\nu^{(j)})$,
$\vec{\tau}(i,j) \in I_j$
and $\vec{\mu} = \biguplus_{i,j} \nu_i^{(j)} \vec{\tau}(i,j)$.
\end{cor}

\begin{proof}
In Proposition \ref{prop:B} we showed
\[
    b_{\vec{\mu}}^{\vec{\nu}} = \frac{1}{\mathfrak{m}(\vec{\nu})!}\sum_{\vec{\tau}(\bullet,\bullet)}{\mathfrak{m}(\vec{\mu})\choose \mathfrak{m}(\vec{\tau}(1,1)),
    \mathfrak{m}(\vec{\tau}(2,1)), \ldots, \mathfrak{m}(\vec{\tau}(k,k))}
    \prod_{\substack{1\leqslant j\leqslant k\\
        1 \leqslant i \leqslant \ell(\nu^{(j)})}} (\nu_i^{(j)})^{\ell(\vec{\tau}(i,j))-1}.
\]
We multiply the numerator and denominator of the right-hand side by
$\nu_i^{(j)}$ for all $(i, j)$ satisfying $1\leqslant j\leqslant k$ and $1 \leqslant i \leqslant \ell(\nu^{(j)})$.
In the denominator, we get $\mathbf{z}_{\vec{\nu}}$ and in the numerator all the powers of the $\nu_i^{(j)}$ increase by 1; therefore, we have
\[
    b_{\vec{\mu}}^{\vec{\nu}} = \frac{1}{\mathbf{z}_{\vec{\nu}}}\sum_{\vec{\tau}(\bullet,\bullet)}{\mathfrak{m}(\vec{\mu})\choose \mathfrak{m}(\vec{\tau}(1,1)),
    \mathfrak{m}(\vec{\tau}(2,1)), \ldots, \mathfrak{m}(\vec{\tau}(k,k))}
    \prod_{\substack{1\leqslant j\leqslant k\\
        1 \leqslant i \leqslant \ell(\nu^{(j)})}} (\nu_i^{(j)})^{\ell(\vec{\tau}(i,j))}.
\]
Note that
\begin{equation*}
    (\nu_i^{(j)})^{\ell(\vec{\tau}(i,j))} = \prod\limits_{\substack{1\leqslant a\leqslant k\\
    1 \leqslant t \leqslant \ell(\tau(i,j)^{(a)})}} \frac{\mu_{t'}^{(a)}}{\tau(i,j)_t^{(a)}},
\end{equation*}
where for each $t$, $\mu_{t'}^{(a)}$ is the cycle length corresponding to $\tau(i,j)_t^{(a)}$
used in the construction of the cycle of length $\nu_i^{(j)}$,
i.e., cycles that satisfy $\mu_{t'}^{(a)} = \nu_i^{(j)} \tau(i,j)_t^{(a)}$
for some $t$ and $t'$.
Once we substitute this expression for $\nu_i^{(j)}$, in the numerator we set the product of all the parts of $\vec{\mu}$ and in the
denominator we get products of the parts of all $\vec{\tau}(i,j)$ for all $i$ and $j$.  Now expanding the multinomial coefficient and regrouping gives
Equation~\eqref{eq:BinZ}.
\end{proof}

\section{Characters and symmetric functions}
\label{sec:charsSym}

This section presents a formula for the irreducible characters of $\mathcal{U}_k$ in terms of symmetric functions.

\subsection{Class functions and a scalar product} \label{subsec:classfunctions}

A \defn{class function} of $\mathcal{U}_k$ is a map $\alpha: \mathcal{U}_k \to
\CC$ that is constant on the generalized conjugacy classes of $\mathcal{U}_k$.
In light of Proposition~\ref{prop:conjugacyclasses},
$\alpha: \mathcal{U}_k \to \CC$ is a class function of $\mathcal{U}_k$ if
$x, y \in \mathcal{U}_k$,
$$\mathsf{cycletype}(x) = \mathsf{cycletype}(y)\hbox{ implies }\alpha(x) = \alpha(y)~.$$
Denote the set of class functions on $\mathcal{U}_k$ by
$\class({\mathcal{U}_k})$.  The class functions form a $\CC$-algebra
under pointwise product (also called the Kronecker product).

Let $\vec{\mu} \in I_k$ and define
$\ind_{\vec{\mu}} : \mathcal{U}_k \rightarrow \CC$ to be the
indicator function of the generalized conjugacy class indexed by
$\vec{\mu}$.  That is, for $x \in \mathcal{U}_k$,
\begin{equation}\label{eq:indicator}
    \ind_{\vec{\mu}}(x) = \begin{cases}1,&\hbox{ if }\mathsf{cycletype}(x)=\vec{\mu},\\
    0,&\hbox{ otherwise.}\end{cases}
\end{equation}
These functions form a
basis of the $\CC$-vector space of class functions of $\mathcal{U}_k$.

The restriction of a class function $\alpha$ of $\mathcal{U}_k$ to its
representative maximal subgroups $G_\lambda$ results in a class function of
$G_\lambda$. Moreover, by \cite[Proposition 7.5 and Proposition 7.6]{Steinberg.2016},
the function
\begin{equation}\label{restrictionmap}
    \mathsf{Res} \colon \class(\mathcal{U}_k) \longrightarrow \prod_{\lambda \vdash k} \class(G_{\lambda})
\end{equation}
defined by
$$\mathsf{Res}(\alpha) = \alpha|_{\prod_{\lambda \vdash k} G_{\lambda}}$$
is an isomorphism of $\CC$-algebras.

As with the class functions of a finite group, there is a scalar product
defined on the class functions of finite monoids.
The \defn{scalar product} of two class functions $\alpha, \beta \in \class(\mathcal{U}_k)$
is
\begin{equation}\label{eq:defspclass}
\left< \alpha, \beta \right>_{\class(\mathcal{U}_k)} =
\sum_{\lambda \vdash k} \frac{1}{|G_{\lambda}|} \sum_{x \in G_{\lambda}} \alpha(x) \overline{\beta(x)}.
\end{equation}
The indicator functions defined in Equation~\ref{eq:indicator} form an orthogonal basis
with respect to this scalar product.  That is,
$$
    \left< \ind_{\vec{\lambda}}, \ind_{\vec{\mu}} \right>_{\class(\mathcal{U}_k)}
    = \begin{cases}\frac{1}{\mathbf{z}_{\vec{\lambda}}}, &\hbox{ if }\vec{\lambda} = \vec{\mu},\\
    0, & \hbox{else.}\end{cases}
$$
The irreducible characters of $\mathcal{U}_k$ also form a basis of $\class(\mathcal{U}_k)$
\cite[Proposition 7.9 and 7.10]{Steinberg.2016};
note however that they are not orthogonal with respect to this
scalar product.

\subsection{Symmetric functions on multiple alphabets}
Let $\XXX = X_1, X_2, \ldots$ be an infinite number of alphabets.  We define the polynomial ring
\[
    \mathsf{Sym}^\ast_\XXX := \CC[p_i[X_j]\, |\, i,j\geqslant 1],
\]
where the degree of $p_i[X_j]$ is $ij$.  If $\mu = (1^{a_1}2^{a_2}\ldots r^{a_r})$ is a partition,
then we define
$$p_\mu[X_j] := p_1[X_j]^{a_1} p_2[X_j]^{a_2} \cdots p_r[X_j]^{a_r};$$
for $\vec{\mu} \in I_k$, we define
$$\mathbf{p}_{\vec{\mu}}[\XXX] := p_{\mu^{(1)}}[X_1] p_{\mu^{(2)}}[X_2] \cdots p_{\mu^{(k)}}[X_k],$$
which is an element in $\mathsf{Sym}^\ast_\XXX$ of degree $k$.

For any $f[\XXX] \in \mathsf{Sym}^\ast_\XXX$, we say that $f[\XXX]$ is of \defn{homogeneous
degree $k$} if
$f[\XXX]$ is in the linear span of $\{ \mathbf{p}_{\vec{\mu}}[\XXX] \}_{\vec{\mu} \in I_k}$
for some non-negative integer $k$.
When $f[\XXX]$ is of homogeneous degree $k$, we denote the degree by $\deg(f[\XXX]) = k$.
The subspace of $\mathsf{Sym}^\ast_\XXX$ of elements of homogeneous degree $k$ is denoted by
$\mathsf{Sym}^\ast_{\XXX,k}$ and it has the set of elements $\{ \mathbf{p}_{\vec{\mu}}[\XXX] \}_{\vec{\mu} \in I_k}$ as a linear basis. Therefore, we have that
\[
    \mathsf{Sym}^\ast_\XXX = \bigoplus_{k \geqslant 0} \mathsf{Sym}^\ast_{\XXX,k}
\]
is a graded ring such that if $f[\XXX] \in \mathsf{Sym}^\ast_{\XXX,k}$ and $g[\XXX] \in \mathsf{Sym}^\ast_{\XXX,\ell}$, then $f[\XXX] g[\XXX] \in 
\mathsf{Sym}^\ast_{\XXX,k+\ell}$.

We define a scalar product on $\mathsf{Sym}^\ast_{\XXX}$ as follows:
\begin{equation}\label{eq:defspsym}
    \left< \mathbf{p}_{\vec{\lambda}}[\XXX], \mathbf{p}_{\vec{\mu}}[\XXX] \right> =
    \begin{cases}\mathbf{z}_{\vec{\mu}}, &\hbox{ if }\vec{\lambda} = \vec{\mu},\\
    0,&\hbox{else.}\end{cases}
\end{equation}
Note that $\big\langle \mathbf{p}_{\vec{\lambda}}[\XXX], \mathbf{p}_{\vec{\mu}}[\XXX] \big\rangle
= \prod_{i=1}^{k} \big\langle p_{\lambda^{(i)}}[X_i], p_{\mu^{(i)}}[X_i] \big\rangle$,
where on the right hand side the scalar products are in the
usual ring of symmetric functions for which the power sum functions form an orthogonal basis.
In particular, for any element $f[\XXX] \in \mathsf{Sym}^\ast_\XXX$,
\begin{equation}
\label{eq:scalar-product-with-power-sums}
\text{$\left< f[\XXX], \mathbf{p}_{\vec{\mu}}[\XXX] \right>$ is equal
to the coefficient of $\frac{\mathbf{p}_{\vec{\mu}}[\XXX]}{\mathbf{z}_{\vec{\mu}}}$
in $f[\XXX]$.}
\end{equation}

Define also an analogue of the \defn{Schur basis}.  For
$\vec{\mu} \in I_k$, set
\[
    \mathbf{s}_{\vec{\mu}}[\XXX] :=
    s_{\mu^{(1)}}[X_1] s_{\mu^{(2)}}[X_2] \cdots s_{\mu^{(k)}}[X_k]~,
\]
where $s_{\mu^{(i)}}[X_i]$ is the Schur function over the alphabet $X_i$.
It follows that
\begin{equation}\label{eq:sorthonormal}
    \left< \mathbf{s}_{\vec{\lambda}}[\XXX],
    \mathbf{s}_{\vec{\mu}}[\XXX] \right> =
    \begin{cases}
    1,&\hbox{ if }\vec{\lambda} = \vec{\mu},\\
    0,&\hbox{ otherwise.}
    \end{cases}
\end{equation}
Since we know for $\lambda, \mu \vdash a$, that the coefficient $\left< s_\lambda[X], p_\mu[X] \right> = \chi^{\lambda}_{\mathfrak{S}_a}(\mu)$, 
it follows that 
\begin{equation}\label{eq:spval}
\left< \mathbf{s}_{\vec{\lambda}}[\XXX],
\mathbf{p}_{\vec{\mu}}[\XXX] \right> =
\chi^{\lambda^{(1)}}_{\mathfrak{S}_{a_1}}(\mu^{(1)})
\chi^{\lambda^{(2)}}_{\mathfrak{S}_{a_2}}(\mu^{(2)})
\cdots
\chi^{\lambda^{(k)}}_{\mathfrak{S}_{a_k}}(\mu^{(k)}),
\end{equation}
where $a_i = |\lambda^{(i)}|$ for $1 \leqslant i \leqslant k$.

Note that the value on the right hand side of Equation~\eqref{eq:spval}
is equal to the irreducible character indexed by $\vec{\lambda}$
evaluated at an element of cycle type $\vec{\mu}$ of the maximal subgroup
$\mathfrak{S}_{a_1} \times \mathfrak{S}_{a_2} \times \cdots \times \mathfrak{S}_{a_k} \simeq
G_\lambda$, where $\lambda = \vtype(\vec{\lambda}) = \vtype(\vec{\mu})=(1^{a_1} 2^{a_2} \ldots k^{a_k})$.

\subsection{A Frobenius characteristic map for the maximal subgroups $G_\lambda$}
By Corollary~\ref{cor:whatisGe},
for each partition $\lambda = (1^{a_1}2^{a_2}\ldots k^{a_k})$,
the maximal subgroup $G_\lambda$ is isomorphic to
\[
    G_\lambda \simeq
    \mathfrak{S}_{a_1} \times \mathfrak{S}_{a_2} \times \cdots \times \mathfrak{S}_{a_k}~.
\]
The usual Frobenius map for $\mathfrak{S}_r$ sends the irreducible character of $\mathfrak{S}_r$
indexed by the partition $\mu \vdash r$ to the Schur function indexed
by that partition.  We denote this map by
\[
    \phi_{\mathfrak{S}_r}( \chi^\mu_{\mathfrak{S}_r} ) = s_\mu~.
\]

Since the maximal subgroups are isomorphic to direct products of
symmetric groups, the Frobenius map can be extended to $G_\lambda$ by
mapping the class functions of $G_\lambda$
to the $k$-fold tensor product of the symmetric functions.
Under this map, the image of the irreducible character of $G_\lambda$ indexed by $\vec{\lambda} \in I_k$
with $\vtype(\vec{\lambda}) = \lambda = (1^{a_1}2^{a_2}\ldots k^{a_k})$
is
\[
    \phi_{G_\lambda}\left( \chi^{\vec{\lambda}}_{G_\lambda}\right) =
    \phi_{G_\lambda}\left( \chi^{\lambda^{(1)}}_{\mathfrak{S}_{a_1}}
    \chi^{\lambda^{(2)}}_{\mathfrak{S}_{a_2}}
    \cdots
    \chi^{\lambda^{(k)}}_{\mathfrak{S}_{a_k}} \right)
    = s_{\lambda^{(1)}}[X_1] s_{\lambda^{(2)}}[X_2] \cdots s_{\lambda^{(k)}}[X_k]
    = \mathbf{s}_{\vec{\lambda}}[\XXX]~.
\]

Equation~\eqref{eq:spval} states that for $\vec{\lambda}, \vec{\mu} \in I_k$
with $\vtype(\vec{\lambda}) = \vtype(\vec{\mu}) = \lambda$,
\[
    \chi^{\vec{\lambda}}_{G_\lambda}( d_{\vec{\mu}})
    = \left< \mathbf{s}_{\vec{\lambda}}[\XXX], \mathbf{p}_{\vec{\mu}}[\XXX]\right>
\]
which by Equation~\ref{eq:scalar-product-with-power-sums} is equal to the coefficient
$\frac{\mathbf{p}_{\vec{\mu}}[\XXX]}{\mathbf{z}_{\vec{\mu}}}$
in $\mathbf{s}_{\vec{\lambda}}[\XXX]$.
Furthermore, by Equation~\eqref{eq:sorthonormal}, the images
of the irreducible characters are all (even if they are not
characters of the same maximal subgroup)
orthonormal in $\mathsf{Sym}^{\ast}_{\XXX,k}$.

\subsection{A Frobenius characteristic map for $\mathcal{U}_k$}
\label{subsec:UkFrob}
We now extend the characteristic map from the class functions
on the maximal subgroups to the ring of class functions
on $\mathcal{U}_k$.

The \defn{Frobenius characteristic map}
$\phi_{\mathcal{U}_k}: \class(\mathcal{U}_k) \to \mathsf{Sym}^{\ast}_{\XXX,k}$
is defined on the basis of indicator functions by
\begin{equation}\label{eq:chmapind}
\phi_{\mathcal{U}_k}( \ind_{\vec{\mu}} ) = \frac{\mathbf{p}_{\vec{\mu}}[\XXX]}{\mathbf{z}_{\vec{\mu}}}~,
\end{equation}
and extended linearly to all of $\class(\mathcal{U}_k)$.
Then for any $\mathcal{U}_k$-class function
$\psi_{\mathcal{U}_k} \colon \mathcal{U}_k \rightarrow \CC$,
\begin{equation}
\label{equation.phi Uk}
    \phi_{\mathcal{U}_k}( \psi_{\mathcal{U}_k} ) = \sum_{\vec{\mu}\in I_k}
    \psi_{\mathcal{U}_k}(d_{\vec{\mu}}) \frac{\mathbf{p}_{\vec{\mu}}[\XXX]}{\mathbf{z}_{\vec{\mu}}},
\end{equation}
where $d_{\vec{\mu}}$ is a representative element from the
generalized conjugacy class $C_{\vec{\mu}}$.

Recall from Equation~\eqref{restrictionmap} that we have an isomorphism
$\class(\mathcal{U}_k) \simeq \prod_{\lambda \vdash k} \class(G_\lambda)$
that is given by mapping a class function of $\mathcal{U}_k$ to the
restrictions to maximal subgroups.
A similar result also holds for the Frobenius characteristic map.
Since $d_{\vec{\mu}}$ belongs to exactly one maximal subgroup,
Equation~\eqref{eq:chmapind} implies that
\begin{equation}\label{eq:restrictimage}
\phi_{\mathcal{U}_k}( f ) = \sum_{\lambda \vdash k}
\phi_{G_\lambda}( f |_{G_\lambda})~
\end{equation}
for any $f \in \class({\mathcal{U}_k})$.

Since the images of the irreducible characters of $G_\lambda$ are the symmetric
functions $\mathbf{s}_{\vec{\lambda}}[\XXX]$
and this basis is orthonormal in $\mathsf{Sym}^{\ast}_{\XXX,k}$,
Equation~\eqref{eq:restrictimage} implies the following important property
of the Frobenius images that we have defined here.

\begin{prop}\label{prop:multiplicity}
Let $\vec{\lambda} \in I_k$, $\lambda = \vtype(\vec{\lambda})$,
and let $\chi$ be a character of $\mathcal{U}_k$.
The multiplicity of $\chi_{G_\lambda}^{\vec{\lambda}}$
in the restriction of $\chi$ from $\mathcal{U}_k$ to $G_\lambda$
is equal to
\[
    \big\langle \phi_{\mathcal{U}_k}( \chi ), \, \mathbf{s}_{\vec{\lambda}}[\XXX] \big\rangle~.
\]
\end{prop}

\begin{proof}
By Equation~\eqref{eq:restrictimage},
\begin{align*}
\left< \phi_{\mathcal{U}_k}( \chi ), \mathbf{s}_{\vec{\lambda}}[\XXX] \right>
&=\left< \phi_{\mathcal{U}_k}( \chi ),
\phi_{G_\lambda}( \chi^{\vec{\lambda}}_{G_\lambda}) \right>\\
&=
\sum_{\gamma \vdash k}
\left< \phi_{G_\gamma}( \chi |_{G_\gamma}),
\phi_{G_\lambda}( \chi^{\vec{\lambda}}_{G_\lambda}) \right>\\
&=
\left< \phi_{G_\lambda}( \chi |_{G_\lambda}),
\phi_{G_\lambda}( \chi^{\vec{\lambda}}_{G_\lambda}) \right>
\end{align*}
because the Frobenius images of the irreducible characters of $G_\lambda$
are orthogonal to the Frobenius images of the
irreducible characters of $G_\gamma$ if $\gamma \neq \lambda$.
Since $\phi_{G_\lambda}( \chi^{\vec{\lambda}}_{G_\lambda}) =
\mathbf{s}_{\vec{\lambda}}[\XXX]$,
\[
    \left< \phi_{\mathcal{U}_k}( \chi ), \mathbf{s}_{\vec{\lambda}}[\XXX] \right>
    =
    \left< \phi_{G_\lambda}( \chi |_{G_\lambda}),
    \phi_{G_\lambda}( \chi^{\vec{\lambda}}_{G_\lambda}) \right>
    = \left< \phi_{G_\lambda}( \chi |_{G_\lambda}), \mathbf{s}_{\vec{\lambda}}[\XXX] \right>,
\]
which is the coefficient of $\mathbf{s}_{\vec{\lambda}}[\XXX]$ in
$\phi_{G_\lambda}( \chi |_{G_\lambda})$;
or in other words, it is the multiplicity of the
irreducible $G_\lambda$-representation
indexed by $\vec{\lambda} \in I_k$ in $\chi |_{G_\lambda}$.
\end{proof}

Moreover, the Frobenius characteristic function $\phi_{\mathcal{U}_k}$ has the property that
$$\left< \ind_{\vec{\mu}},
\ind_{\vec{\lambda}} \right>_{\class(\mathcal{U}_k)}
= \left< \phi_{\mathcal{U}_k}( \ind_{\vec{\mu}} ),
\phi_{\mathcal{U}_k}( \ind_{\vec{\lambda}} )\right>$$
where, to be clear, on the left hand side of the equation the inner
product is on the class functions from Equation~\eqref{eq:defspclass}
and on the right hand side the inner product is on the symmetric functions
from Equation~\eqref{eq:defspsym}. Hence $\phi_{\mathcal{U}_k}$ is an isometry with respect to the
inner products on the class functions of $\mathcal{U}_k$ and
the inner product on $\mathsf{Sym}^{\ast}_{\XXX,k}$.

Let $\mathbf{1}_{\mathcal{U}_r}$ denote the trivial character for $\mathcal{U}_r$
(this is the irreducible character indexed by $(\epart, \ldots, \epart, (1)) \in I_r$ with $r-1$ copies of $\epart$).
This is a class function with the property that $\mathbf{1}_{\mathcal{U}_r}(a) = 1$ for all $a \in \mathcal{U}_r$.
Then let
\begin{align}
E_r:&= E_r[X_1, X_2, \ldots, X_r] := \phi_{\mathcal{U}_r}(\mathbf{1}_{\mathcal{U}_r})\label{def:Ed}\\
&= \sum_{\vec{\mu} \in I_r} \frac{\mathbf{p}_{\vec{\mu}}[\XXX]}{\mathbf{z}_{\vec{\mu}}}\label{eq:Edpexp}\\
&= \sum_{(1^{a_1}2^{a_2}\cdots r^{a_r}) \vdash r}
s_{a_1}[X_1] s_{a_2}[X_2] \cdots s_{a_r}[X_r]~.\label{eq:Edsexp}
\end{align}

The symmetric function $E_r$ is the generating function
for the character values for the trivial representation of $\mathcal{U}_r$. It will serve as
a building block in a formula for the other the irreducible characters of $\mathcal{U}_r$.

The notation we have been using for symmetric functions can be extended
to allow substituting an expression in place of a set of variables.
This is called \defn{plethystic notation} and more details can be found in \cite{LR},
but for our purposes the following should suffice.
For an element $A \in \mathsf{Sym}^{\ast}_{\XXX}$, let $p_k[A]$ be the element
obtained by first expressing $A$ in the power sum basis and then
replacing each $p_r[X_i]$ appearing in the expression with $p_{kr}[X_i]$.
Since every symmetric function $f \in \mathsf{Sym}$ is a polynomial in the power sum elements
$p_1, p_2, p_3, \ldots$, we define $f[A]$ to be the element obtained from $f$ by
replacing each $p_i$ with $p_i[A]$.
This notation is consistent with the expressions we have used thus far once we identify $X_i$ and $p_1[X_i]$.

\begin{remark}
This notational extension is useful for
providing a generating function for
Equation~\eqref{eq:Edsexp}.  For any $r\geqslant0$ and
expressions $A$ and $B$, we have
\begin{equation*}
s_0[A] = 1,
\qquad
s_r[A+B] = \sum_{i=0}^r s_i[A] s_{r-i}[B],
\qquad
s_r[t^i X_i] = t^{ri} s_r[X_i].
\end{equation*}
The middle expression above is sometimes known in the literature as the
\emph{alphabet addition formula}. Applying these to expand the
expression below in $t$, we have
\[
    s_k[1 + t X_1 + t^2 X_2 + \cdots + t^k X_k]
    = 1 + E_1 t + E_2 t^2 + \cdots + E_k t^k + \cdots + s_k[X_k] t^{k^2}~.
\]
The coefficient of $t^r$ for $r>k$ in the expression on the
right hand side above are symmetric functions that are
not necessarily equal to $E_r$ since they will depend on both $r$ and $k$.
\end{remark}

We will use the shorthand notation
$$\mathbf{s}_{\vec{\lambda}}[\EEE] :=s_{\lambda^{(1)}}[E_1]s_{\lambda^{(2)}}[E_2]
\cdots s_{\lambda^{(k)}}[E_k]$$
and
$$\mathbf{p}_{\vec{\lambda}}[\EEE] :=p_{\lambda^{(1)}}[E_1]p_{\lambda^{(2)}}[E_2]
\cdots p_{\lambda^{(k)}}[E_k]~.$$
Then as a corollary to Equation~\eqref{eq:BinZ}, we have that the
coefficients $b_{\vec{\mu}}^{\vec{\nu}}$ are given by the following symmetric function expression.

\begin{cor}\label{cor:bscalar}
For $\vec{\mu}, \vec{\nu} \in I_k$,
\begin{equation}\label{eq:sfbcoef}
b_{\vec{\mu}}^{\vec{\nu}}
= \frac{1}{\mathbf{z}_{\vec{\nu}}}\left< \mathbf{p}_{\vec{\nu}}[\EEE], \mathbf{p}_{\vec{\mu}}[\XXX] \right>~.
\end{equation}
\end{cor}
\begin{proof}
This is found by expanding $\mathbf{p}_{\vec{\nu}}[\EEE]$
in the power sum basis and taking the coefficient of $\mathbf{p}_{\vec{\mu}}[\XXX]$
to show that it agrees with Equation~\eqref{eq:BinZ} using Equation~\eqref{eq:defspsym}.

Using Equation~\eqref{eq:Edpexp},
\begin{align}
\frac{1}{\mathbf{z}_{\vec{\nu}}} \prod_{a=1}^k p_{\nu^{(a)}}[E_a]
&=
\frac{1}{\mathbf{z}_{\vec{\nu}}} \prod_{j=1}^k
\prod_{i=1}^{\ell(\nu^{(j)})} p_{\nu^{(j)}_i}\left[
\sum_{\vec{\tau}(i,j) \in I_j}
\frac{\mathbf{p}_{\vec{\tau}(i,j)}[\XXX]}{\mathbf{z}_{\vec{\tau}(i,j)}}\right]\\
&=\label{eq:bigexp}
\frac{1}{\mathbf{z}_{\vec{\nu}}} \prod_{j=1}^k
\prod_{i=1}^{\ell(\nu^{(j)})}
\sum_{\vec{\tau}(i,j) \in I_j}
\frac{\mathbf{p}_{\nu^{(j)}_i\vec{\tau}(i,j)}[\XXX]}{\mathbf{z}_{\vec{\tau}(i,j)}},
\end{align}
where in the terms of the sum, we note that $\nu^{(j)}_i \in \ZZ_{>0}$
and the expression $\nu^{(j)}_i\vec{\tau}(i,j)$ is to be interpreted
as $( \nu^{(j)}_i \vec{\tau}(i,j)^{(1)}, \nu^{(j)}_i \vec{\tau}(i,j)^{(2)},
\ldots, \nu^{(j)}_i \vec{\tau}(i,j)^{(k)})$ where for a positive integer $a$
and a partition $\lambda \vdash r$ we have
$a \lambda = (a \lambda_1, a \lambda_2,
\ldots, a \lambda_{\ell(\lambda)})\vdash ar$.

The coefficient of $\frac{\mathbf{p}_{\vec{\mu}}[\XXX]}{\mathbf{z}_{\vec{\mu}}}$
is equal to the sum of the coefficients such that
for each $1 \leqslant a \leqslant k$,
\begin{equation}\label{eq:cond}
    \biguplus_{j=1}^k \biguplus_{i=1}^{\ell(\nu^{(j)})} \nu^{(j)}_i\vec{\tau}(i,j)^{(a)} = \vec{\mu}^{(a)}~.
\end{equation}
More specifically the coefficient of $\frac{p^\ast_{\vec{\mu}}[\XXX]}{\mathbf{z}_{\vec{\mu}}}$
in Equation~\eqref{eq:bigexp} is equal to
\begin{equation}\label{eq:coeffp}
    \frac{1}{\mathbf{z}_{\vec{\nu}}}
    \sum_{\vec{\tau}(\bullet,\bullet)}
    \frac{\mathbf{z}_{\vec{\mu}}}
    {\prod_{j=1}^k \prod_{i=1}^{\ell(\nu^{(j)})} \mathbf{z}_{\vec{\tau}(i,j)}},
\end{equation}
where the sum is over all sequences of partitions $\vec{\tau}(i,j) \in I_j$ for $1\leqslant j \leqslant k$,
$1 \leqslant i \leqslant \ell(\nu^{(j)})$ such that
Equation~\eqref{eq:cond} holds.
\end{proof}

As a consequence, we have the following symmetric function expression for
the character table of the uniform block permutation algebra $\CC\mathcal{U}_k$.
\begin{theorem}
\label{conj:frobchar}
Let $\chi^{\vec{\lambda}}_{\mathcal{U}_k}$ be the irreducible character of $\mathcal{U}_k$ indexed by $\vec{\lambda}\in I_k$.
For $\vec{\mu} \in I_k$, let $d_{\vec{\mu}} \in \mathcal{U}_k$ be any element
such that $\mathsf{cycletype}(d_{\vec{\mu}}) = \vec{\mu}$. Then
\begin{equation}\label{eq:sfexpr}
\chi^{\vec{\lambda}}_{\mathcal{U}_k}(d_{\vec{\mu}}) = \left< \mathbf{s}_{\vec{\lambda}}[\EEE], \mathbf{p}_{\vec{\mu}}[\XXX] \right>~.
\end{equation}
As a consequence,
\begin{equation}
\label{eq:frobchar}
    \phi_{\mathcal{U}_k}( \chi^{\vec{\lambda}}_{\mathcal{U}_k} )
    = \mathbf{s}_{\vec{\lambda}}[\EEE]~.
\end{equation}
\end{theorem}

\begin{proof}
Since characters are class functions, they are constant on generalized
conjugacy classes and so it suffices to prove the result for the conjugacy
class representatives $d_{\vec{\mu}}$ defined in Section~\ref{sec:conj-class-reps}.
By Equation~\eqref{eq:defspsym}, $\{ \mathbf{p}_{\vec{\nu}}[\XXX] \}_{\vec{\nu}}$
is an orthogonal basis, hence for any symmetric function alphabet $\mathbf{Y} = Y_1, Y_2, Y_3, \ldots$,
$$\mathbf{s}_{\vec{\lambda}}[\mathbf{Y}] = \sum_{\vec{\nu}}
\left< \mathbf{s}_{\vec{\lambda}}[\mathbf{Y}], \mathbf{p}_{\vec{\nu}}[\mathbf{Y}]\right>
\frac{\mathbf{p}_{\vec{\nu}}[\mathbf{Y}]}{\mathbf{z}_{\vec{\nu}}}~.
$$
We can expand
\begin{align}\label{eq:sfversionXeqCB}
\left< \mathbf{s}_{\vec{\lambda}}[\EEE], \mathbf{p}_{\vec{\mu}}[\XXX] \right>&=
\sum_{\vec{\nu}} \left< \mathbf{s}_{\vec{\lambda}}[\EEE],
\mathbf{p}_{\vec{\nu}}[\EEE] \right>
\left< \frac{\mathbf{p}_{\vec{\nu}}[\EEE]}{\mathbf{z}_{\vec{\nu}}},\mathbf{p}_{\vec{\mu}}[\XXX] \right>~.
\end{align}
This last expression is equal to $\chi^{\vec{\lambda}}_{\mathcal{U}_k}(d_{\vec{\mu}})$
by Equations \eqref{eq:sfbcoef}, \eqref{eq:spval} and \eqref{Eq:X=CB}.
Equation~\eqref{eq:frobchar} follows from~\eqref{equation.phi Uk} and the fact that for any $f$
$$f = \sum_{\vec{\nu}} \left< f, \mathbf{p}_{\vec{\nu}}[\XXX]\right>
\frac{\mathbf{p}_{\vec{\nu}}[\XXX]}{\mathbf{z}_{\vec{\nu}}}~.
$$
\end{proof}

Theorem~\ref{conj:frobchar} allows us to compute the character table for
$\mathcal{U}_k$ using symmetric function computations.
In order to write down the character table, we define a total order on the elements in $I_k$.  To do this we first define the \defn{reverse lexicographic order}
on partitions: we say \defn{$\lambda\leqslant_{\mathrm{rl}} \mu$} if we have $\lambda_i > \mu_i$ at the first index $i$ where $\lambda$ and $\mu$ differ.
This is a total order. For example, $(5) \leqslant_{\mathrm{rl}} (4,1) \leqslant_{\mathrm{rl}} (3,2) \leqslant_{\mathrm{rl}} (3,1,1) \leqslant_{\mathrm{rl}} (2,2,1) \leqslant_{\mathrm{rl}} (2,1,1,1) \leqslant_{\mathrm{rl}} (1,1,1,1,1)$.
Now for $\vec{\lambda},\vec{\mu}\in I_k$, we say \defn{$\vec{\lambda}< \vec{\mu}$} if
$\vtype(\vec{\lambda}) <_{\mathrm{rl}} \vtype(\vec{\mu})$ or if $\vtype(\vec{\lambda}) = \vtype(\vec{\mu})$ and there exists an
$1 \leqslant i \leqslant k$ such that $\lambda^{(j)} = \mu^{(j)}$ for all $1 \leqslant j < i$ and
$\lambda^{(i)} <_{\mathrm{rl}} \mu^{(i)}$.

\begin{example} \label{Ex:charactertable}
\label{ex:UBP2chartable}
The matrices are presented below with the rows and columns ordered from
smallest to largest from the top row of the matrix to the bottom.  The elements of
$I_k$ are presented compactly by dropping a layer of enclosing parentheses and commas.

The character table of $\mathcal{U}_2$ is
$$
\begin{array}{l|r|rr}
(\epart,1) & 1 & 1 & 1 \\
\hline
(2) & 0 & 1 & 1 \\
(11) & 0 & -1 & 1
\end{array}
$$

\label{ex:UBP3chartable}
The character table
of $\mathcal{U}_3$ is
$$
\begin{array}{l|r|r|rrr}
(\epart, \epart, 1) & 1 & 1 & 1 & 1 & 1 \\
\hline
(1, 1) & 0 & 1 & 0 & 1 & 3 \\
\hline
(3) & 0 & 0 & 1 & 1 & 1 \\
(21) & 0 & 0 & -1 & 0 & 2 \\
(111) & 0 & 0 & 1 & -1 & 1
\end{array}
$$

\label{ex:UBP4chartable}
The character table
of $\mathcal{U}_4$ is
$$
\begin{array}{l|r|r|rr|rr|rrrrr}
(\epart, \epart, \epart, 1) & 1 & 1 & 1 & 1 & 1 & 1 & 1 & 1 & 1 & 1 & 1 \\
\hline
(1, \epart, 1) &0 & 1 & 0 & 0 & 0 & 2 & 0 & 1 & 0 & 2 & 4 \\
\hline
(\epart, 2) &0 & 0 & 1 & 1 & 1 & 1 & 1 & 0 & 3 & 1 & 3 \\
(\epart, 11) &0 & 0 & -1 & 1 & 1 & 1 & -1 & 0 & -1 & 1 & 3 \\
\hline
(2, 1) &0 & 0 & 0 & 0 & 1 & 1 & 0 & 0 & 2 & 2 & 6 \\
(11, 1) &0 & 0 & 0 & 0 & -1 & 1 & 0 & 0 & -2 & 0 & 6 \\
\hline
(4) &0 & 0 & 0 & 0 & 0 & 0 & 1 & 1 & 1 & 1 & 1 \\
(31) &0 & 0 & 0 & 0 & 0 & 0 & -1 & 0 & -1 & 1 & 3 \\
(22) &0 & 0 & 0 & 0 & 0 & 0 & 0 & -1 & 2 & 0 & 2 \\
(211) &0 & 0 & 0 & 0 & 0 & 0 & 1 & 0 & -1 & -1 & 3 \\
(1111) &0 & 0 & 0 & 0 & 0 & 0 & -1 & 1 & 1 & -1 & 1
\end{array}
$$
\end{example}

\subsection{Factorizations of the character table of $\mathcal{U}_k$}

In \cite[Section 7]{Steinberg-Mobius}, Steinberg describes two factorizations of the character table for
finite inverse semigroups in terms of the character tables of its
maximal subgroups. We describe both of these factorizations here in
Proposition~\ref{prop:matrixAB} and Proposition \ref{prop:UAfactorization} below.

The first factorization uses an upper uni-triangular matrix $B_k$ with
non-negative integer entries. The general description for the entries in $B_k$ is discussed
in~\cite[Proposition 7.1]{Steinberg-Mobius} in which Steinberg remarks that computing this matrix is in
general a ``daunting task.'' In~\cite[Corollary 3.7]{Solomon}, Solomon computes this factorization for
the character table of the rook monoid. For the uniform block permutation monoid $\mathcal{U}_k$,
our formula for the entries of $B_k$ in terms of symmetric functions
will follow from the results in Section~\ref{sec:zeeformula}.

The second factorization uses a different upper uni-triangular with
non-negative integer entries $U_k$.  Its entries are the multiplicities of the
irreducible representations of the maximal subgroups when we restrict an
irreducible representation of $\mathcal{U}_k$ to its maximal subgroups.
Our interest in this matrix arises because of a relation with the operation
of plethysm described in Corollary \ref{cor:plethmult}.

Throughout this section we assume that $I_k$ is totally ordered as in Section~\ref{subsec:UkFrob}.
This order
satisfies the condition that if
$\{e_{\pi_\lambda}\, |\, \lambda \vdash k\}$ are the idempotent
representatives for the $\mathscr{J}$-classes of $\mathcal{U}_k$, then
$\mathcal{U}_k e_{\pi_\mu} \mathcal{U}_k \subseteq \mathcal{U}_k e_{\pi_\nu} \mathcal{U}_k$
implies $\mu\leqslant_{\mathrm{rl}} \nu$. In particular, the largest element is $\mu = (1^k)$
since $e_{\pi_{(1^k)}}$ is the identity element of $\mathcal{U}_k$,
and $\nu = (k)$ is the smallest element since $e_{\pi_{(k)}}$ has one block and
$e_{\pi_{(k)}}$ is the only element in $\mathcal{U}_k e_{\pi_{(k)}} \mathcal{U}_k$.

Let $X_k$ be the character table of $\mathcal{U}_k$, which we view as a matrix,
and whose entries are denoted by $X_{\vec{\lambda}, \vec{\mu}}$.
The following result summarizes
the properties of $X_k$ proved in the previous section. They
will be used to factor $X_k$ as the product of two matrices.

\begin{prop} \label{prop:matrixX}
Let $X_k = \big( X_{ \vec{\lambda}, \vec{\mu}} \big)_{\vec{\lambda}, \vec{\mu}\in I_k}$
be the character table of $\mathcal{U}_k$ viewed as a matrix. Then
\[
X_{ \vec{\lambda}, \vec{\mu}} =
\chi^{\vec{\lambda}}_{\mathcal{U}_k}(d_{\vec{\mu}})
= \left< \mathbf{s}_{\vec{\lambda}}[\EEE], \mathbf{p}_{\vec{\mu}}[\XXX] \right>
\]
and $X_k$ is upper block diagonal with respect to the total order on $I_k$
defined in Section~\ref{subsec:UkFrob}.
\end{prop}

Define $A_k= \big( A_{ \vec{\lambda}, \vec{\mu}} \big)_{\vec{\lambda}, \vec{\mu}\in I_k}$
to be the block diagonal matrix whose diagonal blocks are the character
tables of the maximal subgroups of $\mathcal{U}_k$;
explicitly,
$A_{\vec{\lambda}, \vec{\mu}} = 0$ if $\vtype(\vec{\lambda}) \neq \vtype(\vec{\mu})$,
and otherwise
$A_{\vec{\lambda}, \vec{\mu}} = \chi^{\vec{\lambda}}_{G_\lambda}(d_{\vec{\mu}})$,
where $\lambda = \vtype(\vec{\lambda})$.
By Equation~\eqref{eq:spval},
\[
A_{\vec{\lambda}, \vec{\mu}} = \chi^{\vec{\lambda}}_{G_\lambda}( d_{\vec{\mu}} )
= \left< \mathbf{s}_{\vec{\lambda}}[\XXX], \mathbf{p}_{\vec{\mu}}[\XXX] \right>~.
\]

Define a second square matrix $B_k  = \big( B_{ \vec{\lambda}, \vec{\mu}} \big)_{\vec{\lambda}, \vec{\mu}\in I_k}$, with the entries from Corollary~\ref{cor:bscalar},
\[
B_{\vec{\mu},\vec{\nu}} = b_{\vec{\mu}}^{\vec{\nu}} =
\left< \frac{\mathbf{p}_{\vec{\nu}}[\EEE]}{\mathbf{z}_{\vec{\nu}}},
\mathbf{p}_{\vec{\mu}}[\XXX]\right>.
\]

\begin{prop}\label{prop:matrixAB}
The matrix $A_k$ is block diagonal,
$B_k$ is upper uni-triangular
with non-negative integer entries, and
\[
X_k = A_k \cdot B_k~.
\]
\end{prop}
\begin{proof}
The statement that
$X_k = A_k \cdot B_k$ is a restatement of Equations~\eqref{Eq:X=CB} and~\eqref{eq:sfversionXeqCB}.

If $\vtype(\vec{\nu}) = \vtype(\vec{\mu})$, then
$b_{\vec{\mu}}^{\vec{\nu}}$ is the number of ways
of merging parts of $d_{\vec{\mu}}$ to obtain an element
of cycle type $\vec{\nu}$.  There is of course one way to do this
if $\vec{\nu} = \vec{\mu}$ and zero ways
if $\vtype(\vec{\mu}) <_{\mathrm{rl}} \vtype(\vec{\nu})$.
\end{proof}

\begin{example} \label{Ex:firstfactorization}
If $k = 2$, the character table for the maximal
subgroups in block diagonal form and the matrix
$B$ are
$$
A_2 = \begin{array}{l|r|rr}
(\epart,1) & 1 & 0 & 0 \\
\hline
(2) & 0 & 1 & 1 \\
(11) & 0 & -1 & 1
\end{array}
\qquad\qquad\qquad\qquad
B_2 = \begin{array}{l|r|rr}
(\epart,1) & 1 & 1 & 1 \\
\hline
(2) & 0 & 1 & 0 \\
(11) & 0 & 0 & 1
\end{array}~.
$$

If $k = 3$, the character table for the maximal
subgroups in block diagonal form and the matrix
$B$ are
$$
A_3 = \begin{array}{l|r|r|rrr}
(\epart, \epart, 1) & 1 & 0 & 0 & 0 & 0 \\
\hline
(1, 1) & 0 & 1 & 0 & 0 & 0 \\
\hline
(3) & 0 & 0 & 1 & 1 & 1 \\
(21) & 0 & 0 & -1 & 0 & 2 \\
(111) & 0 & 0 & 1 & -1 & 1
\end{array}
\qquad\qquad\qquad\qquad
B_3 = \begin{array}{l|r|r|rrr}
(\epart, \epart, 1) & 1 & 1 & 1 & 1 & 1 \\
\hline
(1, 1) & 0 & 1 & 0 & 1 & 3 \\
\hline
(3) & 0 & 0 & 1 & 0 & 0 \\
(21) & 0 & 0 & 0 & 1 & 0 \\
(111) & 0 & 0 & 0 & 0 & 1
\end{array}~.
$$

If $k = 4$, the character table for the maximal
subgroups in block diagonal form is
$$
A_4 = \begin{array}{l|r|r|rr|rr|rrrrr}
(\epart, \epart, \epart, 1) & 1 & 0 & 0 & 0 & 0 & 0 & 0 & 0 & 0 & 0 & 0 \\
\hline
(1, \epart, 1) &0 & 1 & 0 & 0 & 0 & 0 & 0 & 0 & 0 & 0 & 0 \\
\hline
(\epart, 2) &0 & 0 & 1 & 1 & 0 & 0 & 0 & 0 & 0 & 0 & 0 \\
(\epart, 11) &0 & 0 & -1 & 1 & 0 & 0 & 0 & 0 & 0 & 0 & 0 \\
\hline
(2, 1) &0 & 0 & 0 & 0 & 1 & 1 & 0 & 0 & 0 & 0 & 0 \\
(11, 1) &0 & 0 & 0 & 0 & -1 & 1 & 0 & 0 & 0 & 0 & 0 \\
\hline
(4) &0 & 0 & 0 & 0 & 0 & 0 & 1 & 1 & 1 & 1 & 1 \\
(31) &0 & 0 & 0 & 0 & 0 & 0 & -1 & 0 & -1 & 1 & 3 \\
(22) &0 & 0 & 0 & 0 & 0 & 0 & 0 & -1 & 2 & 0 & 2 \\
(211) &0 & 0 & 0 & 0 & 0 & 0 & 1 & 0 & -1 & -1 & 3 \\
(1111) &0 & 0 & 0 & 0 & 0 & 0 & -1 & 1 & 1 & -1 & 1
\end{array}
$$
and the matrix
$B_4$ of values $b^{\vec{\lambda}}_{\vec{\mu}}$ is
$$
B_4 = \begin{array}{l|r|r|rr|rr|rrrrr}
(\epart, \epart, \epart, 1) & 1 & 1 & 1 & 1 & 1 & 1 & 1 & 1 & 1 & 1 & 1 \\
\hline
(1, \epart, 1) &0 & 1 & 0 & 0 & 0 & 2 & 0 & 1 & 0 & 2 & 4 \\
\hline
(\epart, 2) &0 & 0 & 1 & 0 & 0 & 0 & 1 & 0 & 2 & 0 & 0 \\
(\epart, 11) &0 & 0 & 0 & 1 & 1 & 1 & 0 & 0 & 1 & 1 & 3 \\
\hline
(2, 1) &0 & 0 & 0 & 0 & 1 & 0 & 0 & 0 & 2 & 1 & 0 \\
(11, 1) &0 & 0 & 0 & 0 & 0 & 1 & 0 & 0 & 0 & 1 & 6 \\
\hline
(4) &0 & 0 & 0 & 0 & 0 & 0 & 1 & 0 & 0 & 0 & 0 \\
(31) &0 & 0 & 0 & 0 & 0 & 0 & 0 & 1 & 0 & 0 & 0 \\
(22) &0 & 0 & 0 & 0 & 0 & 0 & 0 & 0 & 1 & 0 & 0 \\
(211) &0 & 0 & 0 & 0 & 0 & 0 & 0 & 0 & 0 & 1 & 0 \\
(1111) &0 & 0 & 0 & 0 & 0 & 0 & 0 & 0 & 0 & 0 & 1
\end{array}~.
$$
\end{example}

The second factorization arises from the isomorphism in
Equation~\eqref{restrictionmap}, which is induced by restricting class
functions of $\mathcal{U}_k$ to the maximal subgroups $G_\lambda$.
By~\cite[Theorem 6.5]{Steinberg.2016}, the matrix corresponding to the restriction isomorphism,
which we denote by $U_k$, is upper triangular with 1s on the diagonal. Since $\CC\mathcal{U}_k$
is semisimple, the entries of the matrix $U_k$ are the multiplicities of the
irreducible representations of the maximal subgroups when we restrict an
irreducible representation of $\mathcal{U}_k$ to the maximal subgroups.
Thus, $U_k$ is sometimes known as the \defn{decomposition matrix}.

\begin{prop}\label{prop:UAfactorization}
Define the matrix
$U_k = \big( U_{ \vec{\lambda}, \vec{\mu}} \big)_{\vec{\lambda}, \vec{\mu}\in I_k}$
by
\[
U_{ \vec{\lambda}, \vec{\mu}} = \left<
\mathbf{s}_{\vec{\lambda}}[\EEE], \mathbf{s}_{\vec{\mu}}[\XXX] \right>~.
\]
Then $U_k$ is upper uni-triangular with non-negative integer entries, and
\[
X_k = U_k \cdot A_k~,
\]
where $X_k$ is the character table of $\mathcal{U}_k$ (see Proposition~\ref{eq:scalar-product-with-power-sums}),
and $A_k$ is the block diagonal matrix whose blocks
are the character tables of the maximal subgroups (see Proposition~\ref{prop:matrixAB}).
\end{prop}

\begin{proof}
For $\vec{\lambda}, \vec{\nu} \in I_k$ and $\nu = \vtype(\vec{\nu})$,
\[
U_{\vec{\lambda}, \vec{\nu}} =
\left< \phi_{\mathcal{U}_k}(\chi^{\vec{\lambda}}_{\mathcal{U}_k}),
\phi_{G_\nu}(\chi^{\vec{\nu}}_{G_\nu}) \right>
=
\left< \mathbf{s}_{\vec{\lambda}}[\EEE], \mathbf{s}_{\vec{\nu}}[\XXX] \right>~.
\]
By Proposition~\ref{prop:multiplicity},
the entries of this matrix are multiplicities of irreducible representations in a
restriction and hence they are non-negative integers; more precisely,
$U_{\vec{\lambda},\vec{\nu}}$ is the multiplicity of the irreducible
$G_\nu$-representation $V^{\vec{\nu}}_{G_\nu}$ in the restriction of
the irreducible $\mathcal{U}_k$-representation
$W^{\vec{\lambda}}_{\mathcal{U}_k}$ to the maximal subgroup $G_{\nu}$.

The factorization $X_k = U_k \cdot A_k$ is a consequence of the fact
that $\{ \mathbf{s}_{\vec{\nu}}[\XXX] \}_{\vec{\nu} \in I_k}$ is an orthonormal
basis of $\mathsf{Sym}^{\ast}_{\XXX,k}$, so that
\[
\left< \mathbf{s}_{\vec{\lambda}}[\EEE], \mathbf{p}_{\vec{\mu}}[\XXX] \right> =
\sum_{\vec{\nu} \in I_k}
\left< \mathbf{s}_{\vec{\lambda}}[\EEE], \mathbf{s}_{\vec{\nu}}[\XXX] \right>
\left< \mathbf{s}_{\vec{\nu}}[\XXX], \mathbf{p}_{\vec{\mu}}[\XXX] \right>~.
\]

Now if we examine the expansion of
$\mathbf{s}_{\vec{\lambda}}[\EEE]$,
then if $\vtype(\vec{\nu}) <_{\mathrm{rl}} \vtype(\vec{\lambda})$, there
exists an $r$ such that the multiplicity of $r$ in $\vtype(\vec{\nu})$
is $a>0$ and in $\vtype(\vec{\lambda})$ it is smaller than $a$.
This implies that the degree in $X_r$ in
the symmetric function $\phi_{G_\mu}(\chi^{\vec{\mu}}_{G_\mu})$ is
$a$ but that all terms in
$\mathbf{s}_{\vec{\lambda}}[\EEE] = \phi_{\mathcal{U}_k}(\chi^{\vec{\lambda}}_{\mathcal{U}_k})$
have degree in $X_r$ smaller than $a$ and hence $U_{\vec{\lambda},\vec{\nu}} = 0$.

Note that $\mathbf{s}_{\vec{\lambda}}[\EEE]$ is equal to $\mathbf{s}_{\vec{\lambda}}[\XXX]$ plus
terms that are of not of the same degree in the same variables
as $\mathbf{s}_{\vec{\lambda}}[\XXX]$.  Therefore,
$U_{\vec{\lambda},\vec{\lambda}} = 1$ and
$U_{\vec{\lambda},\vec{\nu}} = 0$ if $\vec{\lambda} \neq \vec{\nu}$.
We conclude that $U_k$ is upper uni-triangular.
\end{proof}

\begin{example}  
For $k=2,3,4,$ the matrices corresponding to the multiplicities of an irreducible 
representation in the restriction from the uniform block permutation algebra
to the maximal subgroups are given by
$$
U_2 =
\begin{array}{l|r|rr}
(\epart, 1) & 1 & 1 & 0 \\
\hline
(2) & 0 & 1 & 0 \\
(11) & 0 & 0 & 1
\end{array}~,
$$
$$
U_3 = \begin{array}{l|r|r|rrr}
(\epart, \epart, 1) & 1 & 1 & 1 & 0 & 0 \\
\hline
(1, 1) & 0 & 1 & 1 & 1 & 0 \\
\hline
(3) & 0 & 0 & 1 & 0 & 0 \\
(21) & 0 & 0 & 0 & 1 & 0 \\
(111) & 0 & 0 & 0 & 0 & 1
\end{array}
$$
and
$$
U_4 = \begin{array}{l|r|r|rr|rr|rrrrr}
(\epart, \epart, \epart, 1) & 1 & 1 & 1 & 0 & 1 & 0 & 1 & 0 & 0 & 0 & 0 \\
\hline
(1, \epart, 1) &0 & 1 & 0 & 0 & 1 & 1 & 1 & 1 & 0 & 0 & 0 \\
\hline
(\epart, 2) &0 & 0 & 1 & 0 & 1 & 0 & 1 & 0 & 1 & 0 & 0 \\
(\epart, 11) &0 & 0 & 0 & 1 & 1 & 0 & 0 & 1 & 0 & 0 & 0 \\
\hline
(2, 1) &0 & 0 & 0 & 0 & 1 & 0 & 1 & 1 & 1 & 0 & 0 \\
(11, 1) &0 & 0 & 0 & 0 & 0 & 1 & 0 & 1 & 0 & 1 & 0 \\
\hline
(4) &0 & 0 & 0 & 0 & 0 & 0 & 1 & 0 & 0 & 0 & 0 \\
(31) &0 & 0 & 0 & 0 & 0 & 0 & 0 & 1 & 0 & 0 & 0 \\
(22) &0 & 0 & 0 & 0 & 0 & 0 & 0 & 0 & 1 & 0 & 0 \\
(211) &0 & 0 & 0 & 0 & 0 & 0 & 0 & 0 & 0 & 1 & 0 \\
(1111) &0 & 0 & 0 & 0 & 0 & 0 & 0 & 0 & 0 & 0 & 1
\end{array}
$$
To dispel the impression that the entries $U_{\vec{\lambda},\vec{\mu}}$ are always $0$ or $1$,
we note that
\begin{equation*}
    U_{(\epart,(1),(1)), ((3),(1))} = \left< s_1[E_2] s_1[E_3], s_3[X_1] s_1[X_2] \right> = 2.
\end{equation*}
\end{example}

To indicate the importance of the decomposition matrix of $\mathcal{U}_k$,
we note in the following result that some of the entries of the matrix $U_k$
correspond to the Schur expansion of certain
symmetric function expressions involving plethysm.
One objective of this research is to give a description of
the decomposition matrix in order to provide an interpretation
of these coefficients.

\begin{cor}\label{cor:plethmult}
For $\mu \vdash k$ and $\vec{\lambda} \in I_k$,
the multiplicity of the irreducible $\mathfrak{S}_k$-module
$V^\mu_{\mathfrak{S}_k}$ in the restriction of
the irreducible $\mathcal{U}_k$-module $W^{\vec{\lambda}}_{\mathcal{U}_k}$
to $\mathfrak{S}_k$ is equal to
\begin{equation}\label{eq:multiplicity}
\big\langle s_{\lambda^{(1)}}[s_1]s_{\lambda^{(2)}}[s_2]
\cdots s_{\lambda^{(k)}}[s_k], \, s_\mu \big\rangle~.
\end{equation}
\end{cor}
\begin{proof}
Since $G_{(1^k)} = \mathfrak{S}_k$,
Proposition \ref{prop:multiplicity} states that the multiplicity of
$V^{\mu}_{\mathfrak{S}_k}$ in
${\rm Res}^{\mathcal{U}_k}_{\mathfrak{S}_k} W^{\vec{\lambda}}_{\mathcal{U}_k}$ is
\begin{equation}\label{eq:sfmultiplicity}
\Big\langle \phi_{\mathcal{U}_k}(\chi^{\vec{\lambda}}_{\mathcal{U}_k}),
\phi_{G_{(1^k)}}(\chi^{\mu}_{\mathfrak{S}_k}) \Big\rangle =
\Big\langle \mathbf{s}_{\vec{\lambda}}[\EEE], s_{\mu}[X_1] \Big\rangle~.
\end{equation}
Now $\mathbf{s}_{\vec{\lambda}}[\EEE]$ has symmetric functions involving
the alphabets $X_2, X_3, \ldots, X_k$ while $s_{\mu}[X_1]$ does
not.  Thus, the value in Equation~\eqref{eq:sfmultiplicity} is not changed if we set each of those alphabets
equal to $0$ in $\mathbf{s}_{\vec{\lambda}}[\EEE]$.
From Equation~\eqref{eq:Edsexp}, we know
\[
E_r|_{X_2=X_3= \cdots =X_k=0} = s_r[X_1],
\]
and therefore the right hand side of Equation~\eqref{eq:sfmultiplicity} is equal to
\[
\big\langle s_{\lambda^{(1)}}[s_1[X_1]]
s_{\lambda^{(2)}}[s_2[X_1]]
\cdots
s_{\lambda^{(k)}}[s_k[X_1]], \, s_{\mu}[X_1] \big\rangle,
\]
which is the same as Equation~\eqref{eq:multiplicity}
upon dropping the reference to the variables $X_1$.
\end{proof}

\bibliographystyle{plain}
\bibliography{main}{}

\end{document}